%% file: adaptiveSQP_techreport.tex
\documentclass[11pt]{article}
\usepackage{paper,asb}

\usepackage{amsmath}
\usepackage{amsfonts}
\usepackage{amssymb}
\usepackage{amsthm}
\usepackage{hyperref}

\newtheorem{theorem}{Theorem}
\newtheorem{lemma}[theorem]{Lemma}
\newtheorem{assumption}[theorem]{Assumption}
\newtheorem{remark}[theorem]{Remark}

\newtheorem{condition}[theorem]{Condition}

\usepackage{subcaption}
\newcommand{\asb}[1]{{\color{black}#1}}
\newcommand{\bz}[1]{{\color{black}#1}}

\newcommand{\change}[1]{{\color{black}#1}}

\newcommand{\PAISSQP}{\texttt{PAIS-SQP}}

\newcommand{\papertitle}{An Adaptive Sampling Sequential Quadratic Programming Method for Equality Constrained Stochastic Optimization}
\newcommand{\paperauthora}{Albert~S.~Berahas}
\newcommand{\paperauthoraaffiliation}{Dept. of Industrial and Operations Engineering, University of Michigan}
\newcommand{\paperauthoraemail}{albertberahas@gmail.com}
\newcommand{\paperauthorb}{Raghu~Bollapragada}
\newcommand{\paperauthorbaffiliation}{Operations Research and Industrial Engineering Prog., UT Austin}
\newcommand{\paperauthorbemail}{raghu.bollapragada@utexas.edu}
\newcommand{\paperauthorc}{Baoyu~Zhou}
\newcommand{\paperauthorcaffiliation}{Booth School of Business, University of Chicago}
\newcommand{\paperauthorcemail}{baoyu.zhou@chicagobooth.edu}

\usepackage{booktabs}
\usepackage{caption}
\usepackage{float}
\usepackage{titlesec}
\usepackage{capt-of}

\usepackage{array}
\usepackage{arydshln}
\begin{document}
\title{\papertitle}

\author{\paperauthora\footnotemark[1]\ \footnotemark[2]
   \and \paperauthorb\footnotemark[3]
   \and \paperauthorc\footnotemark[4]}

\maketitle

\renewcommand{\thefootnote}{\fnsymbol{footnote}}
\footnotetext[1]{Corresponding author.}
\footnotetext[2]{\paperauthoraaffiliation. (\url{\paperauthoraemail})}
\footnotetext[3]{\paperauthorbaffiliation. (\url{\paperauthorbemail})}
\footnotetext[4]{\paperauthorcaffiliation. (\url{\paperauthorcemail})}
\renewcommand{\thefootnote}{\arabic{footnote}}

\begin{abstract}{
\input{abstract.tex}}
\end{abstract}

\numberwithin{equation}{section}
\numberwithin{theorem}{section}

\input{body.tex}

\section*{Acknowledgements}
\input{acknowledgements}

\newpage
\bibliographystyle{plain}
\bibliography{references}

\newpage
\appendix
\input{appendix.tex}

\end{document}

%% file: abstract.tex
This paper presents a methodology for using varying sample sizes in sequential quadratic programming (SQP) methods for solving equality constrained stochastic optimization problems. The first part of the paper deals with the delicate issue of dynamic sample selection in the evaluation of the gradient in conjunction with inexact solutions to the SQP subproblems. Under reasonable assumptions on the quality of the employed gradient approximations and the accuracy of the solutions to the SQP subproblems, we establish global convergence results for the proposed method. Motivated by these results, the second part of the paper describes a practical adaptive inexact stochastic sequential quadratic programming (PAIS-SQP) method. We propose criteria for controlling the sample size and the accuracy in the solutions of the SQP subproblems based on estimates of the variance in the stochastic gradient approximations  obtained as the optimization progresses. Finally, we demonstrate the performance of the practical method on a subset of the CUTE problems and constrained classification tasks.  

%% file: body.tex

\section{Introduction}\label{sec.introduction}

We consider stochastic optimization problems with deterministic equality constraints. We design, analyze and implement a stochastic adaptive sampling algorithm based on the sequential quadratic programming (SQP) paradigm. Optimization problems of this form arise in a plethora of real-world applications, including but not limited to computer vision \cite{RaviDinhLokhSing19}, optimal control \cite{Bett10}, network optimization \cite{Bert98}, partial differential equation optimization \cite{ReesDollWath10}
, deep learning \cite{zhu2019physics}, and reinforcement learning \cite{AchiHeldTamaAbbe17}.

The majority of the methods developed for solving constrained stochastic optimization problems, i.e., optimization problems with deterministic constraints and stochastic objective functions, are based on the penalty method approach \cite{CottGuptPfei16,MahdYangJinZhuYi12,NandPathSingSing19,RaviDinhLokhSing19}. Such methods transform the given constrained problem into an unconstrained problem by penalizing constraint violation in the objective function, and then apply classical unconstrained stochastic optimization methods on the penalized objective. While these methods are well studied, recently a number of algorithms endowed with sound theoretical guarantees and superior empirical performance have been proposed \cite{BeraCurtRobiZhou21,NaAnitKola21}. In \cite{BeraCurtRobiZhou21}, a stochastic SQP method with adaptive step size selection for the fully stochastic regime is proposed, and, in \cite{NaAnitKola21}, a stochastic line search SQP method with adaptive gradient accuracy is proposed. Several other extensions of these methods have been developed, e.g., relaxing constraint assumptions \cite{BeraCurtOneiRobi21}, employing inexact computations  \cite{CurtRobiZhou21} and employing variance reduction \cite{berahas2022accelerating}. 

Adaptive sampling is a powerful technique used in stochastic optimization to control the variance in the approximations employed as the optimization progresses. The idea is simple yet powerful; far from the solution inaccurate and cheap (gradient) information can be employed while near the solution accurate (gradient) information is required for both theory and practice. Of course, the key to such methods is the mechanism by which the sample size (or accuracy of the approximation) is selected. In \cite{FrieSchm12, ByrdChinNoceWu12}, algorithms that increase the samples sizes employed with prescribed (geometric) rules are proposed, and in \cite{ByrdChinNoceWu12} the authors showed that these methods achieve optimal worst-case first-order complexity for unconstrained problems. Moreover, these adaptive methods have additional advantages over fixed sample stochastic approximation approaches with regards to ability to exploit parallelism and produce more stable (less variance) iterates due to the increasing sample sizes. Other algorithms utilize gradient approximation tests to control the accuracy in the approximations; e.g., norm test \cite{carter1991global,ByrdChinNoceWu12}, inner product test \cite{BollByrdNoce18,BollNoceMudiShiTang18}, and others \cite{CartSche18, berahas2021global,JinScheXie21}. Adaptive sampling algorithms have also been developed for simulation-based optimization problems \cite{HashGhosPasu14,PasuGlynGhosHash18}. Finally,  \cite{BeisKeitUrbaWohl20,XieBollByrdNoce20} apply adaptive sampling methods to constrained stochastic optimization problems with convex feasible sets. For a detailed review of adaptive sampling methods  see~\cite{CurtSche20}.

\subsection{Contributions}
In this paper, motivated by the successes of adaptive sampling methods (unconstrained stochastic optimization problems) and SQP methods (constrained deterministic optimization problems), we propose an adaptive sampling stochastic sequential quadratic programming algorithm for solving optimization problems with deterministic constraints and stochastic objective functions. This is far from a trivial task as complication arises when incorporating existing adaptive sampling strategies into the SQP paradigm primarily due to the multi-objective nature of constrained optimization. We propose a novel mechanism to control the accuracy in gradient approximations employed and the search directions calculated that balances the goals of achieving feasibility and reducing the objective function value. The main ingredients of our proposed method are: (1) the extension of the norm condition \cite{carter1991global,ByrdChinNoceWu12} for equality constrained stochastic optimization problems, and (2) the adaptation of the stochastic SQP method proposed in \cite{BeraCurtRobiZhou21}. 

We prove convergence guarantees for two different inexactness conditions, predetermined and adaptive. Specifically, under standard  stochastic conditions on the accuracy of the gradient approximations and deterministic conditions on the quality of the solutions of the linear system, we prove that a measure of first-order stationarity evaluated at the iterates generated by our proposed algorithm converges to zero in expectation from arbitrary starting points. In addition, we establish improved iteration complexity, $\mathcal{O}(\epsilon^{-2})$ versus $\mathcal{O}(\epsilon^{-4})$, as compared to the stochastic SQP method \cite{CurtOneiRobi21}. We also derive worst-case sample complexity results when the sample sizes and linear system solutions are controlled using predetermined rates instead of adaptive tests. Our results show that even though the work per iteration is increasing, the overall sample complexity is still $\mathcal{O}(\epsilon^{-2(1 + \nu)})$, for any $\nu>1$, which is arbitrarily close to $\mathcal{O}(\epsilon^{-4})$ achieved by the stochastic SQP method \cite{CurtOneiRobi21}. Moreover, we then propose a practical variant of our algorithm, inspired by \cite{ByrdChinNoceWu12,BollByrdNoce18}. Finally, we present numerical results on binary classification tasks with equality constraints and on equality constrained problems from the CUTE collection that demonstrate the efficiency and efficacy of our proposed practical method.


\subsection{Notation}

The set of natural numbers is denoted by $\N{}:= \{0, 1, 2, \dots \}$. The set of real numbers (i.e., scalars) is denoted by $\R{}$, and $\R{}_{>r}$ ($\R{}_{\geq r}$) denotes the set of real numbers greater than (greater than or equal to) $r \in \R{}$. The set of $n$-dimensional vectors is denoted by $\R{n}$, the set of $m$-by-$n$ matrices is denoted by $\R{m \times n}$, and the set of $n$-by-$n$ symmetric matrices is denoted by $\mathbb{S}^n$. Our proposed algorithms are iterative, and generate a sequence of iterates $\{x_k\}$ with $x_k \in \R{n}$. Let $f_k := f(x_k)$, $g_k := \nabla f(x_k)$, $c_k := c(x_k)$, and $J_k := \nabla c(x_k)^T$ for all $k \in \N{}$.

\subsection{Organization}

The paper is organized as follows. In Section~\ref{sec.prob_and_ass} we formalize the problem statement and main assumptions. The general algorithmic framework of our proposed method is presented in Section~\ref{sec.gen_algs}. Convergence and complexity guarantees are stated and proven in Section~\ref{sec.analysis}. We present a practical adaptive sampling SQP method in Section~\ref{sec.practical}. In Section~\ref{sec.num_res}, we demonstrate the empirical performance of of the proposed method on classification tasks. Finally, in Section~\ref{sec.final_rem}, we make some final remarks and discuss avenues for future research.

\section{Problem Statement}\label{sec.prob_and_ass}

We consider the following potentially nonlinear and/or nonconvex equality constrained optimization problem
\bequation\label{prob.opt}
  \min_{x\in\R{n}}\ f(x)\ \ \st\ \ c(x) = 0,\ \ \text{with}\ \ f(x) = \E[F(x,\xi)],
\eequation
where the objective function $f : \R{n} \to \R{}$ and constraint function $c : \R{n} \to \R{m}$ are smooth, $\xi$ is a random variable with associated probability space $(\Xi,\Fcal,P)$, $F : \R{n} \times \Xi \to \R{}$, and $\E[\cdot]$ denotes the expectation taken with respect to~$P$. Throughout the paper, we assume that the constraint function and its associated  
derivatives can be computed exactly, and that the objective function and its associated derivatives are expensive to compute, but  accurate approximations can be obtained as desired. 
We formalize the notions of accuracy in subsequent sections of the paper. 

We make the following main assumption with regards to \eqref{prob.opt} and the iterates $\{x_k\}$ generated by our proposed algorithms.
\bassumption\label{ass.main}
  Let $\Xcal \subseteq \R{n}$ be an open convex set containing the sequence~$\{x_k\}$ generated by any run of the algorithm.  The objective function $f : \R{n} \to \R{}$ is continuously differentiable and bounded below over~$\Xcal$ and its gradient function $\nabla f : \R{n} \to \R{n}$ is Lipschitz continuous with constant $L \in \R{}_{>0}$  
  and bounded over $\Xcal$.  The constraint function $c : \R{n} \to \R{m}$ $($with $m \leq n)$ is continuously differentiable and bounded over $\Xcal$ and its Jacobian function $J := \nabla c^T : \R{n} \to \R{m \times n}$ is Lipschitz continuous with constant $\Gamma \in \R{}_{>0}$ 
  and bounded over~$\Xcal$.  In addition, for all $x \in \Xcal$, the Jacobian $J(x)$ has singular values that are bounded uniformly below by $\kappa_\sigma \in \R{}_{>0}$.
\eassumption
\bremark 
Under Assumption~\ref{ass.main}, there exist constants $(f_{\inf},\kappa_{g},\kappa_c,\kappa_\sigma,\kappa_J) \in \R{} \times \R{}_{>0} \times \R{}_{>0} \times \R{}_{>0}\times \R{}_{>0}$ such that, for all $k \in \N{}$,
\bequationNN
  f_{\inf} \leq f_k ,\    \|g_k\|_2 \leq \kappa_{g}, \    \|c_k\|_1 \leq \kappa_c, \    \|J_k\|_2 \leq \kappa_J, \ \text{and } \   \|(J_kJ_k^T)^{-1}\|_2\leq \kappa_{\sigma}^{-2}.
\eequationNN
The components pertaining to the objective and constraint functions are standard assumptions in the equality constrained optimization literature. With regards to the algorithmic components, we do not assume that the iterate sequence itself is bounded, however, we do assume that the objective function and constraints, function values and derivatives, are 
bounded over the set $\mathcal{X}$ containing the iterates. While this assumption is reasonable in the deterministic setting, it is not ideal in the stochastic setting. That being said, it is a common assumption in the equality constrained stochastic optimization literature \cite{BeraCurtRobiZhou21,NaAnitKola21}. The justification is that in this constrained setting, iterates are presumably converging towards a determinsitic feasible region. Furthermore, we assume that the accuracy in the gradient approximations can be controlled, and as such  claim that the assumption above is reasonable.
\eremark

Let $\ell : \R{n} \times \R{m} \to \R{}$ be the Lagrangian function corresponding to \eqref{prob.opt},  $\ell(x,y) = f(x) + c(x)^Ty$,
where $y \in \R{m}$ is the vector of Lagrange multipliers. Under Assumption~\ref{ass.main}, the necessary conditions for first-order stationarity for \eqref{prob.opt} are
\bequationNN
  0 = \bbmatrix \nabla_x \ell(x,y) \\ \nabla_y \ell(x,y) \ebmatrix = \bbmatrix \nabla f(x) + J(x)^Ty \\ c(x) \ebmatrix.
\eequationNN

\section{A General Algorithmic Framework}\label{sec.gen_algs}

Our proposed algorithms can be characterized as \emph{adaptive} (selection of the step size and merit parameter/accuracy of approximations), \emph{inexact} (linear system solutions), and \emph{stochastic} (gradient approximation employed) sequential quadratic optimization (SQP)
methods. Specifically, given $x_k$ for all $k \in \N{}$, a search direction $\dbar_k \in \R{n}$ is computed by inexactly solving a quadratic optimization subproblem based on a local quadratic model of the objective function 
\bequation\label{eq.subprob}
  \min_{d\in\R{n}}\ f_k + \gbar_k^Td + \thalf d^TH_kd\ \ \st\ \ c_k + J_kd = 0,
\eequation
where $\gbar_k\in \R{n}$ is a stochastic approximation of the gradient of the objective function and the matrix $H_k \in \R{n\times n}$ satisfies  Assumption~\ref{ass.H} (below). 
\bassumption\label{ass.H}
  The sequence of symmetric matrices $\{H_k\}\subset \mathbb{S}^n$ is bounded in norm by $\kappa_H \in \R{}_{>0}$ such that $\|H_k\|_2\leq \kappa_H$ for all $k\in\N{}$.  In addition, there exists a constant $\zeta \in \R{}_{>0}$ such that, for all $k \in \N{}$, the matrix $H_k$ has the property that $u^TH_ku \geq \zeta \|u\|_2^2$ for all $u \in \R{n}$ where $J_k u = 0$.
\eassumption

Under Assumptions~\ref{ass.main} and~\ref{ass.H}, the optimal solution $\widetilde{d}_k \in \R{n}$ of the subproblem \eqref{eq.subprob}, and an associated displacement in the Lagrange multiplier $\widetilde{\delta}_k \in \R{m}$, can be obtained by solving the linear system of equations given by
\bequation\label{eq.system}
  \bbmatrix H_k & J_k^T \\ J_k & 0 \ebmatrix \bbmatrix \widetilde{d}_k \\ \widetilde{\delta}_k \ebmatrix = - \bbmatrix \gbar_k + J_k^Ty_k \\ c_k \ebmatrix.
\eequation
Solving such linear systems can be expensive, so our algorithms employ inexact solutions to the above linear systems given by $(\dbar_k,\bar{\delta}_k) \approx (\widetilde{d}_k,\widetilde{\delta}_k)$, i.e., 
\begin{equation}\label{eq.system_stochastic}
  \begin{bmatrix} H_k & J_k^T \\ J_k & 0 \end{bmatrix} \begin{bmatrix} \bar{d}_k \\ \bar{\delta}_k \end{bmatrix} = - \begin{bmatrix} \bar{g}_k + J_k^T y_k \\ c_k \end{bmatrix} + \begin{bmatrix}
    \bar{\rho}_k \\ \bar{r}_k
  \end{bmatrix},
\end{equation}
where the tuple $(\bar\rho_k,\bar{r}_k)$ \change{defines} the residuals, and 
$\|(\bar\rho_k,\bar{r}_k)\|$ can be controlled as required. Our algorithms  impose conditions on the norm of the residuals.

\looseness=-1
Given a pair $(\dbar_k,\bar{\delta}_k)$, our algorithms proceed to compute a positive step size in order to update the primal variables $x_k$ and Lagrange multipliers $y_k$. The step size selection strategy is similar to that in \cite{BeraCurtRobiZhou21}. To this end, the algorithms employ a \emph{merit} function $\phi : \R{n} \times \R{}_{> 0} \to \R{}$, parameterized by a \emph{merit parameter} $\tau_k \in \R{}_{> 0}$
, defined as
\bequation\label{eq.merit}
  \phi(x_k,\tau_k) = \tau_k f_k + \|c_k\|_1.
\eequation
The merit parameter is dynamically adjusted as the optimization progresses. We employ \change{a local} model $l : \R{n} \times \R{}_{> 0} \times \R{n} \times \R{n} \to \R{}$ of the merit function \eqref{eq.merit} 
\bequationNN
  l(x_k,\bar{\tau}_k,\gbar_k,\dbar_k) = \bar{\tau}_k (f_k + \gbar_k^T\dbar_k) + \|c_k + J_k\dbar_k\|_1 = \bar{\tau}_k (f_k + \gbar_k^T\dbar_k) + \|\rbar_k\|_1,
\eequationNN
and\change{, to guide the selection of the merit parameter,} its associated reduction function $\Delta l : \R{n} \times \R{}_{> 0} \times \R{n} \times \R{n} \to \R{}$ defined by
\bequation\label{eq.model_reduction}
\baligned
  \Delta l(x_k,\bar{\tau}_k,\gbar_k,\dbar_k) &= l(x_k,\bar{\tau}_k,\gbar_k,0) - l(x_k,\bar{\tau}_k,\gbar_k,\dbar_k) \\
  &= -\bar{\tau}_k \gbar_k^T\dbar_k + \|c_k\|_1 - \|c_k + J_k\dbar_k\|_1\\
  &= -\bar{\tau}_k \gbar_k^T\dbar_k + \|c_k\|_1 - \|\rbar_k\|_1.
\ealigned
\eequation
Note, $\Delta l(x_k,\bar{\tau}_k,\gbar_k,\widetilde{d}_k) = -\bar{\tau}_k \gbar_k^T\widetilde{d}_k + \|c_k\|_1$, where $(\widetilde{d}_k,\widetilde{\delta}_k)$ is the solution of \eqref{eq.system}. 
The mechanism for updating the merit parameter $\bar{\tau}_k$ is similar to that proposed in  \cite{BeraCurtRobiZhou21}, and is motivated by \change{popular} SQP methods \cite{CurtNoceWach10,ByrdCurtNoce10}. \change{The idea is to update (potentially decrease) the merit parameter in order to ensure that the computed search direction is a descent direction for the merit function. }
First, a trial merit parameter $\bar{\tau}_k^{trial}$ is computed. Given user-defined parameters $(\omega_1,\omega_2,\omega_b) \in (0,1) \times (0,1) \times \R{}_{>0}$, if $\|\rbar_k\|_1 \geq (1-\omega_1)\omega_2\|c_k\|_1$ and/or $\|\bar\rho_k\|_1 \geq \omega_b\|c_k\|_1$, we set $\bar\tau_k^{trial}\leftarrow\infty$. Otherwise, we set 
\bequation\label{eq.merit_parameter_trial}
  \bar{\tau}_k^{trial} \gets \bcases \infty & \text{if $\gbar_k^T\dbar_k + \max\{\bar{d}_k^TH_k\bar{d}_k,\epsilon_d\|\bar{d}_k\|_2^2\} \leq 0$} \\[5pt] \displaystyle \tfrac{(1 - \omega_1)(1-\omega_2)\|c_k\|_1 }{\gbar_k^T\dbar_k + \max\{\bar{d}_k^TH_k\bar{d}_k,\epsilon_d\|\bar{d}_k\|_2^2\}} & \text{otherwise,} \ecases
\eequation
where $\epsilon_d \in \left(0,\sfrac{\zeta}{2}\right)$ is user-defined. 
Then, the merit parameter value $\bar{\tau}_k$ is updated via
\bequation\label{eq.tau_update}
  \bar{\tau}_k \gets \bcases \bar{\tau}_{k-1} & \text{if $\bar{\tau}_{k-1} \leq (1 - \epsilon_\tau)\bar{\tau}^{trial}_k$} \\ (1 - \epsilon_\tau) \bar{\tau}^{trial}_k & \text{otherwise,} \ecases
\eequation
where $\epsilon_\tau \in (0,1)$ is user-defined. This rule ensures that $\{\bar{\tau}_k\}$ is a monotonically non-increasing positive sequence with $\bar{\tau}_k \leq (1 - \epsilon_\tau)\bar{\tau}^{trial}_k$ for all $k \in \N{}$. Moreover, this rule aims to ensure that the reduction function \eqref{eq.model_reduction} is non-negative and satisfies \change{(as proved in Lemma~\ref{lem.model_reduction_cond} below)},
\bequation\label{eq.merit_model_reduction_lower_stochastic}
\baligned
  \Delta l(x_k,\bar\tau_k,\bar{g}_k,\bar{d}_k) &\geq \bar\tau_k \omega_1 \max\{\bar{d}_k^TH_k\bar{d}_k,\epsilon_d \| \bar{d}_k\|_2^2\} + \omega_1 \max\{\|c_k\|_1,\|\bar{r}_k\|_1 - \|c_k\|_1\}.
\ealigned
\eequation

Finally, at the $k$th iteration, the step size $\bar\alpha_k \in \mathbb{R}_{>0}$ is set as follows, 
\begin{equation}\label{eq:alpha_min_stoch}
    \bar\alpha_{k} \gets \min\left\{ \tfrac{2(1-\eta)\beta^{(\change{\sigma}-1)}\Delta l(x_k,\bar\tau_k,\bar{g}_k,\bar{d}_k)}{(\bar\tau_kL_{k}+\Gamma_k)\|\bar{d}_k\|_2^2},\bar\alpha_{k}^{opt}, \alpha_u\beta^{(2-\change{\sigma})}, 1 \right\},
\end{equation}
where $\eta\in (0,1)$, $\beta \in (0,1]$, $\change{\sigma\in [1,2]}$ and $\alpha_u\in\mathbb{R}_{>0}$ are user-defined parameters, \change{$L_k$ and $\Gamma_k$ are estimates of the Lipschitz constants of the gradients of the objective and constraint functions, respectively,} and
\begin{equation}\label{eq:alpha_opt_stoch}
    \bar\alpha_{k}^{opt} \gets  \max\left\{\min\left\{ \tfrac{\Delta l(x_k,\bar\tau_k,\bar{g}_k,\bar{d}_k)}{(\bar\tau_kL_{k}+\Gamma_k)\|\bar{d}_k\|_2^2} , 1 \right\} , \tfrac{\Delta l(x_k,\bar\tau_k,\bar{g}_k,\bar{d}_k) - 2\|c_k\|_1 }{(\bar\tau_kL_{k}+\Gamma_k)\|\bar{d}_k\|_2^2} \right\}.
\end{equation}
The step size choice is motivated by that proposed in \cite{BeraCurtRobiZhou21}, and adapted for our specific setting. \change{The goal of the step size procedure is to ensure sufficient decrease in the merit function across iterations (related to the first two terms in $\bar\alpha_{k}$). The third and fourth terms provide flexibility in the step size choice and a natural upper bound (that simplifies the analysis), respectively.} 
The user-defined parameters $\beta$ and $\sigma$ that influence the step size are also related to the accuracy of gradient approximations employed and the accuracy of the solutions of the linear system \eqref{eq.system}, which are introduced and discussed in Section~\ref{sec.analysis}.

Our algorithmic framework, \emph{ Adaptive, Inexact, Stochastic SQP} (\texttt{AIS-SQP}), is presented in Algorithm~\ref{alg.adaptiveSQP}. 
\begin{algorithm}[ht]
  \caption{(\texttt{AIS-SQP}) Adaptive, Inexact, Stochastic SQP Algorithm\label{alg.adaptiveSQP}}
  \begin{algorithmic}[1]
  \Require $x_0\in\mathbb{R}^n$; $y_0\in\mathbb{R}^m$; $\{H_k\}\subset\mathbb{S}^n$; $\bar\tau_{-1}\in\mathbb{R}_{>0}$;  $\{\omega_1,\omega_2,\eta,\epsilon_{\tau}\} \subset (0,1)$; $\omega_a\in\mathbb{R}_{>0}$;  $\omega_b\in\R{}_{>0}$; $\beta\in (0,1]$; $\alpha_u \in \mathbb{R}_{>0}$; $\epsilon_d\in (0,\sfrac{\zeta}{2})$; \change{$\sigma\in [1,2]$}
  \For{\textbf{all} $k \in \mathbb{N}$}
  \State Compute \emph{some} gradient approximation $\bar{g}_k\in\R{n}$ \label{line:comp_g}
  \State Solve \eqref{eq.system} iteratively; compute a step $(\bar{d}_k,\bar{\delta}_k)$ that satisfies \emph{(a)} or \emph{(b)}: \label{line:conditions}
  
  \emph{(a)}  \eqref{eq.merit_model_reduction_lower_stochastic} and $\| \bar{r}_k\|_1 \leq \omega_a\beta^{\change{\sigma}} \Delta l (x_k,\bar\tau_{k},\bar{g}_k,\bar{d}_k)$, \textbf{with} $\bar\tau_k = \bar\tau_{k-1}$, 
  
  \hspace{0.75cm} and, \textbf{additionally},  $\bar{g}_k^T\bar{d}_k + \max\{\bar{d}_k^TH_k\bar{d}_k,\epsilon_d\|\bar{d}_k\|_2^2\} \leq 0$ \textbf{if} $\|c_k\|_1 > 0$, 
  
  \emph{(b)}  $\|\bar{r}_k\|_1 < \min\{(1-\omega_1)\omega_2, \omega_1\omega_a\beta^{\change{\sigma}}\} \|c_k\|_1$
  and $\|\bar\rho_k\|_1 < \omega_b\|c_k\|_1$
    \State \change{Update $\bar\tau_k$: \textbf{if} $(a)$ is satisfied \textbf{then} $\bar\tau_k \leftarrow \bar\tau_{k-1}$, \textbf{else} update $\bar\tau_k$ via \eqref{eq.merit_parameter_trial}--\eqref{eq.tau_update}}
  \State Compute a step size $\bar\alpha_k$ via \eqref{eq:alpha_min_stoch}--\eqref{eq:alpha_opt_stoch}\label{step.alpha_stochastic}
  \State Update $x_{k+1}\leftarrow x_k + \bar{\alpha}_k \bar{d}_k$, and  $y_{k+1}\leftarrow y_k + \bar{\alpha}_k\bar\delta_k$
  \EndFor
  \end{algorithmic}
\end{algorithm}
\begin{remark}
We make the following remarks about Algorithm~\ref{alg.adaptiveSQP}.
\begin{itemize}[leftmargin=0.25cm]
    \item \textbf{Lines \ref{line:comp_g} and \ref{line:conditions}:} For ease of exposition, we leave these two steps arbitrary and specify them later in the paper (Section~\ref{sec.analysis}). We assume that gradient approximations of arbitrary accuracy can be computed, and that the linear system~\change{\eqref{eq.system}} can be solved iteratively and to arbitrary accuracy. 
    In Section~\ref{sec.practical}, we present a practical adaptive sampling strategy and a mechanism for solving the linear system inexactly.  
    \item \textbf{Lines~\ref{line:conditions}(a) and \ref{line:conditions}(b):} Lemma~\ref{lem.algorithm_well_defined} (below) shows that the algorithm is well defined, and that Line~\ref{line:conditions} will terminate in either $(a)$ or $(b)$. One could enforce a simpler condition on the solution to the linear system, however, this simpler condition would come at the cost of having to solve the linear system exactly if/when the iterates are feasible. In the special case where $\|c_k\|_1=0$, by Assumptions~\ref{ass.H} and~\ref{ass.residual} (presented below; pertaining to the linear system solutions) and \eqref{eq.system_stochastic}, Line~\ref{line:conditions} of Algorithm~\ref{alg.adaptiveSQP} is guaranteed to terminate in case $(a)$. The additional condition in case $(a)$ is added for technical reasons discussed in Section~\ref{sec.analysis}.  
    \item \textbf{Merit parameter update:} \change{If Line \ref{line:conditions} terminates and condition $(a)$ is satisfied, then the merit parameter is not updated. Otherwise, the merit parameter value is  updated via \eqref{eq.merit_parameter_trial}--\eqref{eq.tau_update} to ensure  \eqref{eq.merit_model_reduction_lower_stochastic} is satisfied; see Lemma~\ref{lem.model_reduction_cond}.}
    \item \textbf{Step size selection:} The step size selection strategy \eqref{eq:alpha_min_stoch}--\eqref{eq:alpha_opt_stoch} depends on Lipschtiz constants ($L$ and $\Gamma$), or estimates of these quantities. If one knows the Lipschitz constants, one could simply set $L_k = L$ and $\Gamma_k = \Gamma$ for all $k\in\mathbb{N}$. If such Lipschitz constants are unknown, as is the case more often than not, one can approximate these constants following the approaches proposed in \cite{BeraCurtRobiZhou21,FrieSchm12,BollByrdNoce18}. \change{To simplify the analysis (Section~\ref{sec.analysis}) we assume the Lipschitz constants are known.} 
    \item \textbf{Comparison to other algorithms:} 
    The step computation, merit parameter update, and step size selection mechanisms are similar to those proposed in \cite{BeraCurtRobiZhou21,CurtRobiZhou21,BeraCurtOneiRobi21}. However, there are some key differences, primarily due to the fact that in this work we assume that the accuracy in the gradient approximations can be controlled as the optimization progresses. Similar to \cite{CurtRobiZhou21}, the linear system is solved inexactly, but with a simpler approach that does not require an explicit step decomposition.
\end{itemize}
\end{remark}

Before we proceed, we state and prove a few results that hold throughout the paper. For the analysis in subsequent sections we introduce the following notation. (Note, these quantities are never explicitly computed in our algorithms.) Given the iterate $x_k$ and multiplier $y_k$, let the tuple $(d_k,\delta_k)$ denote the solution to 
\begin{equation}\label{eq.system_deterministic}
  \begin{bmatrix} H_k & J_k^T \\ J_k & 0 \end{bmatrix} \begin{bmatrix} d_k \\ \delta_k \end{bmatrix} = - \begin{bmatrix} g_k + J_k^Ty_k \\ c_k \end{bmatrix},
\end{equation}
the deterministic counter-part of \eqref{eq.system}, where $\gbar_k$ is replaced by the true gradient of the objective function. 
Moreover, let $\{\tau_k\}$ and  $\{\alpha_k\}$ be the sequences of merit parameters and step sizes, respectively,  computed at $x_k$ for all $k\in\mathbb{N}$ by the deterministic variant of the algorithm with $\tau_{k-1} = \bar{\tau}_{k-1}$.  
We make the following additional assumption.
\begin{assumption}\label{ass.residual}
  \change{For any $k \in \mathbb{N}$, a sequence of inexact solutions $\{(\bar{d}_{k,t},\bar\delta_{k,t})\}_{t\in\mathbb{N}}$ is generated by some iterative linear system solver $(t$ denotes the iteration counter of the linear system solver$)$, where $\lim_{t\to\infty}\{(\bar{d}_{k,t},\bar\delta_{k,t})\} = (\tilde{d}_k,\tilde\delta_k)$ and $(\bar{d}_k,\bar\delta_k):= (\bar{d}_{k,t},\bar\delta_{k,t})$ for some $t\in\mathbb{N}$. Furthermore, for technical reasons, we also assume that either $\|c_k\|_1 \neq 0$ or $\bar{g}_k\notin\Range(J_k^T)$ for all $k\in\mathbb{N}$.}
\end{assumption}

\begin{remark}
\change{We make the following remarks about Assumption~\ref{ass.residual}. Assumption~\ref{ass.residual} pertains to properties of two main components: $(i)$ the iterative solver, and $(ii)$ the gradient estimates. First, we require that the iterative solver is able to return the exact solution of \eqref{eq.system} in the limit. Second, we assume that the stochastic gradients $(\gbar_k)$ computed do not lie exactly in the range space of the Jacobian of the constraints $(J_k^T)$ for iterates that are feasible. In general, this is not a strong assumption in the stochastic setting.
For details about practical linear system solvers see \cite{trefethen1997numerical} and references therein, and, for details about our implemented linear system solver see Sections~\ref{sec.system} and \ref{sec.num_res}.}
\end{remark}

The first result shows that Algorithm~\ref{alg.adaptiveSQP} is well-defined.
\begin{lemma}\label{lem.algorithm_well_defined}
   \change{Suppose Assumptions~\ref{ass.H} and \ref{ass.residual} hold.} Line~\ref{line:conditions} of Algorithm~\ref{alg.adaptiveSQP} terminates finitely.
\end{lemma}
\begin{proof} 
We consider two cases: $(i)$ $\|c_k\|_1 > 0$ and $(ii)$ $\|c_k\|_1 = 0$. If $\|c_k\|_1 > 0$, by Assumption~\ref{ass.residual} it follows that $\{\max\{\|\bar{r}_{k,t}\|_1,\|\bar\rho_{k,t}\|_2\}\}\to 0$. Therefore, for sufficiently large $t\in\mathbb{N}$, $\|\bar{r}_{k,t}\|_1 < \min\{(1-\omega_1)\omega_2, \omega_1\omega_a\beta^{\change{\sigma}}\}\|c_k\|_1$ and $\|\bar\rho_{k,t}\|_1 < \omega_b\|c_k\|_1$ are  satisfied, and Line~\ref{line:conditions} of Algorithm~\ref{alg.adaptiveSQP} terminates finitely in case $(b)$. On the other hand, if $\|c_k\|_1 = 0$, by \eqref{eq.system} and \eqref{eq.model_reduction}, it follows that
\begin{equation*}
    \Delta l(x_k,\bar\tau_{k-1},\bar{g}_k,\tilde{d}_k) = -\bar\tau_{k-1}\bar{g}_k^T\tilde{d}_k + \|c_k\|_1 - \|c_k + J_k\tilde{d}_k\|_1 = -\bar\tau_{k-1}\bar{g}_k^T\tilde{d}_k.
\end{equation*}
By \eqref{eq.system} and Assumption~\ref{ass.residual} it follows that $\bar{g}_k\notin\Range(J_k^T)$ and $\|\tilde{d}_k\|_2>0$. By Assumption~\ref{ass.H}, \eqref{eq.system}, $\|\tilde{d}_k\|_2 > 0$ and $\epsilon_d\in \left(0,\sfrac{\zeta}{2}\right)$, it follows that $\tilde{d}_k^TH_k\tilde{d}_k > \epsilon_d\|\tilde{d}_k\|_2^2 > 0$. Moreover, by \eqref{eq.system}, $\bar{g}_k^T\tilde{d}_k + \tilde{d}_k^TH_k\tilde{d}_k = 0$. Combining the above and  $\{\bar\tau_k\}\subset\R{}_{>0}$,
\begin{equation*}
\begin{aligned}
    & \Delta l(x_k,\bar\tau_{k-1},\bar{g}_k,\tilde{d}_k) -  \bar\tau_{k-1}\omega_1\max\{\tilde{d}_k^TH_k\tilde{d}_k,\epsilon_d\|\tilde{d}_k\|_2^2\} \\
     = & \ -\bar\tau_{k-1}\bar{g}_k^T\tilde{d}_k - \bar\tau_{k-1}\omega_1\tilde{d}_k^TH_k\tilde{d}_k \\
     = & \ -\bar\tau_{k-1}(\bar{g}_k^T\tilde{d}_k + \tilde{d}_k^TH_k\tilde{d}_k) + \bar\tau_{k-1}(1-\omega_1)\tilde{d}_k^TH_k\tilde{d}_k \\
     = & \ \bar\tau_{k-1}(1-\omega_1)\tilde{d}_k^TH_k\tilde{d}_k > 0.
\end{aligned}
\end{equation*}
Therefore, by $\{\bar{d}_{k,t}\}\to\tilde{d}_k$ and $\|\tilde{d}_k\| > 0$ (Assumption~\ref{ass.residual}), for sufficiently large $t\in\mathbb{N}$ and sufficiently small $\|\bar{r}_{k,t}\|$, it follows that $\Delta l(x_k,\bar\tau_{k-1},\bar{g}_k,\bar{d}_{k,t}) > 0$ and
\begin{equation*}
\baligned
  \Delta l(x_k,\bar\tau_{k-1},\bar{g}_k,\bar{d}_{k,t}) &\geq \bar\tau_{k-1} \omega_1 \max\{\bar{d}_{k,t}^TH_k\bar{d}_{k,t},\epsilon_d \| \bar{d}_{k,t}\|_2^2\} \\
  & \qquad + \omega_1 \max\{\|c_k\|_1,\|\bar{r}_{k,t}\|_1 - \|c_k\|_1\},
\ealigned
\end{equation*}
with $\Delta l(x_k,\bar\tau_{k-1},\bar{g}_k,\bar{d}_{k,t}) \geq \frac{\|\bar{r}_{k,t}\|_1}{\omega_a\beta^{\change{\sigma}}}$.  So Line~\ref{line:conditions}, Algorithm~\ref{alg.adaptiveSQP} terminates finitely. 
\end{proof}

Next, we prove that \eqref{eq.merit_model_reduction_lower_stochastic} is satisfied at every iteration of Algorithm~\ref{alg.adaptiveSQP}. As mentioned above, this component of the algorithm is central in the analysis.

\begin{lemma}\label{lem.model_reduction_cond}
   \change{Suppose Assumptions~\ref{ass.H} and \ref{ass.residual} hold.} The sequence of iterates generated by Algorithm~\ref{alg.adaptiveSQP} satisfy \eqref{eq.merit_model_reduction_lower_stochastic}.
\end{lemma}
\begin{proof}
  If the first condition on Line~\ref{line:conditions} of Algorithm~\ref{alg.adaptiveSQP} is triggered, i.e., case \emph{(a)}, then the result holds trivially. Thus, we focus on the case where the second condition is triggered, i.e., case \emph{(b)}. In this case, the residual vectors satisfy $\|\bar{r}_k\|_1 < \min\{(1-\omega_1)\omega_2, \omega_1\omega_a\beta^{\change{\sigma}}\}\|c_k\|_1$ and $\|\bar\rho_k\|_1 < \omega_b\|c_k\|_1$, and the merit parameter is updated via \eqref{eq.merit_parameter_trial}--\eqref{eq.tau_update}. By the residual conditions, it follows that $\|c_k\|_1 > 0$. (Note,  Lemma~\ref{lem.algorithm_well_defined} showed that if $\|c_k\|=0$, Line~\ref{line:conditions} of Algorithm~\ref{alg.adaptiveSQP} terminates in case $(a)$.) 
  By \eqref{eq.model_reduction}, \eqref{eq.merit_model_reduction_lower_stochastic} and $\|\bar{r}_k\|_1 < (1-\omega_1)\omega_2\|c_k\|_1$, to complete the proof, it is equivalent to show that 
  \begin{equation}\label{eq.model_reduction_equivalent}
      \bar\tau_k\left( \bar{g}_k^T\bar{d}_k + \omega_1\max\{\bar{d}_k^TH_k\bar{d}_k,\epsilon_d\|\bar{d}_k\|_2^2\} \right) \leq (1-\omega_1)\|c_k\|_1 - \|\bar{r}_k\|_1.
  \end{equation}
  \looseness=-1
  Since $\|\bar{r}_k\|_1 < (1-\omega_1)\omega_2\|c_k\|_1 < (1-\omega_1)\|c_k\|_1$, \eqref{eq.model_reduction_equivalent} directly holds if $\bar{g}_k^T\bar{d}_k + \omega_1\max\{\bar{d}_k^TH_k\bar{d}_k,\epsilon_d\|\bar{d}_k\|_2^2\} \leq 0$. If $\bar{g}_k^T\bar{d}_k + \omega_1\max\{\bar{d}_k^TH_k\bar{d}_k,\epsilon_d\|\bar{d}_k\|_2^2\} > 0$, which also implies $\gbar_k^T\dbar_k + \max\{\bar{d}_k^TH_k\bar{d}_k,\epsilon_d\|\bar{d}_k\|_2^2\}>0$, by \eqref{eq.merit_parameter_trial}--\eqref{eq.tau_update}, 
  \begin{equation*}
  \begin{aligned}
      \bar\tau_k \leq (1-\epsilon_{\tau})\bar\tau_k^{trial} = &\tfrac{(1-\epsilon_{\tau})(1 - \omega_1)(1-\omega_2)\|c_k\|_1}{\bar{g}_k^T\bar{d}_k + \max\{\bar{d}_k^TH_k\bar{d}_k,\epsilon_d \| \bar{d}_k\|_2^2\}} 
      < \tfrac{(1 - \omega_1)\|c_k\|_1 - \|\bar{r}_k\|_1}{\bar{g}_k^T\bar{d}_k + \omega_1 \max\{\bar{d}_k^TH_k\bar{d}_k,\epsilon_d \| \bar{d}_k\|_2^2\}},
  \end{aligned}
  \end{equation*}
  which implies that \eqref{eq.model_reduction_equivalent} and \eqref{eq.merit_model_reduction_lower_stochastic} both hold. 
\end{proof}

The next lemma provides an upper bound on the primal residuals~\eqref{eq.system_stochastic}.
\begin{lemma}\label{lem.residual}
   \change{Suppose Assumptions~\ref{ass.H} and \ref{ass.residual} hold.} For all $k \in \N{}$, the residual vector $\bar{r}_k \in \mathbb{R}^m$ \eqref{eq.system_stochastic}
   satisfies, $\| \bar{r}_k \|_1 \leq \omega_a \beta^{\change{\sigma}} \Delta l (x_k,\bar{\tau}_k,\bar{g}_k,\bar{d}_k)$, where  $\omega_a\in\mathbb{R}_{>0}$, $\beta\in (0,1]$ and $\change{\sigma\in [1,2]}$.
\end{lemma}
\begin{proof}
  There are two cases to consider with regards to 
  Line~\ref{line:conditions} of Algorithm~\ref{alg.adaptiveSQP}. If condition \emph{(a)} is triggered, the result holds trivially.  If condition \emph{(b)} is triggered, by Lemma~\ref{lem.model_reduction_cond} and \eqref{eq.merit_model_reduction_lower_stochastic}, it follows that
  \begin{align*}
    \|\bar{r}_k\|_1 < \min\{(1-\omega_1)\omega_2, \omega_1\omega_a\beta^{\change{\sigma}}\}\|c_k\|_1 &\leq \omega_1\omega_a\beta^{\change{\sigma}}\|c_k\|_1 \leq \omega_a\beta^{\change{\sigma}}\Delta l(x_k,\bar\tau_k,\bar{g}_k,\bar{d}_k).
  \end{align*}
  Combining the two cases yields the desired result. 
\end{proof}

The next lemma provides an upper bound on the deterministic dual variable update, $\|y_k + \delta_k\|_{\infty}$, required for the analysis in Section~\ref{sec.adaptive}. We should note that we never require to compute $\delta_k$ in Algorithm~\ref{alg.adaptiveSQP}.
\begin{lemma}\label{lem.bound_ydelta}
\change{Suppose Assumptions~\ref{ass.main} and~\ref{ass.H} hold.} Then, there exists some constant $\kappa_{y\delta}\in\mathbb{R}_{>0}$, such that for all $k\in\mathbb{N}$,  $\|y_k + \delta_k\|_{\infty} \leq \kappa_{y\delta}$.
\end{lemma}
\begin{proof}
For all $k \in \N{}$, let $Z_k\in\mathbb{R}^{n\times(n-m)}$ be an orthonormal basis for the null space of the Jacobian of the constraints $J_k$, i.e., $J_kZ_k = 0$. Then under Assumption~\ref{ass.H}, it follows that $Z_k^TH_kZ_k \succeq \zeta I$. By \eqref{eq.system_deterministic}, it follows that  
\begin{equation*}
\begin{aligned}
    d_k =\ &-J_k^T(J_kJ_k^T)^{-1}c_k - Z_k(Z_k^TH_kZ_k)^{-1}Z_k^T(g_k - H_kJ_k^T(J_kJ_k^T)^{-1}c_k), \\
    \text{and} \ \ y_k + \delta_k = \ &-(J_kJ_k^T)^{-1}J_k(g_k + H_kd_k) \\
    = \ &- (J_kJ_k^T)^{-1}J_k(I - H_kZ_k(Z_k^TH_kZ_k)^{-1}Z_k^T)g_k  \\
    &+ (J_kJ_k^T)^{-1}J_kH_k(I - Z_k(Z_k^TH_kZ_k)^{-1}Z_k^TH_k)J_k^T(J_kJ_k^T)^{-1}c_k .
\end{aligned}
\end{equation*}
By the Cauchy–Schwarz inequality, and Assumptions~\ref{ass.main} and~\ref{ass.H}, it follows that $\|y_k+\delta_k\|_{\infty} \leq \kappa_{\sigma}^{-4}\kappa_J^2\kappa_H\kappa_c + \kappa_{\sigma}^{-2}\kappa_J\kappa_g$. 
Thus, selecting a sufficiently large constant $\kappa_{y\delta} \in \mathbb{R}_{>0}$, completes the proof.  
\end{proof}

Finally, we show that the model reduction function based on deterministic quantities, i.e., $\Delta l(x_k,\tau_k,g_k,d_k)$, is non-negative and bounded above. For non-optimal points, one can show the model reduction function is strictly positive.
\begin{lemma}\label{lem.bound_Delta_l}
\change{Suppose Assumptions~\ref{ass.main} and~\ref{ass.H} hold.} Then, there exists some fixed constant $\kappa_{\Delta l} \in \R{}_{>0}$ such that for all $k\in\mathbb{N}$, $\Delta l(x_k,\tau_k,g_k,d_k) \in [0,\kappa_{\Delta l})$. 
\end{lemma}
\begin{proof}
  Notice that in this lemma we consider only deterministic quantities. First, we show that $\Delta l(x_k,\tau_k,g_k,d_k) \geq 0$ for all $k\in\mathbb{N}$. We consider two cases for the outcome of Line~\ref{line:conditions} of Algorithm~\ref{alg.adaptiveSQP}. If  \eqref{eq.merit_model_reduction_lower_stochastic} is satisfied with $\tau_k = \bar\tau_{k-1}$, then by \eqref{eq.model_reduction} and \eqref{eq.system_deterministic}, we have
  \begin{equation*}
      \Delta l(x_k,\tau_k,g_k,d_k) \geq \tau_k\omega_1\max\left\{d_k^TH_kd_k,\epsilon_d\|d_k\|_2^2\right\} + \omega_1\|c_k\|_1 \geq 0.
  \end{equation*}
  Otherwise, we have $0 = \|c_k+J_kd_k\|_1 < (1-\omega_1)\omega_2\|c_k\|_1$. We consider two subcases. If $g_k^Td_k + \max\{d_k^TH_kd_k,\epsilon_d\|d_k\|_2^2\} \leq 0$, then it follows that $g_k^Td_k \leq 0$, and by $\tau_k > 0$, \eqref{eq.model_reduction} and \eqref{eq.system_deterministic}, we have
  \begin{equation*}
      \Delta l(x_k,\tau_k,g_k,d_k) = -\tau_kg_k^Td_k + \|c_k\|_1 \geq 0.
  \end{equation*}
  On the other hand, if $g_k^Td_k + \max\{d_k^TH_kd_k,\epsilon_d\|d_k\|_2^2\} > 0$, by \eqref{eq.merit_parameter_trial}--\eqref{eq.tau_update} and the fact that $\tau_k \leq (1-\epsilon_{\tau})\tau_k^{trial} < \tau_k^{trial}$, we have
  \begin{equation}\label{eq.Delta_l_positive_middle}
      \tau_kg_k^Td_k < (1-\omega_1)(1-\omega_2)\|c_k\|_1 - \tau_k\max\{d_k^TH_kd_k,\epsilon_d\|d_k\|_2^2\}.
  \end{equation}
  Combining \eqref{eq.model_reduction}, \eqref{eq.system_deterministic} and \eqref{eq.Delta_l_positive_middle}, it follows that
  \begin{equation*}
  \begin{aligned}
      \Delta l(x_k,\tau_k,g_k,d_k) = & \ -\tau_kg_k^Td_k + \|c_k\|_1 \\
      > & \ \tau_k\max\{d_k^TH_kd_k,\epsilon_d\|d_k\|_2^2\} + \left(1-(1-\omega_1)(1-\omega_2)\right)\|c_k\|_1 \geq 0.
  \end{aligned}
  \end{equation*}
  Thus, we have shown that $\Delta l(x_k,\tau_k,g_k,d_k) \geq 0$ for all $k\in\mathbb{N}$.
  
  Next, we show that $\Delta l(x_k,\tau_k,g_k,d_k) \leq \kappa_{\Delta l}$ for all $k\in\mathbb{N}$. For all $k \in \N{}$, let $Z_k\in\R{n\times (n-m)}$ be an orthonormal basis for the null space of the Jacobian of the constraints $J_k$, i.e., $J_kZ_k = 0$, and Assumption~\ref{ass.H} further implies that $Z_k^TH_kZ_k \succeq \zeta I$.
  From \eqref{eq.system_deterministic} we have
  \begin{equation}\label{eq.deterministic_gd}
      g_k^Td_k = -g_k^T(I - Z_k(Z_k^TH_kZ_k)^{-1}Z_k^TH_k)J_k^T(J_kJ_k^T)^{-1}c_k - g_k^TZ_k(Z_k^TH_kZ_k)^{-1}Z_k^Tg_k.
  \end{equation}
  By Assumption~\ref{ass.main} and \eqref{eq.deterministic_gd}, we have
  \begin{equation*}
     \|g_k^Td_k\|_2 \leq \|g_k^TJ_k^T(J_kJ_k^T)^{-1}c_k\|_2 + \|g_k^TZ_k(Z_k^TH_kZ_k)^{-1}Z_k^Tg_k\|_2 \leq \kappa_g\kappa_J\kappa_{\sigma}^{-2}\kappa_c + \kappa_g^2\zeta^{-1}.
  \end{equation*}
  Moreover, by \eqref{eq.model_reduction} and \eqref{eq.system_deterministic}, it follows that for all $k\in\mathbb{N}$,
  \begin{equation*}
      \Delta l(x_k,\tau_k,g_k,d_k) = -\tau_k g_k^Td_k + \|c_k\|_1 \leq \tau_{-1}\left(\kappa_g\kappa_J\kappa_{\sigma}^{-2}\kappa_c + \kappa_g^2\zeta^{-1}\right) + \kappa_c .
  \end{equation*}
  Selecting a sufficiently large constant $\kappa_{\Delta l} \in \mathbb{R}_{>0}$
  completes the proof. 
\end{proof}

\section{\change{Theoretical Analysis}
}\label{sec.analysis}

In this section, we prove under different conditions on the gradient and linear system solution accuracies, that Algorithm~\ref{alg.adaptiveSQP} has convergence properties that match those from the deterministic setting in expectation. First, we consider adaptive error bounds (Section~\ref{sec.adaptive})\change{,} and then consider predetermined sublinear \change{error bounds} (Section~\ref{sec.pred}).

\subsection{Adaptive Iteration-Dependent Errors}\label{sec.adaptive}
In this section, we provide a comprehensive convergence analysis for Algorithm~\ref{alg.adaptiveSQP} under stochastic conditions on the error in the  gradient approximations (\change{Condition}~\ref{ass.stoch_g}), and inexact solutions to the SQP subproblems \eqref{eq.system_stochastic} (\change{Conditions}~\ref{ass.stoch_linear_system_old} and~\ref{assum:stoch_redcond_old}). 
The following two assumptions are central to the analysis presented in this section. 

\begin{condition}\label{ass.stoch_g}
  For all $k \in \mathbb{N}$, the stochastic gradient estimate $\bar{g}_k \in \mathbb{R}^n$ satisfies
  \begin{equation}\label{eq.norm_cond_stoch}
    \mathbb{E}_k\left[\|\bar{g}_k - g_k\|_2^2\right] \leq \theta_1 \beta^{\change{2\sigma}}\Delta l(x_k,\tau_k,g_k,d_k),
  \end{equation}
where $\theta_1\in\mathbb{R}_{>0}$, $\beta\in (0,1)$, and \change{$\sigma \in [1,2]$}. Additionally, for all $k \in \N{}$, the stochastic gradient estimate $\gbar_k \in \R{n}$ is an unbiased estimator of the gradient of $f$ at $x_k$, i.e., $\mathbb{E}_k \left[ \gbar_k \right] = g_k$, 
where $\mathbb{E}_k[\cdot ]$ denotes the expectation with respect to the distribution of $\xi$ conditioned on the event that the algorithm has reached $x_k \in \R{n}$ in iteration $k \in \N{}$.
\end{condition}
\begin{condition}\label{ass.stoch_linear_system_old}
  For all $k \in \mathbb{N}$, the search directions $(\bar{d}_k,\bar{\delta}_k) \in \mathbb{R}^n \times \mathbb{R}^m$ in \eqref{eq.system_stochastic} $($inexact solutions to \eqref{eq.system}$)$ satisfy
  \begin{equation}\label{eq.inexact_stoch_old}
    \left\|\begin{bmatrix} \tilde{d}_k \\ \tilde{\delta}_k \end{bmatrix} - \begin{bmatrix} \bar{d}_k \\ \bar{\delta}_k \end{bmatrix}\right\|_2^2 \leq \theta_2 \beta^{\change{2\sigma}} \Delta l(x_k,\tau_k,g_k,d_k),
  \end{equation}
where $\theta_2\in\mathbb{R}_{>0}$, $\beta\in (0,1)$, and \change{$\sigma \in [1,2]$}. Note, $ (\widetilde{d}_k,\widetilde{\delta}_k)$ and $(\dbar_k,\bar{\delta}_k)$ are the exact and inexact solutions of \eqref{eq.system}
, respectively.
\end{condition}

\begin{remark}
We note upfront that the conditions given in \change{Conditions}~\ref{ass.stoch_g} and \ref{ass.stoch_linear_system_old} are not implementable in our stochastic setting, as the right-hand-side of the inequalities depend on deterministic quantities. That being said, we use these \change{conditions} as this allows us to gain insights into the errors permitted in the algorithm while still retaining  strong convergence guarantees, and will guide the development of our practical algorithm. An important choice in conditions \eqref{eq.norm_cond_stoch}--\eqref{eq.inexact_stoch_old} is the (deterministic) model reduction function \eqref{eq.model_reduction} on the right-hand-side of the \change{inequalities}. As we show in the analysis, and similar to \cite{BeraCurtRobiZhou21,BeraCurtOneiRobi21,CurtRobiZhou21}, we use this quantity as a proxy of convergence, and as such it is an appropriate measure of the accuracy in the gradient approximations. Another interesting question pertains to the analogue of \eqref{eq.norm_cond_stoch} in the unconstrained setting, i.e., no equality constraints. One can show that for appropriately chosen constants $\theta_1$ and $\beta$, in the unconstrained setting \eqref{eq.norm_cond_stoch} is the well-known ``norm'' condition (in expectation) \cite{carter1991global,ByrdChinNoceWu12}. With regards to \eqref{eq.inexact_stoch_old}, under Assumption~\ref{ass.residual} the inequality is well-defined. Finally, we emphasize that the constants $\beta$ and $\sigma$ are the same constants that appear in Algorithm~\ref{alg.adaptiveSQP} and that are used in the step size selection. Thus, the gradient accuracy, \change{the accuracy in the solution of the linear system~\eqref{eq.system}} 
and the step size selection are inherently connected.  The precise permissible ranges of the constants in \change{Conditions}~\ref{ass.stoch_g} and \ref{ass.stoch_linear_system_old} are made explicit in subsequent lemmas and theorems.
\end{remark}

\change{
Next we introduce a technical condition required for the analysis. Specifically, the condition allows us to  establish an integral and useful upper bound for the difference between the stochastic and deterministic \change{merit parameter values}, and as a result is of vital importance in establishing   complexity result.

\begin{condition}\label{assum:stoch_redcond_old}
For all $k \in \mathbb{N}$, the stochastic gradient estimate $\bar{g}_k \in \mathbb{R}^n$ and the search direction $\bar{d}_k \in \mathbb{R}^n$ $($inexact solution of \eqref{eq.system}
$)$ satisfy
\begin{equation*}
\begin{aligned}
  |(\bar{g}_k^T\bar{d}_k + \max\{\bar{d}_k^TH_k\bar{d}_k,\epsilon_d\|\bar{d}_k\|_2^2\}) &- (g_k^Td_k + \max\{d_k^TH_kd_k,\epsilon_d\|d_k\|_2^2\})| \\
  &\quad \leq \theta_3 \beta^{\change{\sigma}} |g_k^Td_k + \max\{d_k^TH_kd_k,\epsilon_d\|d_k\|_2^2\}|,
\end{aligned}
\end{equation*}  
where $\theta_3\in\mathbb{R}_{>0}$, $\beta\in (0,1)$, \change{$\sigma \in [1,2]$}, and $\theta_3\beta^{\change{\sigma}} \in (0,1)$.
\end{condition}

\begin{remark}
  We make a few remarks about \change{Condition}~\ref{assum:stoch_redcond_old}. This is a strong condition, nonetheless, it is necessary in order to establish the strong non-asymptotic convergence and  complexity results presented in this paper. In the unconstrained setting, \change{Condition}~\ref{assum:stoch_redcond_old} does not add any additional restrictions. That is, in the unconstrained setting, consider any method with $\bar{d}_k = -H_k^{-1}\bar{g}_k$ and $d_k = -H_k^{-1}g_k$, where $ \zeta I \preceq H_k \preceq \kappa_H I $ with $\{\kappa_H,\zeta\} \subset \mathbb{R}_{>0}$ defined in Assumption~\ref{ass.H}. Clearly, $\bar{g}_k^T\bar{d}_k + \max\{\bar{d}_k^TH_k\bar{d}_k,\epsilon_d\|\bar{d}_k\|_2^2\} = 0$ and $ g_k^Td_k + \max\{d_k^TH_kd_k,\epsilon_d\|d_k\|_2^2\} = 0$, thus, no additional restrictions are imposed.
  Thus, in the unconstrained setting, \change{Conditions}~\ref{ass.stoch_g}, \ref{ass.stoch_linear_system_old} and \ref{assum:stoch_redcond_old} reduce to the well known ``norm-condition''.
  Several difficulties arise in the setting with constraints. The primary difficulty pertains to the fact that the merit parameter is possibly adjusted across iterations making the merit function a moving target. Thus, in order to establish convergence and complexity results for all iterations, as compared to other papers that only consider the iterations after the merit parameter has stabilized at a sufficiently small constant value, additional control on the permissible differences is required.
  We should note that if one happened to know a sufficiently small merit parameter value, then \change{Condition}~\ref{assum:stoch_redcond_old} would no longer be required for the theory.
  Finally, again we point out the connection between the accuracy in the gradient approximations, the quality of the solution to the linear system, and the step size through user-defined parameters $\beta$ and $\sigma$.
\end{remark}
}

We prove convergence guarantees for Algorithm \ref{alg.adaptiveSQP}, where the stochastic gradients employed satisfy \change{Condition}~\ref{ass.stoch_g} and the search directions employed satisfy \change{Condition}~\ref{ass.stoch_linear_system_old}. Before we delve into the analysis, we discuss the behavior of the merit parameter sequence $\{\bar{\tau}_k\}$, a key component of our algorithmic framework. In the deterministic setting (e.g., Algorithm~\ref{alg.adaptiveSQP} with $\gbar_k = g_k$ and $\dbar_k = d_k$ for all $k\in\mathbb{N}$), under Assumptions~\ref{ass.main} \change{and~\ref{ass.H}} the merit parameter sequence is bounded away from zero\change{; see  \cite{BeraCurtRobiZhou21,ByrdCurtNoce08}}. In the stochastic setting, where the \change{gradient approximations} employed satisfy \change{Condition}~\ref{ass.stoch_g}, boundedness (away from zero) of the merit parameter cannot be guaranteed.
However, \change{if the iterates generated by Algorithm \ref{alg.adaptiveSQP} converge to a stationary point of \eqref{prob.opt}, then, by Conditions~\ref{ass.stoch_g} and \ref{ass.stoch_linear_system_old}, the gradient approximation eventually become sufficient accurate, as do the solutions to the linear system.} That being said, this is not sufficient to prove strong convergence and complexity guarantees for Algorithm \ref{alg.adaptiveSQP} \change{across all iterations}. To this end, we impose \change{Condition}~\ref{assum:stoch_redcond_old}, an additional technical condition on the gradient approximation and the search direction employed, and, to the best of our knowledge, prove the first complexity guarantees in this setting.



We build up to our main results through a series of lemmas. Our first set of lemmas show that the stochastic search directions computed by Algorithm~\ref{alg.adaptiveSQP} are well-behaved. To this end, we invoke the following orthogonal decomposition of the (stochastic) search direction: $\bar{d}_k = \bar{u}_k + \bar{v}_k$ where $\bar{u}_k \in Null(J_k)$ and $\bar{v}_k \in Range(J_k^T)$ for all $k \in \N{}$.

\begin{lemma}\label{lem.bound_v}
  \change{Suppose Assumptions~\ref{ass.main} and Condition~\ref{ass.stoch_linear_system_old} hold.} Then, there exists $\kappa_{\bar{v}} \in \mathbb{R}_{>0}$ such that, for all $k \in \mathbb{N}$, the normal component $\bar{v}_k$ satisfies $\max\{\|\bar{v}_k\|_2,\|\bar{v}_k\|_2^2\} \leq  \kappa_{\vbar} \max\{\|c_k\|_2,\|\bar{r}_k\|_2\}$.
\end{lemma}
\begin{proof}
  Since $\bar{u}_k\in Null(J_k)$ and $\bar{v}_k\in Range(J_k^T)$, 
  \begin{equation*}
      \bar{v}_k = J_k^T(J_kJ_k^T)^{-1}J_k\bar{v}_k = J_k^T(J_kJ_k^T)^{-1}J_k\bar{d}_k = J_k^T(J_kJ_k^T)^{-1}(\bar{r}_k - c_k).
  \end{equation*}
  Thus, by the Cauchy inequality
  \begin{equation*}
      \|\bar{v}_k\|_2 \leq \|J_k^T(J_kJ_k^T)^{-1}\|_2(\|\bar{r}_k\|_2 + \|c_k\|_2) \leq 2\|J_k^T(J_kJ_k^T)^{-1}\|_2\max\{\|c_k\|_2,\|\bar{r}_k\|_2\}.
  \end{equation*}
  Moreover, it follows that
  \change{\begin{equation*}
  \begin{aligned}
      \|\bar{v}_k\|_2^2 &\leq \left(4\|J_k^T(J_kJ_k^T)^{-1}\|_2^2\max\{\|c_k\|_2,\|\bar{r}_k\|_2\}\right) \max\{\|c_k\|_2,\|\bar{r}_k\|_2\} \\
      &\leq \left(4\|J_k^T(J_kJ_k^T)^{-1}\|_2^2\max\{\|c_k\|_2,\|J_k\|_2\|\bar{d}_k - \tilde{d}_k\|_2\}\right) \max\{\|c_k\|_2,\|\bar{r}_k\|_2\}.
  \end{aligned}
  \end{equation*}}
  \change{By Assumption~\ref{ass.main}, Lemma~\ref{lem.bound_Delta_l}, and Condition~\ref{ass.stoch_linear_system_old}, we have that $\|c_k\|_2$, $\|J_k\|_2$, $\|\bar{d}_k - \tilde{d}_k\|_2$, and $\|J_k^T(J_kJ_k^T)^{-1}\|_2$ are uniformly bounded above for all $k \in \N{}$}, 
  which completes the proof.  
\end{proof}

The next lemma shows that if the step $\bar{d}_k$ is tangentially dominated, i.e.,  $\|\bar{u}_k\|_2$ \emph{dominates} $\|\bar{v}_k\|_2$, then $H_k$ is sufficiently positive definite along the computed stochastic search direction.

\begin{lemma}\label{lem.tangential_big}
  \change{Suppose Assumption~\ref{ass.H} holds.} Then, there exists $\kappa_{\overline{uv}} \in \R{}_{>0}$ such that, for any $k \in \mathbb{N}$, if $\|\bar{u}_k\|_2^2 \geq \kappa_{\overline{uv}} \|\bar{v}_k\|_2^2$, then $\bar{d}_k^TH_k\bar{d}_k \geq \sfrac{\zeta}{2} \|\bar{u}_k\|_2^2$ and $\bar{d}_k^TH_k\bar{d}_k \geq \epsilon_d\|\bar{d}_k\|_2^2$, where $\zeta \in \R{}_{>0}$ (Assumption~\ref{ass.H}) and $\epsilon_d \in (0,\sfrac{\zeta}{2})$.
\end{lemma}
\begin{proof}
  When $\|\bar{u}_k\| = 0$, to satisfy the condition in the statement, we require $\|\bar{u}_k\| = \|\bar{v}_k\| = 0$ which implies $\|\bar{d}_k\| = 0$, and the statement holds trivially.
  
  When $\|\bar{u}_k\| > 0$, by Assumption~\ref{ass.H} it follows that
  \begin{equation*}
  \begin{aligned}
      \bar{d}_k^TH_k\bar{d}_k = & \ \bar{u}_k^TH_k\bar{u}_k + 2\bar{u}_k^TH_k\bar{v}_k + \bar{v}_k^TH_k\bar{v}_k \\
      \geq  & \ \zeta\|\bar{u}_k\|_2^2 - 2\kappa_H\|\bar{u}_k\|_2\|\bar{v}_k\|_2 - \kappa_H\|\bar{v}_k\|_2^2 \\
      \geq & \ \left(\zeta - 2\tfrac{\kappa_H}{\sqrt{\kappa_{\overline{uv}}}} - \tfrac{\kappa_H}{\kappa_{\overline{uv}}}\right)\|\bar{u}_k\|_2^2 \geq \tfrac{\zeta}{2}\|\bar{u}_k\|_2^2
  \end{aligned}
  \end{equation*}
  for sufficiently large $\kappa_{\overline{uv}}$. \change{Moreover, for sufficiently large $\kappa_{\overline{uv}}$, it follows that}
  \begin{equation*}
      \bar{d}_k^TH_k\bar{d}_k \geq \tfrac{\zeta}{2}\|\bar{u}_k\|_2^2 \geq \epsilon_d(1 + \tfrac{1}{\kappa_{\overline{uv}}})\|\bar{u}_k\|_2^2 \geq \epsilon_d(\|\bar{u}_k\|_2^2 + \|\bar{v}_k\|_2^2) = \epsilon_d\|\bar{d}_k\|_2^2,
  \end{equation*}
  \change{which completes the proof.} 
\end{proof}

For $\kappa_{\overline{uv}} \in \R{}_{>0}$ (defined in Lemma~\ref{lem.tangential_big}), let $\mathcal{K}_{\ubar} := \{k \in \mathbb{N} : \|\bar{u}_k\|_2^2 \geq \kappa_{\overline{uv}} \|\bar{v}_k\|_2^2\}$ and $\mathcal{K}_{\vbar} := \{ k \in \mathbb{N} : \|\bar{u}_k\|_2^2 < \kappa_{\overline{uv}} \|\bar{v}_k\|_2^2\}$ denote index sets that form a partition of $\mathbb{N}$, and let 
\begin{equation*}
\begin{aligned}
  \bar{\Psi}_k := \begin{cases} \|\bar{u}_k\|_2^2 + \|c_k\|_2 & \text{if $k \in \mathcal{K}_{\ubar}$;} \\ \max\{\|c_k\|_2,\|\rbar_k\|_2\} & \text{if $k \in \mathcal{K}_{\vbar}$.} \end{cases}
\end{aligned}
\end{equation*}
Our next result shows that the squared norms of the stochastic search directions and the constraint violations for all $k \in \N{}$ are bounded above by $\bar{\Psi}_k$.

\begin{lemma}\label{lem.Psi_1}
  \change{Suppose Assumptions~\ref{ass.main} and Condition~\ref{ass.stoch_linear_system_old} hold.} Then, there exists $\kappa_{\bar\Psi} \in \R{}_{>0}$ such that, for all $k \in \mathbb{N}$, the search direction and constraint violation satisfy $\|\bar{d}_k\|_2^2 \leq \kappa_{\bar\Psi} \bar{\Psi}_k$ and $\|\bar{d}_k\|_2^2 + \|c_k\|_2 \leq  (\kappa_{\bar\Psi} + 1) \bar{\Psi}_k$.
\end{lemma}
\begin{proof}
  For all $k\in\mathcal{K}_{\ubar}$, 
  \begin{equation*}
      \|\bar{d}_k\|_2^2 = \|\bar{u}_k\|_2^2 + \|\bar{v}_k\|_2^2 \leq \left(1 + \tfrac{1}{\kappa_{\overline{uv}}}\right)\|\bar{u}_k\|_2^2 \leq \left(1 + \tfrac{1}{\kappa_{\overline{uv}}}\right)\bar{\Psi}_k.
  \end{equation*}
  For all $k\in\mathcal{K}_{\vbar}$, by Lemma~\ref{lem.bound_v}, 
  \begin{equation*}
      \|\bar{d}_k\|_2^2 = \|\bar{u}_k\|_2^2 + \|\bar{v}_k\|_2^2 < \left(1+\kappa_{\overline{uv}}\right)\|\bar{v}_k\|_2^2 \leq \left(1+\kappa_{\overline{uv}}\right) \kappa_{\vbar}\bar{\Psi}_k.
  \end{equation*}
  Therefore, we set $\kappa_{\bar\Psi} := \max\left\{1 + \tfrac{1}{\kappa_{\overline{uv}}}, \left(1+\kappa_{\overline{uv}}\right)\kappa_{\vbar}\right\}$ to satisfy $\|\bar{d}_k\|_2^2 \leq \kappa_{\bar\Psi}\bar{\Psi}_k$.  Finally, since $\|c_k\|_2\leq \bar{\Psi}_k$ trivially, this concludes the proof. 
\end{proof}

The next lemma shows that the stochastic model reduction,  $\Delta l(x_k,\bar\tau_k,\bar{g}_k,\bar{d}_k)$, is  bounded below by a non-negative quantity.

\begin{lemma}\label{lem.Psi}
  \change{Suppose Assumptions~\ref{ass.H} and \ref{ass.residual} hold.} Then, there exists $\kappa_{\lbar} \in \R{}_{>0}$ such that, for all $k \in \mathbb{N}$, the reduction in the model of the merit function satisfies, $\Delta l(x_k,\bar{\tau}_k,\bar{g}_k,\bar{d}_k) \geq \kappa_{\lbar} \bar\tau_k \bar{\Psi}_k$.
\end{lemma}
\begin{proof} For all $k\in\mathcal{K}_{\ubar}$, by Lemmas~\ref{lem.model_reduction_cond} and~\ref{lem.tangential_big},
\begin{equation*}
\begin{aligned}
    \Delta l(x_k,\bar\tau_k,\bar{g}_k,\bar{d}_k) 
    \geq & \ \bar\tau_k \omega_1 \max\{\bar{d}_k^TH_k\bar{d}_k,\epsilon_d \| \bar{d}_k\|_2^2\} + \omega_1 \max\{\|c_k\|_1,\|\bar{r}_k\|_1 - \|c_k\|_1\} \\
    \geq & \ \bar\tau_k \omega_1 \max\left\{\tfrac{\zeta}{2},\epsilon_d\right\}\|\bar{u}_k\|_2^2 + \omega_1\|c_k\|_1
    \geq \ \bar\tau_k \omega_1\min\left\{\tfrac{\zeta}{2},\tfrac{1}{\tau_{-1}}\right\}\bar{\Psi}_k.
\end{aligned}
\end{equation*}
For all $k\in\mathcal{K}_{\vbar}$, by Lemma~\ref{lem.model_reduction_cond}, 
\begin{equation*}
\begin{aligned}
    \Delta l(x_k,\bar\tau_k,\bar{g}_k,\bar{d}_k) 
    \geq & \ \bar\tau_k \omega_1 \max\{\bar{d}_k^TH_k\bar{d}_k,\epsilon_d \| \bar{d}_k\|_2^2\} + \omega_1 \max\{\|c_k\|_1,\|\bar{r}_k\|_1 - \|c_k\|_1\} \\
    \geq & \ \omega_1 \max\{\|c_k\|_1,\|\bar{r}_k\|_1 - \|c_k\|_1\} \geq   \ \tfrac{\omega_1}{2}\max\{\|c_k\|_1,\|\bar{r}_k\|_1\} \geq \tfrac{\omega_1}{2}\bar{\Psi}_k 
\end{aligned}
\end{equation*}
Setting $\kappa_{\lbar} := \omega_1\min\left\{\sfrac{\zeta}{2},\sfrac{1}{2\tau_{-1}}\right\} \in \mathbb{R}_{>0}$ completes the proof. 

\end{proof}

\change{Our next lemma shows that merit parameter sequences $\{\bar\tau_k\}$ and $\{\tau_k\}$ are  bounded away from zero.}

\begin{lemma}\label{lem.stoch_tau_lb}
\change{Under Assumptions~\change{\ref{ass.main} and~\ref{ass.H} and Condition~\ref{assum:stoch_redcond_old}}, there exist constants $\{\bar\tau_{\min},\tau_{\min}\}\subset\mathbb{R}_{>0}$ \change{such that } $\bar\tau_k \geq \bar\tau_{\min}$ and $\tau_k \geq \tau_{\min}$ for all $k\in\mathbb{N}$.}
\end{lemma}

\begin{proof}
Under \change{Assumptions~\ref{ass.main} and~\ref{ass.H}},  
it is \change{well-known that $\{\tau_k^{trial}\}$, the deterministic variant of the sequence of \change{trial} merit parameter values \eqref{eq.merit_parameter_trial}, is always positive and bounded away from zero; see e.g.,~\cite[Lemma~4.7]{ByrdCurtNoce08}}.
By \change{Condition}~\ref{assum:stoch_redcond_old}, it follows that either  $(g_k^Td_k + \max\{d_k^TH_kd_k,\epsilon_d\|d_k\|^2\})$ and $(\bar{g}_k^T\bar{d}_k + \max\{\bar{d}_k^TH_k\bar{d}_k,\epsilon_d\|\bar{d}_k\|^2\})$ are both \change{non-positive} or both \change{non-negative}. 

If $g_k^Td_k + \max\{d_k^TH_kd_k,\epsilon_d\|d_k\|^2\} \leq 0$, by~\eqref{eq.merit_parameter_trial}, $\tau_k^{trial} = \infty$. By \change{Condition}~\ref{assum:stoch_redcond_old} it follows that $\bar{g}_k^T\bar{d}_k + \max\{\bar{d}_k^TH_k\bar{d}_k,\epsilon_d\|\bar{d}_k\|^2\} \leq 0$ and $\bar\tau_k^{trial} = \infty$.


If $g_k^Td_k + \max\{d_k^TH_kd_k,\epsilon_d\|d_k\|^2\} > 0$, by \change{Condition}~\ref{assum:stoch_redcond_old}, it follows that
\begin{equation*}
    0 < (1-\theta_3\beta^{\change{\sigma}}) \leq \tfrac{\bar{g}_k^T\bar{d}_k + \max\{\bar{d}_k^TH_k\bar{d}_k,\epsilon_d\|\bar{d}_k\|^2\}}{g_k^Td_k + \max\{d_k^TH_kd_k,\epsilon_d\|d_k\|^2\}} \leq (1+\theta_3\beta^{\change{\sigma}}).
\end{equation*}
By~\eqref{eq.merit_parameter_trial}, the above inequality implies $\tfrac{\tau_k^{trial}}{(1+\theta_3\beta^{\change{\sigma}})} \leq \bar\tau_k^{trial} \leq \tfrac{\tau_k^{trial}}{(1-\theta_3\beta^{\change{\sigma}})}$.

\change{In both cases considered above, given the fact that $\{\tau_k^{trial}\}$ is  positive and bounded away from zero, one can conclude that there exist constants $\{\bar\tau_{\min}^{trial},\tau_{\min}^{trial}\}\subset\mathbb{R}_{>0}$ such that $\bar\tau_k^{trial} \geq \bar\tau_{\min}^{trial}$ and $\tau_k^{trial} \geq \tau_{\min}^{trial}$ for all $k\in\mathbb{N}$. By \eqref{eq.tau_update}, there exists $\bar\tau_{\min}$ such that $\bar\tau_k \geq \bar\tau_{\min} :=  \min\{\bar\tau_{-1}, (1-\epsilon_{\tau})\bar\tau_{\min}^{trial}\}$. Finally, choosing $\tau_{\min} = \min\{\bar\tau_{\min},(1-\epsilon_{\tau})\tau_{\min}^{trial}\}$ completes the proof.} 
\end{proof}


The next lemma provides a useful lower bound for the reduction in the merit function that is proportional to the stochastic search direction computed.

\begin{lemma}\label{lem.stoch_model_reduction_step_squared}
\change{Suppose Assumptions~\ref{ass.main},~\ref{ass.H} and~\ref{ass.residual} and Conditions~\ref{ass.stoch_linear_system_old} and~\ref{assum:stoch_redcond_old} hold}. Then, there exist constants $\{\kappa_{\bar\alpha},\kappa_{\Delta \bar{l},\bar{d}}\}\subset\mathbb{R}_{>0}$ such that for all $k\in\mathbb{N}$, $\Delta l(x_k,\bar\tau_k,\bar{g}_k,\bar{d}_k) \geq \kappa_{\bar\alpha}(\bar\tau_k\change{L} + \change{\Gamma})\|\bar{d}_k\|_2^2 \geq \kappa_{\Delta \bar{l},\bar{d}}\|\bar{d}_k\|_2^2$.
\end{lemma}
\begin{proof}
By \change{Lemmas~\ref{lem.Psi_1},~\ref{lem.Psi}, and~\ref{lem.stoch_tau_lb}}, it follows that
\begin{equation*}
\begin{aligned}
\Delta l(x_k,\bar\tau_k,\bar{g}_k,\bar{d}_k) &\geq \kappa_{\bar{l}}\bar\tau_k\bar\Psi_k \geq \tfrac{\kappa_{\bar{l}}\bar\tau_k}{\kappa_{\bar\Psi}}\|\bar{d}_k\|_2^2 = \tfrac{\kappa_{\bar{l}}\bar\tau_k}{\kappa_{\bar\Psi}(\bar\tau_k \change{L} + \change{\Gamma})}(\bar\tau_k \change{L} + \change{\Gamma})\|\bar{d}_k\|_2^2 \\
&\geq \tfrac{\kappa_{\bar{l}}\bar\tau_{\min}}{\kappa_{\bar\Psi}(\bar\tau_{\min}\change{L} + \change{\Gamma})}(\bar\tau_k \change{L} + \change{\Gamma})\|\bar{d}_k\|_2^2 \geq \tfrac{\kappa_{\bar{l}}\bar\tau_{\min}}{\kappa_{\bar\Psi}}\|\bar{d}_k\|_2^2.
\end{aligned}
\end{equation*}
The result follows by choosing appropriate constants $\kappa_{\bar\alpha}$ and $\kappa_{\Delta \bar{l},\bar{d}}$. 
\end{proof}

\begin{remark}\label{rm:deter_ld}
    Corollaries of the above lemmas can be derived for the special case under which all quantities employed are deterministic and \eqref{eq.system_deterministic} is solved exactly. 
    Specifically, under the same logic as in Lemmas~\ref{lem.Psi_1}-\ref{lem.stoch_model_reduction_step_squared}, it follows that there exist constants $\{\kappa_{\alpha},\kappa_{\Delta l,d}\}\subset\R{}_{>0}$ such that for all $k\in\N{}$
  \begin{equation}\label{eq.deltal_d}
      \Delta l(x_k,\tau_k,g_k,d_k) \geq \kappa_{\alpha}(\tau_k \change{L} + \change{\Gamma})\|d_k\|_2^2 \geq \kappa_{\Delta l,d}\|d_k\|_2^2.
  \end{equation}
\end{remark}

The next lemma provides upper and lower bounds on the adaptive step sizes employed by Algorithm~\ref{alg.adaptiveSQP}.


\begin{lemma}\label{lem.stoch_stepsize_bound_old}
\change{Suppose Assumptions~\ref{ass.main},~\ref{ass.H} and~\ref{ass.residual} and Conditions~\ref{ass.stoch_linear_system_old} and~\ref{assum:stoch_redcond_old} hold}. Let $\bar{\alpha}_k$ be defined via \eqref{eq:alpha_min_stoch}--\eqref{eq:alpha_opt_stoch}. For all $k\in\mathbb{N}$, there exists a constant $\underline{\alpha}\in\mathbb{R}_{>0}$ such that $\underline{\alpha} \beta \leq \bar{\alpha}_k \leq \alpha_u\beta^{(2-\change{\sigma})}$.
\end{lemma}
\begin{proof}
The upper bound follows from \eqref{eq:alpha_min_stoch}. \change{We note that combining \eqref{eq:alpha_opt_stoch} and Lemma~\ref{lem.stoch_model_reduction_step_squared}, it follows that $\bar\alpha_k^{opt} \geq \min\left\{\tfrac{\Delta l(x_k,\bar\tau_k,\bar{g}_k,\bar{d}_k)}{(\bar\tau_kL + \Gamma)\|\bar{d}_k\|^2},1\right\} \geq \min\{\kappa_{\bar\alpha},1\}$, where $\kappa_{\bar\alpha}\in\mathbb{R}_{>0}$ is defined in Lemma~\ref{lem.stoch_model_reduction_step_squared}. To derive the lower bound, we consider two cases. If $\bar\alpha_k = \min\left\{\alpha_u\beta^{(2-\change{\sigma})}, \bar\alpha_{k}^{opt},1\right\}$, under the  conditions $\beta\in(0,1]$ and \change{$\sigma\in [1,2]$}, it follows that $\bar\alpha_k \geq \min\{\alpha_u,\kappa_{\bar\alpha},1\}\beta$. Otherwise, if $\bar\alpha_k = \tfrac{2(1-\eta)\beta^{(\change{\sigma}-1)}\Delta l(x_k,\bar\tau_k,\bar{g}_k,\bar{d}_k)}{(\bar\tau_kL +\Gamma)\|\bar{d}_k\|_2^2}$, by Lemma~\ref{lem.stoch_model_reduction_step_squared}, $\beta\in (0,1]$ and \change{$\sigma \in [1,2]$}, it follows that $\bar\alpha_k \geq 2(1-\eta)\kappa_{\bar\alpha}\beta$.} Setting $\underline{\alpha} := \min\{\alpha_u,1,\kappa_{\bar\alpha},2(1-\eta)\kappa_{\bar\alpha}\}$ completes the proof. 
\end{proof}

The next result provides an upper bound on the change of merit function value after a step. Central to the proof of this lemma is the step size strategy \eqref{eq:alpha_min_stoch}-\eqref{eq:alpha_opt_stoch}.

\begin{lemma}\label{lem.key_decrease}
  \change{Suppose Assumption~\ref{ass.main} holds.} For all $k\in\mathbb{N}$, it follows that
  \bequationNN
  \baligned
      &\ \phi(x_k + \bar\alpha_k \bar{d}_k, \bar\tau_k) - \phi(x_k, \bar\tau_k) \\
      \leq&\ -\bar\alpha_k \Delta l(x_k,\tau_k,g_k,d_k) + (1-\eta) \bar\alpha_k \beta^{(\change{\sigma}-1)}  \Delta l(x_k,\bar\tau_k,\bar{g}_k,\bar{d}_k) \\
      & \ + \bar\alpha_k \bar\tau_k g_k^T (\bar{d}_k - d_k) + \bar\alpha_k (\bar\tau_k - \tau_k) g_k^T d_k + \bar\alpha_k\|J_k(\bar{d}_k - \tilde{d}_k)\|_1,
  \ealigned
  \eequationNN
  where $\bar{d}_k$ and $d_k$ are defined in \eqref{eq.system_stochastic} and \eqref{eq.system_deterministic}, respectively, $\bar\tau_k$ is updated via \eqref{eq.merit_parameter_trial}--\eqref{eq.tau_update} and $\tau_k$ is its deterministic counterpart, and $\bar\alpha_k$ is computed via \eqref{eq:alpha_min_stoch}-\eqref{eq:alpha_opt_stoch}.
\end{lemma}
\begin{proof}
  By the step size selection strategy \eqref{eq:alpha_min_stoch}-\eqref{eq:alpha_opt_stoch}, for all $k \in \N{}$, we have that $\bar\alpha_k \leq 1$. By the triangle inequality, 
  Assumption~\ref{ass.main}, and \eqref{eq.system}, 
  \begin{align*}
      &\phi(x_k + \bar\alpha_k \bar{d}_k, \bar\tau_k) - \phi(x_k, \bar\tau_k) \\
      =\ &\bar\tau_k(f(x_k + \bar\alpha_k\bar{d}_k) - f(x_k)) + \|c(x_k + \bar\alpha_k\bar{d}_k)\|_1 - \|c_k\|_1 \\
      \leq\ &\bar\tau_k(\bar\alpha_kg_k^T\bar{d}_k + \tfrac{L}{2}\bar\alpha_k^2\|\bar{d}_k\|_2^2) + \|c_k + \bar\alpha_kJ_k\bar{d}_k\|_1 + \tfrac{\Gamma}{2}\bar\alpha_k^2\|\bar{d}_k\|_2^2 - \|c_k\|_1 \\
      \leq\ &\bar\tau_k(\bar\alpha_kg_k^T\bar{d}_k + \tfrac{L}{2}\bar\alpha_k^2\|\bar{d}_k\|_2^2) + |1-\bar\alpha_k|\|c_k\|_1 + \bar\alpha_k\|c_k + J_k\bar{d}_k\|_1 + \tfrac{\Gamma}{2}\bar\alpha_k^2\|\bar{d}_k\|_2^2 - \|c_k\|_1 \\
      =\ &\bar\alpha_k (\bar\tau_k g_k^T \bar{d}_k - \|c_k\|_1) + \tfrac{1}{2} (\bar\tau_k L + \Gamma) \bar\alpha_k^2 \|\bar{d}_k\|_2^2 + \bar\alpha_k\|c_k + J_k\bar{d}_k\|_1 \\
      =\ &\bar\alpha_k (\tau_k g_k^T d_k - \|c_k\|_1) + \tfrac{1}{2} (\bar\tau_k L + \Gamma) \bar\alpha_k^2 \|\bar{d}_k\|_2^2 \\
      &+ \bar\alpha_k \bar\tau_k g_k^T (\bar{d}_k - d_k) + \bar\alpha_k (\bar\tau_k - \tau_k) g_k^T d_k  + \bar\alpha_k\|c_k + J_k\tilde{d}_k + J_k(\bar{d}_k - \tilde{d}_k)\|_1 \\
      =\ &-\bar\alpha_k \Delta l(x_k,\tau_k,g_k,d_k) + \tfrac{1}{2} (\bar\tau_k L + \Gamma) \bar\alpha_k^2 \|\bar{d}_k\|_2^2 \\
      &+ \bar\alpha_k \bar\tau_k g_k^T (\bar{d}_k - d_k) + \bar\alpha_k (\bar\tau_k - \tau_k) g_k^T d_k + \bar\alpha_k\|J_k(\bar{d}_k - \tilde{d}_k)\|_1.
  \end{align*}
  By \eqref{eq:alpha_min_stoch}, we have for all $k \in \N{}$, that $ \bar{\alpha}_k \leq \tfrac{2(1-\eta)\beta^{(\change{\sigma}-1)}\Delta l(x_k,\bar\tau_k,\bar{g}_k,\bar{d}_k)}{(\bar\tau_k L + \Gamma)\|\bar{d}_k\|_2^2}$, and 
  \begin{align*}
      &\ \phi(x_k + \bar\alpha_k \bar{d}_k, \bar\tau_k) - \phi(x_k, \bar\tau_k) \\
      \leq&\ -\bar\alpha_k \Delta l(x_k,\tau_k,g_k,d_k) + \tfrac{1}{2} \bar\alpha_k (\bar\tau_k L + \Gamma) \left(\tfrac{2(1-\eta)\beta^{(\change{\sigma}-1)} \Delta l(x_k,\bar\tau_k,\bar{g}_k,\bar{d}_k)}{(\bar\tau_k L + \Gamma) \|\bar{d}_k\|_2^2}\right) \|\bar{d}_k\|_2^2\\
      &\quad + \bar\alpha_k \bar\tau_k g_k^T (\bar{d}_k - d_k) + \bar\alpha_k (\bar\tau_k - \tau_k) g_k^T d_k + \bar\alpha_k\|J_k(\bar{d}_k - \tilde{d}_k)\|_1 \\
      =&\ -\bar\alpha_k \Delta l(x_k,\tau_k,g_k,d_k) + (1-\eta) \bar\alpha_k \beta^{(\change{\sigma}-1)} \Delta l(x_k,\bar\tau_k,\bar{g}_k,\bar{d}_k) \\
      &\quad + \bar\alpha_k \bar\tau_k g_k^T (\bar{d}_k - d_k) + \bar\alpha_k (\bar\tau_k - \tau_k) g_k^T d_k + \bar\alpha_k\|J_k(\bar{d}_k - \tilde{d}_k)\|_1,
  \end{align*}
  which is the desired result. 
\end{proof}

We now proceed to state and prove a series of lemmas (Lemmas~\ref{lem.stoch_g_diff_old}--\ref{lem.gd_deltal}) that provide bounds for the differences between the deterministic and stochastic gradients, search directions, and their inner products.

\begin{lemma}\label{lem.stoch_g_diff_old}
   \change{Suppose Condition~\ref{ass.stoch_g} holds.} For all $k\in\mathbb{N}$, $\mathbb{E}_k\left[\|\bar{g}_k - g_k\|_2\right] \leq \sqrt{\theta_1}\beta^{\change{\sigma}} \sqrt{\Delta l(x_k,\tau_k,g_k,d_k)}$.
\end{lemma}
\begin{proof}
  By \change{Condition}~\ref{ass.stoch_g} and Jensen's inequality, it follows that 

\noindent  
  $\mathbb{E}_k\left[\|\bar{g}_k - g_k\|_2\right] \leq \sqrt{\mathbb{E}_k\left[\|\bar{g}_k - g_k\|_2^2\right]} \leq \sqrt{\theta_1\beta^{\change{2\sigma}} \Delta l(x_k,\tau_k,g_k,d_k)}$. 
\end{proof}

\begin{lemma}\label{lem.dtilde_stochastic_old}
  \change{Suppose Assumptions~\ref{ass.main}and~\ref{ass.H} and Condition~\ref{ass.stoch_g} hold.} For all $k\in\mathbb{N}$, $\mathbb{E}_k[\tilde{d}_k] = d_k$, $\mathbb{E}_k[\tilde{u}_k] = u_k$, and $\mathbb{E}_k[\tilde{\delta}_k] = \delta_k$.  Moreover, there exists $ \kappa_{\tilde{d}} \in \R{}_{>0}$, such that $\mathbb{E}_k\left[\|\tilde{d}_k - d_k\|_2\right]  \leq \kappa_{\tilde{d}} \beta^{\change{\sigma}} \sqrt{\Delta l(x_k,\tau_k,g_k,d_k)}$. 
\end{lemma}
\begin{proof}
The first statement follows from the facts that, the matrix on the left-hand side of \eqref{eq.system} is invertible and deterministic under \change{Assumptions~\ref{ass.main} and~\ref{ass.H}} and  conditioned on Algorithm~\ref{alg.adaptiveSQP} having reached iterate~$x_k$ at iteration $k$, and due to the fact that expectation is a linear operator. 
For the second statement, for any realization of $\bar{g}_k$,
  \begin{equation*}
    \begin{bmatrix} \tilde{d}_k - d_k \\ \tilde{\delta}_k - \delta_k \end{bmatrix} = -\begin{bmatrix} H_k & J_k^T \\ J_k & 0 \end{bmatrix}^{-1} \begin{bmatrix} \bar{g}_k - g_k \\ 0 \end{bmatrix} \implies \|\tilde{d}_k - d_k\|_2 \leq \kappa_L \|\bar{g}_k - g_k\|_2,
  \end{equation*}
  where (under \change{Assumptions~\ref{ass.main} and~\ref{ass.H}}) $\kappa_L \in \R{}_{>0}$ is an upper bound on the norm of the matrix shown above.  Thus, it follows by Lemma~\ref{lem.stoch_g_diff_old} that
  \begin{equation*}
      \mathbb{E}_k\left[\|\tilde{d}_k - d_k \|_2\right] \leq \mathbb{E}_k\left[\kappa_L \|\bar{g}_k - g_k \|_2 \right] \leq \kappa_L \sqrt{\theta_1} \beta^{\change{\sigma}} \sqrt{\Delta l(x_k,\tau_k,g_k,d_k)}.
  \end{equation*}
  
\end{proof}

\begin{lemma}\label{lem.gT(d_diff)_stoch_old}
\change{Suppose Assumptions~\ref{ass.main},~\ref{ass.H} and~\ref{ass.residual} and Condition~\ref{ass.stoch_linear_system_old} hold.} For all $k\in\N{}$,
\begin{align*}
    |g_k^T(d_k - \tilde{d}_k)| \leq \ &\kappa_{g,d\tilde{d}} \|g_k - \bar{g}_k\|_2\sqrt{\Delta l(x_k,\tau_k,g_k,d_k)},\\
    |\bar{g}_k^T\tilde{d}_k - g_k^Td_k| \leq \ &\kappa_{\bar{g}g,\tilde{d}d}\|g_k - \bar{g}_k\|_2\sqrt{\Delta l(x_k,\tau_k,g_k,d_k)} + \kappa_L\|\bar{g}_k - g_k\|_2^2,\\
    |\bar{g}_k^T(\bar{d}_k - \tilde{d}_k)| \leq \ &\kappa_{\bar{g},\bar{d}\tilde{d}} \beta^{\change{\sigma}} \Delta l(x_k,\tau_k,g_k,d_k) + \bar{\kappa}_{\bar{g},\bar{d}\tilde{d}} \beta^{\change{\sigma}} \Delta l(x_k,\bar{\tau}_k,\bar{g}_k,\bar{d}_k) \\
    &\change{+ \sqrt{\theta_2}\beta^{\change{\sigma}}\|\bar{g}_k - g_k\|_2\sqrt{\Delta l(x_k,\tau_k,g_k,d_k)},} \\ 
    |g_k^T (\bar{d}_k - d_k)| \leq \ &\kappa_{g,d\tilde{d}}\|g_k - \bar{g}_k\|_2\sqrt{\Delta l(x_k,\tau_k,g_k,d_k)} \\&+ \kappa_{g,\bar{d}d}\beta^{\change{\sigma}}\Delta l(x_k,\tau_k,g_k,d_k) 
    + \bar\kappa_{g,\bar{d}d}\beta^{\change{\sigma}}\Delta l(x_k,\bar\tau_k,\bar{g}_k,\bar{d}_k), \\ 
    |g_k^Td_k - \bar{g}_k^T\bar{d}_k| \leq \ &\kappa_{\bar{g}g,\tilde{d}d}\|g_k - \bar{g}_k\|_2\sqrt{\Delta l(x_k,\tau_k,g_k,d_k)} + \change{\kappa_L\|\bar{g}_k - g_k\|_2^2} \\
    &+ \kappa_{g\bar{g},d\bar{d}} \beta^{\change{\sigma}} \Delta l(x_k,\tau_k,g_k,d_k) + \bar{\kappa}_{g\bar{g},d\bar{d}} \beta^{\change{\sigma}} \Delta l(x_k,\bar{\tau}_k,\bar{g}_k,\bar{d}_k),\\
    \text{and} \ \ \|J_k(\bar{d}_k - \tilde{d}_k)\|_1 \leq \ &  \bar{\kappa}_{J,\bar{d}\tilde{d}} \beta^{\change{\sigma}} \Delta l(x_k,\bar{\tau}_k,\bar{g}_k,\bar{d}_k), 
\end{align*}
where 
$\kappa_{g,d\tilde{d}} = \tfrac{\kappa_H\kappa_L}{\sqrt{\kappa_{\Delta l,d}}} \in \R{}_{>0}$, 
$\kappa_{\bar{g}g,\tilde{d}d} = \kappa_{g,d\tilde{d}} + \tfrac{1}{\sqrt{\kappa_{\Delta l,d}}} \change{+ \sqrt{\theta_2}} \in \R{}_{>0}$, 
$\kappa_{L} \in \R{}_{>0}$, 
$\kappa_{\bar{g},\bar{d}\tilde{d}} = \tfrac{\kappa_H\sqrt{\theta_2}}{\sqrt{\kappa_{\Delta l,d}}} \in \R{}_{>0}$, 
$\bar{\kappa}_{\bar{g},\bar{d}\tilde{d}} = \kappa_{y\delta}\omega_a \in \R{}_{>0}$, 
$\kappa_{g,\bar{d}d} = \tfrac{\kappa_H\sqrt{\theta_2}}{\sqrt{\kappa_{\Delta l,d}}} \in \R{}_{>0}$, 
$\bar{\kappa}_{g,\bar{d}d} = \kappa_{y\delta}\omega_a \in \R{}_{>0}$, $\kappa_{g\bar{g},d\bar{d}} = \kappa_{\bar{g},\bar{d}\tilde{d}} \in \R{}_{>0}$, 
$\bar{\kappa}_{g\bar{g},d\bar{d}} = \bar{\kappa}_{\bar{g},\bar{d}\tilde{d}} \in \R{}_{>0}$, and 
$\bar{\kappa}_{J,\bar{d}\tilde{d}} = \omega_a \in \R{}_{>0}$. 
\end{lemma}
\begin{proof}
    \emph{(First inequality)} 
    By   
    Assumption~\ref{ass.H}, and \eqref{eq.system}, \eqref{eq.system_deterministic},  \eqref{eq.deltal_d}, 
    \begin{equation}\label{eq.res1_stoch_old}
    \begin{aligned}
        |g_k^T(d_k - \tilde{d}_k)|
        =\ & |(g_k + J_k^T(y_k + \delta_k))^T(d_k - \tilde{d}_k)| \\
        =\ & |(H_kd_k)^T(d_k - \tilde{d}_k)| 
        \leq \kappa_H\|d_k\|_2 \|d_k - \tilde{d}_k\|_2 \\ \leq\ &\kappa_H  \sqrt{\tfrac{\Delta l(x_k,\tau_k,g_k,d_k)}{\kappa_{\Delta l,d}}} \kappa_{L} \|g_k - \bar{g}_k\|_2, 
    \end{aligned}
    \end{equation}
    where the result follows using the definition of $\kappa_{g,d\tilde{d}}$.

    \emph{(Second inequality)} 
    By 
    Lemma~\ref{lem.dtilde_stochastic_old}, and
    \eqref{eq.system}, \eqref{eq.system_deterministic}, \eqref{eq.deltal_d}, \eqref{eq.res1_stoch_old}, 
  \begin{equation}\label{eq.res2_stoch_old}
    \begin{aligned}
      &|\bar{g}_k^T\tilde{d}_k - g_k^Td_k| \\
      \leq\ & |(\bar{g}_k - g_k)^Td_k| + |g_k^T(\tilde{d}_k - d_k)| + |(\bar{g}_k - g_k)^T(\tilde{d}_k - d_k)| \\
      \leq\ &\|\bar{g}_k - g_k\|_2\|d_k\|_2 + \kappa_{g,d\tilde{d}} \|g_k - \bar{g}_k\|_2\sqrt{\Delta l(x_k,\tau_k,g_k,d_k)} + \|\bar{g}_k - g_k\|_2\|\tilde{d}_k - d_k\|_2 \\
      \leq\ &\|\bar{g}_k - g_k\|_2\sqrt{\tfrac{\Delta l(x_k,\tau_k,g_k,d_k)}{\kappa_{\Delta l,d}}} + \kappa_{g,d\tilde{d}} \|g_k - \bar{g}_k\|_2\sqrt{\Delta l(x_k,\tau_k,g_k,d_k)} + \kappa_L\|\bar{g}_k - g_k\|^2_2,
  \end{aligned}
  \end{equation}
  where the result follows using the definition of $\kappa_{\bar{g}g,\tilde{d}d}$. 
  
  \emph{(Third Inequality)} By  \change{Assumption~\ref{ass.H}, Condition  \ref{ass.stoch_linear_system_old}}, Lemmas~\ref{lem.residual},~\ref{lem.bound_ydelta} and~\ref{lem.dtilde_stochastic_old}, and \eqref{eq.system}, \eqref{eq.system_stochastic}, \eqref{eq.system_deterministic}, \eqref{eq.deltal_d}, it follows that
  \begin{equation}\label{eq.res3_stoch_old}
    \begin{aligned}
      |\bar{g}_k^T(\bar{d}_k -\tilde{d}_k)| 
      =\ & |(\bar{g}_k + J_k^T(y_k + \delta_k) - J_k^T(y_k + \delta_k))^T(\bar{d}_k -\tilde{d}_k)| \\
      \leq\ &|(\bar{g}_k + J_k^T(y_k + \delta_k) )^T(\bar{d}_k -\tilde{d}_k)| + |( J_k^T(y_k + \delta_k))^T(\bar{d}_k -\tilde{d}_k)| \\
      \leq\ &|(g_k + J_k^T(y_k + \delta_k) )^T(\bar{d}_k -\tilde{d}_k)| + |(\bar{g}_k - g_k)^T(\bar{d}_k -\tilde{d}_k)| \\
      &+ |( J_k^T(y_k + \delta_k))^T(\bar{d}_k -\tilde{d}_k)| \\
      \leq\ &|(H_k d_k )^T(\bar{d}_k -\tilde{d}_k)|  
      + \change{\|\bar{g}_k - g_k\|_2\sqrt{\theta_2}\beta^{\change{\sigma}}\sqrt{\Delta l(x_k,\tau_k,g_k,d_k)}} \\
      & + \|y_k + \delta_k\|_{\infty}  \|\bar{r}_k\|_1\\
      \leq\ &\tfrac{\kappa_H\sqrt{\theta_2}}{\sqrt{\kappa_{\Delta l,d}}}\beta^{\change{\sigma}}\Delta l(x_k,\tau_k,g_k,d_k)  \change{+ \sqrt{\theta_2}\beta^{\change{\sigma}}\|\bar{g}_k - g_k\|_2\sqrt{\Delta l(x_k,\tau_k,g_k,d_k)}}\\
      & + \kappa_{y\delta}\omega_a\beta^{\change{\sigma}}\Delta l(x_k,\bar\tau_k,\bar{g}_k,\bar{d}_k), 
  \end{aligned}
  \end{equation}
  where the result follows using the definitions of $\kappa_{\bar{g},\bar{d}\tilde{d}}$ and $\bar{\kappa}_{\bar{g},\bar{d}\tilde{d}}$.

  \emph{(Fourth inequality)} By 
  \change{Assumption~\ref{ass.H}, Condition  \ref{ass.stoch_linear_system_old}}, Lemmas~\ref{lem.residual} and~\ref{lem.bound_ydelta}, and  \eqref{eq.system}, \eqref{eq.system_stochastic}, \eqref{eq.system_deterministic}, \eqref{eq.deltal_d}, it follows that
\begin{align*}
    |g_k^T(\bar{d}_k - \tilde{d}_k)| 
    = \ & |(g_k + J_k^T(y_k + \delta_k) - J_k^T(y_k + \delta_k))^T(\bar{d}_k -\tilde{d}_k)| \\
    \leq\ &|(g_k + J_k^T(y_k + \delta_k) )^T(\bar{d}_k -\tilde{d}_k)| + |( J_k^T(y_k + \delta_k))^T(\bar{d}_k -\tilde{d}_k)| \\
    \leq\ &|(H_k d_k )^T(\bar{d}_k -\tilde{d}_k)| + \|y_k + \delta_k\|_{\infty} \|\bar{r}_k\|_1 \\
    \leq\ &\tfrac{\kappa_H\sqrt{\theta_2}}{\sqrt{\kappa_{\Delta l,d}}}\beta^{\change{\sigma}}\Delta l(x_k,\tau_k,g_k,d_k) + \kappa_{y\delta}\omega_a\beta^{\change{\sigma}}\Delta l(x_k,\bar\tau_k,\bar{g}_k,\bar{d}_k).
\end{align*}
By the triangle inequality and \eqref{eq.res1_stoch_old}, 
\begin{equation*}
\begin{aligned}
    |g_k^T(\bar{d}_k - d_k)| 
    \leq\ & |g_k^T(\tilde{d}_k - d_k)| + |g_k^T(\bar{d}_k - \tilde{d}_k)| \\
    \leq\ &\kappa_H \sqrt{\tfrac{\Delta l(x_k,\tau_k,g_k,d_k)}{\kappa_{\Delta l,d}}} \kappa_{L} \|g_k - \bar{g}_k\|_2 + \tfrac{\kappa_H\sqrt{\theta_2}}{\sqrt{\kappa_{\Delta l,d}}}\beta^{\change{\sigma}} \Delta l(x_k,\tau_k,g_k,d_k) \\
    &+ \kappa_{y\delta} \omega_a \beta^{\change{\sigma}} \Delta l(x_k,\bar\tau_k,\bar{g}_k,\bar{d}_k) \\
\end{aligned}
\end{equation*}
where the result follows using the definition of $\kappa_{g,\bar{d}d}$ and $\bar{\kappa}_{g,\bar{d}d}$. 

\emph{(Fifth inequality)} By 
\eqref{eq.res2_stoch_old},  \eqref{eq.res3_stoch_old}, 
\begin{equation*}
\begin{aligned}
    |g_k^Td_k - \bar{g}_k^T\bar{d}_k| \leq\ & |g_k^Td_k - \bar{g}_k^T\tilde{d}_k| + |\bar{g}_k^T(\tilde{d}_k - \bar{d}_k)| \\
    \leq\ & \kappa_{\bar{g}g,\tilde{d}d}\|g_k - \bar{g}_k\|_2\sqrt{\Delta l(x_k,\tau_k,g_k,d_k)} + \change{\kappa_L\|\bar{g}_k - g_k\|_2^2} \\
    &+ \kappa_{\bar{g},\bar{d}\tilde{d}}\beta^{\change{\sigma}}\Delta l(x_k,\tau_k,g_k,d_k) + \bar\kappa_{\bar{g},\bar{d}\tilde{d}}\beta^{\change{\sigma}}\Delta l(x_k,\bar\tau_k,\bar{g}_k,\bar{d}_k),
\end{aligned}
\end{equation*}
where the result follows using the definitions of $\kappa_{g\bar{g},d\bar{d}}$ and $\bar{\kappa}_{g\bar{g},d\bar{d}}$. 

\emph{(Sixth inequality)} By Lemma~\ref{lem.residual}, and \eqref{eq.system},  \eqref{eq.system_stochastic}, 
\begin{equation*}
    \|J_k(\bar{d}_k - \tilde{d}_k)\|_1 = \|\bar{r}_k\|_1 \leq \omega_a\beta^{\change{\sigma}}\Delta l(x_k,\bar{\tau}_k,\bar{g}_k,\bar{d}_k),
\end{equation*}
where the result follows using the definition of $\bar{\kappa}_{J,\bar{d}\tilde{d}}$. 
\end{proof}


The next lemma provides  an upper bound for $|g_k^Td_k|$ with respect to reduction in the model of the merit function, i.e., $\Delta l(x_k,\tau_k,g_k,d_k)$.
\begin{lemma}\label{lem.gd_deltal}
\change{Suppose Assumptions~~\ref{ass.H} and \ref{ass.residual} hold.} For all $k \in \N{}$, $|g_k^Td_k| \leq \kappa_{gd,\Delta l}\Delta l(x_k,\tau_k,g_k,d_k)$, 
where $\kappa_{gd,\Delta l} = \tfrac{\kappa_H}{\kappa_{\Delta l,d}} + \tfrac{\kappa_{y\delta}}{\sqrt{m}\kappa_l\tau_{\min}} \in \mathbb{R}_{>0}$.
\end{lemma}
\begin{proof}
By Assumption~\ref{ass.H}, Lemma \ref{lem.Psi}, and \eqref{eq.system_deterministic}, 
\eqref{eq.deltal_d},
\begin{equation*}
\begin{aligned}
    |g_k^Td_k| = |d_k^TH_kd_k + d_k^TJ_k^T(y_k + \delta_k)| &\leq \kappa_H\|d_k\|_2^2 + \kappa_{y\delta}\|c_k\|_1 \\
    &\leq \left(\tfrac{\kappa_H}{\kappa_{\Delta l,d}} + \tfrac{\kappa_{y\delta}}{\sqrt{m}\kappa_l\tau_{\min}}\right)\Delta l(x_k,\tau_k,g_k,d_k), 
\end{aligned}
\end{equation*}
where the result follows using the definition of $\kappa_{gd,\Delta l}$. 
\end{proof}

Next we state and prove an upper bound on the difference between the deterministic and stochastic merit parameters. When $\|c_k\| = 0$, by \eqref{eq.merit_parameter_trial}--\eqref{eq.tau_update},  $\tau_k = \bar\tau_k = \bar\tau_{k-1}$. When $\|c_k\|>0$, for the ease of exposition, we equivalently reformulate the merit parameter sequence $\{\bar\tau_k\}$ update \eqref{eq.merit_parameter_trial}--\eqref{eq.tau_update} as 
\begin{equation}\label{eq.tau_udpate_det_new}
    \bar\tau_k\leftarrow\begin{cases} \bar\tau_{k-1} \quad\quad\quad\quad\quad\quad\quad\quad \text{if } {\scriptstyle\bar{g}_k^T\bar{d}_k + \max\{\bar{d}_k^TH_k\bar{d}_k,\epsilon_d\|\bar{d}_k\|^2\} \leq} \tfrac{(1-\epsilon_{\tau})(1-\omega_1)(1-\omega_2)\|c_k\|_1}{\bar\tau_{k-1}}; \\
    \tfrac{(1-\epsilon_{\tau})(1-\omega_1)(1-\omega_2)\|c_k\|_1}{\bar{g}_k^T\bar{d}_k + \max\{\bar{d}_k^TH_k\bar{d}_k,\epsilon_d\|\bar{d}_k\|^2\}} \quad \text{otherwise.} \end{cases}
\end{equation}
(The update formula for the deterministic merit parameter $\tau_k$ can be defined as above with the stochastic quantities replaced by their deterministic counterparts.) It is clear that if the merit parameter is updated from its previous value, i.e., $\bar{\tau}_k \neq \bar{\tau}_{k-1}$, then 
\begin{equation}\label{eq.tau_stoch_update}
    \tfrac{(1-\epsilon_\tau)(1-\omega_1)(1 - \omega_2)\|c_k\|_1}{\gbar_k^T\dbar_k + \max\{\bar{d}_k^TH_k\bar{d}_k,\epsilon_d\|\bar{d}_k\|_2^2\}} < \bar{\tau}_{k-1}.
\end{equation}
Similarly, given $\bar{\tau}_{k-1}$ if the deterministic merit parameter is updated, 
\begin{equation}\label{eq.tau_det_update}
        \tfrac{(1-\epsilon_\tau)(1-\omega_1)(1 - \omega_2)\|c_k\|_1}{g_k^Td_k + \max\{d_k^TH_kd_k,\epsilon_d\|d_k\|_2^2\}} < \bar{\tau}_{k-1}.
\end{equation}

\begin{lemma}\label{lem.tau_bound_old}
    \change{Suppose Assumptions~~\ref{ass.H} and~\ref{ass.residual} and Condition~\ref{assum:stoch_redcond_old} hold.} For all $k\in\N{}$, $|(\bar\tau_k - \tau_k)g_k^Td_k| \leq \kappa_{\bar{\tau}}\beta^{\change{\sigma}} \Delta l(x_k,\tau_k,g_k,d_k)$, 
    where $\kappa_{\bar{\tau}} = 2\theta_3 \bar{\tau}_{-1}\kappa_{gd,\Delta l}$ and $\beta\in\left(0,\sfrac{1}{(2\theta_3)^{\change{1/\sigma}}}\right]$. 
\end{lemma}

\begin{proof} By the merit parameter updating mechanism \eqref{eq.merit_parameter_trial}--\eqref{eq.tau_update}, the merit parameter values ($\bar\tau_k$ and $\tau_k$) are only potentially updated if $\|c_k\|_1 > 0$. We divide the proof into two cases based \change{on} the outcome of Line~\ref{line:conditions} in Algorithm~\ref{alg.adaptiveSQP} ($(a)$ or case $(b)$). Let $h_k = g_k^Td_k + \max\{d_k^TH_kd_k,\epsilon_d\|d_k\|_2^2\} $ and $\bar{h}_k = \bar{g}_k^T\bar{d}_k + \max\{\bar{d}_k^TH_k\bar{d}_k,\epsilon_d\|\bar{d}_k\|_2^2\}$.

If Line~\ref{line:conditions} in Algorithm~\ref{alg.adaptiveSQP} terminates in case $(a)$, and $\|c_k\|_1 > 0$,  
it follows that $\bar{h}_k \leq 0$
, and by \change{Condition}~\ref{assum:stoch_redcond_old}, $h_k \leq 0$; 
otherwise, when $\|c_k\|_1 = 0$, by \eqref{eq.system_deterministic} and \change{Assumption}~\ref{ass.H}, it follows that $d_k\in\Null(J_k)$ and $h_k = g_k^Td_k + d_k^TH_kd_k = -d_k^TJ_k^T(y_k+\delta_k) = c_k^T(y_k+\delta_k) = 0$. Thus, neither the stochastic nor the deterministic merit parameters update, i.e., $\tau_k = \bar\tau_k =\bar\tau_{k-1}$.  
The result holds since $\Delta l(x_k,\tau_k,g_k,d_k)$ is non-negative. 

Next, we consider the case where Line~\ref{line:conditions} in Algorithm~\ref{alg.adaptiveSQP} terminates with case $(b)$. 
We divide merit parameter values $(\tau_k,\bar\tau_k)$ into four cases. 

\textbf{Case (i)}: Neither $\bar\tau_k$ or $\tau_k$ are updated, i.e., $\tau_k = \bar\tau_k =\bar\tau_{k-1}$. In this case, the result holds since $\Delta l(x_k,\tau_k,g_k,d_k)$ is non-negative.

\textbf{Case (ii)}: Both $\bar\tau_k$ and $\tau_k$ are updated, i.e., \eqref{eq.tau_stoch_update} and \eqref{eq.tau_det_update} hold. By \eqref{eq.tau_udpate_det_new}, \eqref{eq.tau_stoch_update} and  \eqref{eq.tau_det_update}, and \change{Condition}~\ref{assum:stoch_redcond_old}, it follows that
\begin{equation*}
\begin{aligned}
    |\bar\tau_k - \tau_k| &= (1 - \epsilon_\tau)(1-\omega_1)(1 - \omega_2)\|c_k\|_1\tfrac{\left|\bar{h}_k - h_k\right|}{\bar{h}_k h_k} \\
    &\leq \tfrac{(1 - \epsilon_\tau)(1-\omega_1)(1 - \omega_2)\|c_k\|_1\theta_3\beta^{\change{\sigma}}}{\bar{h}_k} \leq \theta_3\beta^{\change{\sigma}}\bar\tau_{-1}.
\end{aligned}
\end{equation*}

\textbf{Case (iii)}: The stochastic merit parameter $\bar{\tau}_k$ is updated and the deterministic merit parameter $\tau_k$ is not, i.e., $\bar{\tau}_k \leq \tau_k = \bar{\tau}_{k-1}$. Since the stochastic merit parameter is updated, $\bar\tau_k^{trial} < \infty$ and $\bar{h}_k > 0$,  
and by \change{Condition}~\ref{assum:stoch_redcond_old} it follows that $h_k > 0 $
, and $\tau_k^{trial} < \infty$. Moreover, since the deterministic merit parameter is not updated,  
\begin{equation}\label{eq.tau_det_nupdate}
        \tfrac{(1-\epsilon_\tau)(1-\omega_1)(1 - \omega_2)\|c_k\|_1}{h_k} \geq \bar{\tau}_{k-1}.
\end{equation}
By \eqref{eq.tau_udpate_det_new}, \eqref{eq.tau_stoch_update} and  \eqref{eq.tau_det_nupdate}, and \change{Condition}~\ref{assum:stoch_redcond_old}, it follows that
\begin{equation*}
\begin{aligned}
    |\bar\tau_k - \tau_k| 
    &=\bar\tau_{k-1} - \tfrac{(1- \epsilon_\tau)(1-\omega_1)(1 - \omega_2)\|c_k\|_1}{\bar{h}_k} \\
    &\leq \tfrac{(1- \epsilon_\tau)(1-\omega_1)(1 - \omega_2)\|c_k\|_1}{h_k} - \tfrac{(1- \epsilon_\tau)(1-\omega_1)(1 - \omega_2)\|c_k\|_1}{\bar{h}_k} \\
    &\leq (1 - \epsilon_\tau)(1-\omega_1)(1 - \omega_2)\|c_k\|_1\tfrac{\left|\bar{h}_k - h_k\right|}{\bar{h}_k h_k} \\
    &\leq \tfrac{(1 - \epsilon_\tau)(1-\omega_1)(1 - \omega_2)\|c_k\|_1\theta_3\beta^{\change{\sigma}}}{\bar{h}_k} < \theta_3\beta^{\change{\sigma}}\bar\tau_{-1}.
\end{aligned}
\end{equation*}

\textbf{Case (iv)}: The deterministic merit parameter $\tau_k$ is updated, while the stochastic merit parameter $\bar{\tau}_k$ is not, i.e., $\tau_k \leq \bar{\tau}_k  = \bar{\tau}_{k-1}$. Since the deterministic merit parameter is updated, $\tau_k^{trial} < \infty$ and  $h_k > 0$, and by \change{Condition}~\ref{assum:stoch_redcond_old} it follows that $\bar{h}_k > 0$ and $\bar\tau_k^{trial} < \infty$. Moreover, since the stochastic merit parameter is not updated, 
\begin{equation}\label{eq.tau_stoch_nupdate}
        \tfrac{(1-\epsilon_\tau)(1-\omega_1)(1 - \omega_2)\|c_k\|_1}{\bar{h}_k} \geq \bar{\tau}_{k-1}.
\end{equation}
By \eqref{eq.tau_udpate_det_new}, \eqref{eq.tau_det_update} and  \eqref{eq.tau_stoch_nupdate}, and \change{Condition}~\ref{assum:stoch_redcond_old}, it follows that
\begin{equation*}
\begin{aligned}
    |\bar\tau_k - \tau_k| &=\bar\tau_{k-1} - \tfrac{(1- \epsilon_\tau)(1-\omega_1)(1 - \omega_2)\|c_k\|_1}{h_k} \\
    &\change{\leq }  \tfrac{(1- \epsilon_\tau)(1-\omega_1)(1 - \omega_2)\|c_k\|_1}{\bar{h}_k} - \tfrac{(1- \epsilon_\tau)(1-\omega_1)(1 - \omega_2)\|c_k\|_1}{h_k}\\
    &\leq (1 - \epsilon_\tau)(1-\omega_1)(1 - \omega_2)\|c_k\|_1\tfrac{\left|\bar{h}_k - h_k\right|}{\bar{h}_k h_k} \\
    &\leq \tfrac{(1 - \epsilon_\tau)(1-\omega_1)(1 - \omega_2)\|c_k\|_1\theta_3\beta^{\change{\sigma}}}{\bar{h}_k}\\
    &\leq \tfrac{(1 - \epsilon_\tau)(1-\omega_1)(1 - \omega_2)\|c_k\|_1\theta_3\beta^{\change{\sigma}}}{(1-\theta_3\beta^{\change{\sigma}})h_k} \leq \tfrac{\theta_3\beta^{\change{\sigma}}}{1 - \theta_3 \beta^{\change{\sigma}}} \bar\tau_{-1}.
\end{aligned}
\end{equation*}
Combing the four cases above and Lemma~\ref{lem.gd_deltal} yields the result. 
\end{proof}
The next lemma bounds the stochastic model of the reduction of the merit function.
\begin{lemma}\label{lem.delta_l_bound_stoch_old}
   \change{Suppose Assumptions~\ref{ass.main},~\ref{ass.H} and~\ref{ass.residual} and Conditions~\ref{ass.stoch_linear_system_old} and~\ref{assum:stoch_redcond_old} hold.} For all $k \in \N{}$ and $\beta \in \left(0,\sfrac{1}{\left(2\bar\tau_{-1}\bar{\kappa}_{g\bar{g},d\bar{d}}\right)^{\change{1/\sigma}}}\right]$, 
   \begin{equation*}
   \begin{aligned}
       \Delta l(x_k,\bar\tau_k,
       \bar{g}_k,\bar{d}_k) \leq\ & \left(1 + \kappa_{\overline{\Delta l},\Delta l}\beta^{\change{\sigma}}\right)\Delta l(x_k,\tau_k,g_k,d_k) \\
       &+ 2\bar\tau_{-1}\left( \kappa_{\bar{g}g,\tilde{d}d}\|g_k - \bar{g}_k\|_2\sqrt{\Delta l(x_k,\tau_k,g_k,d_k)} + \change{\kappa_L\|\bar{g}_k - g_k\|_2^2}\right), 
   \end{aligned}
   \end{equation*}
   where $\kappa_{\overline{\Delta l},\Delta l} = 2( \kappa_{\bar{\tau}} + \bar\tau_{-1} \kappa_{g\bar{g},d\bar{d}} + \bar\tau_{-1} \bar{\kappa}_{g\bar{g},d\bar{d}} ) \in \R{}_{>0}$. Additionally, under \change{Condition}~\ref{ass.stoch_g}, for all $k \in \N{}$, 
  \begin{equation*}
      \mathbb{E}_k\left[ \Delta l(x_k,\bar\tau_k, \bar{g}_k,\bar{d}_k) \right] \leq \left(1 + \bar\kappa_{\overline{\Delta l},\Delta l}\beta^{\change{\sigma}}\right)\Delta l(x_k,\tau_k,g_k,d_k),
  \end{equation*}
   where $\bar\kappa_{\overline{\Delta l},\Delta l} = 2\left( \kappa_{\bar{\tau}} + \bar\tau_{-1} (\kappa_{g\bar{g},d\bar{d}} + \bar{\kappa}_{g\bar{g},d\bar{d}} + \kappa_{\bar{g}g,\tilde{d}d}\sqrt{\theta_1} + \change{\kappa_L\theta_1}) \right)\in\mathbb{R}_{>0}$.
\end{lemma}
\begin{proof}
    By \eqref{eq.model_reduction}, and Lemmas~\ref{lem.gT(d_diff)_stoch_old} and \ref{lem.tau_bound_old}, it follows that
  \begin{align*}
        \Delta l(x_k,\bar\tau_k,\bar{g}_k,\bar{d}_k) = \ & -\bar\tau_k\bar{g}_k^T\bar{d}_k + \|c_k\|_1 - \|c_k + J_k\bar{d}_k\|_1 \\
       =\ &-\tau_kg_k^Td_k + \|c_k\|_1 \\
       &+ (\tau_k - \bar\tau_k)g_k^Td_k + \bar\tau_k(g_k^Td_k - \bar{g}_k^T\bar{d}_k) - \|c_k + J_k\bar{d}_k\|_1 \\
       \leq\ &\Delta l(x_k,\tau_k,g_k,d_k) + |(\tau_k - \bar\tau_k)g_k^Td_k|  + \bar\tau_{-1}|g_k^Td_k - \bar{g}_k^T\bar{d}_k| \\
       \leq\ & \Delta l(x_k,\tau_k,g_k,d_k) + \kappa_{\bar\tau}\beta^{\change{\sigma}}\Delta l(x_k,\tau_k,g_k,d_k) \\
       &+ \bar\tau_{-1}\left( \kappa_{\bar{g}g,\tilde{d}d}\|g_k - \bar{g}_k\|_2\sqrt{\Delta l(x_k,\tau_k,g_k,d_k)} + \change{\kappa_L\|\bar{g}_k - g_k\|_2^2} \right. \\
       &\left. + \kappa_{g\bar{g},d\bar{d}} \beta^{\change{\sigma}} \Delta l(x_k,\tau_k,g_k,d_k) + \bar{\kappa}_{g\bar{g},d\bar{d}} \beta^{\change{\sigma}} \Delta l(x_k,\bar{\tau}_k,\bar{g}_k,\bar{d}_k) \right).
  \end{align*}
Thus, by choosing $\beta \in \left(0,\sfrac{1}{\left(2\bar\tau_{-1}\bar{\kappa}_{g\bar{g},d\bar{d}}\right)^{\change{1/\sigma}}}\right]$, 
\begin{equation*}
\begin{aligned}
    \Delta l(x_k,\bar\tau_k,\bar{g}_k,\bar{d}_k) 
    \leq \ &\left(1 + \tfrac{( \kappa_{\bar{\tau}} + \bar\tau_{-1} \kappa_{g\bar{g},d\bar{d}} + \bar\tau_{-1} \bar{\kappa}_{g\bar{g},d\bar{d}} )\beta^{\change{\sigma}}}{1 - \bar\tau_{-1} \bar{\kappa}_{g\bar{g},d\bar{d}} \beta^{\change{\sigma}}}\right)\Delta l(x_k,\tau_k,g_k,d_k) \\
    &+ \tfrac{\bar\tau_{-1}\left( \kappa_{\bar{g}g,\tilde{d}d}\|g_k - \bar{g}_k\|_2\sqrt{\Delta l(x_k,\tau_k,g_k,d_k)} + \change{\kappa_L\|\bar{g}_k - g_k\|_2^2}\right)}{1 - \bar\tau_{-1} \bar{\kappa}_{g\bar{g},d\bar{d}} \beta^{\change{\sigma}}} \\
    \leq \ &\left(1 + 2( \kappa_{\bar{\tau}} + \bar\tau_{-1} \kappa_{g\bar{g},d\bar{d}} + \bar\tau_{-1} \bar{\kappa}_{g\bar{g},d\bar{d}} ) \beta^{\change{\sigma}}\right)\Delta l(x_k,\tau_k,g_k,d_k) \\
    &+ 2\bar\tau_{-1}\left( \kappa_{\bar{g}g,\tilde{d}d}\|g_k - \bar{g}_k\|_2\sqrt{\Delta l(x_k,\tau_k,g_k,d_k)} + \change{\kappa_L\|\bar{g}_k - g_k\|_2^2}\right).
\end{aligned}
\end{equation*}
The first result follows using the definition of~$\kappa_{\overline{\Delta l},\Delta l}$. By \change{Condition}~\ref{ass.stoch_g}, 
\begin{align*}
&\mathbb{E}_k\left[\Delta l(x_k,\bar\tau_k,\bar{g}_k,\bar{d}_k)\right] \\
\leq\ &\mathbb{E}_k\left[ \left(1 + 2( \kappa_{\bar{\tau}} + \bar\tau_{-1} \kappa_{g\bar{g},d\bar{d}} + \bar\tau_{-1} \bar{\kappa}_{g\bar{g},d\bar{d}} ) \beta^{\change{\sigma}}\right)\Delta l(x_k,\tau_k,g_k,d_k) \right. \\
&\left.+ 2\bar\tau_{-1}\left( \kappa_{\bar{g}g,\tilde{d}d}\|g_k - \bar{g}_k\|_2\sqrt{\Delta l(x_k,\tau_k,g_k,d_k)} + \change{\kappa_L\|\bar{g}_k - g_k\|_2^2}\right) \right] \\
\leq \ &\left(1 + 2( \kappa_{\bar{\tau}} + \bar\tau_{-1} \kappa_{g\bar{g},d\bar{d}} + \bar\tau_{-1} \bar{\kappa}_{g\bar{g},d\bar{d}} ) \beta^{\change{\sigma}}\right)\Delta l(x_k,\tau_k,g_k,d_k) \\
&+ 2\bar\tau_{-1}\left( \kappa_{\bar{g}g,\tilde{d}d}\sqrt{\theta_1}\beta^{\change{\sigma}}\Delta l(x_k,\tau_k,g_k,d_k) + \change{\kappa_L\theta_1\beta^{\change{2\sigma}}\Delta l(x_k,\tau_k,g_k,d_k)}\right) \\
\leq\ &\left( 1 + 2\left( \kappa_{\bar{\tau}} + \bar\tau_{-1} (\kappa_{g\bar{g},d\bar{d}} + \bar{\kappa}_{g\bar{g},d\bar{d}} + \kappa_{\bar{g}g,\tilde{d}d}\sqrt{\theta_1} + \change{\kappa_L\theta_1}) \right) \beta^{\change{\sigma}} \right)\Delta l(x_k,\tau_k,g_k,d_k),
\end{align*}
where the second  result follows using the definition of $\bar\kappa_{\overline{\Delta l},\Delta l}$. 
\end{proof}

The next lemma bounds the difference in the merit function after a step.

\begin{lemma}\label{lem:merit_bnd_stoch_old}
    \change{Suppose Assumptions~\ref{ass.main},~\ref{ass.H} and~\ref{ass.residual} and Conditions~\ref{ass.stoch_g},~\ref{ass.stoch_linear_system_old} and~\ref{assum:stoch_redcond_old} hold.} For all $k\in\mathbb{N}$, there exist $\kappa_\phi \in \mathbb{R}_{>0}$ such that 
    \begin{equation*}
        \mathbb{E}_k\left[\phi(x_{k+1},\bar\tau_{k+1}) - \phi(x_k,\bar\tau_k)\right] \leq \mathbb{E}_k\left[(\bar\tau_{k+1} - \bar\tau_k)\right] f_{\inf} -  \beta (\underline{\alpha}\eta - \kappa_{\phi}\beta  ) \Delta l(x_k,\tau_k,g_k,d_k).
    \end{equation*}
\end{lemma}
\begin{proof}
    By Lemmas~\ref{lem.stoch_stepsize_bound_old}, \ref{lem.key_decrease}, and \ref{lem.delta_l_bound_stoch_old}, and \eqref{eq.merit}, it follows that for all $k \in \N{}$
\begin{align*}
      &\mathbb{E}_k\left[\phi(x_{k+1},\bar\tau_{k+1}) - \phi(x_k,\bar\tau_k)\right] \\
      =\ &\mathbb{E}_k\left[\phi(x_{k+1},\bar\tau_{k+1}) - \phi(x_{k+1},\bar\tau_k) + \phi(x_{k+1},\bar\tau_k) - \phi(x_k,\bar\tau_k)\right] \\
      =\ & \mathbb{E}_k\left[(\bar\tau_{k+1} - \bar\tau_k)f_{k+1}\right] + \mathbb{E}_k\left[\phi(x_{k+1},\bar\tau_k) - \phi(x_k,\bar\tau_k)\right] \\
      \leq\ & \mathbb{E}_k\left[(\bar\tau_{k+1} - \bar\tau_k)\right] f_{\inf} \\
      &- \mathbb{E}_k\left[\bar\alpha_k\Delta l(x_k,\tau_k,g_k,d_k) - (1-\eta) \bar\alpha_k \beta^{(\change{\sigma}-1)} \Delta l(x_k,\bar\tau_k,\bar{g}_k,\bar{d}_k) \right]\\
      &+ \mathbb{E}_k\left[\bar\alpha_k \bar\tau_k g_k^T (\bar{d}_k - d_k)\right] + \mathbb{E}_k\left[\bar\alpha_k (\bar\tau_k - \tau_k) g_k^T d_k\right] + \mathbb{E}_k\left[\bar\alpha_k\|J_k(\bar{d}_k - \tilde{d}_k)\|_1 \right]\\
      \leq\ & \mathbb{E}_k\left[(\bar\tau_{k+1} - \bar\tau_k)\right] f_{\inf} - \mathbb{E}_k\left[\bar\alpha_k \Delta l(x_k,\tau_k,g_k,d_k)\right] \\
      &+ \mathbb{E}_k\left[(1-\eta)\bar\alpha_k\beta^{(\change{\sigma}-1)}(1 + \kappa_{\overline{\Delta l},\Delta l}\beta^{\change{\sigma}}) \Delta l(x_k,\tau_k,g_k,d_k) \right]\\
      &+ \mathbb{E}_k\left[(1-\eta)2\alpha_u\beta \bar\tau_{-1}\left(\kappa_{\bar{g}g,\tilde{d}d}\|g_k - \bar{g}_k\|_2\sqrt{\Delta l(x_k,\tau_k,g_k,d_k)} + \change{\kappa_L\|\bar{g}_k - g_k\|_2^2}\right)\right] \\
      &+ \alpha_u\beta^{(2-\change{\sigma})}\mathbb{E}_k\left[\bar\tau_k |g_k^T (\bar{d}_k - d_k)|\right]  \\
      & + \alpha_u\beta^{(2-\change{\sigma})}\mathbb{E}_k\left[ |(\bar\tau_k - \tau_k) g_k^T d_k|\right] + \alpha_u\beta^{(2-\change{\sigma})}\mathbb{E}_k\left[\|J_k(\bar{d}_k - \tilde{d}_k)\|_1 \right].
    \end{align*}
    Continuing from the above, by   Lemmas~\ref{lem.stoch_stepsize_bound_old}, \ref{lem.gT(d_diff)_stoch_old}, \ref{lem.tau_bound_old} and \ref{lem.delta_l_bound_stoch_old}, 
\begin{align*}
      &\mathbb{E}_k\left[\phi(x_{k+1},\bar\tau_{k+1}) - \phi(x_k,\bar\tau_k)\right] \\
      \leq\ & \mathbb{E}_k\left[(\bar\tau_{k+1} - \bar\tau_k)\right] f_{\inf} \\
      & - \mathbb{E}_k\left[\eta\bar\alpha_k \Delta l(x_k,\tau_k,g_k,d_k) - (1-\eta)\kappa_{\overline{\Delta l},\Delta l}\bar\alpha_k\beta^{(\change{2\sigma}-1)} \Delta l(x_k,\tau_k,g_k,d_k) \right] \\
      &+ 2(1-\eta)\alpha_u\beta\bar\tau_{-1}\left( \kappa_{\bar{g}g,\tilde{d}d}\sqrt{\theta_1}\beta^{\change{\sigma}} + \change{\kappa_L\theta_1\beta^{\change{2\sigma}}} \right)\Delta l(x_k,\tau_k,g_k,d_k) \\
      &+ \alpha_u\beta^{2}\bar\tau_{-1} \left((\kappa_{g,d\tilde{d}}\sqrt{\theta_1} + \kappa_{g,\bar{d}d})\Delta l(x_k,\tau_k,g_k,d_k) + \bar\kappa_{g,\bar{d}d}\mathbb{E}_k\left[\Delta l(x_k,\bar\tau_k,\bar{g}_k,\bar{d}_k)\right] \right)  \\
      &+ \alpha_u\beta^{2}\kappa_{\bar{\tau}}\Delta l(x_k,\tau_k,g_k,d_k) + \alpha_u\beta^{2} \bar{\kappa}_{J,\bar{d}\tilde{d}} \mathbb{E}_k \left[\Delta l(x_k,\bar{\tau}_k,\bar{g}_k,\bar{d}_k)\right]\\
      \leq\ & \mathbb{E}_k\left[(\bar\tau_{k+1} - \bar\tau_k)\right] f_{\inf} \\
      & - \eta\underline{\alpha}\beta \Delta l(x_k,\tau_k,g_k,d_k) + (1-\eta)\kappa_{\overline{\Delta l},\Delta l}\alpha_u\beta^{(\change{\sigma}+1)} \Delta l(x_k,\tau_k,g_k,d_k)\\
      &+ 2(1-\eta)\alpha_u\beta^{(\change{\sigma}+1)}\bar\tau_{-1}\left( \kappa_{\bar{g}g,\tilde{d}d}\sqrt{\theta_1} + \change{\kappa_L\theta_1\beta^{\change{\sigma}}} \right)\Delta l(x_k,\tau_k,g_k,d_k) \\
      &+ \alpha_u\beta^{2}\bar\tau_{-1} \left((\kappa_{g,d\tilde{d}}\sqrt{\theta_1} + \kappa_{g,\bar{d}d}) + \bar\kappa_{g,\bar{d}d}\left(1 + \bar\kappa_{\overline{\Delta l},\Delta l}\beta^{\change{\sigma}}\right) \right)\Delta l(x_k,\tau_k,g_k,d_k)  \\
      &+ \alpha_u\beta^{2}\kappa_{\bar{\tau}}\Delta l(x_k,\tau_k,g_k,d_k) + \alpha_u\beta^{2} \bar{\kappa}_{J,\bar{d}\tilde{d}} \left(1 + \bar\kappa_{\overline{\Delta l},\Delta l}\beta^{\change{\sigma}}\right)\Delta l(x_k,\tau_k,g_k,d_k) \\
      \leq\ & \mathbb{E}_k\left[(\bar\tau_{k+1} - \bar\tau_k)\right] f_{\inf} - \beta\eta\underline{\alpha} \Delta l(x_k,\tau_k,g_k,d_k) + \beta^2 \kappa_{\phi} \Delta l(x_k,\tau_k,g_k,d_k). 
  \end{align*}
  The result follows by setting
  \begin{equation*}
    \begin{aligned}
    \kappa_{\phi} := \  &\alpha_u\left((1-\eta)\left(\kappa_{\overline{\Delta l},\Delta l} + 2\bar\tau_{-1}\left( \kappa_{\bar{g}g,\tilde{d}d}\sqrt{\theta_1} + \change{\kappa_L\theta_1} \right)\right) + \kappa_{\bar\tau} \right. \\
    &\left. + \bar\tau_{-1}(\kappa_{g,d\tilde{d}}\sqrt{\theta_1} + \kappa_{g,\bar{d}d}) + (\bar\tau_{-1}\bar\kappa_{g,\bar{d}d} + \bar\kappa_{J,\bar{d}\tilde{d}})\left(1 + \bar\kappa_{\overline{\Delta l},\Delta l}\right) \right) \in \mathbb{R}_{>0}.
    \end{aligned}
    \end{equation*} 
\end{proof}
Lemma~\ref{lem:merit_bnd_stoch_old} provides an upper bound on the change in the merit function across a step. The bound has two terms: the first term is related to the difference in the merit parameter 
across iterations and the second term is negative, conditioned on $\beta$ being sufficiently small, and proportional to the model of the reduction of the merit function. We are now ready to prove the main theorem of this section.

\begin{theorem}\label{thm.main_stochastic_old}
   \change{Suppose Assumptions~\ref{ass.main},~\ref{ass.H} and~\ref{ass.residual} and Conditions~\ref{ass.stoch_g},~\ref{ass.stoch_linear_system_old} and~\ref{assum:stoch_redcond_old} hold.} By choosing $\beta \in \left(0, \min\left\{ \tfrac{1}{(2\theta_3)^{\change{1/\sigma}}},\tfrac{(1-\gamma)\eta\underline{\alpha}}{\kappa_{\phi}},\tfrac{1}{\left(2\tau_{-1}\bar{\kappa}_{g\bar{g},d\bar{d}}\right)^{\change{1/\sigma}}}\right\}\right]$ for any $\gamma\in (0,1)$, 
   \begin{equation*}
       \lim_{k\to\infty}\mathbb{E}\left[\sum_{j=0}^{k-1}\Delta l(x_j,\tau_j,g_j,d_j)\right] < \infty,
   \end{equation*}
   from which it follows that, $\lim_{k\to\infty}\mathbb{E}\left[\Delta l(x_k,\tau_k,g_k,d_k)\right] = 0$.
\end{theorem}
\begin{proof}
    By Lemma~\ref{lem:merit_bnd_stoch_old} and $\beta \in \left(0, \sfrac{(1-\gamma)\eta\underline{\alpha}}{\kappa_{\phi}}\right]$, it follows that 
    \bequation\label{eq.main_dec_stoch_old}
    \baligned
      \mathbb{E}_k\left[\phi(x_{k+1},\bar\tau_{k+1}) - \phi(x_k,\bar\tau_k)\right] \leq\ & \mathbb{E}_k\left[(\bar\tau_{k+1} - \bar\tau_k)\right] f_{\inf} - \beta (\eta\underline{\alpha} - \kappa_{\phi}\beta  )\Delta l(x_k,\tau_k,g_k,d_k)\\
      \leq\ & \mathbb{E}_k\left[(\bar\tau_{k+1} - \bar\tau_k)\right] f_{\inf} - \underline{\alpha}\beta\gamma\eta \Delta l(x_k,\tau_k,g_k,d_k).
  \ealigned
    \eequation
    Applying a telescopic sum to \eqref{eq.main_dec_stoch_old} and taking the total expectation, 
    \begin{equation*}
    \begin{aligned}
        -\infty <\ &\phi_{\inf} - \phi(x_0,\bar\tau_0) \leq \mathbb{E}[\phi(x_k,\bar\tau_k) - \phi(x_0,\bar\tau_0)] \\
        =\ &\mathbb{E}\left[ \sum_{j=0}^{k-1} \left(\phi(x_{j+1},\bar\tau_{j+1}) - \phi(x_j,\bar\tau_j)\right)\right] \\
        \leq\ &\mathbb{E}\left[\sum_{j=0}^{k-1}(\bar\tau_{j+1} - \bar\tau_j)f_{\inf} - \sum_{j=0}^{k-1}\underline{\alpha}\beta\gamma\eta\Delta l(x_j,\tau_j,g_j,d_j)\right] \\
        \leq\ &\bar\tau_{-1}|f_{\inf}| - \underline{\alpha}\beta\gamma\eta\mathbb{E}\left[\sum_{j=0}^{k-1}\Delta l(x_j,\tau_j,g_j,d_j)\right],
    \end{aligned}
    \end{equation*}
    which completes the proof. 
\end{proof}

Theorem~\ref{thm.main_stochastic_old} describes the behavior of the model of the reduction of the merit function evaluated at the iterates generated by the Algorithm~\ref{alg.adaptiveSQP} in expectation. We connect the result of Theorem~\ref{thm.main_stochastic_old} to feasibility and
stationarity measures below.

\bcorollary\label{cor.stoch} 
Under the conditions of Theorem~\ref{thm.main_stochastic_old}, Algorithm~\ref{alg.adaptiveSQP} yields a sequence of iterates $\{(x_k,y_k)\}$ for which
\begin{equation*}
    \lim_{k\to\infty} \mathbb{E}\left[\|d_k\|_2^2 \right] = 0, \; \lim_{k\to\infty} \mathbb{E}\left[\|c_k\|_2 \right] = 0, \; \text{and} \; \lim_{k\to\infty} \mathbb{E}\left[\|g_k + J_k^T(y_k+\delta_k)\|_2 \right] = 0.
\end{equation*}
\ecorollary
\begin{proof} By the deterministic analogue of Lemma~\ref{lem.Psi}, $\Delta l(x_k,\tau_k,g_k,d_k) \geq \kappa_l \tau_{\min} \Psi_k$.
By Theorem~\ref{thm.main_stochastic_old} and the definitions of $\kappa_l$ and $\tau_{\min}$, we have
$\lim_{k\to\infty} \mathbb{E}\left[ \Psi_k \right] = 0$, and thus, our first two results follow from the deterministic variant of Lemma~\ref{lem.Psi_1}. The final result follows from \eqref{eq.system_deterministic}, i.e., $\| g_k + J_k^T(y_k+\delta_k)\|_2 = \|H_kd_k\|_2 \leq \|H_k\|_2 \|d_k\|_2 \leq \kappa_H \|d_k\|_2$, which completes the proof. 
\end{proof}

Corollary~\ref{cor.stoch} shows that in the limit in expectation the search direction, the constraint violation and a first-order
stationarity measure converge to zero.

\begin{remark} 
We make a few remarks about the main theoretical results (Theorem~\ref{thm.main_stochastic_old} and Corollary~\ref{cor.stoch}).
\begin{itemize}[leftmargin=0.25cm]
    \item \textbf{Comparison to determinisic results:} The result in Corollary~\ref{cor.stoch}  is similar, albeit in expectation, to what can be proven for an exact deterministic SQP method, i.e., $\gbar_k = g_k$ and $\dbar_k = d_k$,  under the same assumptions; see e.g., \cite{BeraCurtRobiZhou21,ByrdCurtNoce08}.
    \item \textbf{Comparison to \cite{BeraCurtRobiZhou21}:} 
    The main difference in the result of Corollary~\ref{cor.stoch} and similar results for the stochastic SQP algorithm proposed in \cite{BeraCurtRobiZhou21} pertain to the requirements on the $\{\beta_k\}$ sequence. In \cite{BeraCurtRobiZhou21} (and other works, e.g., \cite{BeraCurtOneiRobi21,CurtOneiRobi21,CurtRobiZhou21}) a diminishing $\{\beta_k\}$ sequence is required to guarantee convergence, whereas in this work convergence with a constant $\{\beta_k\}$ sequence is derived due to the variance reduction achieved. 
    \item \textbf{Comparison to \cite{NaAnitKola21}:} 
    In \cite{NaAnitKola21},  a stochastic line search SQP method for equality constrained problems  that utilizes an exact differentiable merit function is proposed. Under deterministic conditions on the function and derivative approximations and exact solutions to the linear systems, the authors show convergence analogous to that of a deterministic algorithm. We note that under the same deterministic conditions, similar results can be established for our proposed adaptive sampling algorithm.
    \item \textbf{Comparison to \cite{berahas2022accelerating}:} A result analogous to Theorem~\ref{thm.main_stochastic_old} is proven in \cite{berahas2022accelerating}. In both works this is possible due to variance reduction in the approximations employed; the algorithm proposed in \cite{berahas2022accelerating} makes use of predictive variance reduction via SVRG gradients, whereas in this work achieves variance reduction via adaptive sampling.
\end{itemize}
 \end{remark}

The final result we show in this section is a complexity result for our proposed algorithm, i.e., the number of iterations required to achieved an $\epsilon$-accurate solution in expectation. Specifically, we consider the following complexity metric,
\bequation\label{eq.complexity_det}
    \mathbb{E}[\| g_k + J_k^T(y_k + \delta_k)\|_2] \leq \epsilon_L, \quad \text{and} \quad  \mathbb{E}[\|c_k\|_1] \leq \epsilon_c,
\eequation
for $\epsilon_L \in (0,1)$ and $\epsilon_c \in (0,1)$. 
\bcorollary\label{cor.final_results}
Under the conditions of Theorem~\ref{thm.main_stochastic_old}, Algorithm \ref{alg.adaptiveSQP} generates an iterate $ (x_k,y_k) $ that satisfies \eqref{eq.complexity_det} in at most
\begin{equation}\label{eq.maxiters}
    K_\epsilon =  \left( \tfrac{\bar{\tau}_{-1}(f(x_0) - f_{\inf}) + \|c_0\|_1}{\underline{\alpha}\beta \eta \kappa_x}\right)  \max \left\{ \epsilon_L^{-2}, \epsilon_c^{-1}\right\}
\end{equation}
iterations. Moreover, if $\epsilon_L = \epsilon$ and $\epsilon_c = \epsilon^2$, then $K_\epsilon = \mathcal{O}(\epsilon^{-2})$.
\ecorollary
\begin{proof} 
First we show that if $\E[\|g_k + J_k^T(y_k + \delta_k)\|_2] > \epsilon_L$ or $\E[\| c_k\|_1] > \epsilon_c$, then 
\begin{equation}\label{eq.deltal_complexity}
    \E[\Delta l(x_k,\tau_k,g_k,d_k)] \geq \kappa_x \min\{ \epsilon_L^2, \epsilon_c\},
\end{equation}
where $\kappa_x = \min\left\{ \omega_1 , \tfrac{\tau_{\min}\omega_1\epsilon_d}{\kappa_H^2}\right\} \in \mathbb{R}_{>0}$. Consider arbitrary $(k,\epsilon_L,\epsilon_c) \in \N{}\times(0,1)\times(0,1)$ for which $\E[\|g_k + J_k^T(y_k + \delta_k)\|_2] > \epsilon_L$ and/or $\E[\| c_k\|_1] > \epsilon_c$. First, suppose that $\E[\|c_k\|_1] > \epsilon_c$. By the deterministic variant of \eqref{eq.merit_model_reduction_lower_stochastic}, 
\begin{equation*}
    \E[\Delta l(x_k,\tau_k,g_k,d_k)] \geq \E[\tau_k \omega_1 \max\{d_k^TH_kd_k,\epsilon_d \| d_k\|_2^2\} + \omega_1 \|c_k\|_1] \geq \E[\omega_1 \|c_k\|_1] > \omega_1 \epsilon_c.
\end{equation*}
Next, suppose that 
$\E[\|g_k + J_k^T(y_k + \delta_k)\|_2] > \epsilon_L$. By Assumption~\ref{ass.H} and \eqref{eq.system_deterministic}, 
\begin{equation*}
    \epsilon_L < \E[\|g_k + J_k^T(y_k + \delta_k)\|_2] = \E[\|H_k d_k\|_2] \leq \E[\kappa_H \|d_k\|_2],
\end{equation*}
and thus, by the deterministic variant of \eqref{eq.merit_model_reduction_lower_stochastic}, the fact that $\tau_k$ is bounded below and the definition of $\epsilon_d$, it follows that
\begin{equation*}
\begin{aligned}
    \E[\Delta l(x_k,\tau_k,g_k,d_k)] &\geq \E[\tau_k \omega_1 \max\{d_k^TH_kd_k,\epsilon_d \| d_k\|_2^2\} + \omega_1 \|c_k\|_1] \\
    &\geq \E[\tau_{\min}\omega_1\epsilon_d \| d_k\|_2^2] > \tfrac{\tau_{\min}\omega_1\epsilon_d}{\kappa_H^2} \epsilon_L^2.
\end{aligned}
\end{equation*}
Combining the results of the two cases and using the definition of $\kappa_x$ yields \eqref{eq.deltal_complexity}. If \eqref{eq.complexity_det} is violated, then by Lemma~\ref{lem.stoch_stepsize_bound_old}, \eqref{eq.merit}, \eqref{eq.main_dec_stoch_old} and \eqref{eq.deltal_complexity}
\begin{equation*}
\begin{aligned}
    &\mathbb{E}[\bar\tau_{-1} (f(x_0) - f_{\inf}) + \|c_0\|_1] 
    \geq\  \mathbb{E}[\bar\tau_0 (f(x_0) - f_{\inf}) + \|c_0\|_1] \\
    \geq\ & \E\left[\bar\tau_0 f(x_0) + \|c_0\|_1 - \bar\tau_{k} f(x_{k}) - \|c_{k}\|_1 + (\bar\tau_{k} - \bar\tau_0) f_{\inf}\right]\\
    =\ & \E\left[\sum_{j=0}^{k-1}\left(\phi(x_{j},\bar\tau_j) - \phi(x_{j+1},\bar\tau_{j+1}) + (\bar\tau_{j+1} - \bar\tau_j) f_{\inf} \right)\right] \\
    \geq \ & \E \left[\sum_{j=0}^{k-1} \bar{\alpha}_j \eta \Delta l(x_j,\tau_j,g_j,d_j)\right] \geq \E \left[\sum_{j=0}^{k-1}\underline{\alpha} \beta \eta \Delta l(x_j,\tau_j,g_j,d_j)\right] \\
    \geq \ &k\underline{\alpha} \beta \eta \kappa_x \min \left\{ \epsilon_L^2, \epsilon_c\right\},
\end{aligned}
\end{equation*}
which implies that $k$ is bounded above by \eqref{eq.maxiters}. 
\end{proof}
The result of Corollary~\ref{cor.final_results} is similar to that of determinsitc SQP methods, albeit in expectation,  under the same assumptions; see e.g., \cite[Theorem 1]{CurtOneiRobi21}. To the best of our knowledge, this is the first time that such complexity results have been derived.

\subsection{Predetermined Sublinear Errors}\label{sec.pred}
\change{In this subsection, we consider conditions that control the errors at predetermined rates and provide results similar to those in Section~\ref{sec.adaptive} and that are implementable. From Theorem~\ref{thm.main_stochastic_old} and Corollary~\ref{cor.final_results}, we have that the quantity $\Delta l (x_k, \tau_k,g_k,d_k)$ goes to zero at a sublinear rate (in expectation). As such, one can expect to be able to derive similar results by replacing this quantity in  Conditions~\ref{ass.stoch_g} and~\ref{ass.stoch_linear_system_old} with a term that goes to zero at a sublinear rate as the algorithm progresses. Motivated by this observation, we consider the setting in which the accuracy in the gradient estimation is increased at a sublinear rate and the accuracy in the linear system solves is increased at a sublinear rate, to achieve convergence to stationarity 
in expectation.}
The reason we include this result is to emphasize that predetermined sampling strategies and error sequences suffice to provide convergence guarantees without the need for \change{Conditions~\ref{ass.stoch_g} and~\ref{ass.stoch_linear_system_old} that require accessing $\Delta l(x_k,\tau_k,g_k,d_k)$}. Also, this result provides guidance on the total \change{sample complexity}. For brevity, we only introduce the assumptions and present the main theoretical results\footnote{\change{We refer interested readers to an online pre-print of the manuscript for the full technical results and proofs: \url{https://arxiv.org/pdf/2206.00712.pdf}.}}. 

The two assumptions below are analogues of \change{Conditions}~\ref{ass.stoch_g} and \ref{ass.stoch_linear_system_old}.
\begin{condition}\label{ass.stoch_g_sublin}
  For all $k \in \mathbb{N}$, the stochastic gradient estimate $\bar{g}_k \in \mathbb{R}^n$ satisfies, $\mathbb{E}_k\left[\|\bar{g}_k - g_k\|_2^2\right] \leq \tfrac{\theta_1 \beta^{\change{2\sigma}}}{(k+1)^{\nu}}$,
  where $\theta_1\in\mathbb{R}_{>0}$, $\beta\in (0,1)$, $\nu \in \mathbb{R}_{>1}$ and \change{$\sigma \in [1,2]$}. 
  Additionally, for all $k \in \N{}$, the stochastic gradient estimate $\gbar_k \in \R{n}$ is an unbiased estimator of the gradient of $f$ at $x_k$, i.e., $\mathbb{E}_k \left[ \gbar_k \right] = g_k$.
\end{condition}

\begin{condition}\label{ass.stoch_linear_system}
  For all $k \in \mathbb{N}$, the search directions $(\bar{d}_k,\bar{\delta}_k) \in \mathbb{R}^n \times \mathbb{R}^m$ in \eqref{eq.system_stochastic} $($inexact solutions to \eqref{eq.system}
  $)$ satisfy, $\left\|\begin{bmatrix} \tilde{d}_k \\ \tilde{\delta}_k \end{bmatrix} - \begin{bmatrix} \bar{d}_k \\ \bar{\delta}_k \end{bmatrix}\right\|_2^2 \leq \tfrac{\theta_2 \beta^{\change{2\sigma}}}{(k+1)^{\nu}}$, 
  where $\theta_2\in\mathbb{R}_{>0}$, $\beta\in (0,1)$, $\nu \in \mathbb{R}_{>1}$ and \change{$\sigma \in [1,2]$}. Note, $ (\widetilde{d}_k,\widetilde{\delta}_k)$ and $(\dbar_k,\bar{\delta}_k)$ are the exact and inexact solutions of \eqref{eq.system}
, respectively.
\end{condition}

Next, we state the main results of this subsection. Lemma~\ref{lem:merit_bnd_stoch_mp} (analogue of Lemma~\ref{lem:merit_bnd_stoch_old}) bounds the difference of the merit function across iterations. 
\begin{lemma}\label{lem:merit_bnd_stoch_mp}
    \change{Suppose Assumptions~\ref{ass.main},~\ref{ass.H} and~\ref{ass.residual} and Conditions~\ref{ass.stoch_g_sublin},~\ref{ass.stoch_linear_system} and~\ref{assum:stoch_redcond_old} hold.} For all $k\in\mathbb{N}$, there exist  $(\bar{\kappa}_\phi,\bar{\kappa}_{\phi,\nu}) \in \mathbb{R}_{>0} \times  \mathbb{R}_{>0}$ 
    such that
    \begin{equation*}
    \begin{aligned}
        &\mathbb{E}_k\left[\phi(x_{k+1},\bar\tau_{k+1}) - \phi(x_k,\bar\tau_k)\right] \\
        \leq\ & \mathbb{E}_k\left[(\bar\tau_{k+1} - \bar\tau_k)\right] f_{\inf} 
         -  \beta (\underline{\alpha}\eta - \bar{\kappa}_\phi\beta  ) \Delta l(x_k,\tau_k,g_k,d_k) 
          + \beta^2 \tfrac{\bar{\kappa}_{\phi,\nu}}{(k+1)^{\nu}}.
    \end{aligned}
    \end{equation*}
\end{lemma}
Lemma~\ref{lem:merit_bnd_stoch_mp} has an additional \change{term} as compared to Lemma~\ref{lem:merit_bnd_stoch_old}. This is due to the fact that we control the gradient error and linear system accuracy at a sublinear rate instead of controlling it relative to the algorithmic progress, as measured in terms of $\Delta l(x_k,\tau_k,g_k,d_k)$. That being said, the additional term is proportional to $\beta^2$ and whose accumulation is finite in the limit.

Theorem~\ref{thm.main_stochastic_mp} and Corollary~\ref{cor.iter_complexity_raghu} (analogues of Theorem~\ref{thm.main_stochastic_old} and Corollary~\ref{cor.final_results}, respectively) provide  convergence and iteration complexity results, respectively. Corollary~\ref{cor.iter_complexity_raghu} also provides sample complexity results.
\begin{theorem}\label{thm.main_stochastic_mp}
   \change{Suppose Assumptions~\ref{ass.main},~\ref{ass.H} and~\ref{ass.residual} and Conditions~\ref{ass.stoch_g_sublin},~\ref{ass.stoch_linear_system} and~\ref{assum:stoch_redcond_old} hold.} By choosing $\beta \in \left(0, \min\left\{\sfrac{1}{(2\theta_3)^{\change{\tfrac{1}{\sigma}}}},\sfrac{(1-\gamma)\eta\underline{\alpha}}{\bar{\kappa}_{\phi}},\sfrac{1}{\left(2\bar\tau_{-1}\bar\kappa_{\bar{g},\bar{d}\tilde{d}}\right)^{\change{\tfrac{1}{\sigma}}}}\right\}\right]$ for any $\gamma\in (0,1)$ and $\nu \in \mathbb{R}_{>1}$,  $\lim_{k\to\infty}\mathbb{E}\left[\sum_{j=0}^{k-1}\Delta l(x_j,\tau_j,g_j,d_j)\right] < \infty$, 
   from which it follows that,  $\lim_{k\to\infty}\mathbb{E}\left[\Delta l(x_k,\tau_k,g_k,d_k)\right] = 0$.
\end{theorem}

The conclusion of Theorem~\ref{thm.main_stochastic_mp} is the same as that of Theorem~\ref{thm.main_stochastic_old}, up to constants, and, as a result, Corollary~\ref{cor.stoch} holds for Theorem~\ref{thm.main_stochastic_mp}.

\bcorollary\label{cor.iter_complexity_raghu}
Under the conditions of Theorem~\ref{thm.main_stochastic_mp}, Algorithm~\ref{alg.adaptiveSQP} generates an iterate $ (x_k,y_k) $ that satisfies
\eqref{eq.complexity_det} in at most $K_\epsilon =\mathcal{O}\left(\max \left\{ \epsilon_L^{-2}, \epsilon_c^{-1}\right\}\right)$ iterations and $W_\epsilon = \mathcal{O}\left(\left(\max \left\{ \epsilon_L^{-2}, \epsilon_c^{-1}\right\}\right)^{(\nu +1)}\right)$  stochastic gradient evaluations $(\nu \in \mathbb{R}_{>1})$. Moreover, if $\epsilon_L = \epsilon$ and $\epsilon_c = \epsilon^2$, then $K_\epsilon = \mathcal{O}\left(\epsilon^{-2}\right)$ and $W_\epsilon = \mathcal{O}(\epsilon^{-2(\nu +1)})$.
\ecorollary
Corollary~\ref{cor.iter_complexity_raghu} shows that one can achieve the same iteration complexity as the deterministic variant of the algorithm under \change{Conditions}~\ref{ass.stoch_g_sublin} and ~\ref{ass.stoch_linear_system} (and other assumptions stated earlier) at an increased overall sample complexity. 

\section{\change{A} Practical \change{Adaptive, Inexact and Stochastic SQP} Method}\label{sec.practical}

\looseness=-1

\change{In this section, we present our proposed practical adaptive  inexact stochastic SQP method (\PAISSQP). We describe the sample size selection mechanism, iterative linear system solver and early termination conditions employed.}


\subsection{Sample Size Selection}
We describe the mechanism by which the sample size is selected at every iteration. 
Condition \eqref{eq.norm_cond_stoch} involves computing population variances and deterministic quantities which are not available in our setting, and possibly requires solving multiple linear systems. 
That being said, one can approximate these quantities with sample variances and sampled stochastic counterparts of the deterministic quantities required following the ideas proposed in \cite{ByrdChinNoceWu12,BollByrdNoce18}. 



Condition~\eqref{eq.norm_cond_stoch} is approximated as follows. Let $\bar{g}_k \in \R{n}$ be defined as
\bequation\label{eq.grad_s_k}
    \bar{g}_k := \tfrac{1}{|\mathcal{S}_k|} \sum_{i \in \mathcal{S}_k} \nabla F(x_k,\xi_i),
\eequation
where $\mathcal{S}_k$ is  a set consisting of indices drawn at random from the distribution of $\xi$, and $\xi_i$ is a realization of $\xi$. The left-and-side of  \eqref{eq.norm_cond_stoch} can be expressed as\footnote{\change{We should note that if $f$ is of finite sum structure, i.e., $f(x) = \tfrac{1}{N}\sum f_i(x)$, and the component functions are sampled without replacement, this results in a lower variance compared to \eqref{eq.grad_approx} by a factor of $1 - \tfrac{|S_k|}{N}$. This factor goes to zero as $|S_k| \rightarrow N$; see \cite{ByrdChinNoceWu12}.}}
\bequation\label{eq.grad_approx}
    \E_k\left[\| \bar{g}_k - \nabla f(x_k)\|_2^2\right] = \tfrac{\E_k\left[ \|\nabla F(x_k,\xi) - \nabla f(x_k)\|_2^2\right]}{|S_k|}.
\eequation
Computing the population variance on the right-hand-side of \eqref{eq.grad_approx} is prohibitively expensive in our setting, and thus we approximate it with the sample variance (see left-hand-side of \eqref{eq.practical2}). Moreover, the right-hand-side of  \eqref{eq.norm_cond_stoch} is approximated with its stochastic counter-part. This results in the following approximation to condition~\eqref{eq.norm_cond_stoch},
\bequation\label{eq.practical2}
   \tfrac{\text{Var}_{i \in \mathcal{S}_k}[\nabla F(x_k,\xi_i)]}{|S_k|} \leq \theta_1 \beta^{\change{2\sigma}}\Delta l(x_k,\bar\tau_k,\bar g_k,\bar d_k),
\eequation
where $\text{Var}_{i \in \mathcal{S}_k}[\nabla F(x_k,\xi_i)] = \frac{1}{|\mathcal{S}_k| - 1} \sum_{i \in \mathcal{S}_k} \|\nabla F(x_k,\xi_i) - \bar{g}_k\|_2^2$, used in \PAISSQP{}. 

In our practical algorithm (Algorithm~\ref{alg.adaptiveSQP_practical}), if inequality \eqref{eq.practical2} is not satisfied (given the set $\mathcal{S}_k$), we \change{choose a new sample $\hat{\mathcal{S}}_k$ with a larger sample size with the intent of satisfying \eqref{eq.practical2}.} 
The heuristic we propose to do this is as follows. Suppose we wish to find a \change{new} larger sample $\hat{\mathcal{S}}_k$ ($|\hat{\mathcal{S}}_k| > |\mathcal{S}_k|$) \change{that satisfies \eqref{eq.practical2} with $\hat{\mathcal{S}}_k$}, and let us assume that the change in sample size is gradual enough that for any $x_k$, such that $\text{Var}_{i \in \mathcal{S}_k}[\nabla F(x_k,\xi_i)] \approx \text{Var}_{i \in \hat{\mathcal{S}}_k}[\nabla F(x_k,\xi_i)]$ and $\Delta l(x_k,\bar{\tau}_k,\bar{g}_k,\bar{d}_k) \approx \Delta l(x_k,\hat{\tau}_k,\hat{g}_k,\hat{d}_k)$, 
where $\hat{\tau}_k,\hat{g}_k,\hat{d}_k$ are stochastic realizations of these quantities computed based on the sample $\hat{\mathcal{S}}_k$. Under this assumption, it is clear that \eqref{eq.practical2} is satisfied if 
\bequation\label{eq.increase_cond}
    |\hat{\mathcal{S}}_k| \geq \left\lceil \tfrac{\text{Var}_{i \in \mathcal{S}_k}[\nabla F(x_k,\xi_i)]}{\theta_1 \beta^{\change{2\sigma}}\Delta l(x_k,\bar\tau_k,\bar g_k,\bar d_k)} \right\rceil.
\eequation
\change{
Finally, to further reduce the computational efforts, we use $\hat{\mathcal{S}}_k$ for the next iteration instead of the current iterate. That is, we set $\mathcal{S}_{k+1} = \hat{\mathcal{S}}_k$.}
The ideas above are used in Algorithm~\ref{alg.adaptiveSQP_practical}.

\subsection{Inexact Linear System Solutions}\label{sec.system}

Our proposed practical algorithm makes use of an iterative solver with \emph{early termination tests} to solve the Newton-SQP linear system  \eqref{eq.system_stochastic}. We use the minimum residual (MINRES) method \cite{ChoiPaigSaun11,PaigSaun75}, with early termination conditions, for solving the system inexactly, but note that other iterative algorithms could also be used. For all $k \in \N{}$, let $\{(\bar{d}_{k,t},\bar\delta_{k,t},\bar\rho_{k,t},\bar{r}_{k,t})\}_{t\in\mathbb{N}}$ denote the steps (and residuals) generated in iteration  $t\in\mathbb{N}$ of MINRES, and $(\bar{d}_k,\bar{r}_k,\bar\rho_{k},\bar{r}_{k}) \leftarrow (\bar{d}_{k,t'},\bar{r}_{k,t'},\bar\rho_{k,t'},\bar{r}_{k,t'})$ where $t'$ is the last MINRES iteration. 
The MINRES method \change{is} terminated for the minimum $t$ such that either condition $(a)$
\begin{equation}\label{cond1}
    \text{\eqref{eq.merit_model_reduction_lower_stochastic}} \; \text{ and } \; 
  \| \bar{r}_k\|_1 \leq \omega_a\beta^{\change{\sigma}} \Delta l (x_k,\bar\tau_{k},\bar{g}_k,\bar{d}_k) \;  \textbf{ with } \; \bar\tau_k = \bar\tau_{k-1}
\end{equation}
or, condition $(b)$ 
\begin{equation}\label{cond2}
    \|\bar{r}_k\|_1 < \min\{(1-\omega_1)\omega_2, \omega_1\omega_a\beta^{\change{\sigma}}\} \|c_k\|_1 \;\; \textbf{  and } \;\; \|\bar\rho_k\|_1 < \omega_b\|c_k\|_1
\end{equation}  
hold. These conditions are inspired by the theory, but relaxed for practicality. Specifically, the additional condition in case $(a)$ is not checked and neither is \change{Condition}~\ref{assum:stoch_redcond_old}.





\subsection{\PAISSQP{}} In this section, we present our practical algorithm \PAISSQP.

\begin{algorithm}[H]
  \caption{(\PAISSQP) Practical, Adaptive, Inexact, Stochastic SQP Algorithm\label{alg.adaptiveSQP_practical}}
  \begin{algorithmic}[1]
  \Require $x_0\in\mathbb{R}^n$; $y_0\in\mathbb{R}^m$; $\{H_k\}\subset\mathbb{S}^n$; $\bar\tau_{-1}\in\mathbb{R}_{>0}$; $\hat{\mathcal{S}}_{-1}\subset\mathbb{N}$;  $\{\omega_1,\omega_2,\eta,\epsilon_{\tau}\} \subset (0,1)$; $\omega_a\in\mathbb{R}_{>0}$;  $\omega_b\in\R{}_{>0}$; $\beta\in (0,1]$; $\alpha_u \in \mathbb{R}_{>0}$; $\epsilon_d\in (0,\sfrac{\zeta}{2})$; \change{$\sigma\in [1,2]$}
  \For{\textbf{all} $k \in \mathbb{N}$}
  \State Compute  $\bar{g}_k$ via \eqref{eq.grad_s_k} with $\mathcal{S}_k = \hat{\mathcal{S}}_{k-1}$
  \State Solve \eqref{eq.system} iteratively using MINRES; 
  
  Compute a step $(\bar{d}_k,\bar{\delta}_k)$ that satisfies either \eqref{cond1} or \eqref{cond2} \label{line:conditions_practical}
  \State Update $\bar\tau_k$ via \eqref{eq.merit_parameter_trial}--\eqref{eq.tau_update}\label{line.tau_update_practical}
  \State Compute a step size $\bar\alpha_k$ via \eqref{eq:alpha_min_stoch}--\eqref{eq:alpha_opt_stoch}\label{step.alpha_stochastic1}
  \State Update $x_{k+1}\leftarrow x_k + \bar{\alpha}_k \bar{d}_k$, and  $y_{k+1}\leftarrow y_k + \bar{\alpha}_k\bar\delta_k$ 
  \State Choose a \change{new} sample $\hat{\mathcal{S}}_{k}$ such that $|\hat{\mathcal{S}}_{k}| = |\mathcal{S}_{k}|$
  \State  \textbf{If} condition \eqref{eq.practical2} is not satisfied, augment $\hat{\mathcal{S}}_{k}$ using formula \eqref{eq.increase_cond}
  \EndFor
  \end{algorithmic}
\end{algorithm}


\section{Numerical Results}\label{sec.num_res}

The main goal of this section is to illustrate the efficiency of our proposed practical algorithm (Algorithm~\ref{alg.adaptiveSQP_practical}). We characterize efficiency in terms of two metrics that capture the major costs in solving \eqref{prob.opt}. The first metric is the number of objective function gradient evaluations (or accessed data points in the context of the machine learning problems presented below). The second metric is the number of iterations of an iterative solver used to solve the linear system \eqref{eq.system}. To illustrate the efficiency of Algorithm~\ref{alg.adaptiveSQP_practical}, we present results on two classes of problems, constrained classification problems that arise in machine learning (Section~\ref{sec.log_reg}) and standard CUTE collection of nonlinear optimization problems (Section~\ref{sec.cute}), and compare exact and inexact variants (all implementation details are given in Section~\ref{sec.impl_dets}).

In both Sections~\ref{sec.log_reg} and \ref{sec.cute}, results are given in terms of feasibility and stationarity errors defined as, $\|c_k\|_{\infty}$ and $\| g_k + J_k^Ty_k \|_{\infty}$, respectively, where the vector $y_k \in \mathbb{R}^m$ is computed as a least-squares multiplier \change{i.e., $y_k = \arg\min_{y\in\mathbb{R}^m}\|g_k + J_k^Ty\|_2^2$}. Moreover, in both sections we terminate all algorithms solely due to iteration, sampled gradient evaluation or linear system iteration budgets. 

\subsection{Implementation Details}
\label{sec.impl_dets}

We compare different variants of Algorithm~\ref{alg.adaptiveSQP_practical}. Specifically, we compare variants with different levels of accuracy in the gradient approximations employed, as well as variants with and without the early termination conditions. Precise \change{characterizations} of the accuracy levels are given in Sections~\ref{sec.log_reg} and~\ref{sec.cute}. For all variants MINRES was used to solve the linear systems, and all variants employed the same adaptive step size selection strategy described by~\eqref{eq:alpha_min_stoch}--\eqref{eq:alpha_opt_stoch}. For all problems estimates of $L$ and $\Gamma$ were computed using gradient differences around the initial point, and kept constant throughout the course of optimization. This procedure was performed in such a way that, for each problem instance, all algorithms used the same values for these
estimates. For all algorithms, $H_k \gets I$ for all $k \in \mathbb{N}$, $\bar{\tau}_{-1} \gets 1$, $\beta \gets 1$, $\alpha_u \gets 10^2$, $\eta \gets \sfrac{1}{2}$, $\omega_1 \gets \sfrac{1}{2}$, $\omega_2 \gets \sfrac{1}{2}$, $\omega_a \gets 10^2$, $\omega_b \gets 10^2$, $\epsilon_{\tau} \gets 10^{-4}$, and \change{$\sigma \gets 1$}. For \PAISSQP{}, we additionally set $\theta_1 \gets 0.99$. (Note, $\theta_2$ and $\theta_3$ are not required by the \PAISSQP{} algorithm.) The initial sample size $|\hat{\mathcal{S}}_{-1}|$ is defined in Sections~\ref{sec.log_reg} and \ref{sec.cute}. For variants that solve the linear system without the early termination conditions, the termination tolerance used in MINRES was $\epsilon_{\text{MINRES}} \gets 10^{-8}$, in order to obtain accurate solutions.

\subsection{Constrained Logistic Regression}
\label{sec.log_reg}

In our first set of experiments, we consider constrained logistic regression problems,
\bequation\label{eq.logistic}
  \min_{x \in \R{n}}\ f(x) = \frac{1}{N} \sum_{i=1}^N \log\left(1 + e^{-y_i(X_i^Tx)} \right)\ \ \st\ \ Ax = b_1,\ \ \|x\|_2^2 = b_2,
\eequation
where \change{$N$ denotes the number of data points (samples),} $X \in \mathbb{R}^{n \times N}$ is the data matrix, $X_i$ is the $i$th column of matrix $X$, $y \in \{ -1,1\}^N$ contains corresponding label data, $A \in \mathbb{R}^{m \times n}$, $b_1 \in \mathbb{R}^{m}$ and $b_2 \in \mathbb{R}$. We present results on two data sets from the LIBSVM collection~\cite{chang2011libsvm}; \texttt{australian} and \texttt{mushroom}.  
For the linear constraints, the data was generated as follows: the entries of the matrix $A$ and the vector $b_1$ were drawn from a standard normal distribution (for each data set the same $A$ and $b_1$ were used for all methods)\change{, with $m=10$}. For the $\ell_2$-norm constraint, $b_2 =1$. For all problems and algorithms, the initial primal iterate was set to the vector of all ones of appropriate dimension, and the initial dual variables $y_0$ were set as the least-squares multipliers. 

For each data set, we consider exact and inexact (linear system solutions) variants, i.e., variants with and without the early termination conditions, and three different sample sizes ($|\mathcal{S}_k| = |\mathcal{S}| \in \{2,128,N\}$ for all $k\in\mathbb{N}$) for a total of 6 variants, and compare against \PAISSQP{}. \change{For all methods, sampling was done without replacement.} A budget of $50$ epochs was given to every method. The results for the two data sets are presented in Figures~\ref{fig.australian} and \ref{fig.mushroom}. For every method, we report the feasibility and stationarity errors in terms of iterations, epochs (gradient evaluations) and linear system iterations. The results indicate that our proposed practical inexact SQP method \PAISSQP{} strikes a good balances between reducing constraint violation while attempting to find a point that satisfies approximate  first-order stationarity across all three evaluation metrics. In Figure~\ref{fig.step_batch} we show the sample size and step size selected by the different variants. While the sample sizes increase relatively quickly, there are significant savings that can be achieved by employing inexact information. The step size figures illustrate that the adaptive step size mechanism \eqref{eq:alpha_min_stoch}--\eqref{eq:alpha_opt_stoch} is stable.

\begin{figure}[]
    \centering
    \begin{subfigure}[b]{0.32\textwidth}
    \includegraphics[width=\textwidth,clip=true,trim=30 180 50 200]{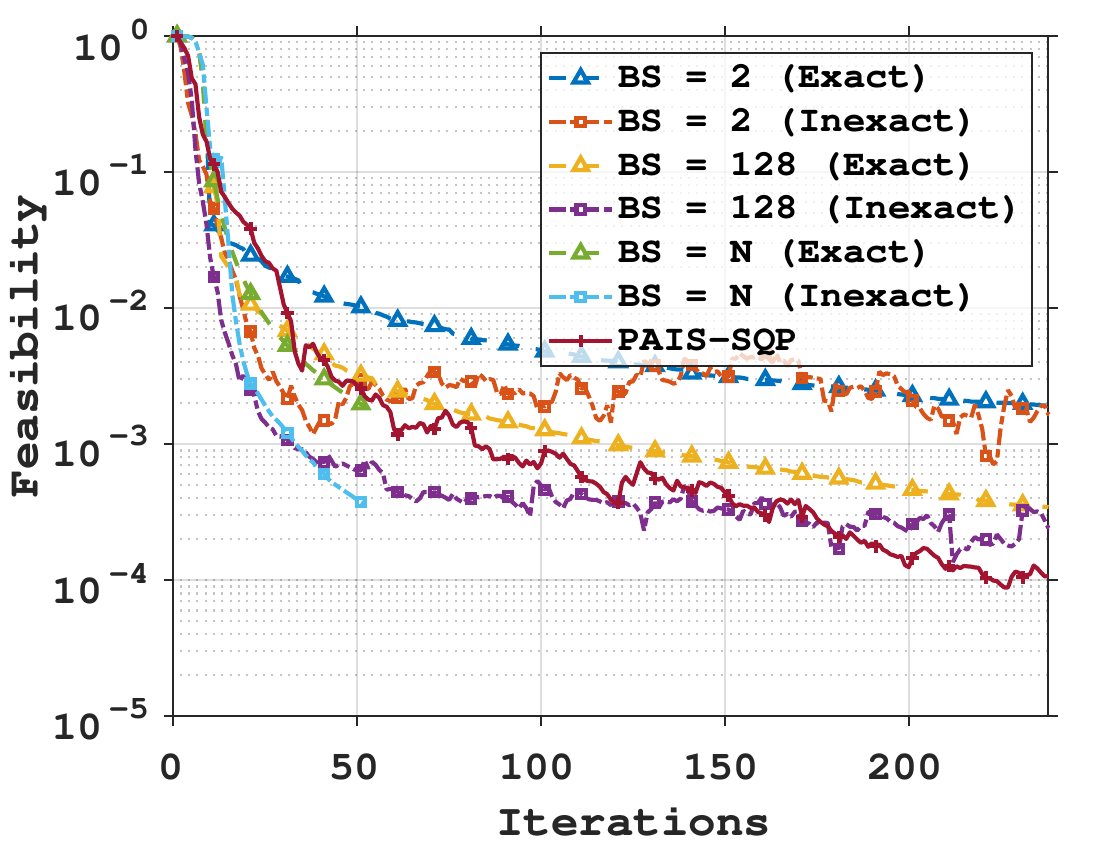}
    \caption{Feasibility vs. Iterations}
    \end{subfigure}
    \begin{subfigure}[b]{0.32\textwidth}
    \includegraphics[width=\textwidth,clip=true,trim=30 180 50 200]{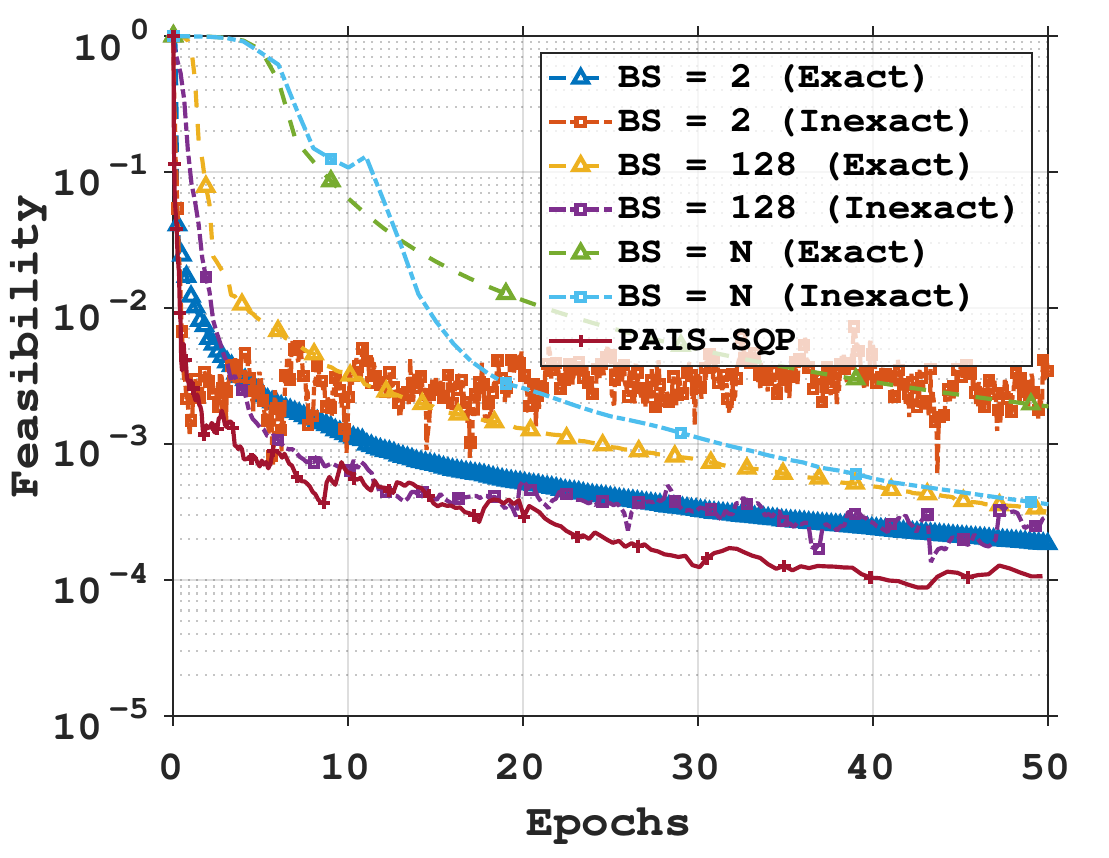}
    \caption{ Feasibility vs. Epochs} 
    \end{subfigure}
    \begin{subfigure}[b]{0.32\textwidth}
    \includegraphics[width=\textwidth,clip=true,trim=30 180 50 200]{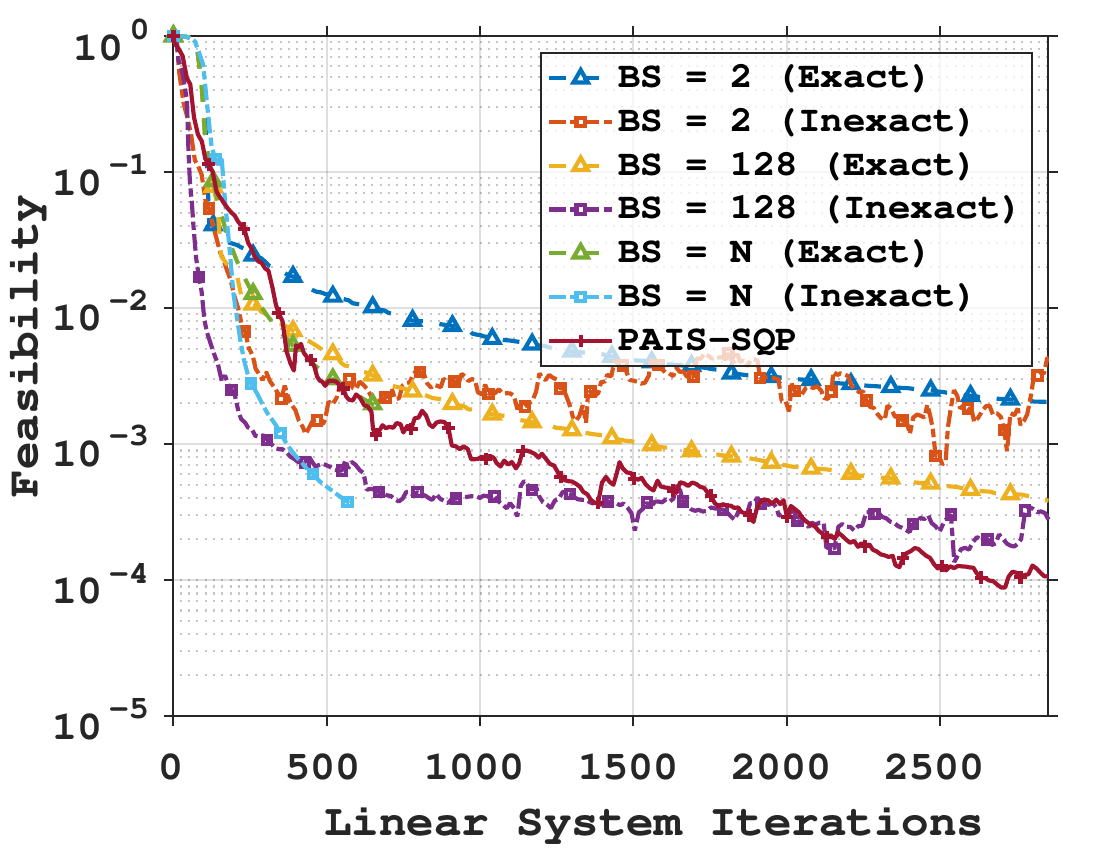}
    \caption{ Feasibility vs. LS  Iters} 
    \end{subfigure}
  
    \begin{subfigure}[b]{0.32\textwidth}
    \includegraphics[width=\textwidth,clip=true,trim=30 180 50 200]{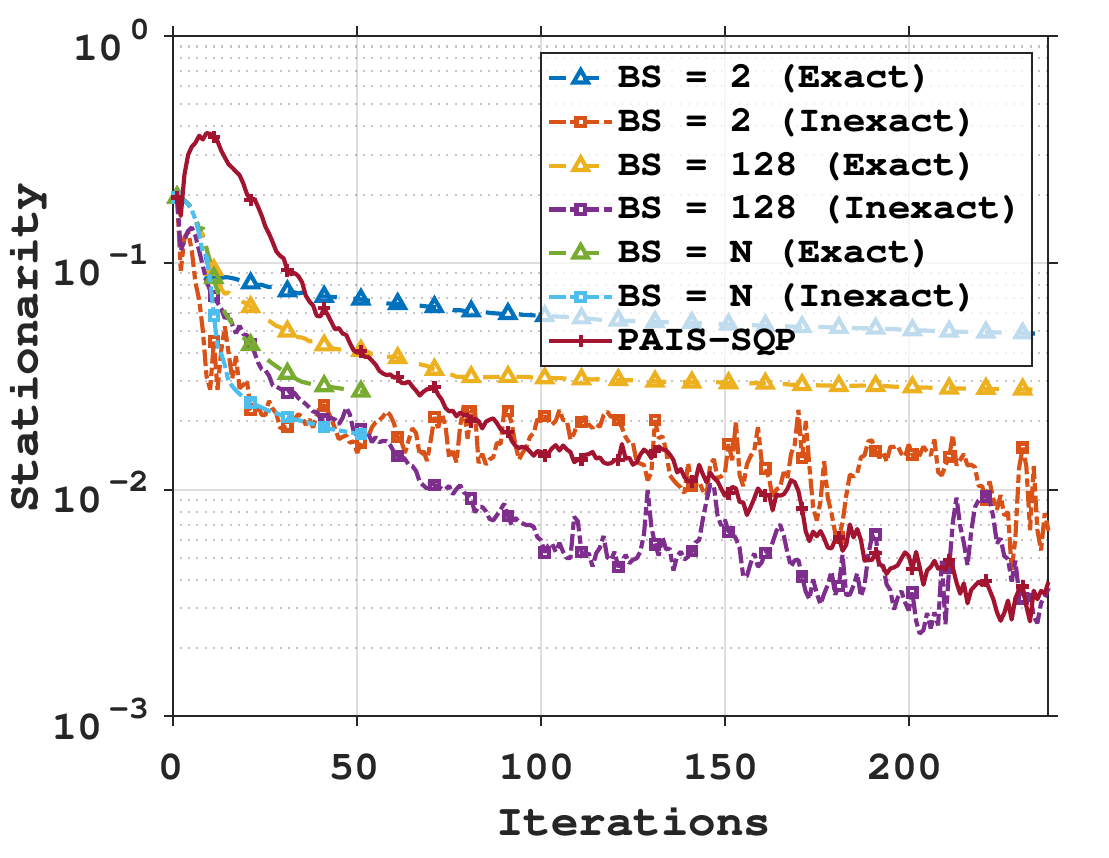}
    \caption{Stationarity vs. Iterations}
    \end{subfigure}
    \begin{subfigure}[b]{0.32\textwidth}
    \includegraphics[width=\textwidth,clip=true,trim=30 180 50 200]{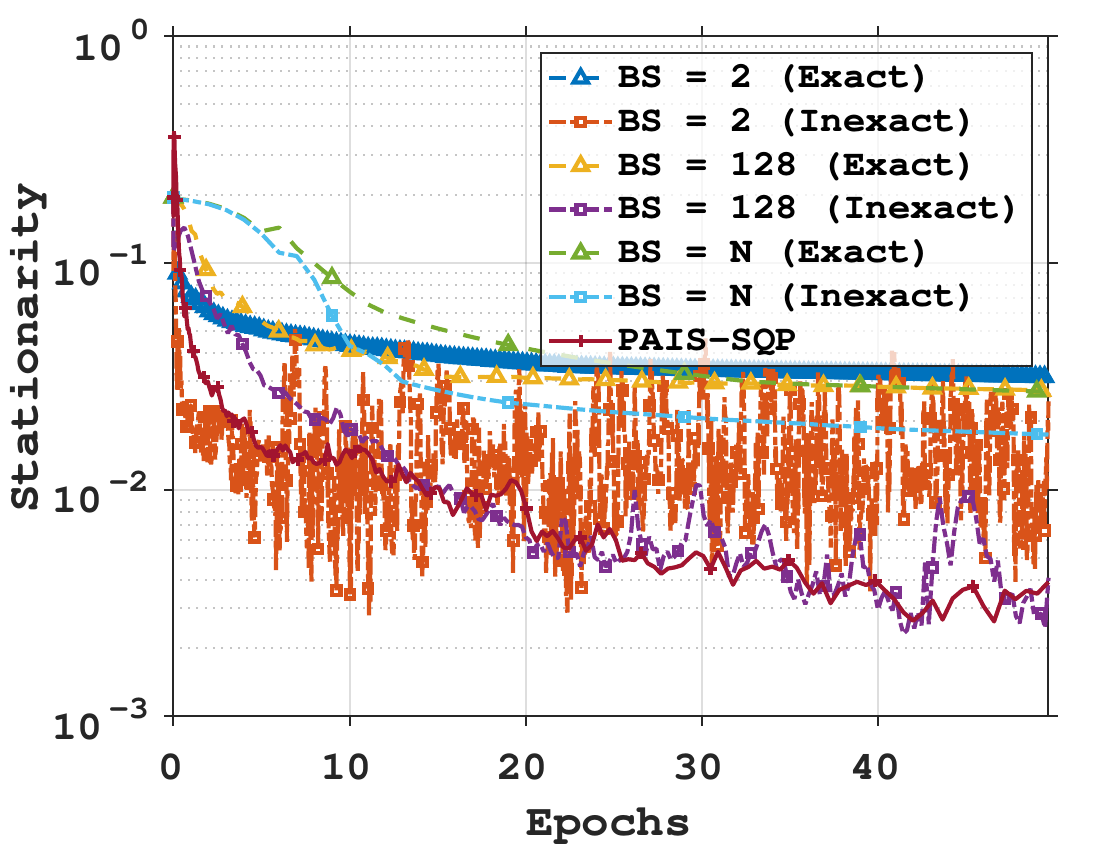}
    \caption{ Stationarity vs. Epochs}
    \end{subfigure}
    \begin{subfigure}[b]{0.32\textwidth}
    \includegraphics[width=\textwidth,clip=true,trim=30 180 50 200]{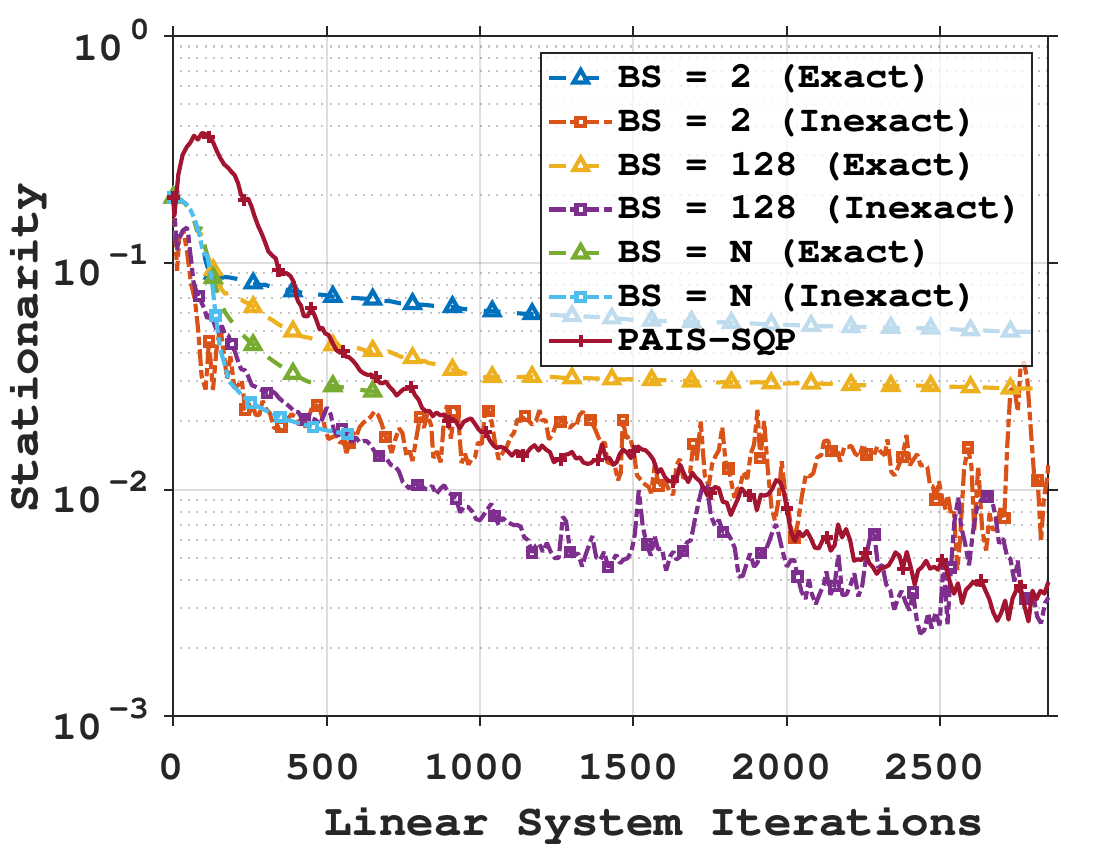}
    \caption{ Stationarity vs. LS  Iters} 
    \end{subfigure}
    \caption{\texttt{australian}: Feasibility \& stationarity errors versus iterations/epochs/linear system iterations for exact and inexact variants of Algorithm~\ref{alg.adaptiveSQP_practical} on \eqref{eq.logistic}.}\label{fig.australian}
\end{figure}

\begin{figure}[]
    \centering
    \begin{subfigure}[b]{0.32\textwidth}
    \includegraphics[width=\textwidth,clip=true,trim=30 180 50 200]{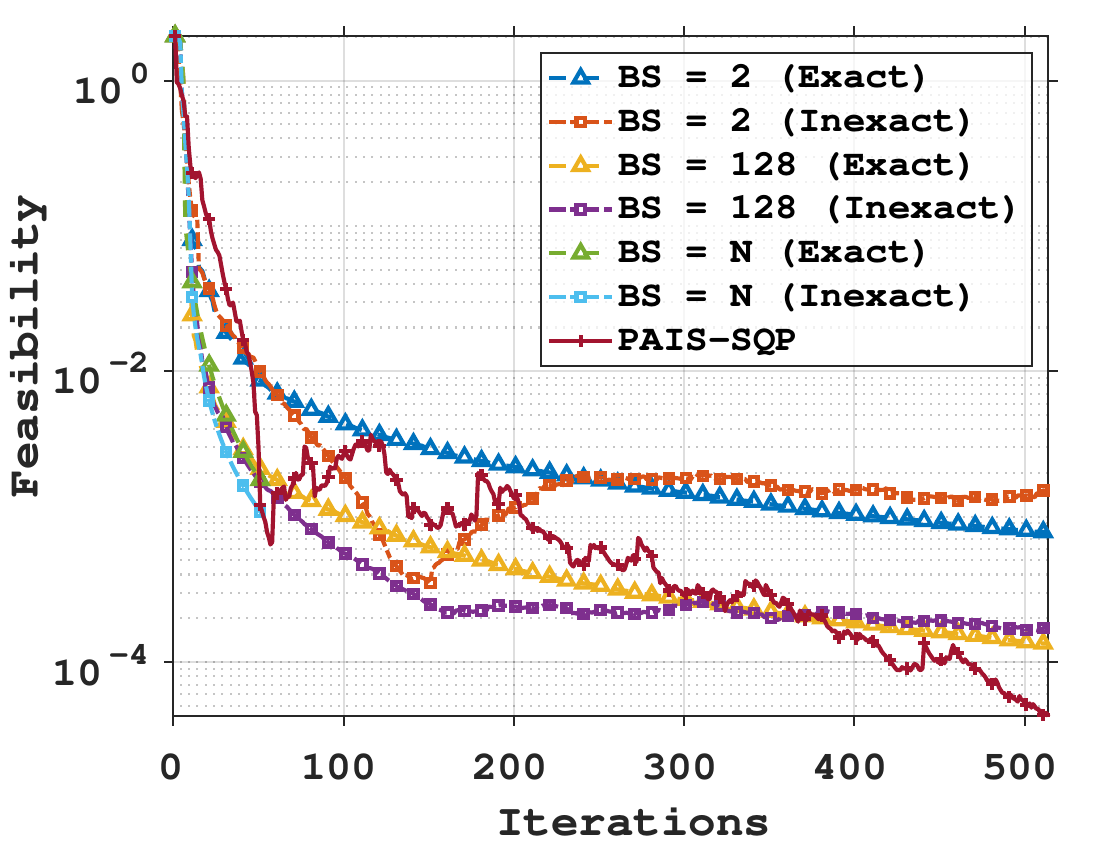}
    \caption{Feasibility vs. Iterations}
    \end{subfigure}
    \begin{subfigure}[b]{0.32\textwidth}
    \includegraphics[width=\textwidth,clip=true,trim=30 180 50 200]{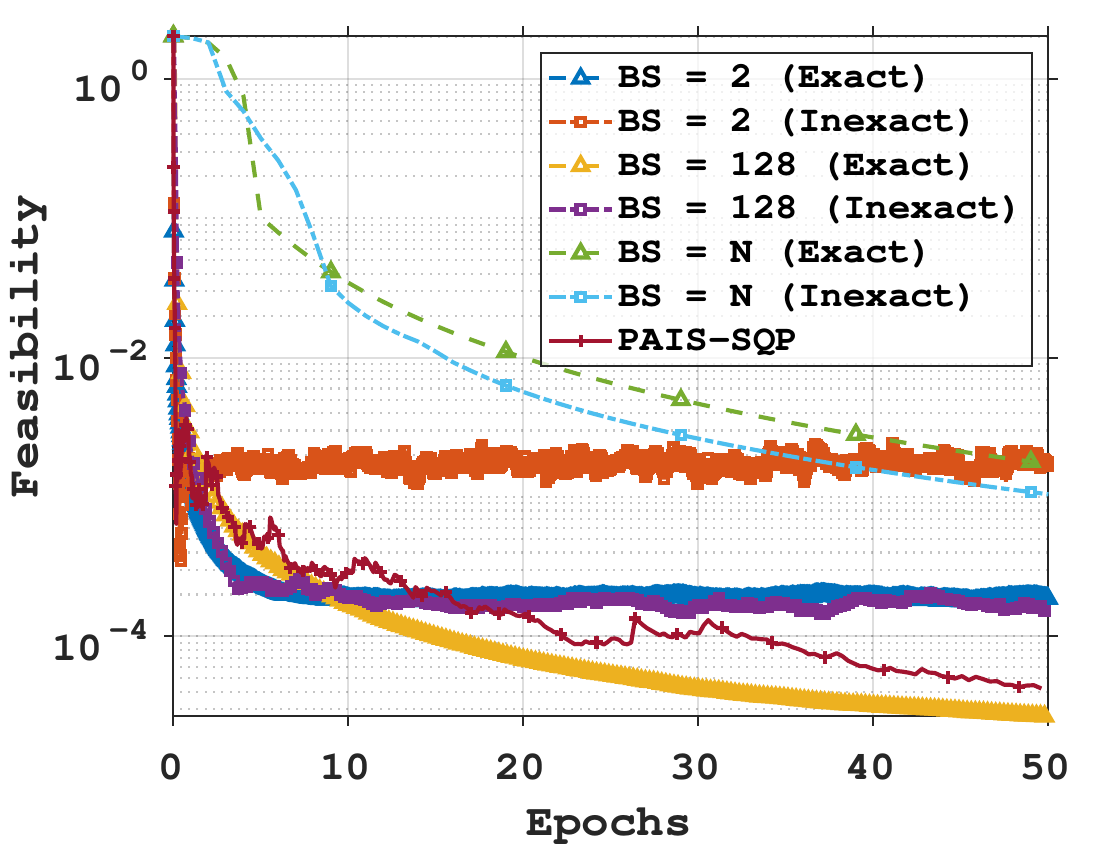}
    \caption{ Feasibility vs. Epochs} 
    \end{subfigure}
    \begin{subfigure}[b]{0.32\textwidth}
    \includegraphics[width=\textwidth,clip=true,trim=30 180 50 200]{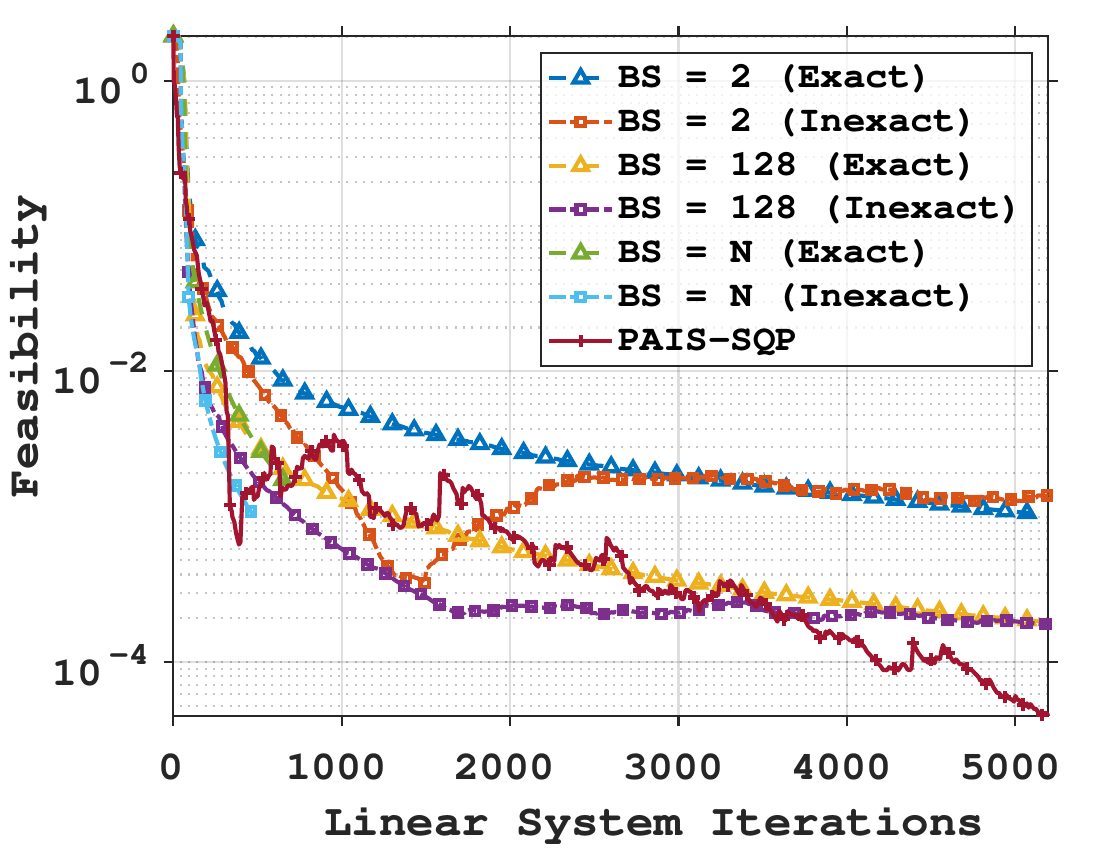}
    \caption{ Feasibility vs. LS  Iters} 
    \end{subfigure}
  
    \begin{subfigure}[b]{0.32\textwidth}
    \includegraphics[width=\textwidth,clip=true,trim=30 180 50 200]{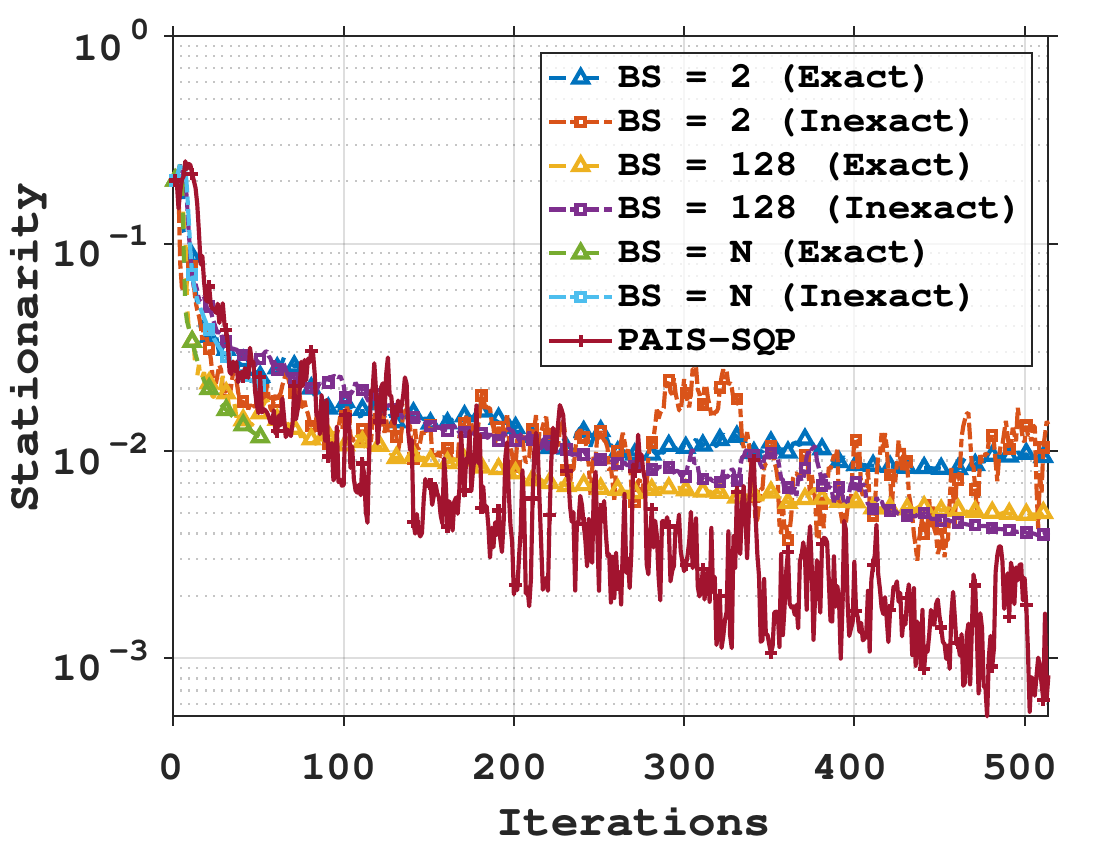}
    \caption{Stationarity vs. Iterations}
    \end{subfigure}
    \begin{subfigure}[b]{0.32\textwidth}
    \includegraphics[width=\textwidth,clip=true,trim=30 180 50 200]{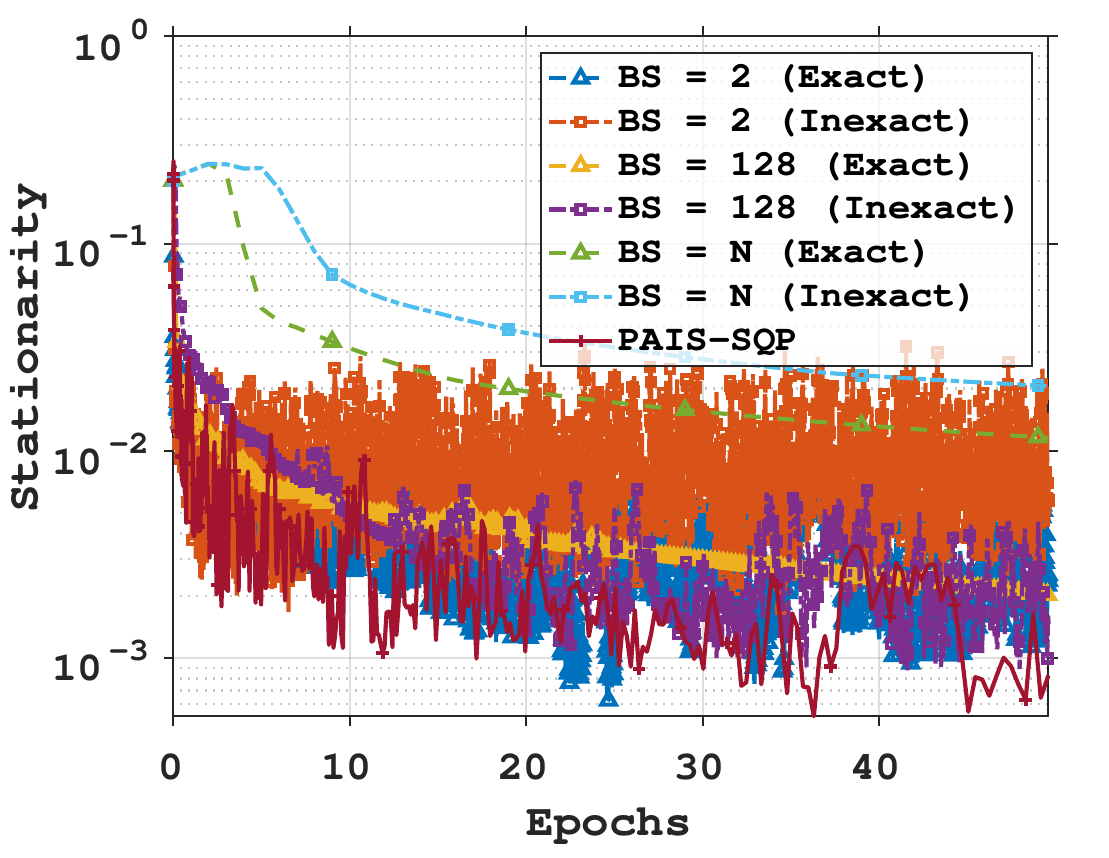}
    \caption{ Stationarity vs. Epochs}
    \end{subfigure}
    \begin{subfigure}[b]{0.32\textwidth}
    \includegraphics[width=\textwidth,clip=true,trim=30 180 50 200]{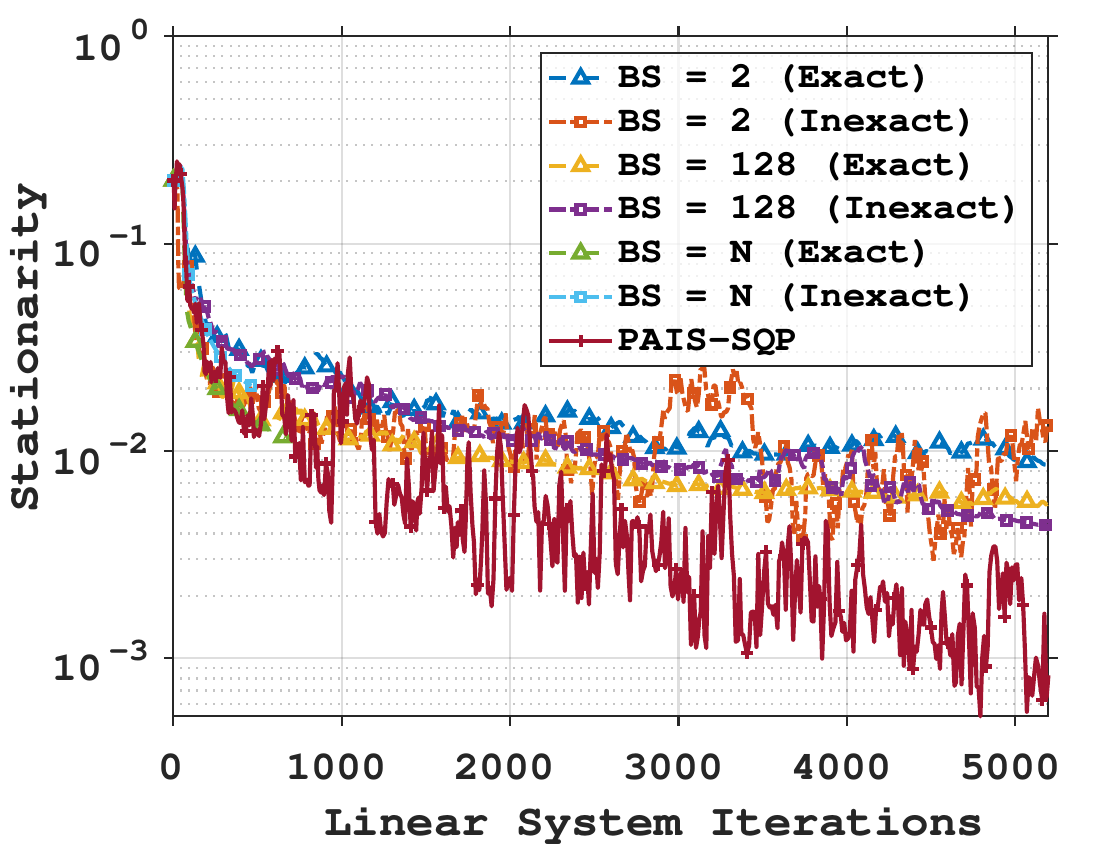}
    \caption{ Stationarity vs. LS  Iters} 
    \end{subfigure}
    \caption{\texttt{mushroom}: Feasibility/stationarity errors versus iterations/epochs/linear system iterations for exact and inexact variants of Algorithm~\ref{alg.adaptiveSQP_practical} on \eqref{eq.logistic}.}\label{fig.mushroom}
\end{figure}

\begin{figure}[]
    \centering
    \begin{subfigure}[b]{0.24\textwidth}
    \includegraphics[width=\textwidth,clip=true,trim=30 180 50 200]{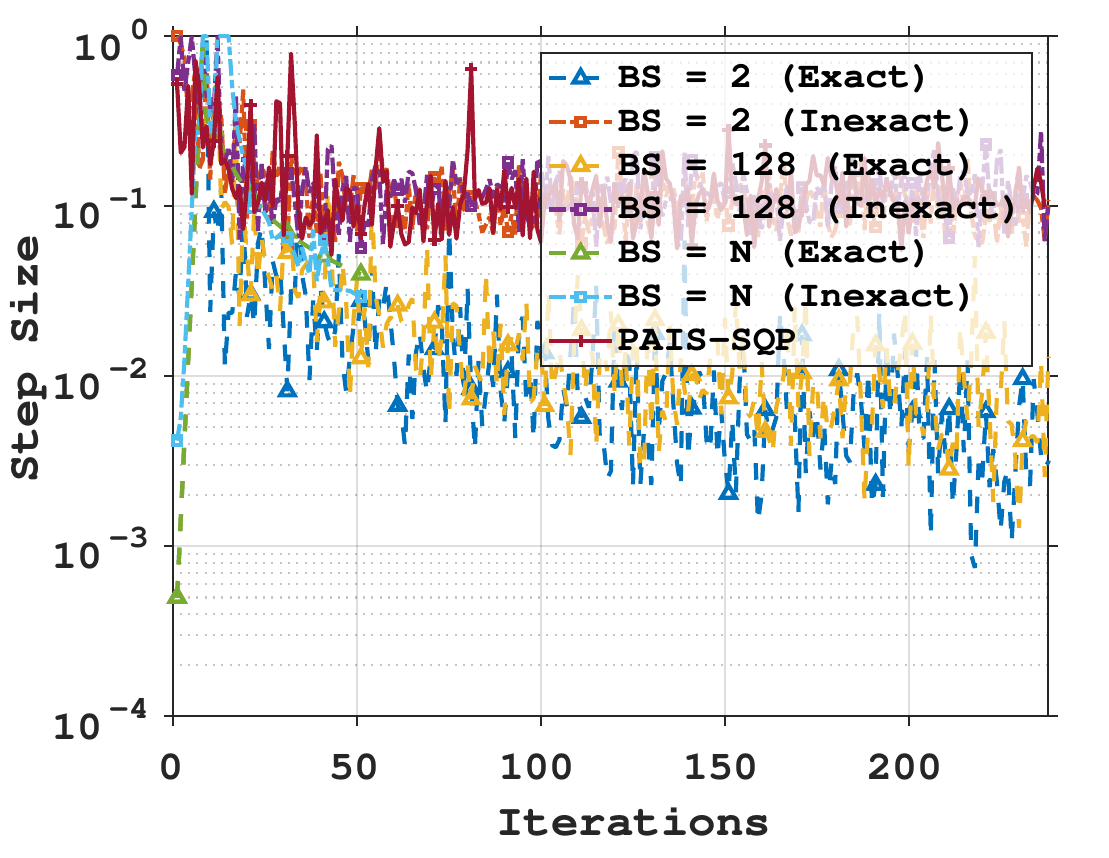}
    \caption{ Step Size vs. Iterations} 
    \end{subfigure}
    \begin{subfigure}[b]{0.24\textwidth}
    \includegraphics[width=\textwidth,clip=true,trim=30 180 50 200]{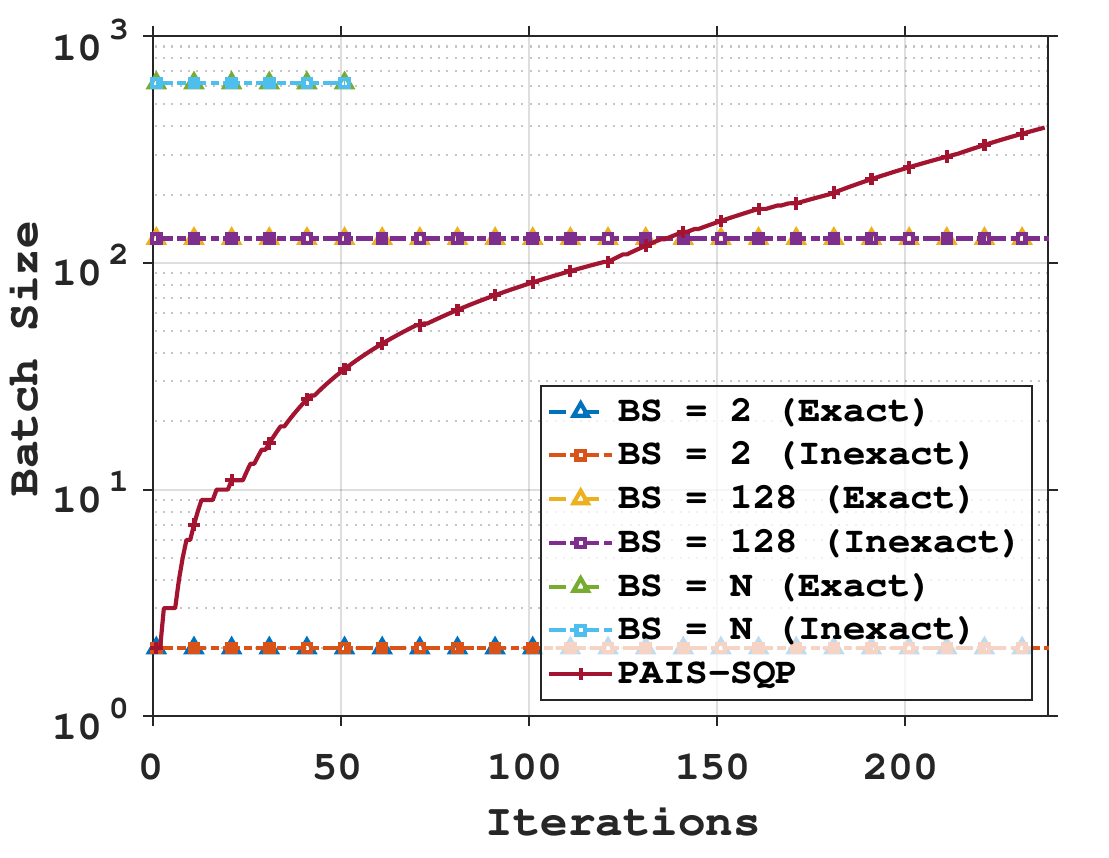}
    \caption{ Batch Size vs. Iterations} 
    \end{subfigure}
    \begin{subfigure}[b]{0.24\textwidth}
    \includegraphics[width=\textwidth,clip=true,trim=30 180 50 200]{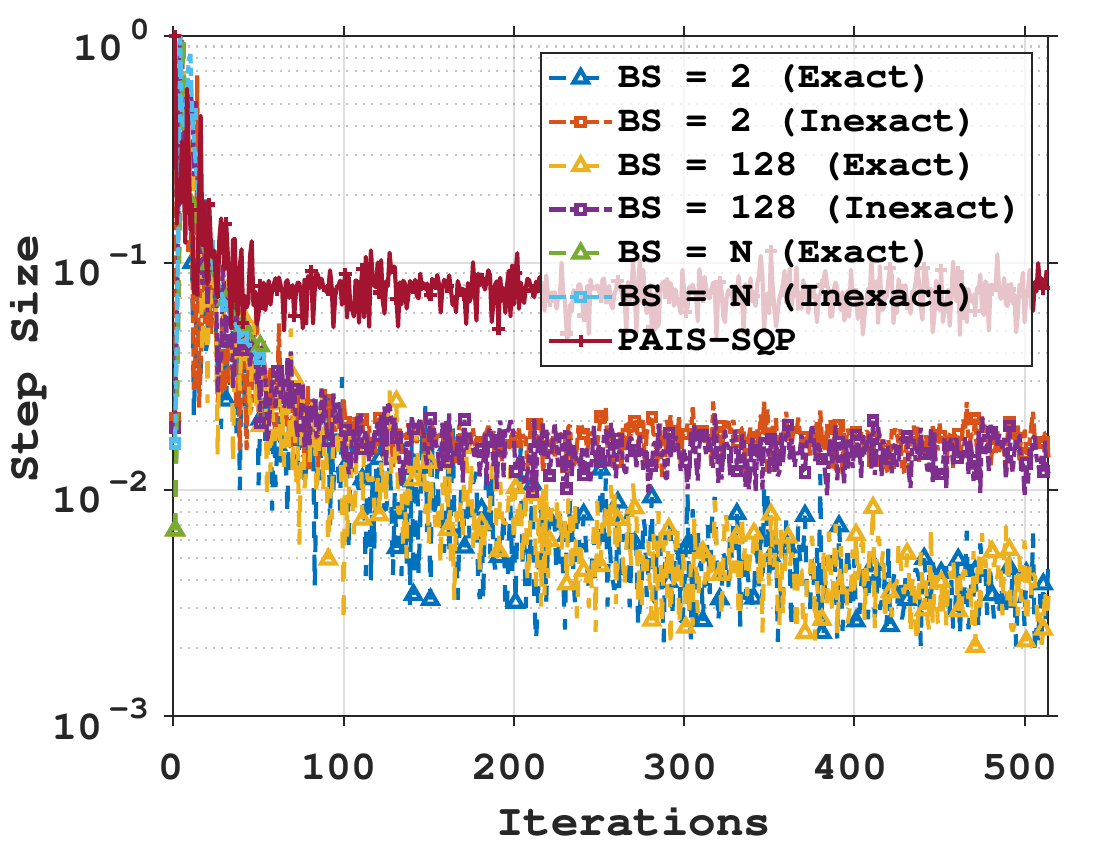}
    \caption{ Step Size vs. Iterations} 
    \end{subfigure}
    \begin{subfigure}[b]{0.24\textwidth}
    \includegraphics[width=\textwidth,clip=true,trim=30 180 50 200]{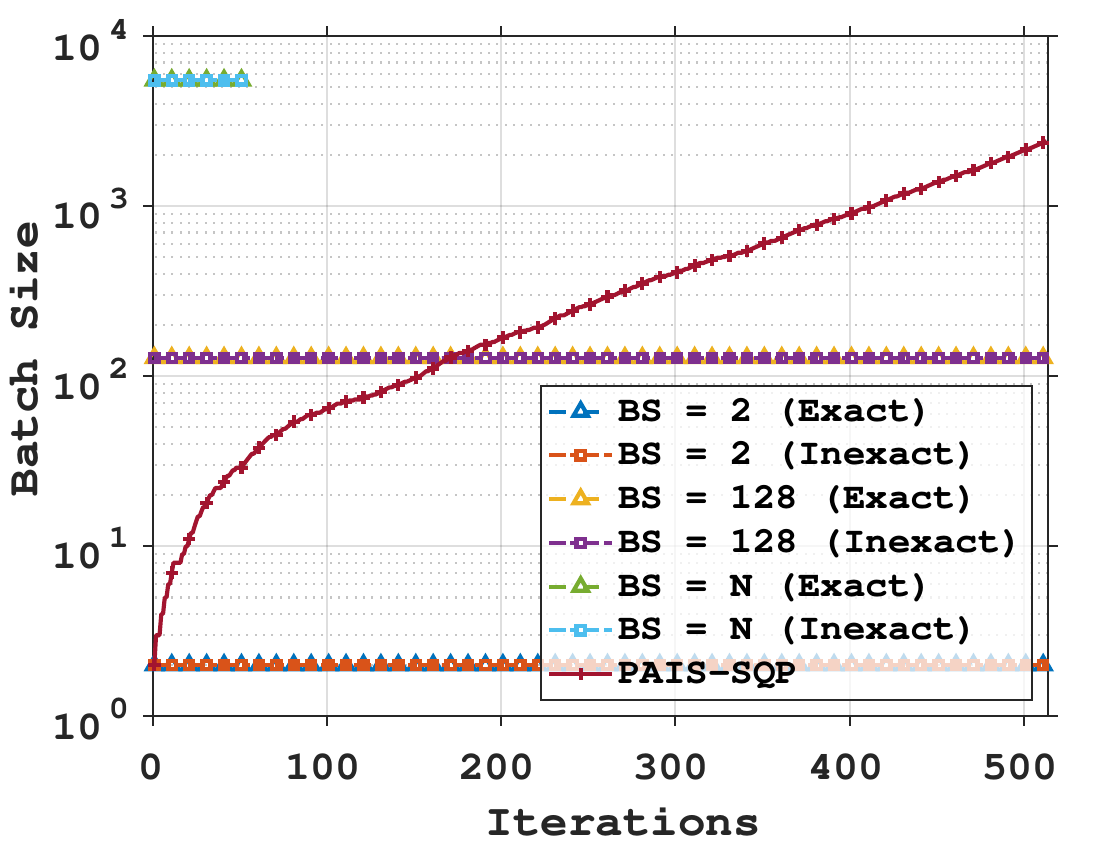}
    \caption{ Batch Size vs. Iterations} 
    \end{subfigure}
    \caption{\texttt{australian} ((a), (b)); \texttt{mushroom} ((c), (d)): Step sizes and batch sizes versus iterations.}\label{fig.step_batch}
\end{figure}

\newpage

\subsection{CUTE Problems}
\label{sec.cute}

 Next, we consider equality constrained problems from the CUTE collection of nonlinear optimization problems  \cite{bongartz1995cute}. Specifically, of
the $123$ such problems in the collection we considered $49$ problems. (We only used those for which: \change{$(i)$ the linear independence constraint qualification (LICQ)} held for all iterations of all algorithms, $(ii)$ $f$ is not a constant function, and $(iii)$ $n+m \leq 1000$.) We used the prescribed starting point for all problems and all algorithms. The CUTE problems are deterministic, so we added noise to the gradient computations to make the problems stochastic. Specifically, we consider additive noise where the gradient was computed as
\begin{equation*}
    \bar{g}_k = \tfrac{1}{|\mathcal{S}_k|} \sum_{i \in \mathcal{S}_k} (\nabla f(x_k) + \mathcal{N}(0,\epsilon_{N,i}I)),
\end{equation*}
where $\epsilon_{N,i} \gets 10^{-1}$ for all $i$, and $\mathcal{S}_k$ is prescribed by the variant and determines the level of noise.

For each problem, we again consider variants that compute exact and inexact (early termination conditions) linear system solutions. We compare \PAISSQP{}, to non-adaptive sampling variants with ($|\mathcal{S}_k| = |\mathcal{S}| \in \{2,128,1024\}$ for all $k\in\mathbb{N}$), and limit the maximum sample size employed by \PAISSQP{} to $1024$. For each problem, we ran $10$ instances with different random seeds. This led to a total of $490$ runs of each algorithm for each noise level. We terminated the methods on the following budget: $1024\cdot10^3$ gradient evaluations or $1024\cdot10^2$ linear system iterations (whichever comes first).


The results of these experiments are reported in Figure~\ref{fig.perf} in the form of performance profiles \cite{more2009benchmarking}. We present results in terms of feasibility and stationarity with respect to gradient evaluations and linear system iterations. The performance profiles were constructed as follows. For each problem, method and seed, the iterate used in the performance profile $x_{pp}$ was chosen as: either the point with minimum $\|g_k + J_k^Ty_k\|_{\infty}$ among all points with $\|c_k\|_{\infty} \leq 10^{-6}$, or if no such point exists, then the point with minimum $\|c_k\|_{\infty}$. Following 
\cite{more2009benchmarking}, for the two metrics an algorithm was deemed to have solved a given problem for a given seed if $m(x_0) - m(x_{pp}) \geq (1-\epsilon_{pp})(m(x_0) - m(x_{b}))$, where $m(x_l)$ is $\|g_l + J_l^Ty_l\|_{\infty}$ (for stationarity) and $\|c_l\|_{\infty}$ (for feasibility), respectively,  $m(x_{b})$ denotes the best possible value of either metric for each problem and seed, and tolerance $\epsilon_{pp} \in (0,1)$. Overall, across all tolerances and metrics, the \PAISSQP{} method appears to be the most robust (as seen by the right-most points on the figures). The ability of \PAISSQP{} to make sufficient progress with inexact information in the intial stages of the optimization, combined with its ability to increase accuracy of the approximations employed, as needed, as the optimization progresses allows the algorithm to balance convergence and cost. As a result, \PAISSQP{} is efficient and robust in \change{terms} of all metrics. 


\begin{figure}[ht]
    \centering
    \begin{subfigure}[b]{0.24\textwidth}
    \includegraphics[width=\textwidth,clip=true,trim=30 180 50 200]{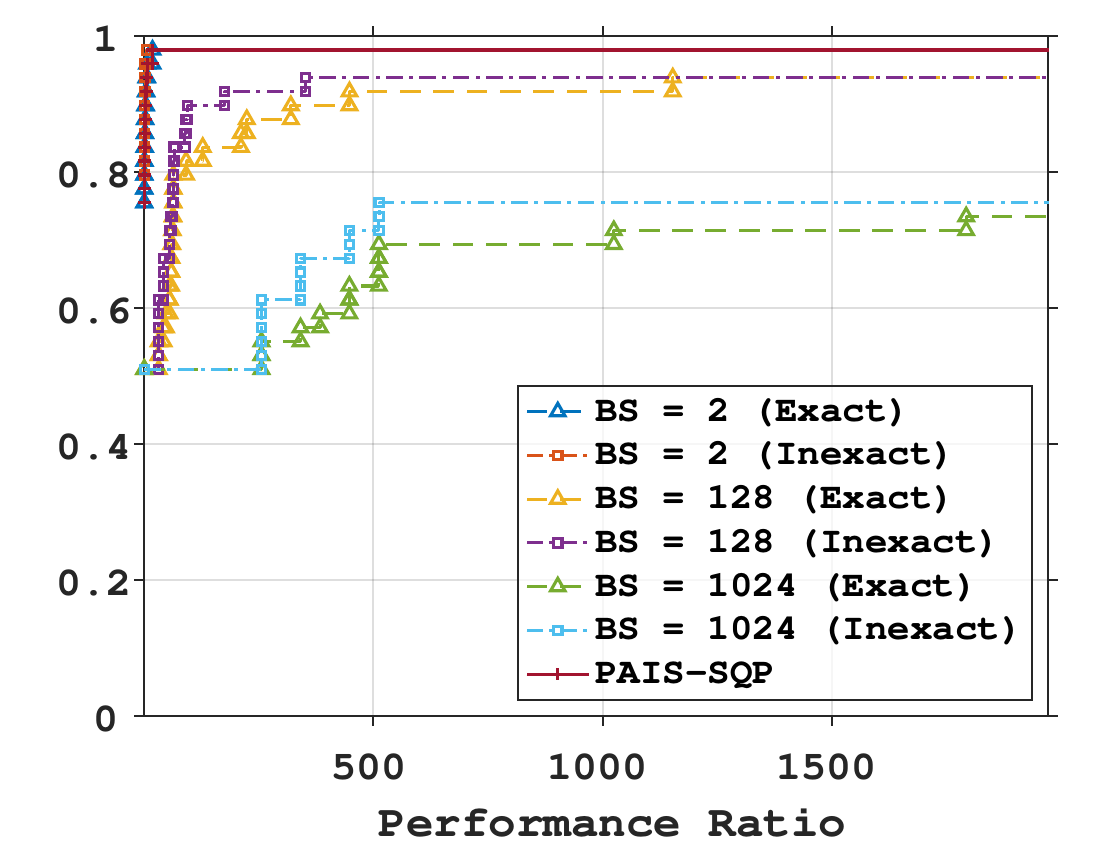}
    \caption{Feas. vs. Grad.}
    \end{subfigure}
    \begin{subfigure}[b]{0.24\textwidth}
    \includegraphics[width=\textwidth,clip=true,trim=30 180 50 200]{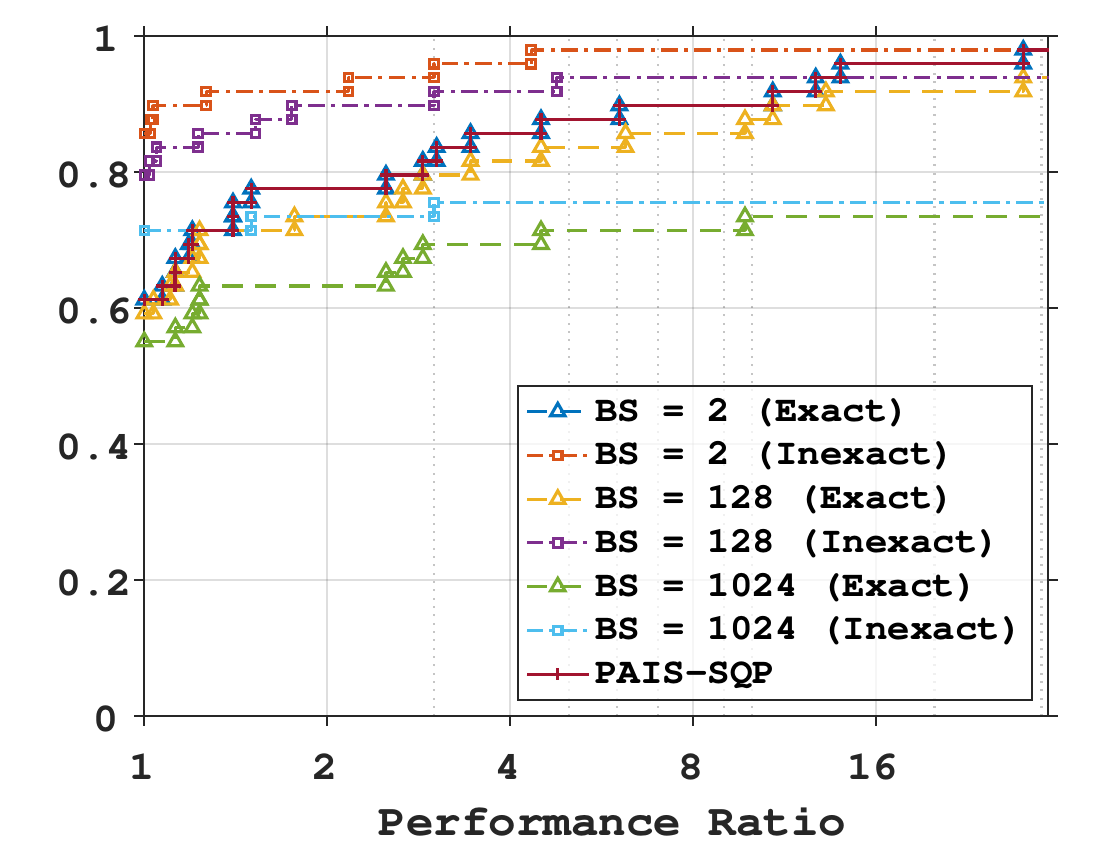}
    \caption{Feas. vs. LS Iters}
    \end{subfigure}
    \begin{subfigure}[b]{0.24\textwidth}
    \includegraphics[width=\textwidth,clip=true,trim=30 180 50 200]{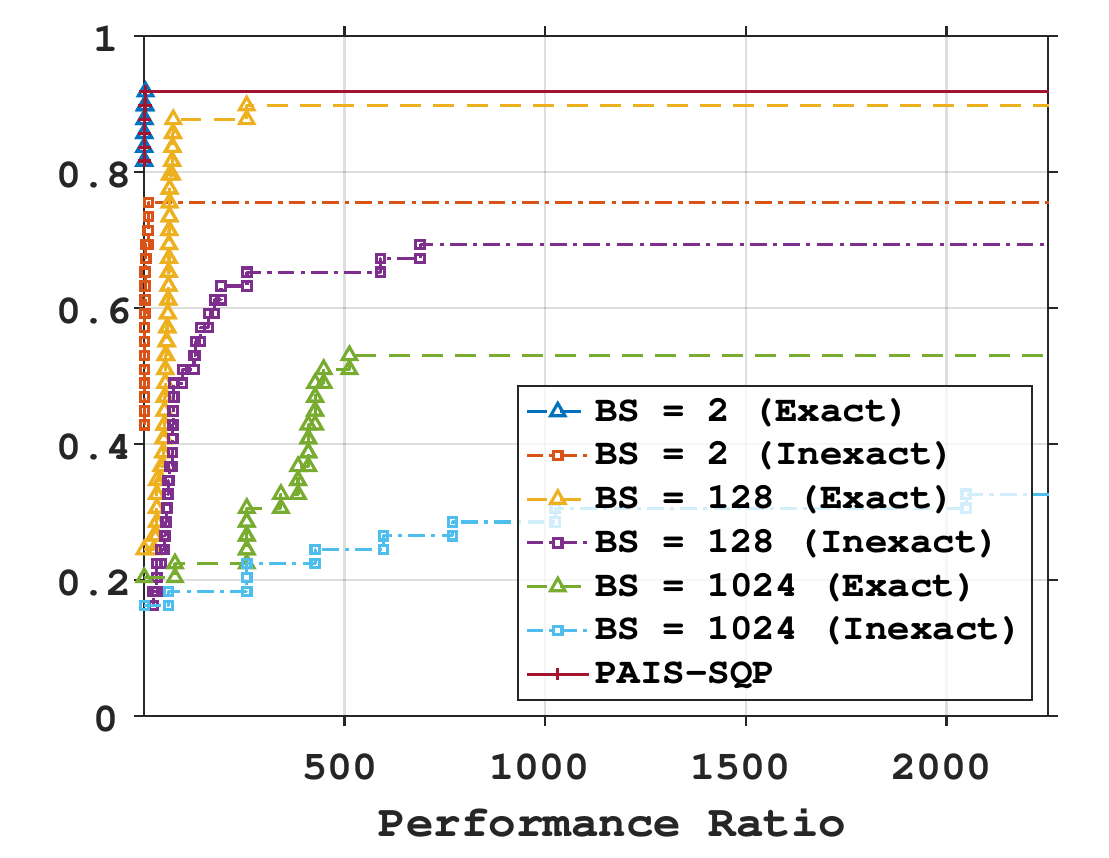}
    \caption{Stat. vs. Grad.} 
    \end{subfigure}
    \begin{subfigure}[b]{0.24\textwidth}
    \includegraphics[width=\textwidth,clip=true,trim=30 180 50 200]{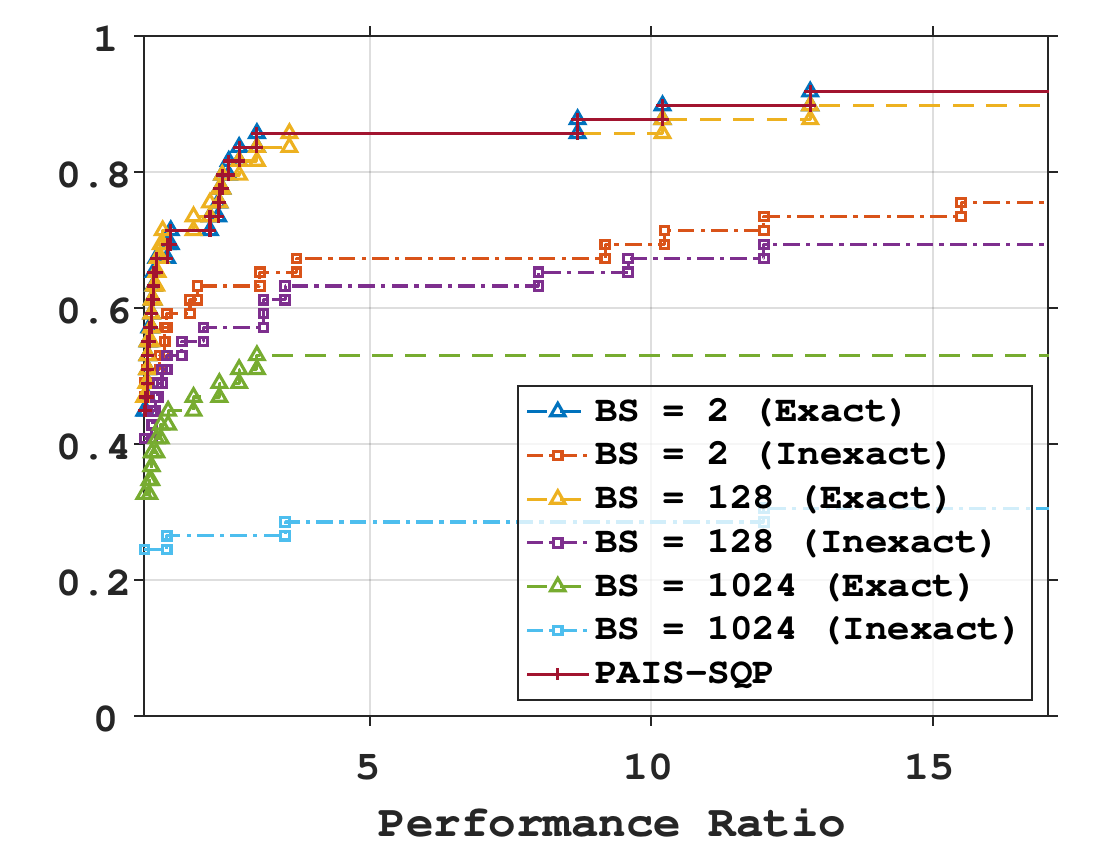}
    \caption{Stat. vs. LS Iters}
    \end{subfigure}
  
    \begin{subfigure}[b]{0.24\textwidth}
    \includegraphics[width=\textwidth,clip=true,trim=30 180 50 200]{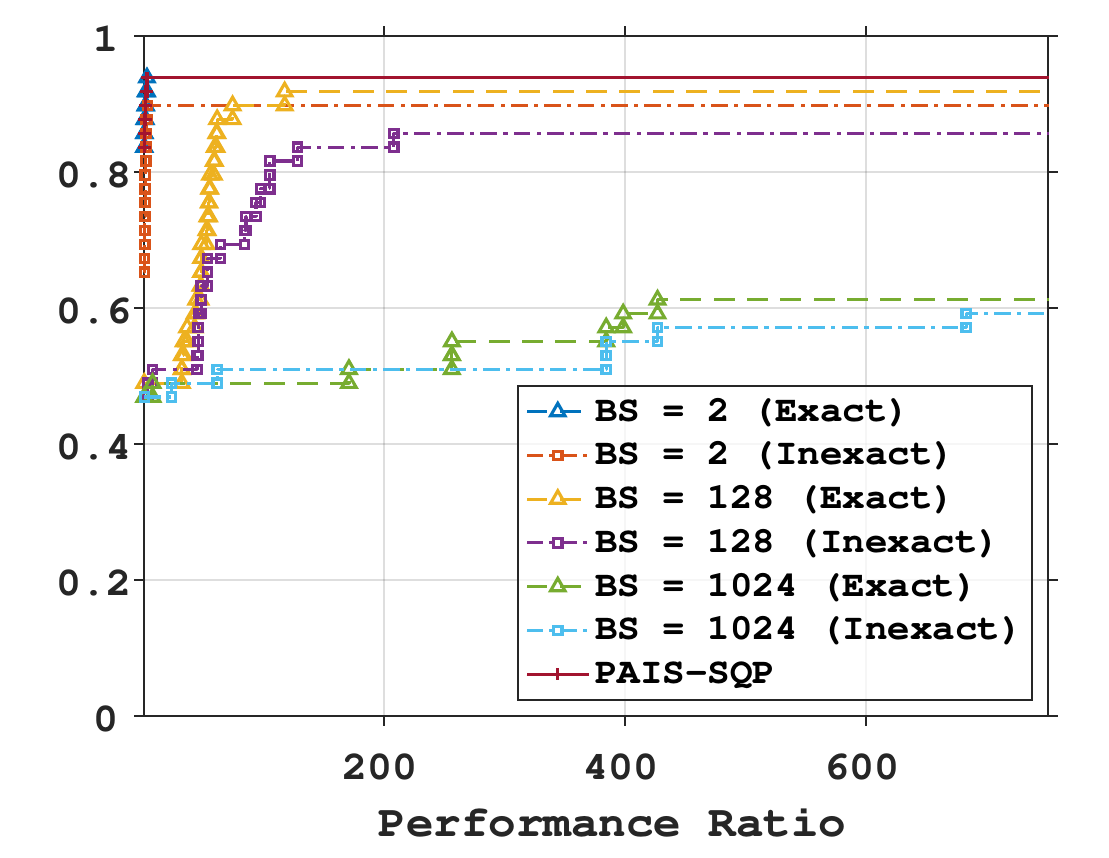}
    \caption{Feas. vs. Grad.}
    \end{subfigure}
    \begin{subfigure}[b]{0.24\textwidth}
    \includegraphics[width=\textwidth,clip=true,trim=30 180 50 200]{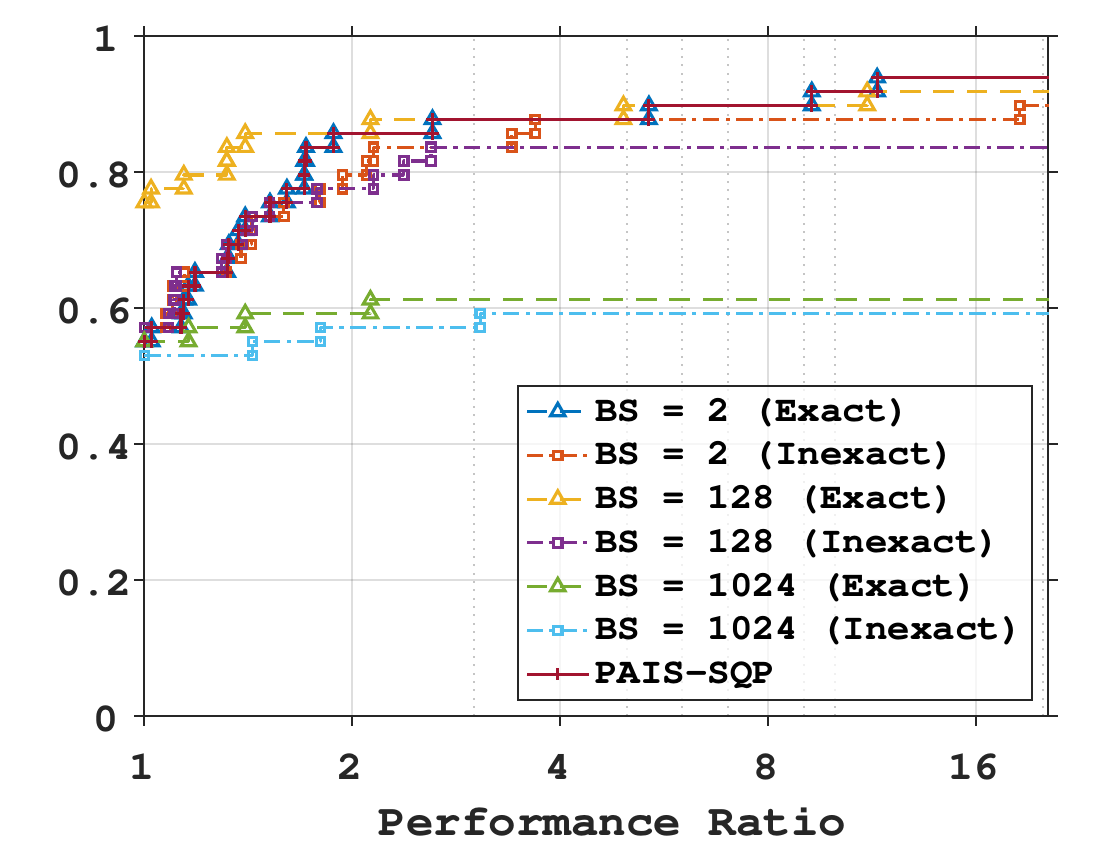}
    \caption{Feas. vs. LS Iters}
    \end{subfigure}
    \begin{subfigure}[b]{0.24\textwidth}
    \includegraphics[width=\textwidth,clip=true,trim=30 180 50 200]{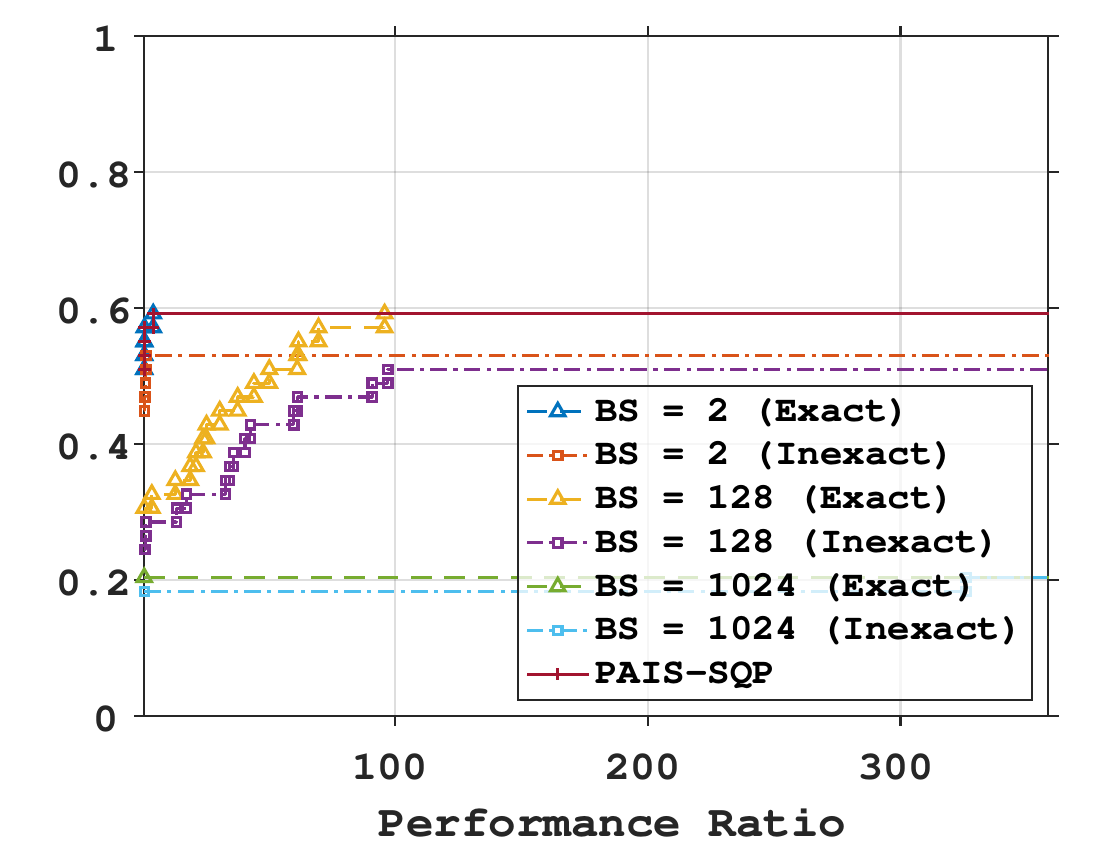}
    \caption{Stat. vs. Grad.} 
    \end{subfigure}
    \begin{subfigure}[b]{0.24\textwidth}
    \includegraphics[width=\textwidth,clip=true,trim=30 180 50 200]{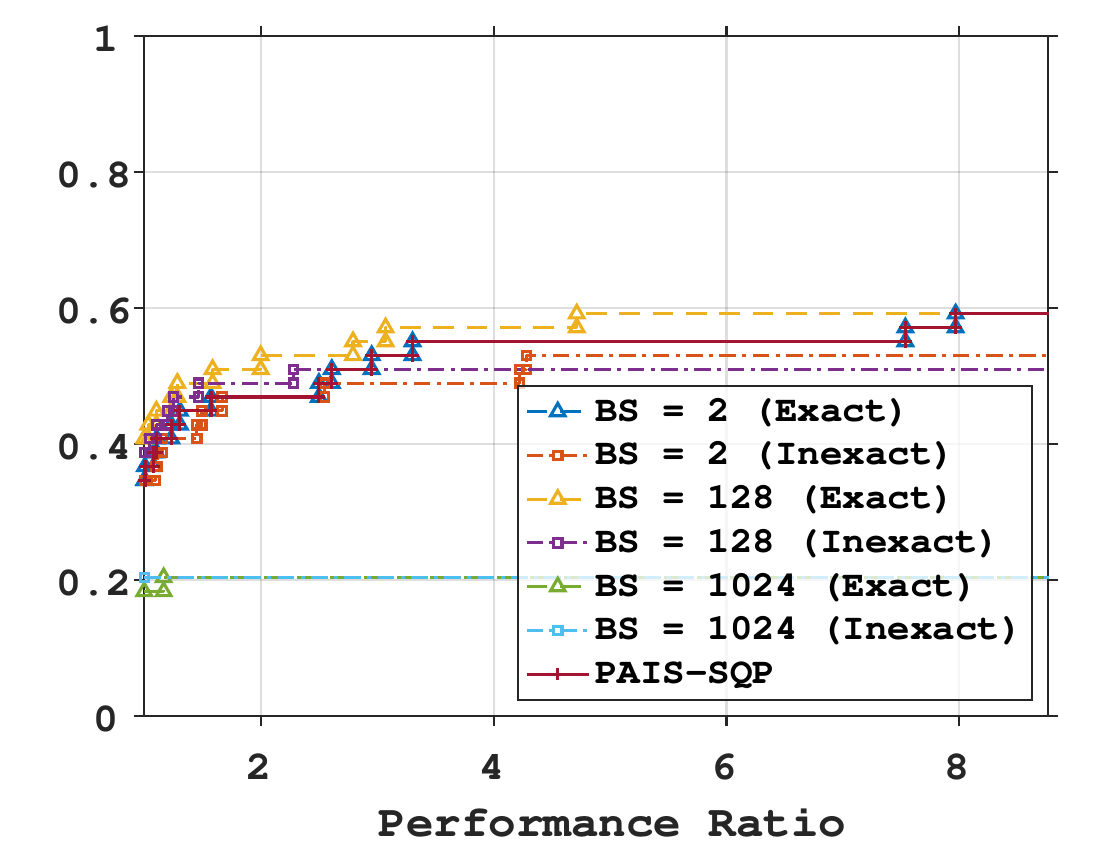}
    \caption{Stat. vs. LS Iters}
    \end{subfigure}
    
    \begin{subfigure}[b]{0.24\textwidth}
    \includegraphics[width=\textwidth,clip=true,trim=30 180 50 200]{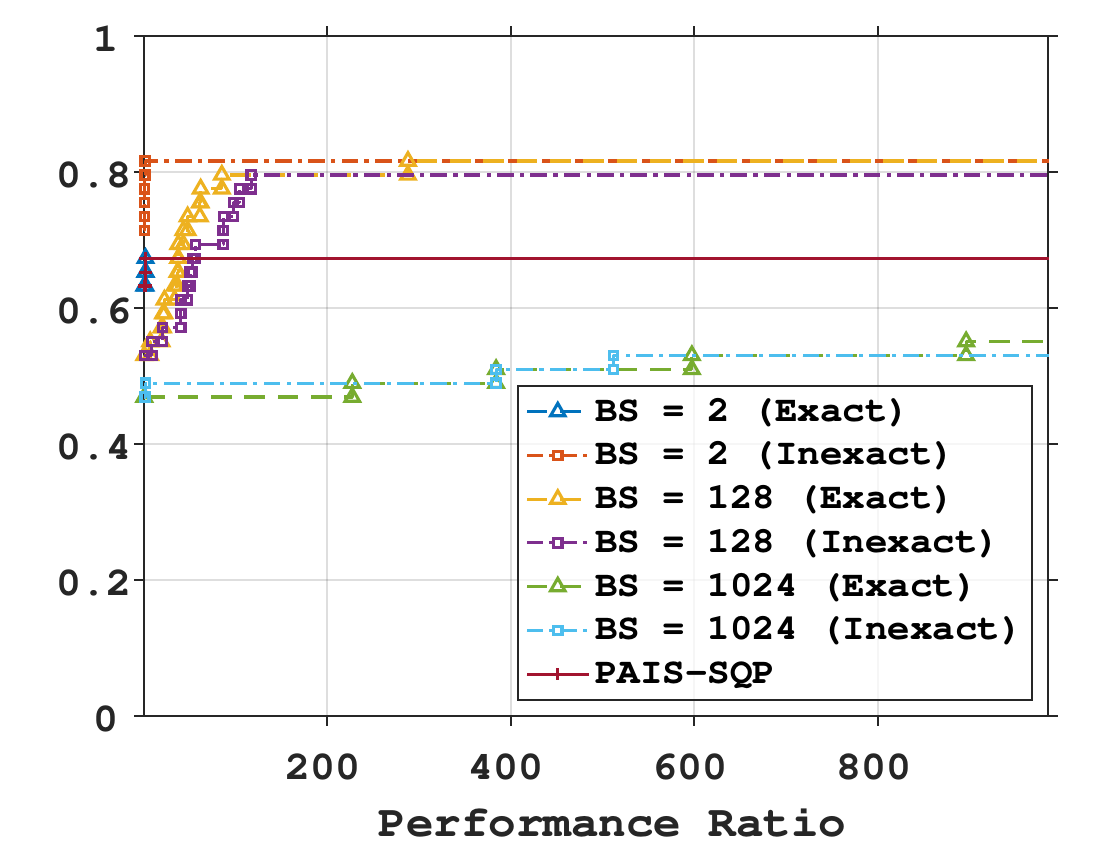}
    \caption{Feas. vs. Grad.}
    \end{subfigure}
    \begin{subfigure}[b]{0.24\textwidth}
    \includegraphics[width=\textwidth,clip=true,trim=30 180 50 200]{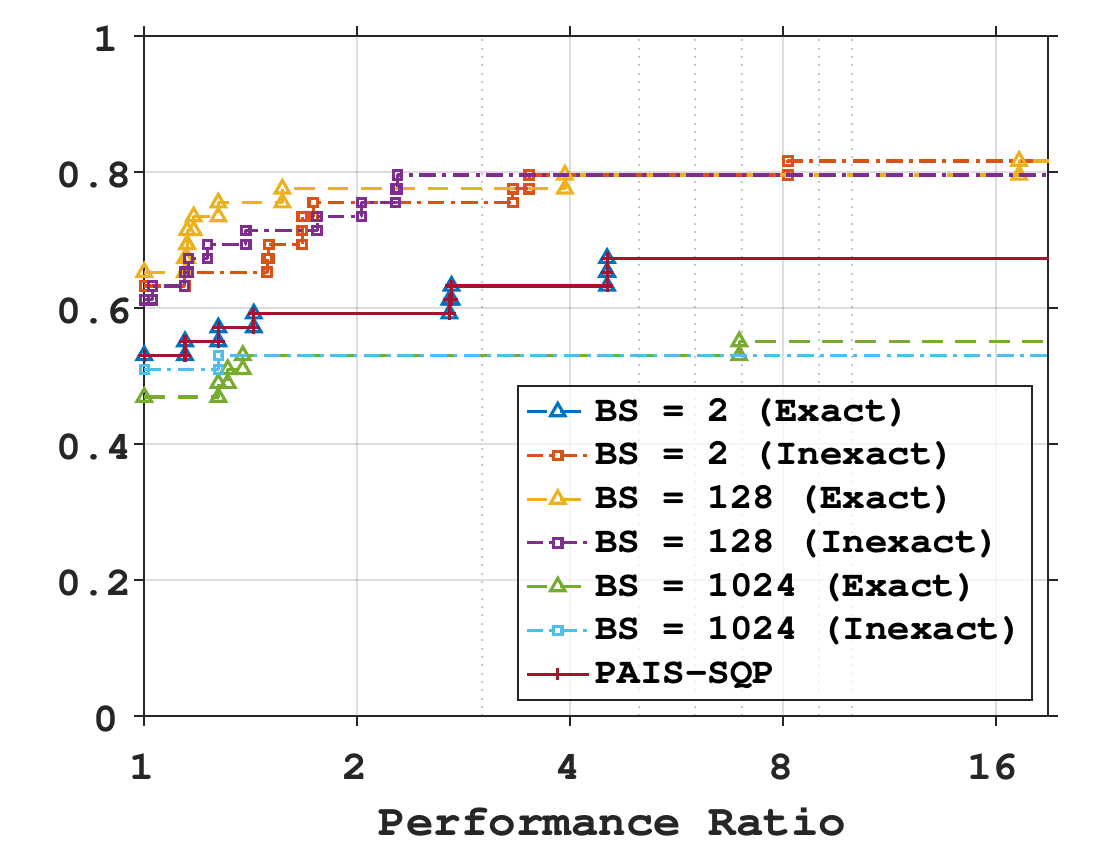}
    \caption{Feas. vs. LS Iters}
    \end{subfigure}
    \begin{subfigure}[b]{0.24\textwidth}
    \includegraphics[width=\textwidth,clip=true,trim=30 180 50 200]{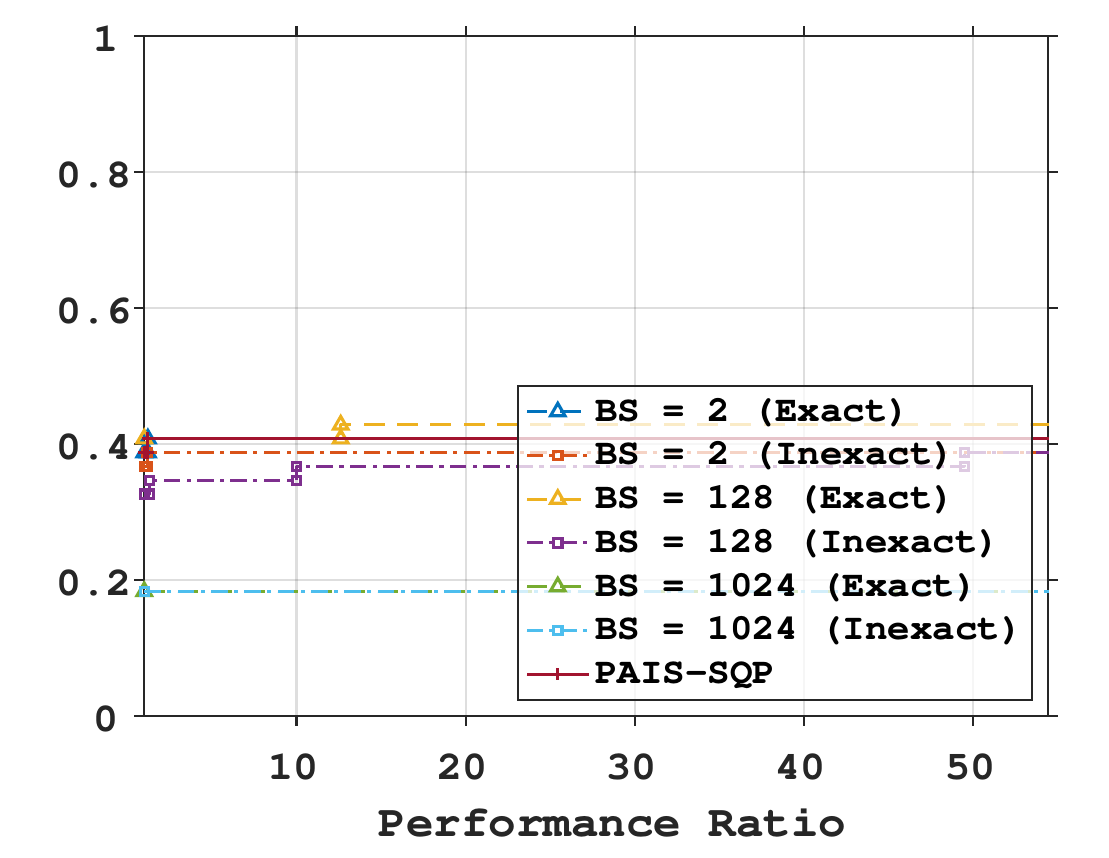}
    \caption{Stat. vs. Grad.} 
    \end{subfigure}
    \begin{subfigure}[b]{0.24\textwidth}
    \includegraphics[width=\textwidth,clip=true,trim=30 180 50 200]{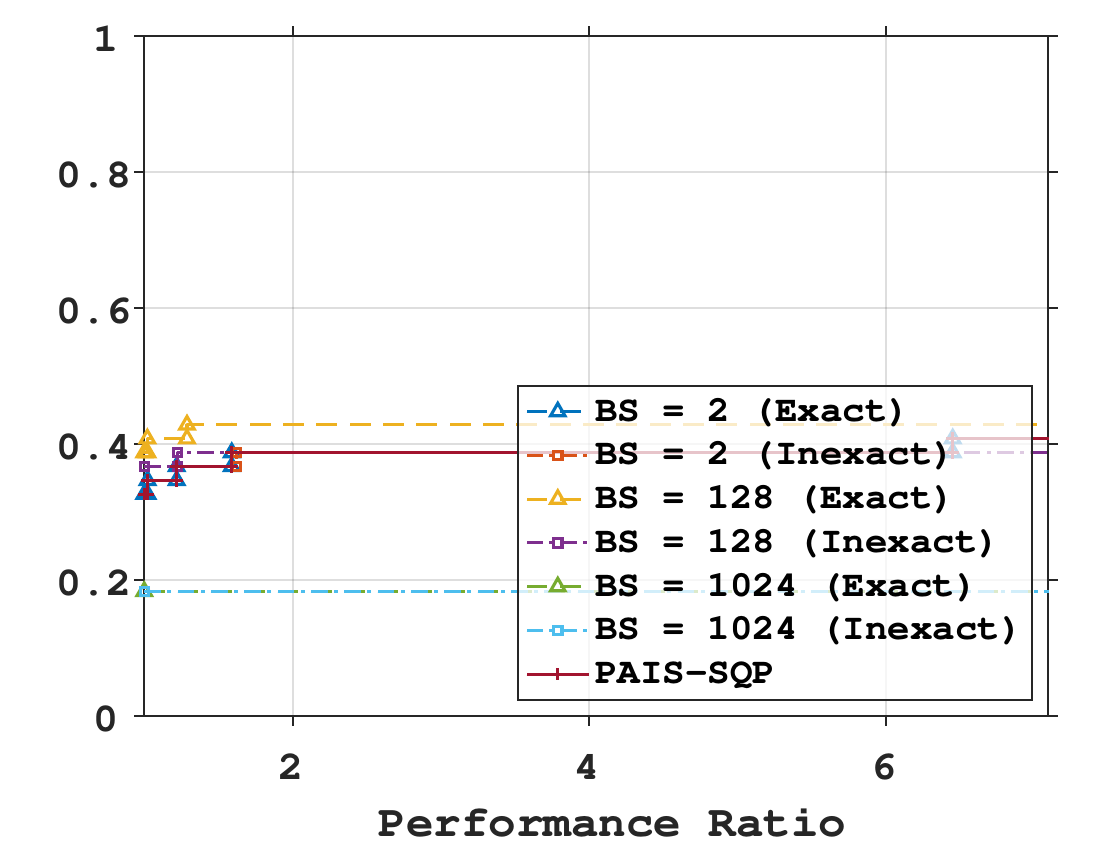}
    \caption{Stat. vs. LS Iters}
    \end{subfigure}
    \caption{\texttt{CUTE}: Performance profiles for exact and inexact variants of Algorithm~\ref{alg.adaptiveSQP_practical} on CUTE collection. 
    First row accuracy $\epsilon_{pp} = 10^{-1}$; Second row accuracy $\epsilon_{pp} = 10^{-3}$; Third row accuracy $\epsilon_{pp} = 10^{-5}$.}\label{fig.perf}
\end{figure}

\section{Final Remarks}\label{sec.final_rem}
In this paper, we have designed and analyzed a stochastic SQP algorithm (\texttt{AIS-SQP}) for solving optimization problems involving deterministic nonlinear equality constraints and a stochastic objective function. At each iteration, the \texttt{AIS-SQP} method computes a stochastic approximation of the gradient of the objective function, computes a step by solving a stochastic Newton-SQP linear system inexactly, potentially updates the merit parameter, and adaptively selects a step size and updates the iterate. Our algorithm is adaptive in several ways. We have proposed accuracy conditions for the stochastic gradient approximation and the quality of the linear system solutions. Moreover, we have proposed adaptive update rules for the merit and step size parameters. Our algorithmic development and analysis \change{have} revealed an intrinsic relationship between the accuracy of stochastic objective gradient realizations, the quality of inexact solutions to the Newton-SQP linear systems, and the adaptive step sizes selected. That is, higher accuracy in the gradient estimation and linear system solution can potentially lead to the acceptance of larger step sizes as in the deterministic counterparts.

We have proved that our algorithm generates a sequence of iterates whose first-order stationarity measure converges to zero in expectation. While similar results have been established in the literature, e.g., \cite{BeraCurtRobiZhou21,CurtOneiRobi21}, these works only consider asymptotic regimes after which the penalty parameter is sufficiently small and has stabilized. In this work, we have analyzed the complete behavior of the algorithm across all potential merit parameter changes, under \change{reasonable assumptions}, and have provided iteration complexity analysis for \texttt{AIS-SQP}, which matches that of deterministic SQP methods \cite{CurtOneiRobi21}, in expectation. We have also established sublinear (gradient) sample complexity results of the proposed algorithm when the gradient and linear system accuracies are controlled at predetermined sublinear rates.

Inspired by \texttt{AIS-SQP}, we have developed, implemented and tested a practical variant \texttt{PAIS-SQP} of the adaptive stochastic SQP method. Our results on two different sets of experiments, constrained logistic regression and standard nonlinear optimization test problems, suggest that our practical algorithm strikes a good balance between minimizing constraint violation while also minimizing the objective function in terms of importance evaluation metrics such as iterations, gradient evaluations and linear system iterations.

%% file: acknowledgements.tex
We would like to thank Dr. Frank E. Curtis for all his useful suggestions and feedback. Moreover, we would like to thank the Office of Naval Research (award number: N00014-21-1-2532) and Lawrence Livermore National Laboratory for their support of this project.

%% file: appendix.tex
\newpage
\clearpage

\section{Technical Results (Section~\ref{sec.pred})}\label{sec.stoch_sublin}

\input{raghu}

\vspace{1cm}
\section{Additional Numerical Results: Constrained Logistic Regression}\label{app.numerical}

\input{numerical_results}

%% file: raghu.tex
In this appendix, we present the complete theoretical results for Algorithm~\ref{alg.adaptiveSQP} under Conditions~\ref{ass.stoch_g_sublin} and \ref{ass.stoch_linear_system}. We present all the technical lemmas required to establish the main theoretical results presented in Section~\ref{sec.pred} (Theorem~\ref{thm.main_stochastic_mp} and Corollary~\ref{cor.iter_complexity_raghu}).


Lemmas~\ref{lem.bound_v}--\ref{lem.Psi} from Section~\ref{sec.adaptive} hold without change. For completeness, we unify those results in the following lemma.

\blemma\label{lem.det2stoch} Under Assumption~\ref{ass.main}, for all $k \in \N{}$, \eqref{eq.system} has a unique solution. In addition, under Assumptions~\ref{ass.main},~\ref{ass.H} and~\ref{ass.residual} and Condition~\ref{ass.stoch_linear_system}, for the same constants $(\kappa_{\vbar}, \kappa_{\overline{uv}}, \kappa_{\bar\Psi}, \kappa_{\bar{l}}) \in \R{}_{>0} \times \R{}_{>0} \times \R{}_{>0} \times \R{}_{>0}$
that appear in the respective lemmas (listed in parenthesis below), the following statements hold for all
$k \in \N{}$:
\benumerate
    \item[a.] (Lemma \ref{lem.bound_v}) The normal component $\bar{v}_k$ is bounded as $\max\{\|\bar{v}_k\|_2,\|\bar{v}_k\|_2^2\} \leq \kappa_{\vbar} \max\{\|c_k\|_2,\|\bar{r}_k\|_2\}$.
    \item[b.] (Lemma \ref{lem.tangential_big}) If $\|\bar{u}_k\|_2^2 \geq \kappa_{\overline{uv}} \|\bar{v}_k\|_2^2$, then $\bar{d}_k^TH_k\bar{d}_k \geq \tfrac{1}{2} \zeta \|\bar{u}_k\|_2^2$ and $\bar{d}_k^TH_k\bar{d}_k \geq \epsilon_d\|\bar{d}_k\|_2^2$ where $\epsilon_d\in (0,\tfrac{\zeta}{2})$ is an user-defined parameter in Algorithm~\ref{alg.adaptiveSQP}.
    \item[c.] (Lemma \ref{lem.Psi_1}) The search direction satisfies $\|\bar{d}_k\|_2^2 \leq \kappa_{\bar\Psi} \bar{\Psi}_k$ and $\|\bar{d}_k\|_2^2 + \|c_k\|_2 \leq  (\kappa_{\bar\Psi} + 1) \bar{\Psi}_k$.
    \item[d.] (Lemma \ref{lem.Psi}) The model reduction satisfies $\Delta l(x_k,\bar{\tau}_k,\bar{g}_k,\bar{d}_k) \geq \kappa_{\bar{l}} \bar\tau_k \bar{\Psi}_k$.
\eenumerate
\elemma

Additionally, if \asb{Condition}~\ref{assum:stoch_redcond_old} holds, then the following results also hold \asb{(Lemmas~\ref{lem.stoch_tau_lb}, \ref{lem.stoch_model_reduction_step_squared}, \ref{lem.stoch_stepsize_bound_old} and~\ref{lem.key_decrease})}.
\begin{lemma}\label{lemma.combo}
    Under \asb{Assumptions~\ref{ass.main},~\ref{ass.H} and~\ref{ass.residual} and Conditions~\ref{assum:stoch_redcond_old} and~\ref{ass.stoch_linear_system}}, \asb{for the same constants that appear in the respective lemmas (listed in parenthesis below), the following statements hold for all $k\in\mathbb{N}$:} 
    \benumerate
        \item [a.] (Lemma \ref{lem.stoch_tau_lb}) There exists a constant $\bar\tau_{\min}\in\mathbb{R}_{>0}$ such that Algorithm~\ref{alg.adaptiveSQP} generates a sequence of $\{\bar\tau_k\}$, where $\bar\tau_k \geq \bar\tau_{\min}$.
        
        \item [b.] (Lemma \ref{lem.stoch_model_reduction_step_squared}) There exist constants $\{\kappa_{\bar\alpha},\kappa_{\Delta \bar{l},\bar{d}}\}\subset\mathbb{R}_{>0}$ such that for all $k\in\mathbb{N}$, $\Delta l(x_k,\bar\tau_k,\bar{g}_k,\bar{d}_k) \geq \kappa_{\bar\alpha}(\bar\tau_kL_k + \Gamma_k)\|\bar{d}_k\|_2^2 \geq \kappa_{\Delta \bar{l},\bar{d}}\|\bar{d}_k\|_2^2$.
    \item[c.] (Lemma \ref{lem.stoch_stepsize_bound_old}) For all $k \in \N{}$, there exists a constant $\underline{\alpha} \in \mathbb{R}_{>0}$ such that, $\underline{\alpha} \beta \leq \bar{\alpha}_k \leq \alpha_u\beta^{(2-\bz{\sigma})}$.
    \item [d.] (Lemma \ref{lem.key_decrease}) Finally, for all $k \in \N{}$, it follows that
\bequation
  \baligned\label{eq.merit_red_det_stoch}
      &\ \phi(x_k + \bar\alpha_k \bar{d}_k, \bar\tau_k) - \phi(x_k, \bar\tau_k) \\
      \leq&\ -\bar\alpha_k \Delta l(x_k,\tau_k,g_k,d_k) + (1-\eta) \bar\alpha_k \beta^{(\bz{\sigma}-1)}  \Delta l(x_k,\bar\tau_k,\bar{g}_k,\bar{d}_k) \\
      & \ + \bar\alpha_k \bar\tau_k g_k^T (\bar{d}_k - d_k) + \bar\alpha_k (\bar\tau_k - \tau_k) g_k^T d_k + \bar\alpha_k\|J_k(\bar{d}_k - \tilde{d}_k)\|_1.
  \ealigned
  \eequation
    \eenumerate
\end{lemma}

We now proceed to state and prove a series of lemmas that are analogues to those proven in Section \ref{sec.adaptive}. 
We begin by bounding the variance in the stochastic gradient approximations and the variance in the search directions computed in Lemmas~\ref{lem.new} and \ref{lem.dtilde_stochastic}.

\begin{lemma}\label{lem.new}
      \asb{Suppose Condition~\ref{ass.stoch_g_sublin} holds.} For all $k\in\mathbb{N}$, $\mathbb{E}_k\left[\|\bar{g}_k - g_k\|_2\right] \leq \tfrac{\sqrt{\theta_1} \beta^{\bz{\sigma}}}{(k+1)^{\nu/2}}$.
\end{lemma}
\begin{proof}
By \asb{Condition}~\ref{ass.stoch_g_sublin} and Jensen's inequality, we have
\begin{equation*}
      \mathbb{E}_k\left[\|\bar{g}_k - g_k\|_2\right] \leq \sqrt{\mathbb{E}_k\left[\|\bar{g}_k - g_k\|_2^2\right]} \leq \tfrac{\sqrt{\theta_1} \beta^{\bz{\sigma}}}{(k+1)^{\nu/2}},
\end{equation*}
which completes the proof.
\end{proof}
\begin{lemma}\label{lem.dtilde_stochastic}
  \asb{Suppose Assumptions~\ref{ass.main}and~\ref{ass.H} and Condition~\ref{ass.stoch_g_sublin} hold.} For all $k$, $\mathbb{E}_k[\tilde{d}_k] = d_k$, $\mathbb{E}_k[\tilde{u}_k] = u_k$, and $\mathbb{E}_k[\tilde{\delta}_k] = \delta_k$.  Moreover, there exists $\kappa_L \in \R{}_{>0}$, 
  such that $\|\tilde{d}_k - d_k\|_2 \leq \kappa_L \|\bar{g}_k - g_k\|_2$ and $\mathbb{E}_k\left[\|\tilde{d}_k - d_k\|_2\right]  \leq  \tfrac{\kappa_{\tilde{d}}\beta^{\bz{\sigma}}}{(k+1)^{\nu/2}}$, 
where $\kappa_{\tilde{d}} = \kappa_L \sqrt{\theta_1} \in \mathbb{R}_{>0}$.
\end{lemma}
\begin{proof}
The first two statements follow the same arguments as in the proof of Lemma~\ref{lem.dtilde_stochastic_old}. By \eqref{eq.system} and Lemma~\ref{lem.new},
\begin{equation*}
    \mathbb{E}_k\left[\|\tilde{d}_k - d_k \|_2\right] \leq \mathbb{E}_k\left[\kappa_L \|\bar{g}_k - g_k \|_2 \right] = \kappa_L \mathbb{E}_k\left[\|\bar{g}_k - g_k \|_2 \right] \leq \kappa_L \beta^{\bz{\sigma}}\tfrac{\sqrt{\theta_1}}{(k+1)^{\nu/2}},
\end{equation*}
which proves the last statement.
\end{proof}

Similar to Lemma~\ref{lem.gT(d_diff)_stoch_old}, the next lemma provides bounds on the differences between stochastic and deterministic gradient approximations and exact and inexact search directions.
\begin{lemma}\label{lem.gT(d_diff)_stoch}
\asb{Suppose Assumptions~\ref{ass.main},~\ref{ass.H} and~\ref{ass.residual} and Condition~\ref{ass.stoch_linear_system} hold.} For all $k\in\N{}$,
\bequationNN
  \baligned
    |g_k^T(d_k - \tilde{d}_k)| \leq\ &\kappa_{g,d\tilde{d},\sqrt{\nu}} \|g_k - \bar{g}_k\|_2 \sqrt{\Delta l(x_k,\tau_k,g_k,d_k)},\\
    |\bar{g}_k^T\tilde{d}_k - g_k^Td_k| \leq\  &\kappa_{\bar{g}g,\tilde{d}d,\sqrt{\nu}}\|g_k - \bar{g}_k\|_2\sqrt{\Delta l(x_k,\tau_k,g_k,d_k)} + \kappa_L\|\bar{g}_k - g_k\|_2^2,\\
    |\bar{g}_k^T(\bar{d}_k - \tilde{d}_k)| \leq\     &\kappa_{\bar{g},\bar{d}\tilde{d},\sqrt{\nu}}\beta^{\bz{\sigma}}\tfrac{\sqrt{\Delta l(x_k,\tau_k,g_k,d_k)}}{\sqrt{(k+1)^{\nu}}} + \tfrac{\sqrt{\theta_2}\beta^{\bz{\sigma}}}{\sqrt{(k+1)^{\nu}}}\|\bar{g}_k - g_k\|_2 \\
    &+ \bar{\kappa}_{\bar{g},\bar{d}\tilde{d}} \beta^{\bz{\sigma}} \Delta l(x_k,\bar{\tau}_k,\bar{g}_k,\bar{d}_k),\\
    |g_k^T (\bar{d}_k - d_k)| \leq\ &\kappa_{g,d\tilde{d},\sqrt{\nu}} \|g_k - \bar{g}_k\|_2 \sqrt{\Delta l(x_k,\tau_k,g_k,d_k)} \\
    &+ \kappa_{g,\bar{d}d,\sqrt{\nu}}\beta^{\bz{\sigma}}\tfrac{\sqrt{\Delta l(x_k,\tau_k,g_k,d_k)}}{\sqrt{(k+1)^{\nu}}} + \bar{\kappa}_{g,\bar{d}d} \beta^{\bz{\sigma}} \Delta l(x_k,\bar{\tau}_k,\bar{g}_k,\bar{d}_k),\\
    |g_k^Td_k - \bar{g}_k^T\bar{d}_k| \leq\  &\kappa_{\bar{g}g,\tilde{d}d,\sqrt{\nu}} \|g_k - \bar{g}_k\|_2 \sqrt{\Delta l(x_k,\tau_k,g_k,d_k)} + \kappa_L\|\bar{g}_k - g_k\|^2_2 \\
    &+ \kappa_{\bar{g},\bar{d}\tilde{d},\sqrt{\nu}}\beta^{\bz{\sigma}}\tfrac{\sqrt{\Delta l(x_k,\tau_k,g_k,d_k)}}{\sqrt{(k+1)^{\nu}}} + \tfrac{\sqrt{\theta_2}\beta^{\bz{\sigma}}}{\sqrt{(k+1)^{\nu}}}\|\bar{g}_k - g_k\|_2 \\
    &+ \bar{\kappa}_{\bar{g},\bar{d}\tilde{d}} \beta^{\bz{\sigma}} \Delta l(x_k,\bar{\tau}_k,\bar{g}_k,\bar{d}_k),\\
    \text{and} \ \ \|J_k(\bar{d}_k - \tilde{d}_k)\|_1 \leq\  &\bar{\kappa}_{J,\bar{d}\tilde{d}} \beta^{\bz{\sigma}} \Delta l(x_k,\bar{\tau}_k,\bar{g}_k,\bar{d}_k), 
  \ealigned
\eequationNN
 where $\kappa_{g,d\tilde{d},\sqrt{\nu}} = \tfrac{\kappa_H\kappa_L}{\sqrt{\kappa_{\Delta l,d}}} \in \R{}_{>0}$,
 $\kappa_{\bar{g}g,\tilde{d}d,\sqrt{\nu}} = \kappa_{g,d\tilde{d},\sqrt{\nu}} + \tfrac{1}{\sqrt{\kappa_{\Delta l,d}}}  \in \R{}_{>0}$,
 $\kappa_{\bar{g},\bar{d}\tilde{d},\sqrt{\nu}} = \kappa_H \sqrt{\tfrac{\theta_2}{\kappa_{\Delta l,d}}} \in \R{}_{>0}$,
 $\bar{\kappa}_{\bar{g},\bar{d}\tilde{d}} = \kappa_{y\delta}\omega_a \in \R{}_{>0}$,
 $\kappa_{g,\bar{d}d,\sqrt{\nu}} = \kappa_H\sqrt{\tfrac{\theta_2}{\kappa_{\Delta l,d}}}\in \R{}_{>0}$,
 $\bar{\kappa}_{g,\bar{d}d} = \kappa_{y\delta}\omega_a \in \R{}_{>0}$,
 and $\bar{\kappa}_{J,\bar{d}\tilde{d}} = \omega_a \in \R{}_{>0}$. 
\end{lemma}

\begin{proof}
    \emph{(First inequality)} 
    By Assumption~\ref{ass.H}, Lemma~\ref{lem.dtilde_stochastic} and \eqref{eq.system}, \eqref{eq.system_deterministic}, \eqref{eq.deltal_d}, 
    \begin{equation}\label{eq.res1_stoch}
    \begin{aligned}
        |g_k^T(d_k - \tilde{d}_k)| =\ & |(g_k + J_k^T(y_k + \delta_k))^T(d_k - \tilde{d}_k)| \\
        =\ & |(H_kd_k)^T(d_k - \tilde{d}_k)|
        \leq\ \kappa_H\|d_k\|_2 \|d_k - \tilde{d}_k\|_2 \\
        \leq\ & \kappa_H \tfrac{\sqrt{\Delta l(x_k,\tau_k,g_k,d_k)}}{\sqrt{\kappa_{\Delta l,d}}} \kappa_L\|g_k - \bar{g}_k\|_2,
    \end{aligned}
    \end{equation}
    where the result follows using the definition of $\kappa_{g,d\tilde{d},\sqrt{\nu}}$.
    
    \emph{(Second inequality)} 
    By the Cauchy–Schwarz inequality, the triangle inequality, Assumption~\ref{ass.H}, \eqref{eq.system}, \eqref{eq.system_deterministic}, \eqref{eq.deltal_d} and \eqref{eq.res1_stoch}, and Lemma~\ref{lem.dtilde_stochastic}, it follows that
  \begin{equation}\label{eq.res2_stoch}
    \begin{aligned}
      &|\bar{g}_k^T\tilde{d}_k - g_k^Td_k| \\
      \leq\ & |(\bar{g}_k - g_k)^Td_k| + |g_k^T(\tilde{d}_k - d_k)| + |(\bar{g}_k - g_k)^T(\tilde{d}_k - d_k)| \\
      \leq\ &\|\bar{g}_k - g_k\|_2\|d_k\|_2 + |g_k^T(\tilde{d}_k - d_k)| + \|\bar{g}_k - g_k\|_2\|\tilde{d}_k - d_k\|_2 \\
      \leq\ &\left(\kappa_{g,d\tilde{d},\sqrt{\nu}} + \tfrac{1}{\sqrt{\kappa_{\Delta l,d}}} \right) \|g_k - \bar{g}_k\|_2 \sqrt{\Delta l(x_k,\tau_k,g_k,d_k)} + \kappa_L\|\bar{g}_k - g_k\|^2_2,
  \end{aligned}
  \end{equation}
  where the result follows using the definitions of $\kappa_{\bar{g}g,\tilde{d}d,\sqrt{\nu}}$. 
  
  \emph{(Third inequality)} 
  By Assumption~\ref{ass.H}, Lemmas~\ref{lem.residual},~\ref{lem.bound_ydelta} and~\ref{lem.dtilde_stochastic}, \eqref{eq.system}, \eqref{eq.system_stochastic}, \eqref{eq.system_deterministic}, and \eqref{eq.deltal_d}, we have for all $k\in\mathbb{N}$ that
  \begin{equation}\label{eq.res3_stoch}
    \begin{aligned}
      |\bar{g}_k^T(\bar{d}_k -\tilde{d}_k)| =\ & |(\bar{g}_k + J_k^T(y_k + \delta_k) - J_k^T(y_k + \delta_k))^T(\bar{d}_k -\tilde{d}_k)| \\
      \leq\ & |(\bar{g}_k + J_k^T(y_k + \delta_k) )^T(\bar{d}_k -\tilde{d}_k)| + |( J_k^T(y_k + \delta_k))^T(\bar{d}_k -\tilde{d}_k)| \\
      \leq\ & |(g_k + J_k^T(y_k + \delta_k) )^T(\bar{d}_k -\tilde{d}_k)| + |(\bar{g}_k - g_k)^T(\bar{d}_k -\tilde{d}_k)| \\
      &+ |( J_k^T(y_k + \delta_k))^T(\bar{d}_k -\tilde{d}_k)| \\
      =\ & |(H_k d_k )^T(\bar{d}_k -\tilde{d}_k)| + |(\bar{g}_k - g_k)^T(\bar{d}_k -\tilde{d}_k)| + |(y_k + \delta_k)^T \bar{r}_k| \\
      \leq\ & |(H_k d_k )^T(\bar{d}_k -\tilde{d}_k)| + \|\bar{g}_k - g_k\|_2\|\bar{d}_k - \tilde{d}_k\|_2 + \|y_k + \delta_k\|_{\infty}  \|\bar{r}_k\|_1 \\
      \leq\ & \kappa_H  \sqrt{\tfrac{\Delta l(x_k,\tau_k,g_k,d_k)}{\kappa_{\Delta l,d}}} \tfrac{\beta^{\bz{\sigma}}\sqrt{\theta_2}}{\sqrt{(k+1)^{\nu}}} + \tfrac{\beta^{\bz{\sigma}}\sqrt{\theta_2}}{\sqrt{(k+1)^{\nu}}}\|\bar{g}_k - g_k\|_2 \\
      & + \kappa_{y\delta} \omega_a \beta^{\bz{\sigma}} \Delta l(x_k,\bar\tau_k,\bar{g}_k,\bar{d}_k), 
  \end{aligned}
  \end{equation}
  where the result follows using the definitions of $\kappa_{\bar{g},\bar{d}\tilde{d},\sqrt{\nu}}$ and $\bar{\kappa}_{\bar{g},\bar{d}\tilde{d}}$. 
  
  \emph{(Fourth inequality)} 
  By Assumption~\ref{ass.H}, Lemmas~\ref{lem.residual} and~\ref{lem.bound_ydelta}, \eqref{eq.system}, \eqref{eq.system_stochastic}, \eqref{eq.system_deterministic} and \eqref{eq.deltal_d}, it follows that
\begin{equation*}
\begin{aligned}
    |g_k^T(\bar{d}_k - \tilde{d}_k)| =\ & |(g_k + J_k^T(y_k + \delta_k) - J_k^T(y_k + \delta_k))^T(\bar{d}_k -\tilde{d}_k)| \\
    \leq\ & |(g_k + J_k^T(y_k + \delta_k) )^T(\bar{d}_k -\tilde{d}_k)| + |( J_k^T(y_k + \delta_k))^T(\bar{d}_k -\tilde{d}_k)| \\
    \leq\ & |(H_k d_k )^T(\bar{d}_k -\tilde{d}_k)| + \|y_k + \delta_k\|_{\infty} \|\bar{r}_k\|_1 \\
    \leq\ & \kappa_H  \sqrt{\tfrac{\Delta l(x_k,\tau_k,g_k,d_k)}{\kappa_{\Delta l,d}}} \beta^{\bz{\sigma}} \tfrac{\sqrt{\theta_2}}{\sqrt{(k+1)^{\nu}}}  + \kappa_{y\delta} \omega_a \beta^{\bz{\sigma}} \Delta l(x_k,\bar\tau_k,\bar{g}_k,\bar{d}_k).
\end{aligned}
\end{equation*}
By the triangle inequality and \eqref{eq.res1_stoch},
\begin{equation*}
\begin{aligned}
    |g_k^T(\bar{d}_k - d_k)| \leq\ & |g_k^T(\tilde{d}_k - d_k)| + |g_k^T(\bar{d}_k - \tilde{d}_k)| \\
    \leq\ &\kappa_{g,d\tilde{d},\sqrt{\nu}}\|g_k - \bar{g}_k\|_2\sqrt{\Delta l(x_k,\tau_k,g_k,d_k)} \\
    & + \kappa_H\sqrt{\tfrac{\theta_2}{\kappa_{\Delta l,d}}} \beta^{\bz{\sigma}} \tfrac{\sqrt{\Delta l(x_k,\tau_k,g_k,d_k)}}{\sqrt{(k+1)^{\nu}}} + \kappa_{y\delta} \omega_a \beta^{\bz{\sigma}} \Delta l(x_k,\bar\tau_k,\bar{g}_k,\bar{d}_k),
\end{aligned}
\end{equation*}
where the result follows using the definition of $\kappa_{g,\bar{d}d,\sqrt{\nu}}$ and $\bar{\kappa}_{g,\bar{d}d}$. 

\emph{(Fifth inequality)} 
By \eqref{eq.res2_stoch}, \eqref{eq.res3_stoch},
\begin{equation*}
\begin{aligned}
    |g_k^Td_k - \bar{g}_k^T\bar{d}_k|  \leq\ & |g_k^Td_k - \bar{g}_k^T\tilde{d}_k| + |\bar{g}_k^T(\tilde{d}_k - \bar{d}_k)| \\
    \leq\ & \kappa_{\bar{g}g,\tilde{d}d,\sqrt{\nu}} \|g_k - \bar{g}_k\|_2 \sqrt{\Delta l(x_k,\tau_k,g_k,d_k)} + \kappa_L\|\bar{g}_k - g_k\|^2_2 \\
    &+ \kappa_{\bar{g},\bar{d}\tilde{d},\sqrt{\nu}}\beta^{\bz{\sigma}}\tfrac{\sqrt{\Delta l(x_k,\tau_k,g_k,d_k)}}{\sqrt{(k+1)^{\nu}}} + \tfrac{\sqrt{\theta_2}\beta^{\bz{\sigma}}}{\sqrt{(k+1)^{\nu}}}\|\bar{g}_k - g_k\|_2\\
    &+ \bar{\kappa}_{\bar{g},\bar{d}\tilde{d}} \beta^{\bz{\sigma}} \Delta l(x_k,\bar{\tau}_k,\bar{g}_k,\bar{d}_k).
\end{aligned}
\end{equation*}

\emph{(Sixth inequality)} 
By Lemma~\ref{lem.residual}, and \eqref{eq.system},  \eqref{eq.system_stochastic},
\begin{equation*}
    \|J_k(\bar{d}_k - \tilde{d}_k)\|_1 = \|\bar{r}_k\|_1 \leq \omega_a\beta^{\bz{\sigma}} \Delta l(x_k,\bar{\tau}_k,\bar{g}_k,\bar{d}_k),
\end{equation*}
where the result follows using the definition of $\bar\kappa_{J,\bar{d}\tilde{d}}$.
\end{proof}

The next lemma provides a useful upper bound on the difference between the deterministic and stochastic merit parameters.

\begin{lemma}\label{lem.tau_bound_stoch}
    \asb{Suppose Assumptions~~\ref{ass.H} and~\ref{ass.residual} and Condition~\ref{assum:stoch_redcond_old} hold.} For all $k\in\N{}$, $|(\bar\tau_k - \tau_k)g_k^Td_k|
    \leq \kappa_{\bar{\tau}}\beta^{\bz{\sigma}}\Delta l(x_k,\tau_k,g_k,d_k)$, 
where $\kappa_{\bar{\tau}} = 2\theta_3 \bar{\tau}_{-1}\kappa_{gd,\Delta l}$ and $\beta\in \left(0,\sfrac{1}{(2\theta_3)^{\bz{1/\sigma}}}\right]$. 
\end{lemma}
\begin{proof}
The proof of this statement is identical to the proof of Lemma \ref{lem.tau_bound_old}.
\end{proof}

Similar to Lemma~\ref{lem.delta_l_bound_stoch_old}, the next lemma bounds the stochastic model reduction function with respect to its deterministic counterpart (with additional terms).
\begin{lemma}\label{lem.delta_l_bound_stoch}
      \asb{Suppose Assumptions~\ref{ass.main},~\ref{ass.H} and~\ref{ass.residual} and Conditions~\ref{assum:stoch_redcond_old} and~\ref{ass.stoch_linear_system} hold.} For all $k \in \N{}$ and $\beta \in \left(0,\sfrac{1}{\left(2\bar\tau_{-1}\bar\kappa_{\bar{g},\bar{d}\tilde{d}}\right)^{\bz{1/\sigma}}}\right]$, 
   \begin{equation*}
   \begin{aligned}
   \Delta l(x_k,\bar\tau_k,\bar{g}_k,\bar{d}_k) \leq\ &(1 + \kappa_{\overline{\Delta l},\Delta l}\beta^{\bz{\sigma}})\Delta l(x_k,\tau_k,g_k,d_k) \\
   &+ 2\bar\tau_{-1}\left( \kappa_L\|\bar{g}_k - g_k\|^2_2 + \kappa_{\bar{g},\bar{d}\tilde{d},\sqrt{\nu}}\tfrac{\beta^{\bz{\sigma}}\sqrt{\Delta l(x_k,\tau_k,g_k,d_k)}}{\sqrt{(k+1)^{\nu}}} \right. \\
   &\left. + \kappa_{\bar{g}g,\tilde{d}d,\sqrt{\nu}} \|g_k - \bar{g}_k\|_2 \sqrt{\Delta l(x_k,\tau_k,g_k,d_k)} + \tfrac{\sqrt{\theta_2}\beta^{\bz{\sigma}}\|\bar{g}_k - g_k\|_2}{\sqrt{(k+1)^{\nu}}} \right),
   \end{aligned}
   \end{equation*}
   where $\kappa_{\overline{\Delta l},\Delta l} =  2(\kappa_{\bar\tau} + \bar\tau_{-1}\bar\kappa_{\bar{g},\bar{d}\tilde{d}}) \in \mathbb{R}_{>0}$. Additionally, under \asb{Condition~\ref{ass.stoch_g_sublin}}, for all $k \in \N{}$
   \begin{equation*}
   \begin{aligned}
       \mathbb{E}_k\left[\Delta l(x_k,\bar\tau_k,
       \bar{g}_k,\bar{d}_k)\right] \leq\ &(1 + \bar\kappa_{\overline{\Delta l},\Delta l}\beta^{\bz{\sigma}}) \Delta l(x_k,\tau_k,g_k,d_k) \\
       &+ \bar\kappa_{\overline{\Delta l},\Delta l,\sqrt{\nu}}\beta^{\bz{\sigma}}\tfrac{\sqrt{\Delta l(x_k,\tau_k,g_k,d_k)}}{\sqrt{(k+1)^{\nu}}} + \bar\kappa_{\overline{\Delta l},\Delta l,\nu} \beta^{\bz{2\sigma}}\tfrac{1}{(k+1)^{\nu}},
   \end{aligned}
   \end{equation*}
   where $\bar\kappa_{\overline{\Delta l},\Delta l} = \kappa_{\overline{\Delta l},\Delta l}   \in \mathbb{R}_{>0}$, $\bar\kappa_{\overline{\Delta l},\Delta l,\sqrt{\nu}} = 2\bar\tau_{-1}(\kappa_{\bar{g},\bar{d}\tilde{d},\sqrt{\nu}} + \kappa_{\bar{g}g,\tilde{d}d,\sqrt{\nu}}\sqrt{\theta_1}) \in \mathbb{R}_{>0}$, and $\bar\kappa_{\overline{\Delta l},\Delta l,\nu} = 2\bar\tau_{-1}(\kappa_L\theta_1 + \sqrt{\theta_1\theta_2}) \in \mathbb{R}_{>0}$.
\end{lemma}

\begin{proof}
By \eqref{eq.model_reduction}, and Lemmas~\ref{lem.gT(d_diff)_stoch} and \ref{lem.tau_bound_stoch}, it follows that
\begin{equation*}
\begin{aligned}
    & \Delta l(x_k,\bar\tau_k,\bar{g}_k,\bar{d}_k) = -\bar\tau_k\bar{g}_k^T\bar{d}_k + \|c_k\|_1 - \|c_k + J_k\bar{d}_k\|_1 \\
    \leq\ &-\tau_kg_k^Td_k + \|c_k\|_1 + (\tau_k - \bar\tau_k)g_k^Td_k + \bar\tau_k(g_k^Td_k - \bar{g}_k^T\bar{d}_k) \\
    \leq\ &\Delta l(x_k,\tau_k,g_k,d_k) + |(\tau_k - \bar\tau_k)g_k^Td_k| + \bar\tau_{-1}|g_k^Td_k - \bar{g}_k^T\bar{d}_k| \\
    \leq\ & \Delta l(x_k,\tau_k,g_k,d_k) + \kappa_{\bar{\tau}}\beta^{\bz{\sigma}}\Delta l(x_k,\tau_k,g_k,d_k) \\
    &+ \bar\tau_{-1}\left(\kappa_{\bar{g}g,\tilde{d}d,\sqrt{\nu}} \|g_k - \bar{g}_k\|_2 \sqrt{\Delta l(x_k,\tau_k,g_k,d_k)} + \kappa_L\|\bar{g}_k - g_k\|^2_2 \right. \\
    &\left. + \kappa_{\bar{g},\bar{d}\tilde{d},\sqrt{\nu}}\tfrac{\beta^{\bz{\sigma}}\sqrt{\Delta l(x_k,\tau_k,g_k,d_k)}}{\sqrt{(k+1)^{\nu}}} + \tfrac{\sqrt{\theta_2}\beta^{\bz{\sigma}}\|\bar{g}_k - g_k\|_2}{\sqrt{(k+1)^{\nu}}} + \bar{\kappa}_{\bar{g},\bar{d}\tilde{d}} \beta^{\bz{\sigma}} \Delta l(x_k,\bar{\tau}_k,\bar{g}_k,\bar{d}_k)\right).
\end{aligned}
\end{equation*}
The first result follows by re-arranging the above, invoking the restriction on $\beta$, and the definition of $\kappa_{\overline{\Delta l},\Delta l}$.

Taking the conditional expectation, by Assumption~\ref{ass.stoch_g_sublin} and Lemma~\ref{lem.new}, 
\begin{equation*}
\begin{aligned}
\mathbb{E}_k\left[ \Delta l(x_k,\bar\tau_k,\bar{g}_k,\bar{d}_k) \right] \leq\ &\left(1 + 2(\kappa_{\bar\tau} + \bar\tau_{-1}\bar\kappa_{\bar{g},\bar{d}\tilde{d}})\beta^{\bz{\sigma}}\right)\Delta l(x_k,\tau_k,g_k,d_k) \\
&+ 2\bar\tau_{-1}\left( \tfrac{\kappa_L\theta_1\beta^{\bz{2\sigma}}}{(k+1)^{\nu}} + \kappa_{\bar{g},\bar{d}\tilde{d},\sqrt{\nu}}\tfrac{\beta^{\bz{\sigma}}\sqrt{\Delta l(x_k,\tau_k,g_k,d_k)}}{\sqrt{(k+1)^{\nu}}} \right. \\
&\left. + \kappa_{\bar{g}g,\tilde{d}d,\sqrt{\nu}}\sqrt{\theta_1} \tfrac{\beta^{\bz{\sigma}}\sqrt{\Delta l(x_k,\tau_k,g_k,d_k)}}{\sqrt{(k+1)^{\nu}}} + \tfrac{\sqrt{\theta_1\theta_2}\beta^{\bz{2\sigma}}}{(k+1)^{\nu}} \right).
\end{aligned}
\end{equation*}
Using the definitions of $\kappa_{\overline{\Delta l},\Delta l}$, $\bar\kappa_{\overline{\Delta l},\Delta l}$, $\bar\kappa_{\overline{\Delta l},\Delta l,\sqrt{\nu}}$, and $\bar\kappa_{\overline{\Delta l},\Delta l,\nu}$, completes the proof.
\end{proof}

Finally, we restate and prove the theoretical results stated in Section~\ref{sec.pred}. Specifically, we state and prove Lemma~\ref{lem:merit_bnd_stoch_mp}, Theorem~\ref{thm.main_stochastic_mp} and Corollary~\ref{cor.iter_complexity_raghu}.
\begin{lemma}\label{lem:merit_bnd_stoch}
    (Lemma~\ref{lem:merit_bnd_stoch_mp})  \asb{Suppose Assumptions~\ref{ass.main},~\ref{ass.H} and~\ref{ass.residual} and Conditions~\ref{ass.stoch_g_sublin},~\ref{ass.stoch_linear_system} and~\ref{assum:stoch_redcond_old} hold.} For all $k\in\mathbb{N}$,
    \begin{equation*}
    \begin{aligned}
        \mathbb{E}_k\left[\phi(x_{k+1},\bar\tau_{k+1}) - \phi(x_k,\bar\tau_k)\right] \leq & \mathbb{E}_k\left[(\bar\tau_{k+1} - \bar\tau_k)\right] f_{\inf} -  \beta (\underline{\alpha}\eta - \bar{\kappa}_\phi\beta  ) \Delta l(x_k,\tau_k,g_k,d_k) \\
         & + \beta^2 \tfrac{\bar{\kappa}_{\phi,\nu}}{(k+1)^{\nu}}.
    \end{aligned}
    \end{equation*}
    where $\bar{\kappa}_{\phi} = \kappa_{\phi} + \tfrac{\kappa_{\phi,\sqrt{\nu}}}{2} \in \mathbb{R}_{>0}$ and $\bar{\kappa}_{\phi,\nu} = \kappa_{\phi,\nu} + \tfrac{\kappa_{\phi,\sqrt{\nu}}}{2} \in \mathbb{R}_{>0}$, and 
    \begin{align*}
        \kappa_{\phi} &= \alpha_u\left( (1-\eta)\kappa_{\overline{\Delta l},\Delta l} + \kappa_{\bar\tau} + (\bar\tau_{-1}\bar\kappa_{g,\bar{d}d} + \bar\kappa_{J,\bar{d}\tilde{d}})(1 + \bar\kappa_{\overline{\Delta l},\Delta l}) \right) \in \mathbb{R}_{>0}, \\
        \kappa_{\phi,\nu} &= \alpha_u\left( 2\bar\tau_{-1}(1-\eta)(\kappa_L\theta_1 + \sqrt{\theta_1\theta_2}) + (\bar\tau_{-1}\bar\kappa_{g,\bar{d}d} + \bar\kappa_{J,\bar{d}\tilde{d}})\bar\kappa_{\overline{\Delta l},\Delta l,\nu} \right) \in \mathbb{R}_{>0},\\
        \kappa_{\phi,\sqrt{\nu}} &= \alpha_u\left( 2\bar\tau_{-1}(1-\eta)\left(\kappa_{\bar{g},\bar{d}\tilde{d},\sqrt{\nu}} + \sqrt{\theta_1}\kappa_{\bar{g}g,\tilde{d}d,\sqrt{\nu}}\right) \right.\\
        & \quad + \left. \bar\tau_{-1} \left( \sqrt{\theta_1}\kappa_{g,d\tilde{d},\sqrt{\nu}} + \kappa_{g,\bar{d}d,\sqrt{\nu}} \right) + (\bar\tau_{-1}\bar\kappa_{g,\bar{d}d} + \bar\kappa_{J,\bar{d}\tilde{d}})\bar\kappa_{\overline{\Delta l},\Delta l,\sqrt{\nu}} \right) \in \mathbb{R}_{>0}.
    \end{align*}
\end{lemma}
\begin{proof} 

By Lemmas~\ref{lemma.combo}, \ref{lem.tau_bound_stoch}, \ref{lem.gT(d_diff)_stoch},  \ref{lem.delta_l_bound_stoch} and \ref{lem.stoch_stepsize_bound_old}, and \eqref{eq.merit}, it follows that for all $k \in \N{}$
    \begin{align*}
      &\mathbb{E}_k\left[\phi(x_{k+1},\bar\tau_{k+1}) - \phi(x_k,\bar\tau_k)\right] \\
      =\ & \mathbb{E}_k\left[(\bar\tau_{k+1} - \bar\tau_k)f_{k+1}\right] + \mathbb{E}_k\left[\phi(x_{k+1},\bar\tau_k) - \phi(x_k,\bar\tau_k)\right] \\
      \leq\ & \mathbb{E}_k\left[(\bar\tau_{k+1} - \bar\tau_k)\right] f_{\inf} \\
      &- \mathbb{E}_k\left[\bar\alpha_k\Delta l(x_k,\tau_k,g_k,d_k) - (1-\eta) \bar\alpha_k \beta^{(\bz{\sigma}-1)} \Delta l(x_k,\bar\tau_k,\bar{g}_k,\bar{d}_k) \right]\\
      &+ \mathbb{E}_k\left[\bar\alpha_k \bar\tau_k g_k^T (\bar{d}_k - d_k)\right] + \mathbb{E}_k\left[\bar\alpha_k (\bar\tau_k - \tau_k) g_k^T d_k\right] + \mathbb{E}_k\left[\bar\alpha_k\|J_k(\bar{d}_k - \tilde{d}_k)\|_1 \right]\\
      \leq\ & \mathbb{E}_k\left[(\bar\tau_{k+1} - \bar\tau_k)\right] f_{\inf} - \mathbb{E}_k\left[\bar\alpha_k \Delta l(x_k,\tau_k,g_k,d_k)\right] \\
      & + \mathbb{E}_k\left[(1-\eta) \bar\alpha_k \beta^{(\bz{\sigma}-1)} \left((1 + \kappa_{\overline{\Delta l},\Delta l}\beta^{\bz{\sigma}})\Delta l(x_k,\tau_k,g_k,d_k) \right. \right. \\
      & \quad + 2\bar\tau_{-1}\left( \kappa_L\|\bar{g}_k - g_k\|^2_2 + \kappa_{\bar{g},\bar{d}\tilde{d},\sqrt{\nu}}\tfrac{\beta^{\bz{\sigma}}\sqrt{\Delta l(x_k,\tau_k,g_k,d_k)}}{\sqrt{(k+1)^{\nu}}} \right.\\
      & \quad\quad \left.\left.\left.+ \kappa_{\bar{g}g,\tilde{d}d,\sqrt{\nu}} \|g_k - \bar{g}_k\|_2 \sqrt{\Delta l(x_k,\tau_k,g_k,d_k)} + \tfrac{\sqrt{\theta_2}\beta^{\bz{\sigma}}\|\bar{g}_k - g_k\|_2}{\sqrt{(k+1)^{\nu}}} \right)\right) \right]\\
      & + \alpha_u\beta^{(2-\bz{\sigma})}\bar\tau_{-1}\mathbb{E}_k\left[|g_k^T (\bar{d}_k - d_k)|\right] + \alpha_u\beta^{(2-\bz{\sigma})}\mathbb{E}_k\left[ |(\bar\tau_k - \tau_k) g_k^T d_k|\right] \\
      & + \alpha_u\beta^{(2-\bz{\sigma})}\mathbb{E}_k\left[\|J_k(\bar{d}_k - \tilde{d}_k)\|_1 \right]\\
      \leq\ & \mathbb{E}_k\left[(\bar\tau_{k+1} - \bar\tau_k)\right] f_{\inf} - \mathbb{E}_k\left[\eta\bar\alpha_k \Delta l(x_k,\tau_k,g_k,d_k)\right] \\
      & + (1-\eta)\kappa_{\overline{\Delta l},\Delta l}\alpha_u\beta^{(1+\bz{\sigma})} \Delta l(x_k,\tau_k,g_k,d_k) \\
      & + 2\bar\tau_{-1}(1-\eta)\alpha_u\beta\left( \kappa_L\tfrac{\theta_1\beta^{\bz{2\sigma}}}{(k+1)^{\nu}} + \kappa_{\bar{g},\bar{d}\tilde{d},\sqrt{\nu}}\tfrac{\beta^{\bz{\sigma}}\sqrt{\Delta l(x_k,\tau_k,g_k,d_k)}}{\sqrt{(k+1)^{\nu}}} \right.\\
      & \quad\quad \left. + \kappa_{\bar{g}g,\tilde{d}d,\sqrt{\nu}} \tfrac{\sqrt{\theta_1}\beta^{\bz{\sigma}}}{\sqrt{(k+1)^{\nu}}} \sqrt{\Delta l(x_k,\tau_k,g_k,d_k)} + \tfrac{\sqrt{\theta_1\theta_2}\beta^{\bz{2\sigma}}}{(k+1)^{\nu}} \right) \\
      & + \alpha_u\beta^{(2-\bz{\sigma})}\bar\tau_{-1}\left( \kappa_{g,d\tilde{d},\sqrt{\nu}}\tfrac{\sqrt{\theta_1}\beta^{\bz{\sigma}}\sqrt{\Delta l(x_k,\tau_k,g_k,d_k)}}{\sqrt{(k+1)^{\nu}}} + \kappa_{g,\bar{d}d,\sqrt{\nu}}\tfrac{\beta^{\bz{\sigma}}\sqrt{\Delta l(x_k,\tau_k,g_k,d_k)}}{\sqrt{(k+1)^{\nu}}} \right.\\
      & \quad\quad \left.+ \bar\kappa_{g,\bar{d}d}\beta^{\bz{\sigma}}\mathbb{E}_k\left[\Delta l(x_k,\bar\tau_k,\bar{g}_k,\bar{d}_k)\right] \right)
    \end{align*}  
Continuing from the above, by Lemmas~\ref{lem.delta_l_bound_stoch} and \ref{lem.stoch_stepsize_bound_old}, it follows that for all $k \in \N{}$
    \begin{align*}
    &\mathbb{E}_k\left[\phi(x_{k+1},\bar\tau_{k+1}) - \phi(x_k,\bar\tau_k)\right] \\
      \leq\ & \mathbb{E}_k\left[(\bar\tau_{k+1} - \bar\tau_k)\right] f_{\inf} \\
      & - \underline{\alpha}\beta\eta \Delta l(x_k,\tau_k,g_k,d_k) + (1-\eta)\kappa_{\overline{\Delta l},\Delta l}\alpha_u\beta^{(1 + \bz{\sigma})} \Delta l(x_k,\tau_k,g_k,d_k)\\
      & + 2\bar\tau_{-1}(1-\eta)\alpha_u\beta^{(1+\bz{\sigma})}\left( \left(\kappa_{\bar{g},\bar{d}\tilde{d},\sqrt{\nu}} + \sqrt{\theta_1}\kappa_{\bar{g}g,\tilde{d}d,\sqrt{\nu}}\right)\tfrac{\sqrt{\Delta l(x_k,\tau_k,g_k,d_k)}}{\sqrt{(k+1)^{\nu}}} \right.\\
      & \quad\quad\left.+ \left(\kappa_L\theta_1 + \sqrt{\theta_1\theta_2}\right)\tfrac{\beta^{\bz{\sigma}}}{(k+1)^{\nu}} \right)  \\
      & + \alpha_u\bar\tau_{-1}\beta^2 \left( \sqrt{\theta_1}\kappa_{g,d\tilde{d},\sqrt{\nu}} + \kappa_{g,\bar{d}d,\sqrt{\nu}} \right)\tfrac{\sqrt{\Delta l(x_k,\tau_k,g_k,d_k)}}{\sqrt{(k+1)^{\nu}}} + \alpha_u\kappa_{\bar\tau}\beta^2\Delta l(x_k,\tau_k,g_k,d_k) \\
      & + \alpha_u\beta^2 (\bar\tau_{-1}\bar\kappa_{g,\bar{d}d} + \bar\kappa_{J,\bar{d}\tilde{d}})\left( (1 + \bar\kappa_{\overline{\Delta l},\Delta l}\beta^{\bz{\sigma}}) \Delta l(x_k,\tau_k,g_k,d_k) \right.\\
      & \left. \quad\quad +  \bar\kappa_{\overline{\Delta l},\Delta l,\sqrt{\nu}}\beta^{\bz{\sigma}}\tfrac{\sqrt{\Delta l(x_k,\tau_k,g_k,d_k)}}{\sqrt{(k+1)^{\nu}}} + \bar\kappa_{\overline{\Delta l},\Delta l,\nu} \beta^{\bz{2\sigma}}\tfrac{1}{(k+1)^{\nu}} \right) \\
      \leq\ & \mathbb{E}_k\left[(\bar\tau_{k+1} - \bar\tau_k)\right] f_{\inf} - \underline{\alpha}\beta\eta \Delta l(x_k,\tau_k,g_k,d_k) \\
      &+ \alpha_u\left( (1-\eta)\kappa_{\overline{\Delta l},\Delta l} + \kappa_{\bar\tau} + (\bar\tau_{-1}\bar\kappa_{g,\bar{d}d} + \bar\kappa_{J,\bar{d}\tilde{d}})(1 + \bar\kappa_{\overline{\Delta l},\Delta l}) \right)\beta^{2}\Delta l(x_k,\tau_k,g_k,d_k) \\
      &+ \alpha_u\left( 2\bar\tau_{-1}(1-\eta)\left(\kappa_{\bar{g},\bar{d}\tilde{d},\sqrt{\nu}} + \sqrt{\theta_1}\kappa_{\bar{g}g,\tilde{d}d,\sqrt{\nu}}\right) + \bar\tau_{-1} \left( \sqrt{\theta_1}\kappa_{g,d\tilde{d},\sqrt{\nu}} + \kappa_{g,\bar{d}d,\sqrt{\nu}} \right) \right. \\
      &\left. + (\bar\tau_{-1}\bar\kappa_{g,\bar{d}d} + \bar\kappa_{J,\bar{d}\tilde{d}})\bar\kappa_{\overline{\Delta l},\Delta l,\sqrt{\nu}} \right)\beta^2 \tfrac{\sqrt{\Delta l(x_k,\tau_k,g_k,d_k)}}{\sqrt{(k+1)^{\nu}}} \\
      &+ \alpha_u\left( 2\bar\tau_{-1}(1-\eta)(\kappa_L\theta_1 + \sqrt{\theta_1\theta_2}) + (\bar\tau_{-1}\bar\kappa_{g,\bar{d}d} + \bar\kappa_{J,\bar{d}\tilde{d}})\bar\kappa_{\overline{\Delta l},\Delta l,\nu} \right)\beta^2 \tfrac{1}{(k+1)^{\nu}}\\
      =\ & \mathbb{E}_k\left[(\bar\tau_{k+1} - \bar\tau_k)\right] f_{\inf} - \underline{\alpha}\beta\eta \Delta l(x_k,\tau_k,g_k,d_k) + \kappa_{\phi} \beta^{2}\Delta l(x_k,\tau_k,g_k,d_k) \\
      & + \kappa_{\phi,\sqrt{\nu}}\beta^2\tfrac{\sqrt{\Delta l(x_k,\tau_k,g_k,d_k)}}{\sqrt{(k+1)^{\nu}}} + \kappa_{\phi,\nu}\beta^2\tfrac{1}{(k+1)^{\nu}} \\
      \leq\ & \mathbb{E}_k\left[(\bar\tau_{k+1} - \bar\tau_k)\right] f_{\inf} - \underline{\alpha}\beta\eta \Delta l(x_k,\tau_k,g_k,d_k) + \left(\kappa_{\phi} + \tfrac{\kappa_{\phi,\sqrt{\nu}}}{2} \right) \beta^{2}\Delta l(x_k,\tau_k,g_k,d_k) \\
      & +  \left(\kappa_{\phi,\nu} + \tfrac{\kappa_{\phi,\sqrt{\nu}}}{2}\right)\beta^2\tfrac{1}{(k+1)^{\nu}},
  \end{align*}
  where the result follows using the definitions of $\kappa_{\phi}$, $\kappa_{\phi,\sqrt{\nu}}$, $\kappa_{\phi,\nu}$, $\bar{\kappa}_{\phi}$ and $\bar{\kappa}_{\phi,\nu}$.
\end{proof}

We are now ready to prove the main theorem of this section.
\begin{theorem}\label{thm.main_stochastic}
   (Theorem~\ref{thm.main_stochastic_mp})  \asb{Suppose Assumptions~\ref{ass.main},~\ref{ass.H} and~\ref{ass.residual} and Conditions~\ref{ass.stoch_g_sublin},~\ref{ass.stoch_linear_system} and~\ref{assum:stoch_redcond_old} hold. By choosing $\beta \in \left(0, \min\left\{\sfrac{1}{(2\theta_3)^{\change{\tfrac{1}{\sigma}}}},\sfrac{(1-\gamma)\eta\underline{\alpha}}{\bar{\kappa}_{\phi}},\sfrac{1}{\left(2\bar\tau_{-1}\bar\kappa_{\bar{g},\bar{d}\tilde{d}}\right)^{\change{\tfrac{1}{\sigma}}}}\right\}\right]$ for any $\gamma\in (0,1)$ and $\nu \in \mathbb{R}_{>1}$,}
   \begin{equation*}
       \lim_{k\to\infty}\mathbb{E}\left[\sum_{j=0}^{k-1}\Delta l(x_j,\tau_j,g_j,d_j)\right] < \infty,
   \end{equation*}
   from which it follows that $\lim_{k\to\infty}\mathbb{E}\left[\Delta l(x_k,\tau_k,g_k,d_k)\right] = 0$.
\end{theorem}
\begin{proof}
    By Lemma~\ref{lem:merit_bnd_stoch} and $\beta \in \left(0, \tfrac{(1-\gamma)\eta\underline{\alpha}}{\bar{\kappa}_{\phi}}\right]$, it follows that 
    \bequation\label{eq.main_dec_stoch}
    \baligned
      &\mathbb{E}_k\left[\phi(x_{k+1},\bar\tau_{k+1}) - \phi(x_k,\bar\tau_k)\right] \\
      \leq\ & \mathbb{E}_k\left[(\bar\tau_{k+1} - \bar\tau_k)\right] f_{\inf} -  \beta \left(\underline{\alpha}\eta - \bar{\kappa}_{\phi}\beta\right)\Delta l(x_k,\tau_k,g_k,d_k) + \beta^2\tfrac{\bar{\kappa}_{\phi,\nu}}{(k+1)^{\nu}}\\
      \leq\ & \mathbb{E}_k\left[(\bar\tau_{k+1} - \bar\tau_k)\right] f_{\inf} - \underline{\alpha}\beta\gamma\eta \Delta l(x_k,\tau_k,g_k,d_k) + \beta^2\tfrac{\bar{\kappa}_{\phi,\nu}}{(k+1)^{\nu}}.
  \ealigned
    \eequation
    Applying a telescopic sum to \eqref{eq.main_dec_stoch} and taking the total expectation, it follows that
    \begin{equation*}
    \begin{aligned}
        -\infty <\ &\phi_{\inf} - \phi(x_0,\bar\tau_0) \leq \mathbb{E}[\phi(x_k,\bar\tau_k) - \phi(x_0,\bar\tau_0)] \\
        =\ &\mathbb{E}\left[ \sum_{j=0}^{k-1} \left(\phi(x_{j+1},\bar\tau_{j+1}) - \phi(x_j,\bar\tau_j)\right)\right] \\
        \leq\ &\mathbb{E}\left[\sum_{j=0}^{k-1}(\bar\tau_{j+1} - \bar\tau_j)f_{\inf} - \sum_{j=0}^{k-1}\underline{\alpha}\beta\gamma\eta\Delta l(x_j,\tau_j,g_j,d_j) + \sum_{j=0}^{k-1}\beta^2\tfrac{\bar{\kappa}_{\phi,\nu}}{(j+1)^{\nu}}\right] \\
        \leq\ &\bar\tau_{-1}|f_{\inf}| - \underline{\alpha}\beta\gamma\eta\mathbb{E}\left[\sum_{j=0}^{k-1}\Delta l(x_j,\tau_j,g_j,d_j)\right] + \beta^2\sum_{j=0}^
        {k-1}\tfrac{\bar{\kappa}_{\phi,\nu}}{(j+1)^{\nu}}.
    \end{aligned}
    \end{equation*}
    Finally, using the fact that $\sum_{j=0}^{k-1}\beta^2 \tfrac{\bar{\kappa}_{\phi,\nu}}{(j+1)^{\nu}} < \infty$ for any $\nu > 1$, we may complete the proof.
\end{proof}

\bcorollary\label{cor.stoch_raghu_appendix} 
Under the  conditions of Theorem~\ref{thm.main_stochastic}, Algorithm~\ref{alg.adaptiveSQP} yields a sequence of iterates $\{(x_k,y_k)\}$ for which
\begin{equation*}
    \lim_{k\to\infty} \mathbb{E}\left[\|d_k\|_2^2 \right] = 0, \quad \lim_{k\to\infty} \mathbb{E}\left[\|c_k\|_2 \right] = 0,\quad\text{and} \quad \lim_{k\to\infty} \mathbb{E}\left[\|g_k + J_k^T(y_k+\delta_k)\|_2 \right] = 0.
\end{equation*}
\ecorollary
\begin{proof} The proof of this corollary follows the exact same arguments as the proof of Corollary~\ref{cor.stoch} (Section~\ref{sec.adaptive}).
\end{proof}

The final result we show in this section is a complexity result for our proposed algorithm, i.e., the number of iterations and the total number of stochastic gradient evaluations required to achieved an $\epsilon$-accurate solution in expectation. Specifically, to measure the complexity of our algorithm, we consider the minimum number of iterations, $K_{\epsilon}$, and the minimum total number of stochastic gradient evaluations, $W_\epsilon$, required to achieve the following approximate stationary measure 
\bequation\label{eq.complexity_stochastic_old_ap}
   \mathbb{E}[\| g_k + J_k^T(y_k + \delta_k)\|_2] \leq \epsilon_L \quad \text{and} \quad  \mathbb{E}[\|c_k\|_1] \leq \epsilon_c,
\eequation
for $\epsilon_L \in (0,1)$ and $\epsilon_c \in (0,1)$.
\bcorollary \label{cor.iter_complexity_raghu_appendix}
(Corollary~\ref{cor.iter_complexity_raghu})
\asb{Under the conditions of Theorem~\ref{thm.main_stochastic}, Algorithm~\ref{alg.adaptiveSQP} generates an iterate $ (x_k,y_k) $ that satisfies
\eqref{eq.complexity_det} in at most $K_\epsilon =\mathcal{O}\left(\max \left\{ \epsilon_L^{-2}, \epsilon_c^{-1}\right\}\right)$ iterations and $W_\epsilon = \mathcal{O}\left(\left(\max \left\{ \epsilon_L^{-2}, \epsilon_c^{-1}\right\}\right)^{(\nu +1)}\right)$  stochastic gradient evaluations $(\nu \in \mathbb{R}_{>1})$. Moreover, if $\epsilon_L = \epsilon$ and $\epsilon_c = \epsilon^2$, then $K_\epsilon = \mathcal{O}\left(\epsilon^{-2}\right)$ and $W_\epsilon = \mathcal{O}(\epsilon^{-2(\nu +1)})$.}
\ecorollary
\begin{proof}
Using the same logic as the proof of Corollary~\ref{cor.final_results}, we know that if \eqref{eq.complexity_stochastic_old_ap} is violated, then for all $k\in\{0,\ldots,K_{\epsilon}-1\}$

\begin{equation}\label{eq.lb_Deltal_appendix}
    \E[\Delta l(x_k,\tau_k,g_k,d_k)] \geq \kappa_x \min\{ \epsilon_L^2, \epsilon_c\},
\end{equation}
where $\kappa_x = \min\left\{ \omega_1 , \tfrac{\tau_{\min}\omega_1\epsilon_d}{\kappa_H^2}\right\}\in\mathbb{R}_{>0}$. Re-arranging terms in \eqref{eq.main_dec_stoch}, and using \eqref{eq.merit} and \eqref{eq.lb_Deltal_appendix} to do a telescoping, it follows that
\begin{equation*}
\begin{aligned}
&\mathbb{E}[\bar\tau_{-1} (f(x_0) - f_{\inf}) + \|c_0\|_1] \geq \mathbb{E}[\bar\tau_0 (f(x_0) - f_{\inf}) + \|c_0\|_1] \\
\geq\ & \mathbb{E}\left[\bar\tau_0 f(x_0) + \|c_0\|_1 - \bar\tau_{K_{\epsilon}} f(x_{K_{\epsilon}}) - \|c_{K_{\epsilon}}\|_1 + (\bar\tau_{K_{\epsilon}} - \bar\tau_0) f_{\inf}\right]\\
=\ &\mathbb{E}\left[\sum_{k=0}^{K_{\epsilon}-1}\left(\bar\tau_k f(x_k) + \|c_k\|_1 - \bar\tau_{k+1} f(x_{k+1}) - \|c_{k+1}\|_1 + (\bar\tau_{k+1} - \bar\tau_k) f_{\inf}\right)\right]\\
=\ &\mathbb{E}\left[ \sum_{k=0}^{K_{\epsilon}-1}\left(\phi(x_{k},\bar\tau_k) - \phi(x_{k+1},\bar\tau_{k+1}) + (\bar\tau_{k+1} - \bar\tau_k) f_{\inf} \right)\right] \\
\geq \ &\mathbb{E}\left[ \sum_{k=0}^{K_{\epsilon}-1} \left( \underline{\alpha}\beta\gamma\eta \Delta l(x_k,\tau_k,g_k,d_k) - \beta^2\tfrac{\bar{\kappa}_{\phi,\nu}}{(k+1)^{\nu}} \right) \right] \\
\geq \ &\sum_{k=0}^{K_{\epsilon}-1} \left( \underline{\alpha}\beta\gamma\eta \kappa_x\min\{ \epsilon_L^2, \epsilon_c\} - \beta^2\tfrac{\bar{\kappa}_{\phi,\nu}}{(k+1)^{\nu}} \right) \\ \geq \ & \underline{\alpha}\beta\gamma\eta \kappa_x\min\{ \epsilon_L^2, \epsilon_c\}K_{\epsilon} - \beta^2\bar{\kappa}_{\phi,\nu}\sum_{k=0}^{\infty}\tfrac{1}{(k+1)^{\nu}}.
\end{aligned}
\end{equation*}
It further implies that $K_{\epsilon}$ is bounded as
\begin{equation}\label{eq:iter_comp}
    K_{\epsilon} \leq \tfrac{\bar\tau_{-1} (f(x_0) - f_{\inf}) + \|c_0\|_1 + \beta^2\kappa_{\phi,2}\sum_{k=0}^{\infty}\tfrac{1}{(k+1)^{\nu}}}{\underline{\alpha}\beta\gamma\eta \kappa_x\min\{ \epsilon_L^2, \epsilon_c\}} = \mathcal{O}(\min\{ \epsilon_L^2, \epsilon_c\})
\end{equation}
under the condition that $\nu \in \mathbb{R}_{>1}$. Next, we analyze the sampling complexity.
Suppose, we use $|S_k|$ samples to estimate the stochastic gradient $\bar{g}_k$, then, 
\begin{equation*}
    \mathbb{E}_k\left[\|\bar{g}_k - g_k\|_2^2\right] = \tfrac{\bz{\sigma_g^2}}{|S_k|}
\end{equation*}
where $\bz{\sigma_g^2}$ is the population variance. Therefore, the minimum number of samples required to satisfy  \asb{Condition}~\ref{ass.stoch_g_sublin} is
\begin{equation*}
    |S_k| = \tfrac{\bz{\sigma_g^2}(k+1)^{\nu}}{\theta_1 \beta^{\bz{2\sigma}}}.
\end{equation*}
By~\eqref{eq:iter_comp}, it follows that the total number of stochastic gradient evaluations required to satisfy  \eqref{eq.complexity_stochastic_old_ap} can be expressed as
\begin{equation*}
    W_\epsilon = \sum_{k=0}^{\bz{K_\epsilon}-1} |S_k| = \sum_{k=0}^{\bz{K_\epsilon}-1} \tfrac{\bz{\sigma_g^2}(k+1)^{\nu}}{\theta_1 \beta^{\bz{2\sigma}}} = \tfrac{\bz{\sigma_g^2}}{\theta_1 \beta^{\bz{2\sigma}}}\sum_{k=0}^{\bz{K_\epsilon}-1}(k+1)^{\nu}.
\end{equation*}
Using Faulhaber's formula, sum of $\nu$-th power of first $k$ positive integers is a function of polynomial $\nu + 1$. Therefore, there exists some constant $\lambda_k \in (0, \infty)$ such that $\sum_{k=0}^{\bz{K_\epsilon}-1}(k+1)^{\nu} \leq \lambda_k \bz{K_{\epsilon}^{\nu + 1}}$. Therefore,
\begin{equation*}
    W_\epsilon \leq \tfrac{\bz{\sigma_g^2} \lambda_k}{\theta_1 \beta^{\bz{2\sigma}}}K_\epsilon^{\nu + 1}
\end{equation*}
Substituting $K_\epsilon = \mathcal{O} (\epsilon^{-2})$ yields the desired result. 
\end{proof}

%% file: numerical_results.tex
In this section, we provide additional numerical results for constrained logistic problems. We consider all the data sets in Table~\ref{tab:data} from the \cite{chang2011libsvm} collection.


\begin{table}[ht]
\caption{\label{tab:data}Binary classification data set details. For more information see \cite{chang2011libsvm}. }
  \centering
  {\footnotesize
\begin{tabular}{lcc}\toprule
\textbf{data set}    & \textbf{dimension ($\pmb{n}$)} & \textbf{ datapoints ($\pmb{N}$)} \\ \midrule
\texttt{australian}      & $14$                           & $690$                                                                                \\ \hdashline
\texttt{ionosphere}      & $34$                           & $351$                                                                                \\ \hdashline
\texttt{mushrooms}       & $112$                          & $8,124$                                                                              \\ \hdashline
\texttt{sonar}           & $60$                           & $208$                                                                                \\ \hdashline
\texttt{splice}          & $60$                           & $1,000$                                                                              \\ 
\bottomrule                                                                     
\end{tabular}}
\end{table}

\begin{figure}[ht]
    \centering
    \begin{subfigure}[b]{0.32\textwidth}
    \includegraphics[width=\textwidth,clip=true,trim=30 180 50 200]{figs/logistic/australian_feas_it.pdf}
    \caption{Feasibility vs. Iterations}
    \end{subfigure}
    \begin{subfigure}[b]{0.32\textwidth}
    \includegraphics[width=\textwidth,clip=true,trim=30 180 50 200]{figs/logistic/australian_feas_ep.pdf}
    \caption{ Feasibility vs. Epochs} 
    \end{subfigure}
    \begin{subfigure}[b]{0.32\textwidth}
    \includegraphics[width=\textwidth,clip=true,trim=30 180 50 200]{figs/logistic/australian_feas_ls.pdf}
    \caption{ Feasibility vs. LS  Iters} 
    \end{subfigure}
  
    \begin{subfigure}[b]{0.32\textwidth}
    \includegraphics[width=\textwidth,clip=true,trim=30 180 50 200]{figs/logistic/australian_stat_it.pdf}
    \caption{Stationarity vs. Iterations}
    \end{subfigure}
    \begin{subfigure}[b]{0.32\textwidth}
    \includegraphics[width=\textwidth,clip=true,trim=30 180 50 200]{figs/logistic/australian_stat_ep.pdf}
    \caption{ Stationarity vs. Epochs}
    \end{subfigure}
    \begin{subfigure}[b]{0.32\textwidth}
    \includegraphics[width=\textwidth,clip=true,trim=30 180 50 200]{figs/logistic/australian_stat_ls.pdf}
    \caption{ Stationarity vs. LS  Iters} 
    \end{subfigure}

    \begin{subfigure}[b]{0.32\textwidth}
    \includegraphics[width=\textwidth,clip=true,trim=30 180 50 200]{figs/logistic/australian_alpha_it.pdf}
    \caption{ Step Size vs. Iterations} 
    \end{subfigure}
    \begin{subfigure}[b]{0.32\textwidth}
    \includegraphics[width=\textwidth,clip=true,trim=30 180 50 200]{figs/logistic/australian_bs_it.pdf}
    \caption{ Batch Size vs. Iterations} 
    \end{subfigure}
    \caption{\texttt{australian}: First \& Second Row: Feasibility \& stationarity errors versus iterations/epochs/linear system iterations for exact and inexact variants of Algorithm~\ref{alg.adaptiveSQP_practical} on \eqref{eq.logistic}. Last Row: Step sizes and Batch sizes versus iterations. }
\end{figure}

\begin{figure}[ht]
    \centering
    \begin{subfigure}[b]{0.32\textwidth}
    \includegraphics[width=\textwidth,clip=true,trim=30 180 50 200]{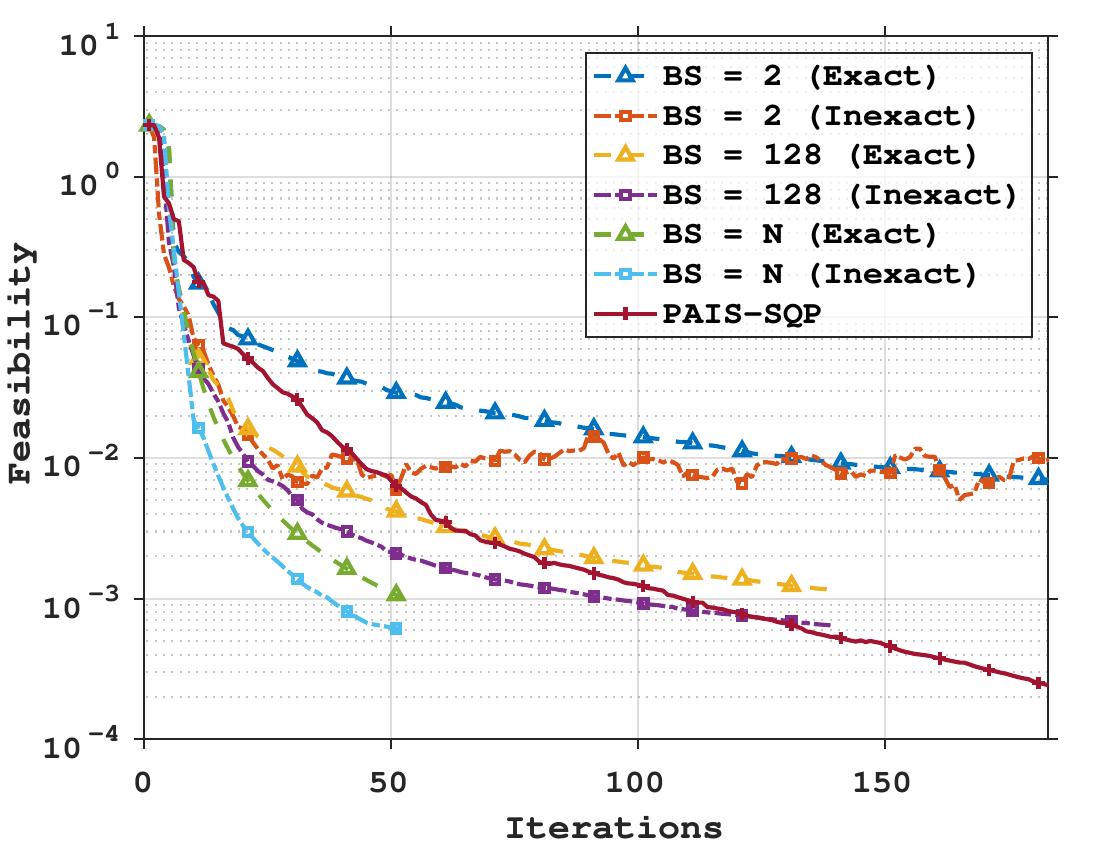}
    \caption{Feasibility vs. Iterations}
    \end{subfigure}
    \begin{subfigure}[b]{0.32\textwidth}
    \includegraphics[width=\textwidth,clip=true,trim=30 180 50 200]{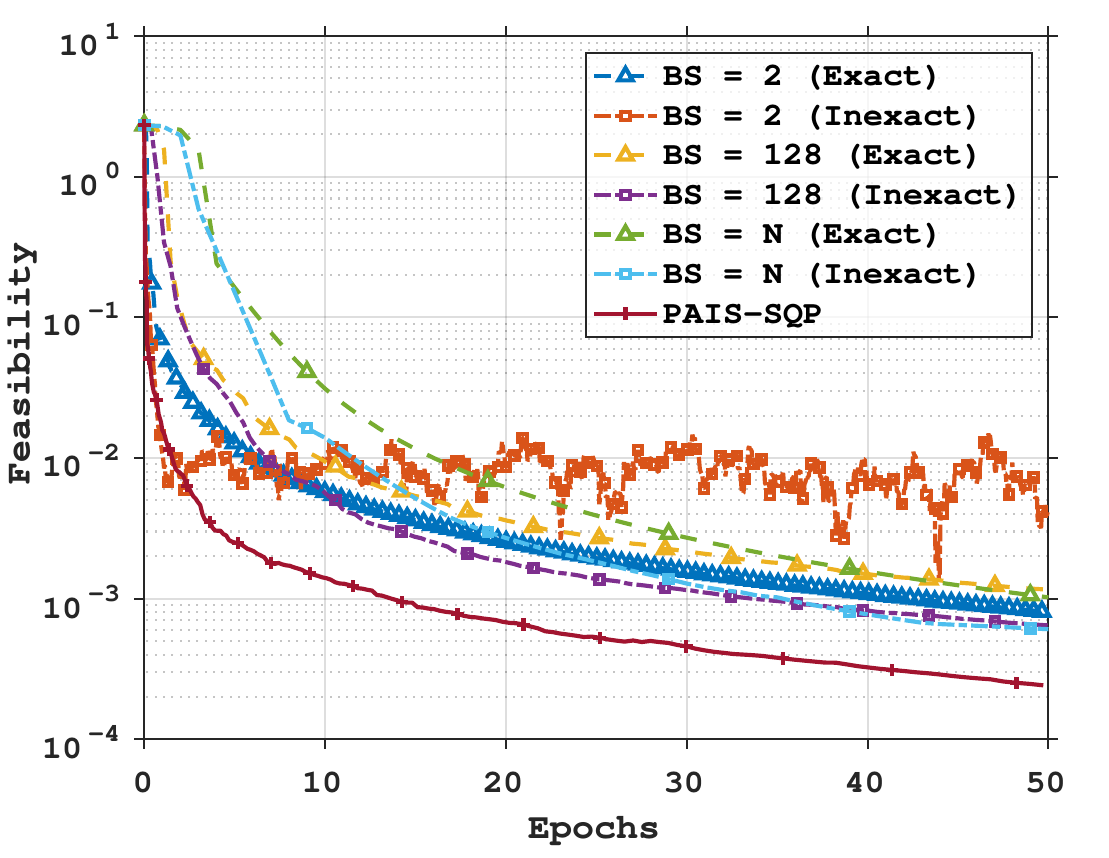}
    \caption{ Feasibility vs. Epochs} 
    \end{subfigure}
    \begin{subfigure}[b]{0.32\textwidth}
    \includegraphics[width=\textwidth,clip=true,trim=30 180 50 200]{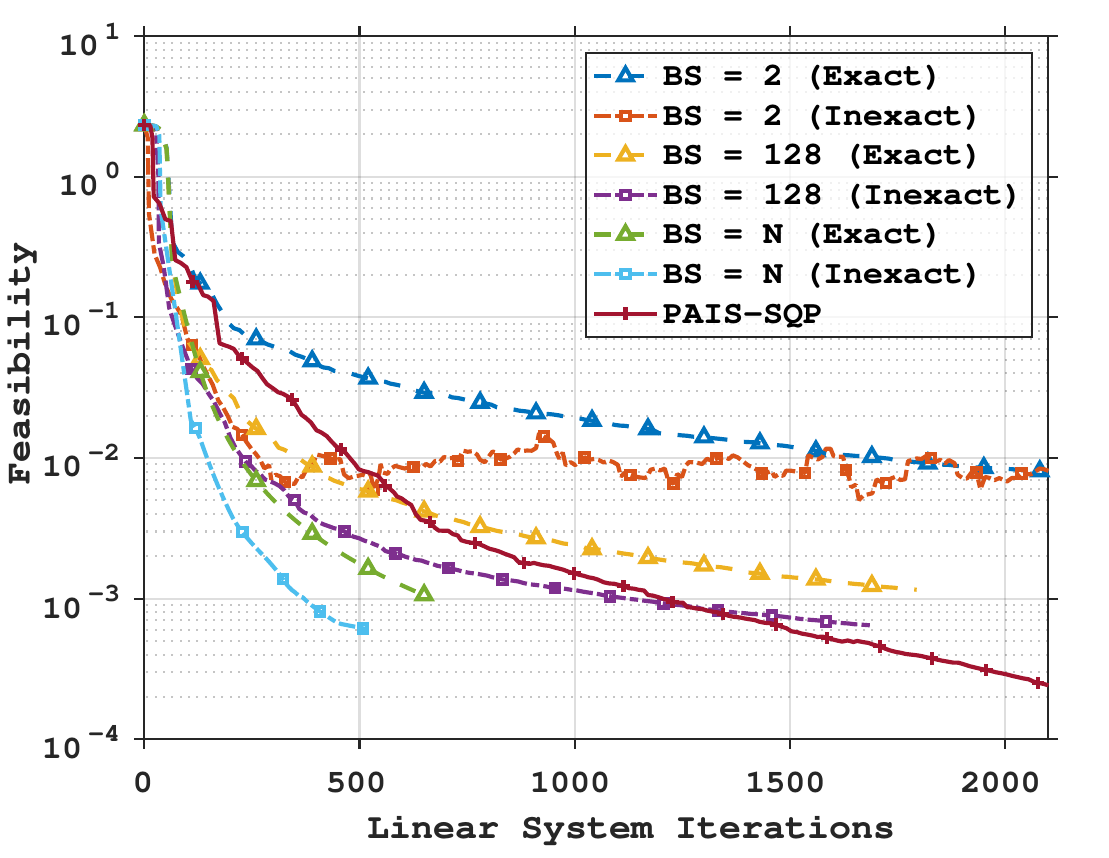}
    \caption{ Feasibility vs. LS  Iters} 
    \end{subfigure}
  
    \begin{subfigure}[b]{0.32\textwidth}
    \includegraphics[width=\textwidth,clip=true,trim=30 180 50 200]{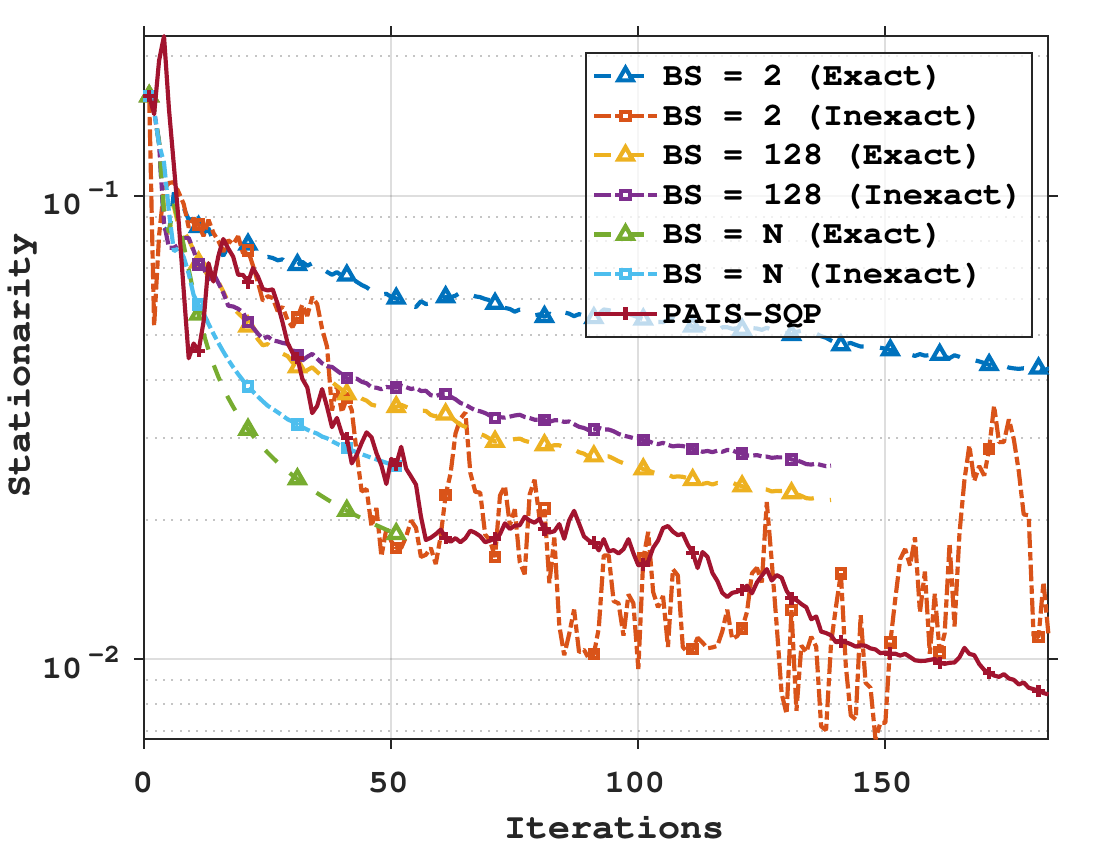}
    \caption{Stationarity vs. Iterations}
    \end{subfigure}
    \begin{subfigure}[b]{0.32\textwidth}
    \includegraphics[width=\textwidth,clip=true,trim=30 180 50 200]{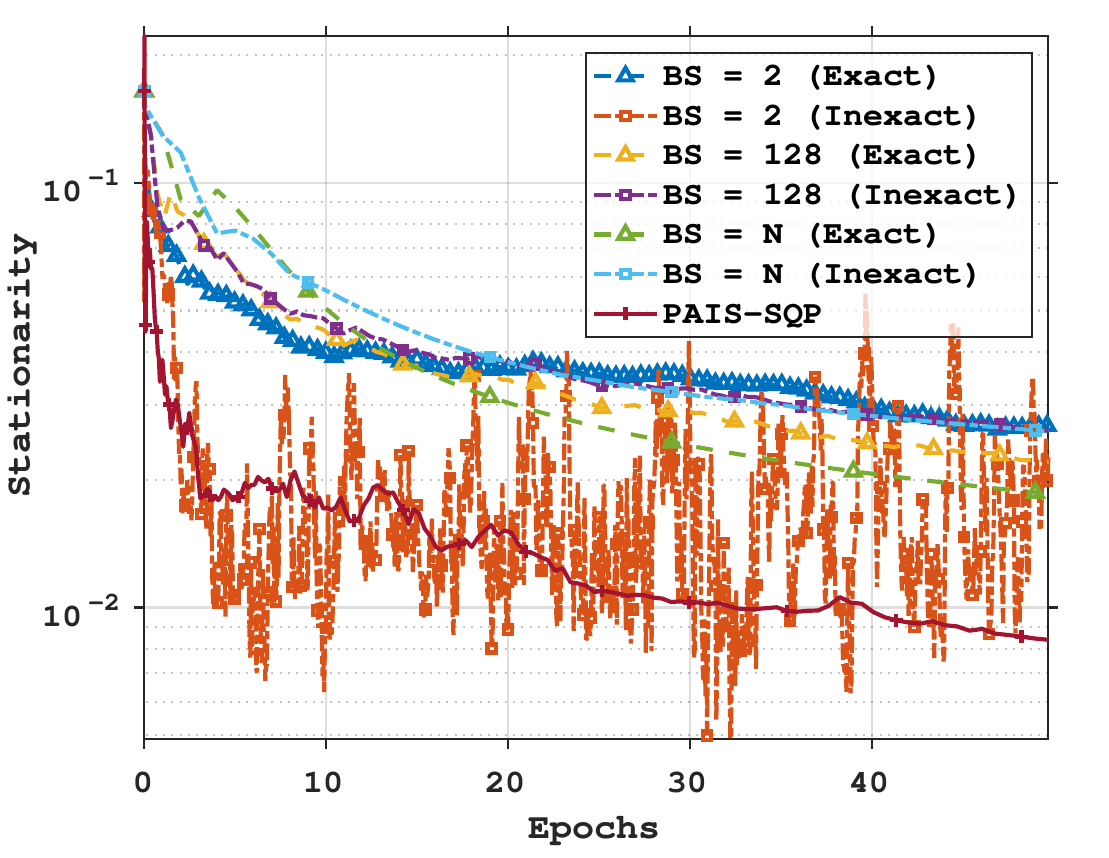}
    \caption{ Stationarity vs. Epochs}
    \end{subfigure}
    \begin{subfigure}[b]{0.32\textwidth}
    \includegraphics[width=\textwidth,clip=true,trim=30 180 50 200]{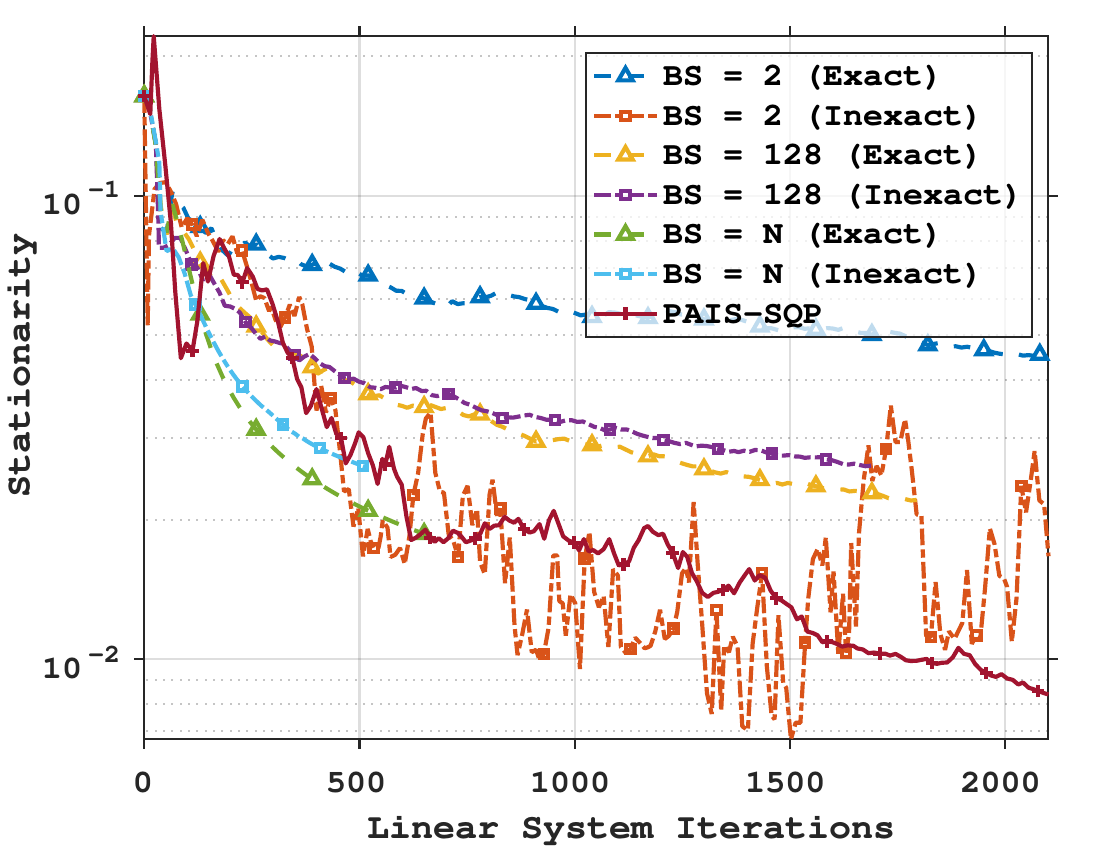}
    \caption{ Stationarity vs. LS  Iters} 
    \end{subfigure}
  
    \begin{subfigure}[b]{0.32\textwidth}
    \includegraphics[width=\textwidth,clip=true,trim=30 180 50 200]{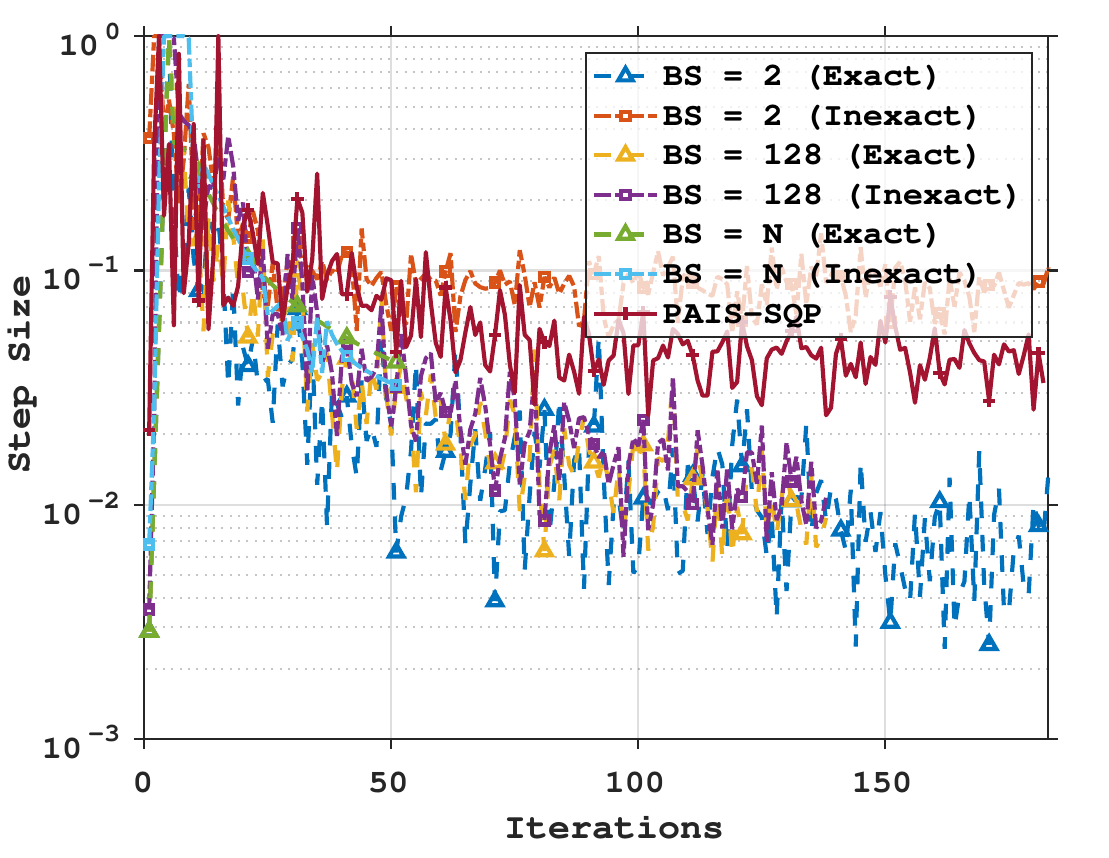}
    \caption{ Step Size vs. Iterations} 
    \end{subfigure}
    \begin{subfigure}[b]{0.32\textwidth}
    \includegraphics[width=\textwidth,clip=true,trim=30 180 50 200]{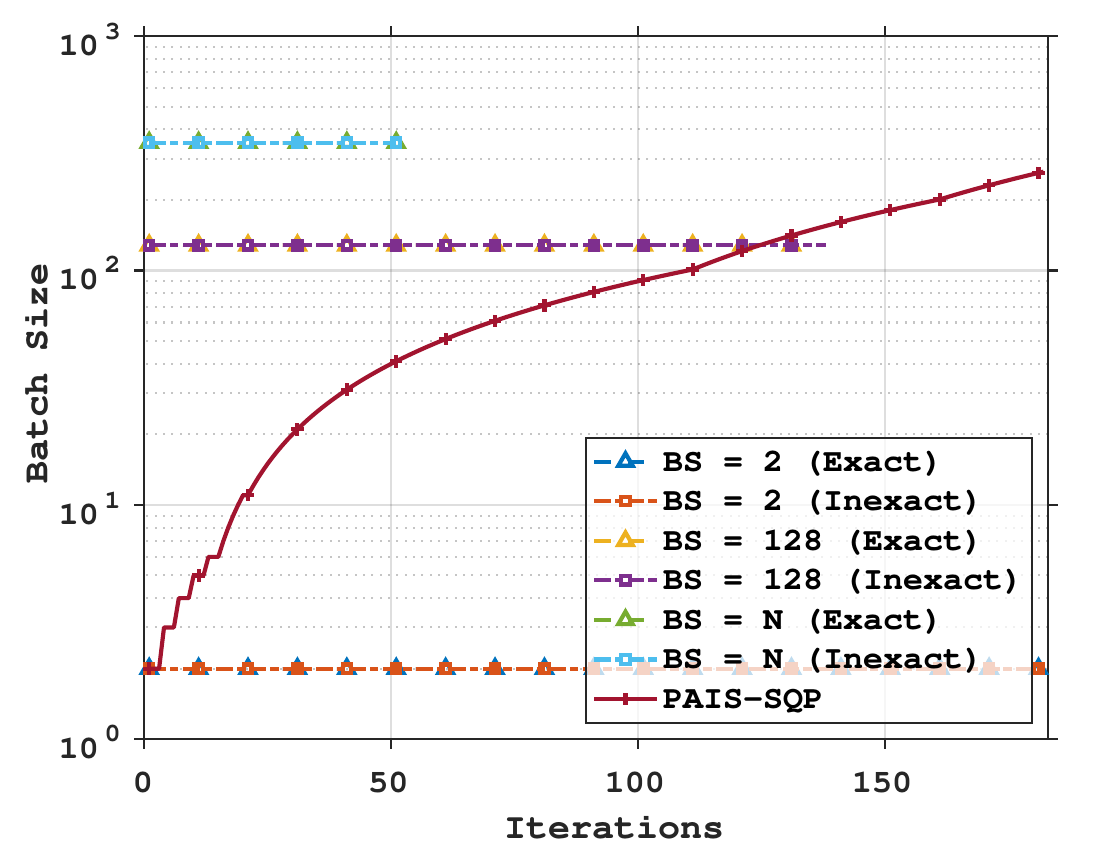}
    \caption{ Batch Size vs. Iterations} 
    \end{subfigure}
    \caption{\texttt{ionosphere}: First \& Second Row: Feasibility \& stationarity errors versus iterations/epochs/linear system iterations for exact and inexact variants of Algorithm~\ref{alg.adaptiveSQP_practical} on \eqref{eq.logistic}. Last Row: Step sizes and Batch sizes versus iterations. }
\end{figure}

\begin{figure}[ht]
    \centering
    \begin{subfigure}[b]{0.32\textwidth}
    \includegraphics[width=\textwidth,clip=true,trim=30 180 50 200]{figs/logistic/mushroom_feas_it.pdf}
    \caption{Feasibility vs. Iterations}
    \end{subfigure}
    \begin{subfigure}[b]{0.32\textwidth}
    \includegraphics[width=\textwidth,clip=true,trim=30 180 50 200]{figs/logistic/mushroom_feas_ep.pdf}
    \caption{ Feasibility vs. Epochs} 
    \end{subfigure}
    \begin{subfigure}[b]{0.32\textwidth}
    \includegraphics[width=\textwidth,clip=true,trim=30 180 50 200]{figs/logistic/mushroom_feas_ls.pdf}
    \caption{ Feasibility vs. LS  Iters} 
    \end{subfigure}
  
    \begin{subfigure}[b]{0.32\textwidth}
    \includegraphics[width=\textwidth,clip=true,trim=30 180 50 200]{figs/logistic/mushroom_stat_it.pdf}
    \caption{Stationarity vs. Iterations}
    \end{subfigure}
    \begin{subfigure}[b]{0.32\textwidth}
    \includegraphics[width=\textwidth,clip=true,trim=30 180 50 200]{figs/logistic/mushroom_stat_ep.pdf}
    \caption{ Stationarity vs. Epochs}
    \end{subfigure}
    \begin{subfigure}[b]{0.32\textwidth}
    \includegraphics[width=\textwidth,clip=true,trim=30 180 50 200]{figs/logistic/mushroom_stat_ls.pdf}
    \caption{ Stationarity vs. LS  Iters} 
    \end{subfigure}
  
    \begin{subfigure}[b]{0.32\textwidth}
    \includegraphics[width=\textwidth,clip=true,trim=30 180 50 200]{figs/logistic/mushroom_alpha_it.pdf}
    \caption{ Step Size vs. Iterations} 
    \end{subfigure}
    \begin{subfigure}[b]{0.32\textwidth}
    \includegraphics[width=\textwidth,clip=true,trim=30 180 50 200]{figs/logistic/mushroom_bs_it.pdf}
    \caption{ Batch Size vs. Iterations} 
    \end{subfigure}
    \caption{\texttt{mushroom}: First \& Second Row: Feasibility \& stationarity errors versus iterations/epochs/linear system iterations for exact and inexact variants of Algorithm~\ref{alg.adaptiveSQP_practical} on \eqref{eq.logistic}. Last Row: Step sizes and Batch sizes versus iterations. }
\end{figure}

\begin{figure}[ht]
    \centering
    \begin{subfigure}[b]{0.32\textwidth}
    \includegraphics[width=\textwidth,clip=true,trim=30 180 50 200]{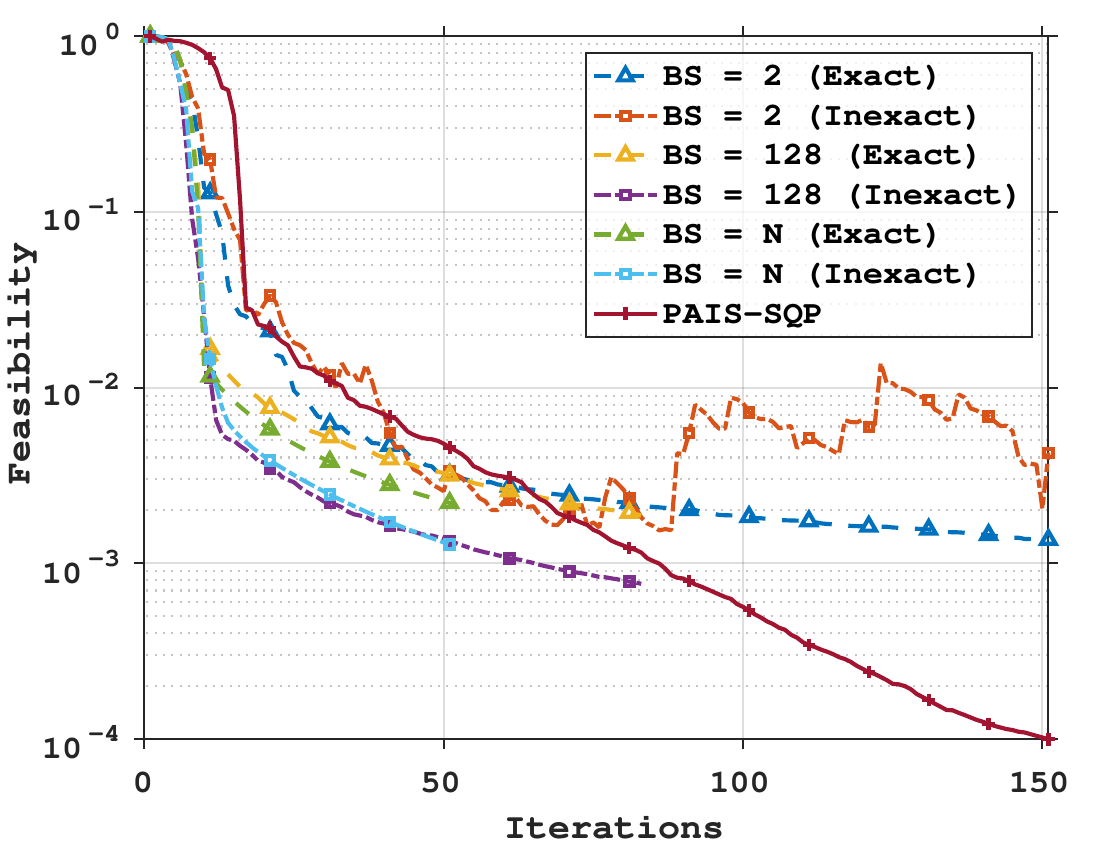}
    \caption{Feasibility vs. Iterations}
    \end{subfigure}
    \begin{subfigure}[b]{0.32\textwidth}
    \includegraphics[width=\textwidth,clip=true,trim=30 180 50 200]{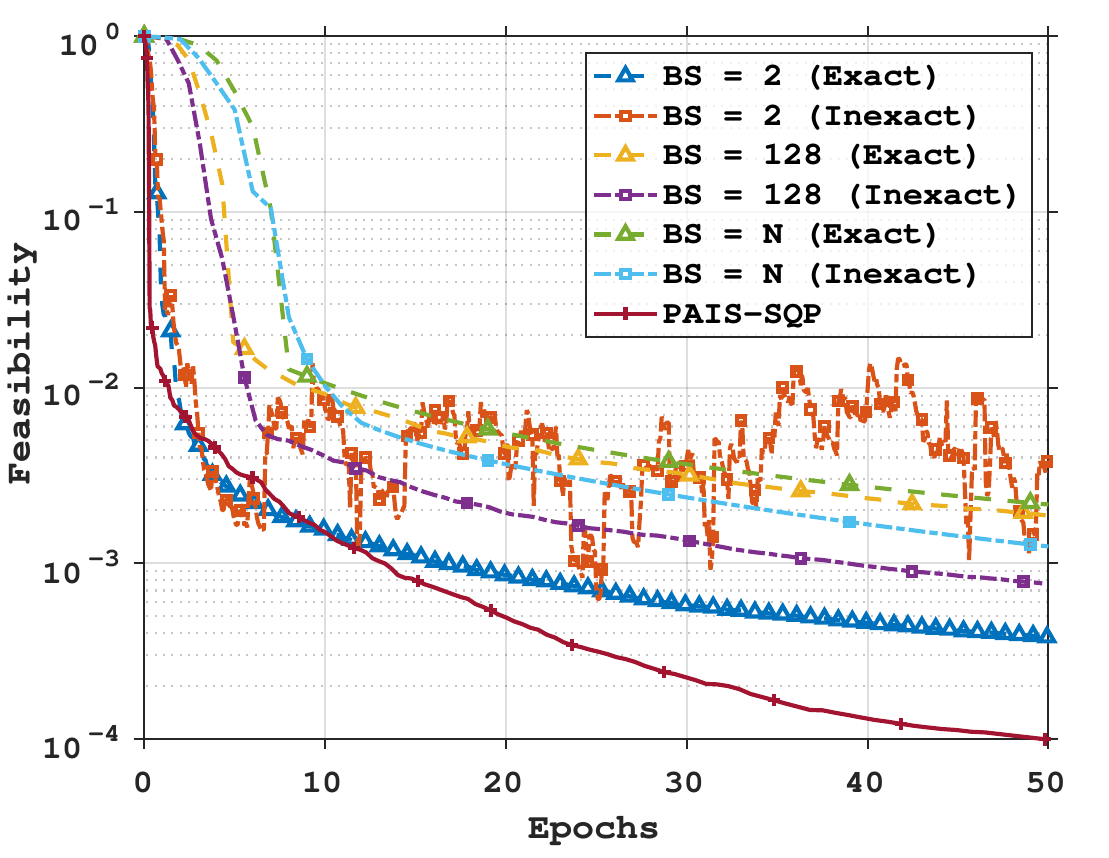}
    \caption{ Feasibility vs. Epochs} 
    \end{subfigure}
    \begin{subfigure}[b]{0.32\textwidth}
    \includegraphics[width=\textwidth,clip=true,trim=30 180 50 200]{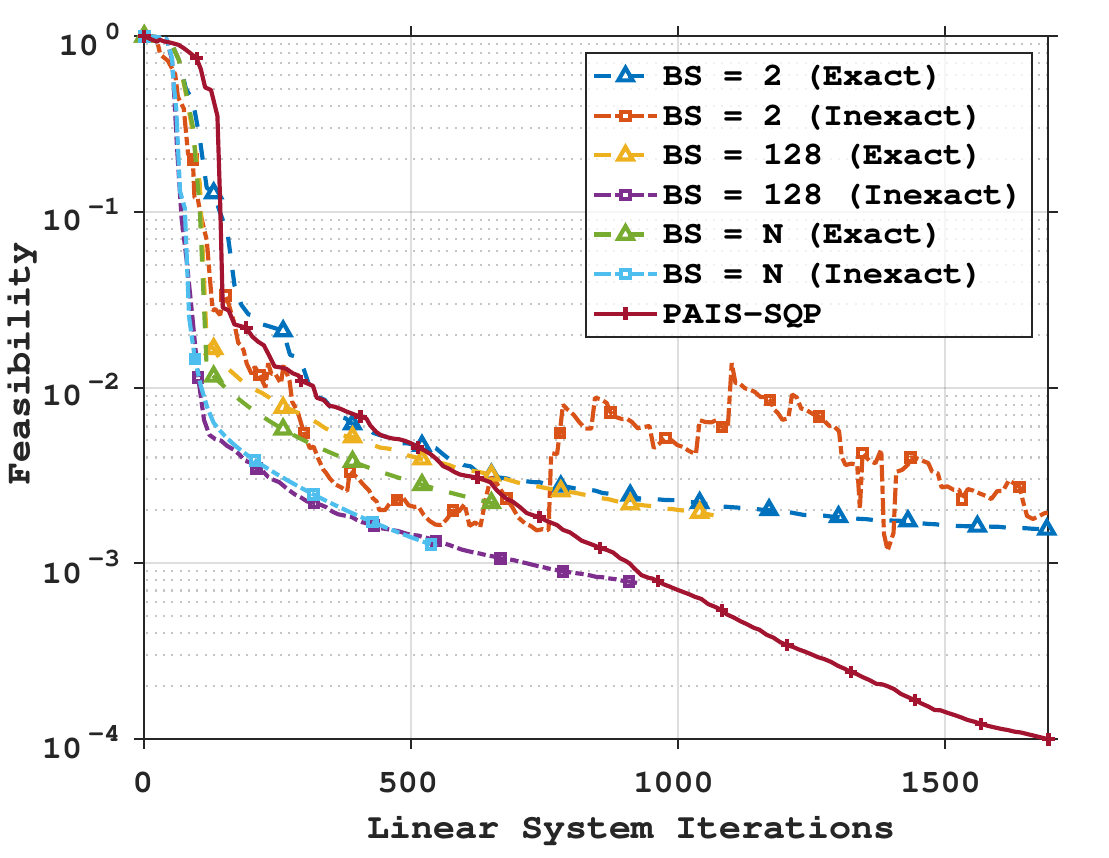}
    \caption{ Feasibility vs. LS  Iters} 
    \end{subfigure}
  
    \begin{subfigure}[b]{0.32\textwidth}
    \includegraphics[width=\textwidth,clip=true,trim=30 180 50 200]{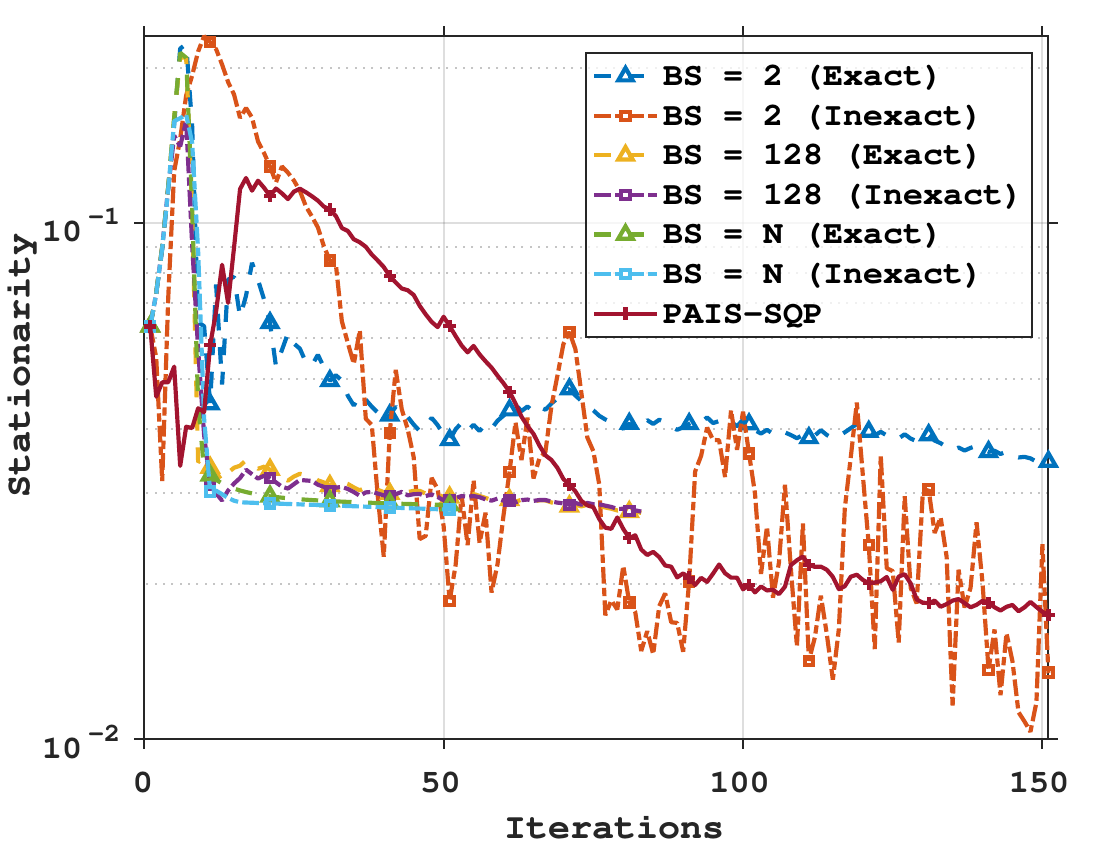}
    \caption{Stationarity vs. Iterations}
    \end{subfigure}
    \begin{subfigure}[b]{0.32\textwidth}
    \includegraphics[width=\textwidth,clip=true,trim=30 180 50 200]{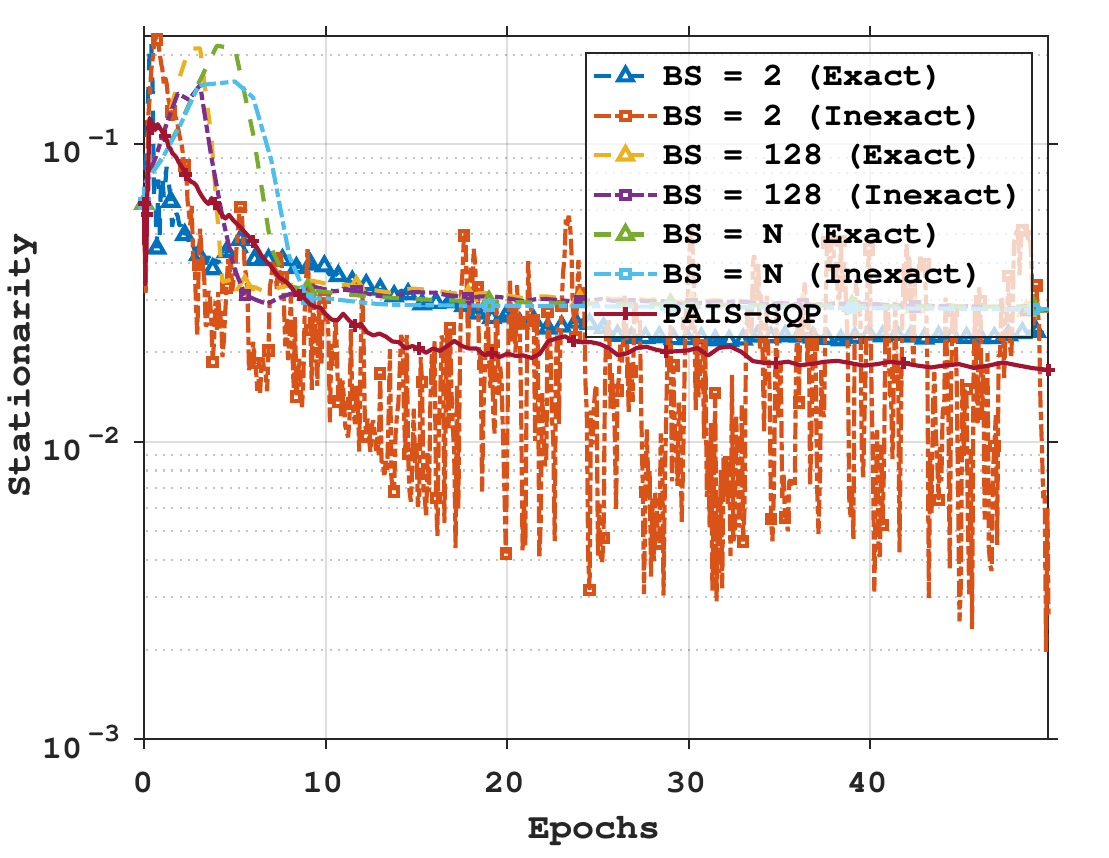}
    \caption{ Stationarity vs. Epochs}
    \end{subfigure}
    \begin{subfigure}[b]{0.32\textwidth}
    \includegraphics[width=\textwidth,clip=true,trim=30 180 50 200]{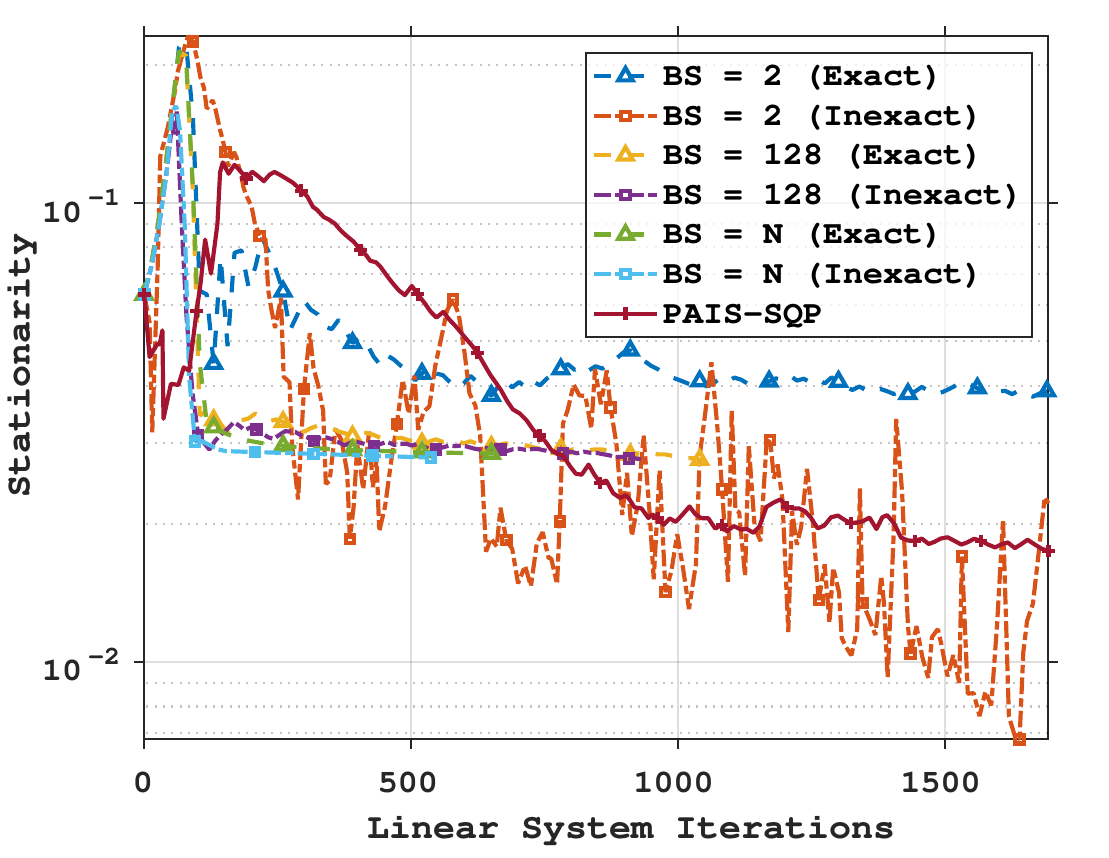}
    \caption{ Stationarity vs. LS  Iters} 
    \end{subfigure}
  
    \begin{subfigure}[b]{0.32\textwidth}
    \includegraphics[width=\textwidth,clip=true,trim=30 180 50 200]{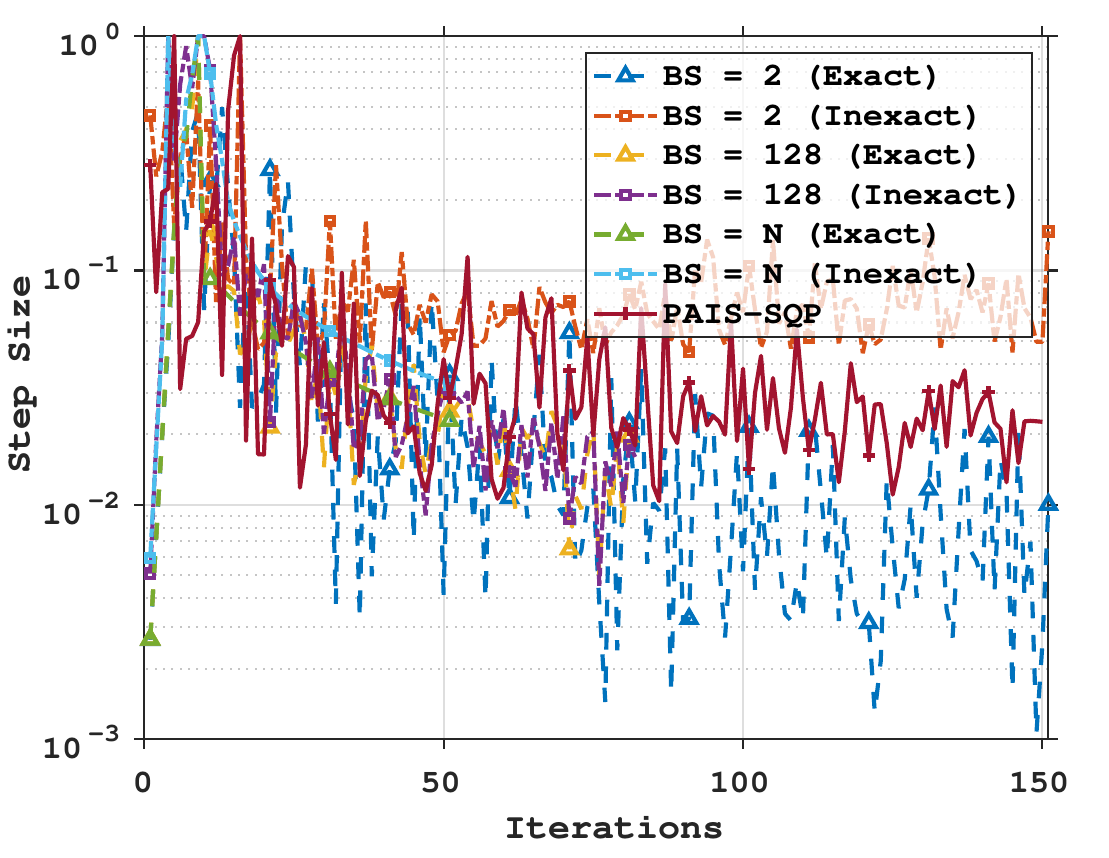}
    \caption{ Step Size vs. Iterations} 
    \end{subfigure}
    \begin{subfigure}[b]{0.32\textwidth}
    \includegraphics[width=\textwidth,clip=true,trim=30 180 50 200]{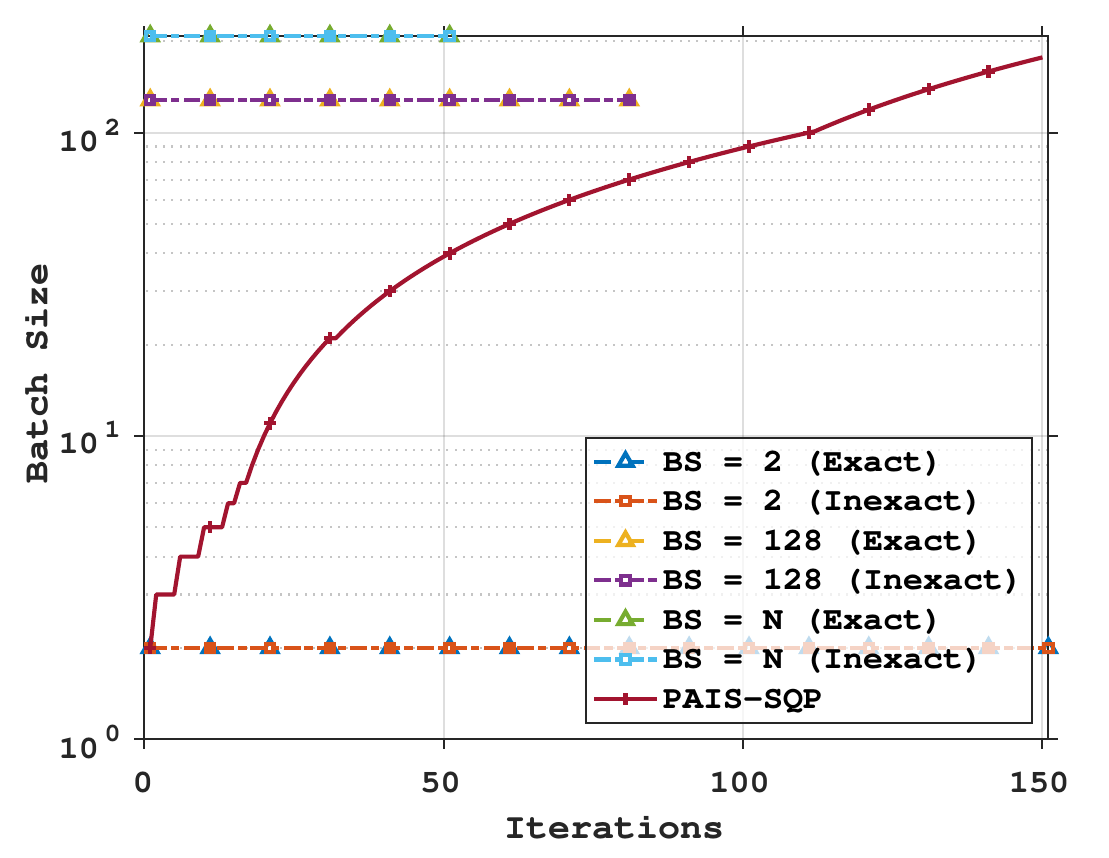}
    \caption{ Batch Size vs. Iterations} 
    \end{subfigure}
    \caption{\texttt{sonar}: First \& Second Row: Feasibility \& stationarity errors versus iterations/epochs/linear system iterations for exact and inexact variants of Algorithm~\ref{alg.adaptiveSQP_practical} on \eqref{eq.logistic}. Last Row: Step sizes and Batch sizes versus iterations. }
\end{figure}

\begin{figure}[ht]
    \centering
    \begin{subfigure}[b]{0.32\textwidth}
    \includegraphics[width=\textwidth,clip=true,trim=30 180 50 200]{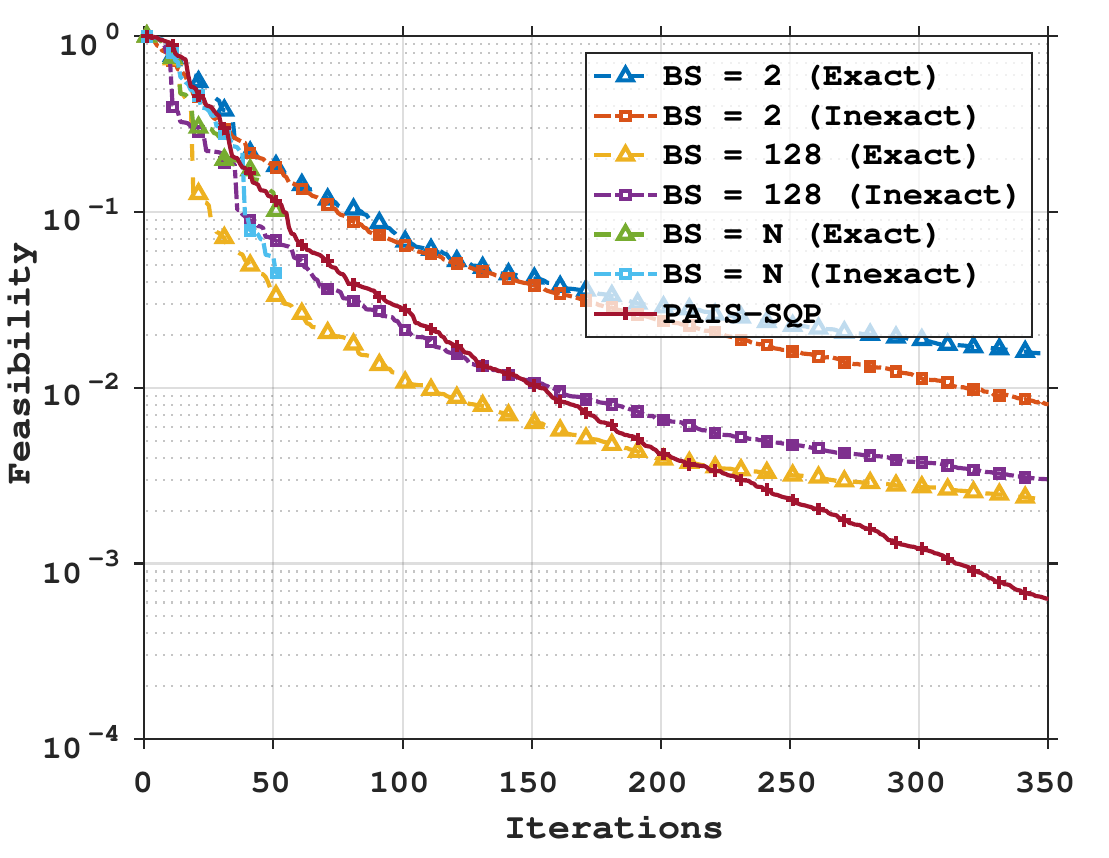}
    \caption{Feasibility vs. Iterations}
    \end{subfigure}
    \begin{subfigure}[b]{0.32\textwidth}
    \includegraphics[width=\textwidth,clip=true,trim=30 180 50 200]{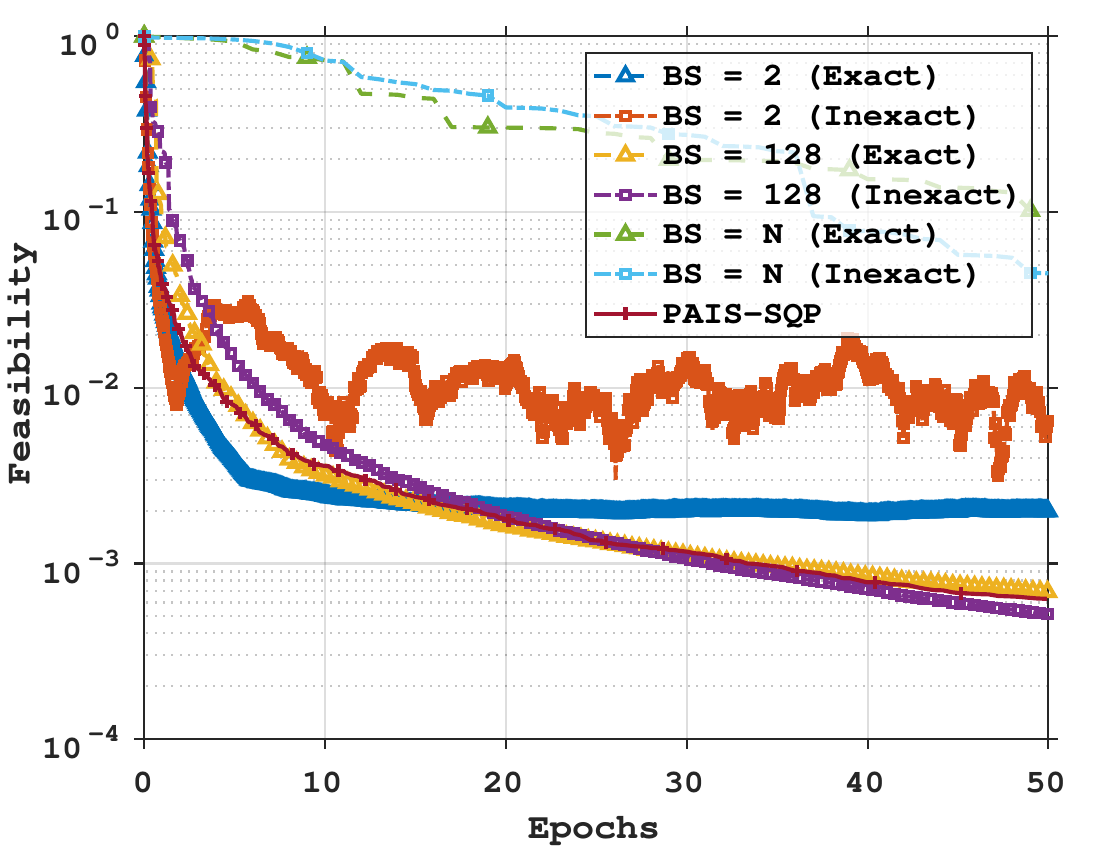}
    \caption{ Feasibility vs. Epochs} 
    \end{subfigure}
    \begin{subfigure}[b]{0.32\textwidth}
    \includegraphics[width=\textwidth,clip=true,trim=30 180 50 200]{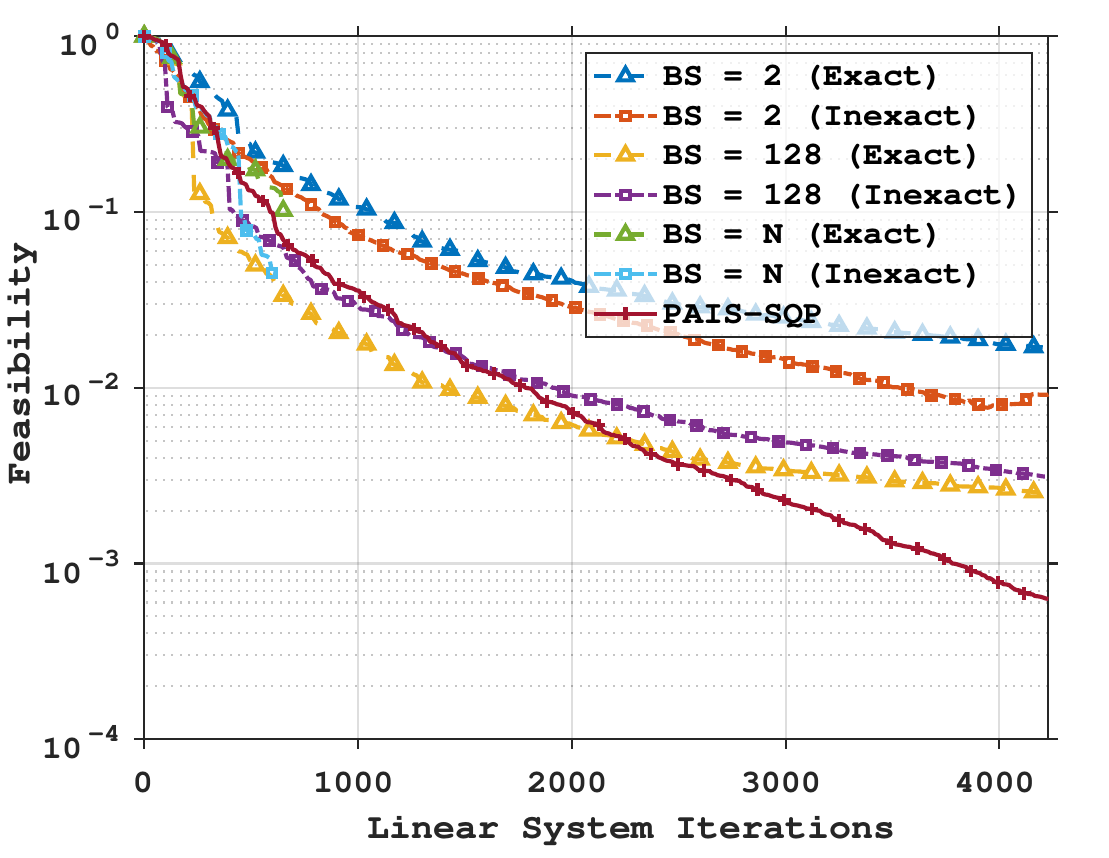}
    \caption{ Feasibility vs. LS  Iters} 
    \end{subfigure}
  
    \begin{subfigure}[b]{0.32\textwidth}
    \includegraphics[width=\textwidth,clip=true,trim=30 180 50 200]{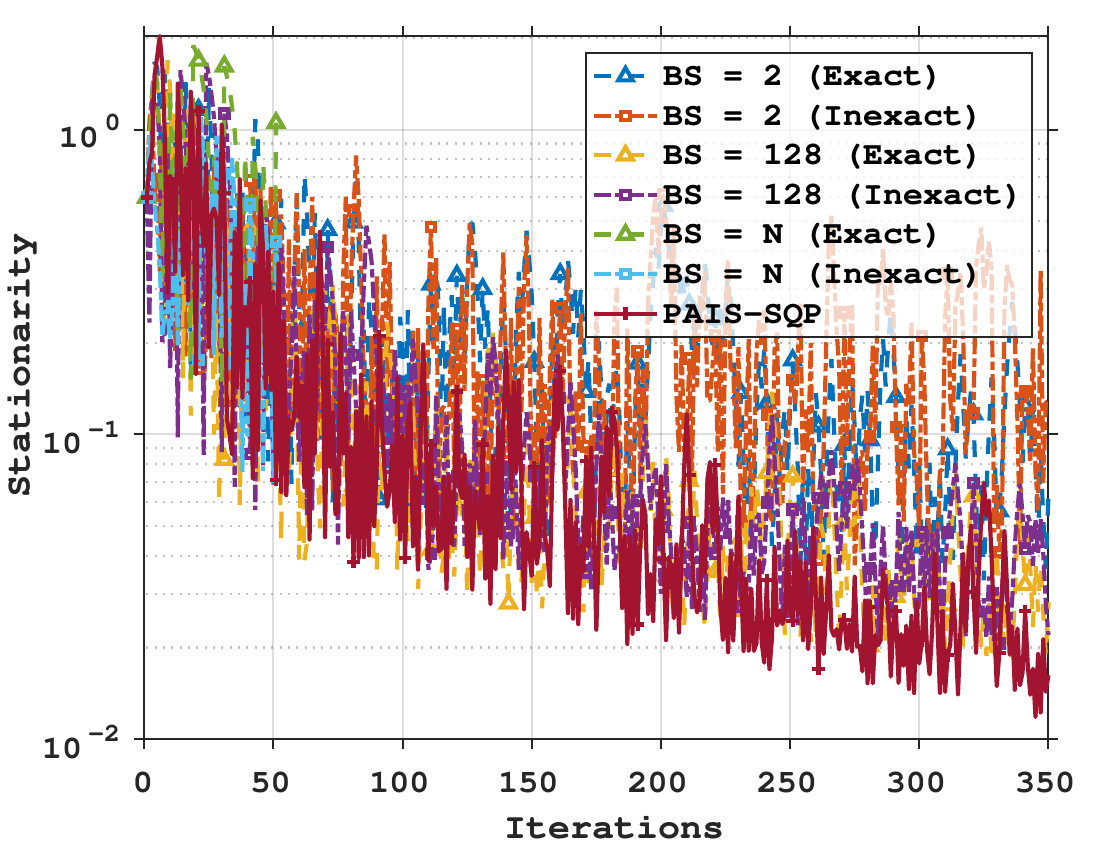}
    \caption{Stationarity vs. Iterations}
    \end{subfigure}
    \begin{subfigure}[b]{0.32\textwidth}
    \includegraphics[width=\textwidth,clip=true,trim=30 180 50 200]{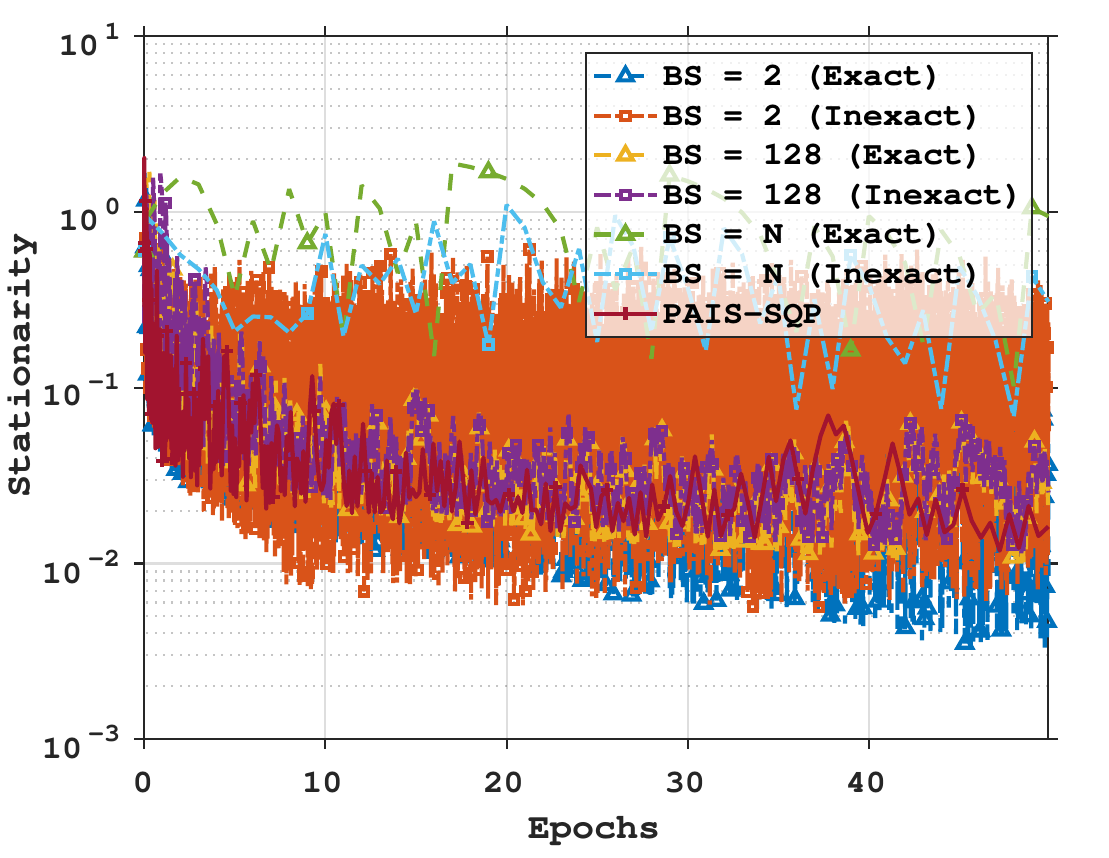}
    \caption{ Stationarity vs. Epochs}
    \end{subfigure}
    \begin{subfigure}[b]{0.32\textwidth}
    \includegraphics[width=\textwidth,clip=true,trim=30 180 50 200]{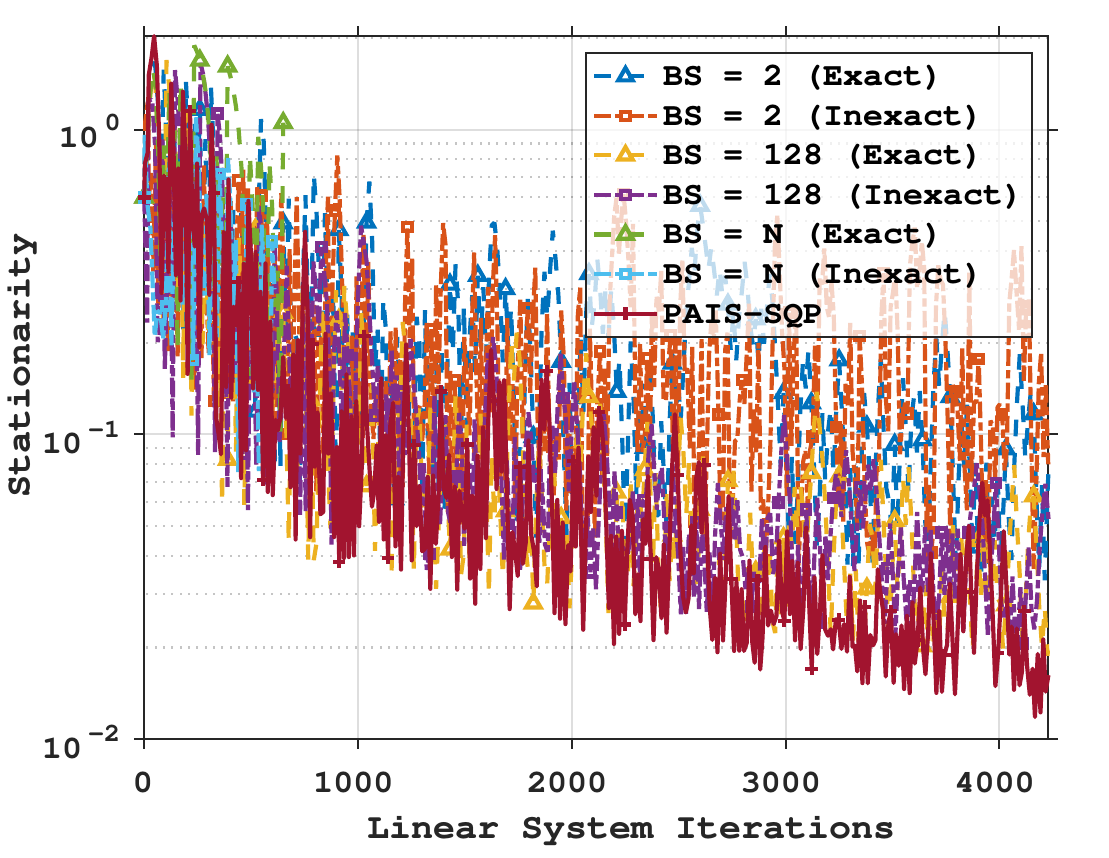}
    \caption{ Stationarity vs. LS  Iters} 
    \end{subfigure}
  
    \begin{subfigure}[b]{0.32\textwidth}
    \includegraphics[width=\textwidth,clip=true,trim=30 180 50 200]{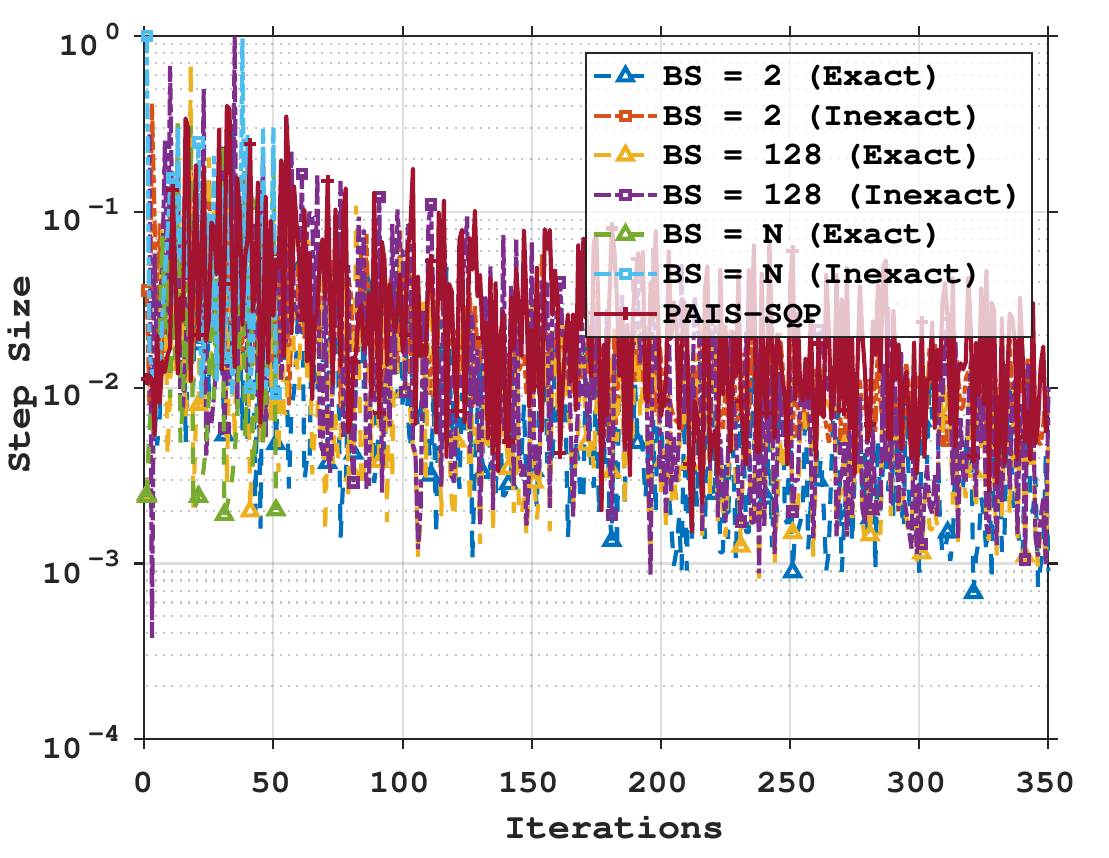}
    \caption{ Step Size vs. Iterations} 
    \end{subfigure}
    \begin{subfigure}[b]{0.32\textwidth}
    \includegraphics[width=\textwidth,clip=true,trim=30 180 50 200]{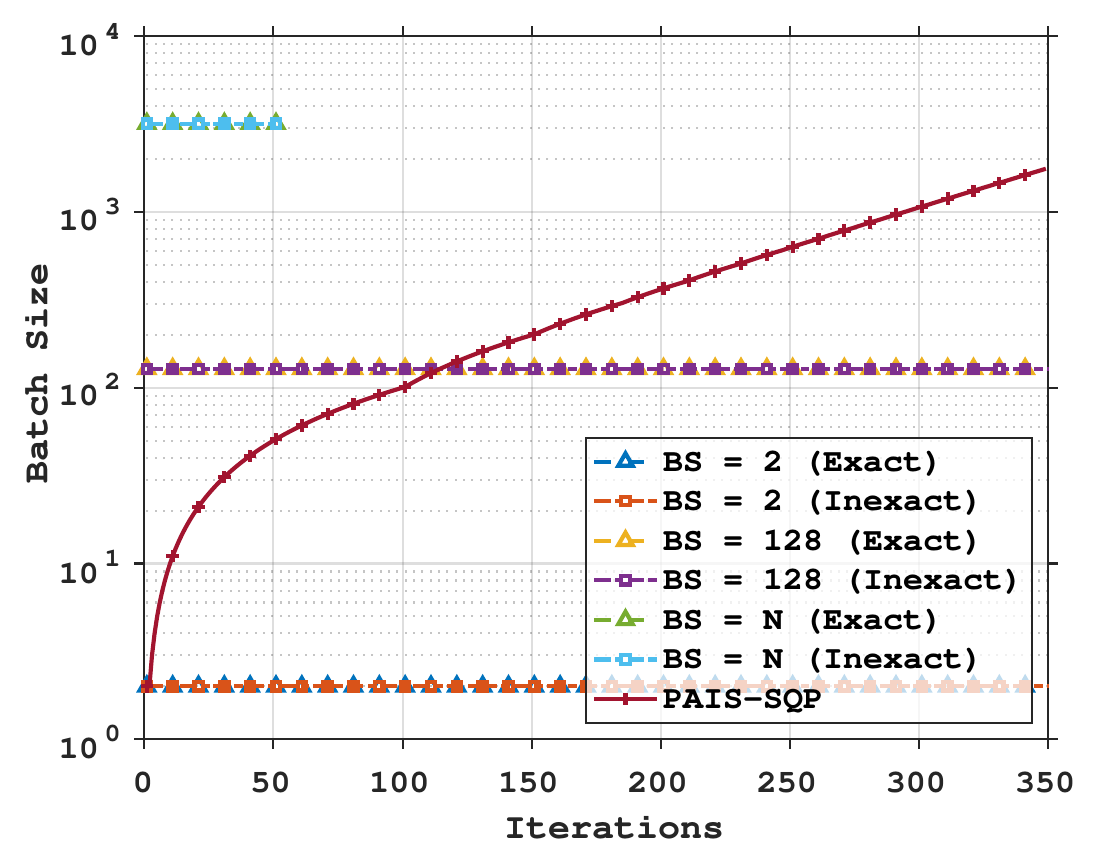}
    \caption{ Batch Size vs. Iterations} 
    \end{subfigure}
    \caption{\texttt{splice}: First \& Second Row: Feasibility \& stationarity errors versus iterations/epochs/linear system iterations for exact and inexact variants of Algorithm~\ref{alg.adaptiveSQP_practical} on \eqref{eq.logistic}. Last Row: Step sizes and Batch sizes versus iterations. }
\end{figure}

%% file: adaptiveSQP_techreport.bbl
\begin{thebibliography}{10}

\bibitem{AchiHeldTamaAbbe17}
Joshua Achiam, David Held, Aviv Tamar, and Pieter Abbeel.
\newblock Constrained policy optimization.
\newblock In {\em International Conference on Machine Learning}, pages 22--31.
  PMLR, 2017.

\bibitem{BeisKeitUrbaWohl20}
Florian Beiser, Brendan Keith, Simon Urbainczyk, and Barbara Wohlmuth.
\newblock Adaptive sampling strategies for risk-averse stochastic optimization
  with constraints.
\newblock {\em arXiv preprint arXiv:2012.03844}, 2020.

\bibitem{berahas2021global}
Albert~S Berahas, Liyuan Cao, and Katya Scheinberg.
\newblock Global convergence rate analysis of a generic line search algorithm
  with noise.
\newblock {\em SIAM Journal on Optimization}, 31(2):1489--1518, 2021.

\bibitem{BeraCurtOneiRobi21}
Albert~S Berahas, Frank~E Curtis, Michael~J O'Neill, and Daniel~P Robinson.
\newblock A stochastic sequential quadratic optimization algorithm for
  nonlinear equality constrained optimization with rank-deficient jacobians.
\newblock {\em arXiv preprint arXiv:2106.13015}, 2021.

\bibitem{BeraCurtRobiZhou21}
Albert~S Berahas, Frank~E Curtis, Daniel Robinson, and Baoyu Zhou.
\newblock Sequential quadratic optimization for nonlinear equality constrained
  stochastic optimization.
\newblock {\em SIAM Journal on Optimization}, 31(2):1352--1379, 2021.

\bibitem{berahas2022accelerating}
Albert~S Berahas, Jiahao Shi, Zihong Yi, and Baoyu Zhou.
\newblock {Accelerating Stochastic Sequential Quadratic Programming for
  Equality Constrained Optimization using Predictive Variance Reduction}.
\newblock {\em arXiv preprint arXiv:2204.04161}, 2022.

\bibitem{Bert98}
Dimitri Bertsekas.
\newblock {\em Network optimization: continuous and discrete models}.
\newblock Athena Scientific, 1998.

\bibitem{Bett10}
John~T Betts.
\newblock {\em Practical methods for optimal control and estimation using
  nonlinear programming}.
\newblock SIAM, 2010.

\bibitem{BollByrdNoce18}
Raghu Bollapragada, Richard Byrd, and Jorge Nocedal.
\newblock Adaptive sampling strategies for stochastic optimization.
\newblock {\em SIAM Journal on Optimization}, 28(4):3312--3343, 2018.

\bibitem{BollNoceMudiShiTang18}
Raghu Bollapragada, Jorge Nocedal, Dheevatsa Mudigere, Hao-Jun Shi, and Ping
  Tak~Peter Tang.
\newblock A progressive batching l-bfgs method for machine learning.
\newblock In {\em International Conference on Machine Learning}, pages
  620--629. PMLR, 2018.

\bibitem{bongartz1995cute}
Ingrid Bongartz, Andrew~R Conn, Nick Gould, and Ph~L Toint.
\newblock Cute: Constrained and unconstrained testing environment.
\newblock {\em ACM Transactions on Mathematical Software (TOMS)},
  21(1):123--160, 1995.

\bibitem{ByrdChinNoceWu12}
Richard~H Byrd, Gillian~M Chin, Jorge Nocedal, and Yuchen Wu.
\newblock Sample size selection in optimization methods for machine learning.
\newblock {\em Mathematical programming}, 134(1):127--155, 2012.

\bibitem{ByrdCurtNoce08}
Richard~H Byrd, Frank~E Curtis, and Jorge Nocedal.
\newblock An inexact sqp method for equality constrained optimization.
\newblock {\em SIAM Journal on Optimization}, 19(1):351--369, 2008.

\bibitem{ByrdCurtNoce10}
Richard~H Byrd, Frank~E Curtis, and Jorge Nocedal.
\newblock An inexact newton method for nonconvex equality constrained
  optimization.
\newblock {\em Mathematical programming}, 122(2):273--299, 2010.

\bibitem{carter1991global}
Richard~G Carter.
\newblock On the global convergence of trust region algorithms using inexact
  gradient information.
\newblock {\em SIAM Journal on Numerical Analysis}, 28(1):251--265, 1991.

\bibitem{CartSche18}
Coralia Cartis and Katya Scheinberg.
\newblock Global convergence rate analysis of unconstrained optimization
  methods based on probabilistic models.
\newblock {\em Mathematical Programming}, 169(2):337--375, 2018.

\bibitem{chang2011libsvm}
Chih-Chung Chang and Chih-Jen Lin.
\newblock {LIBSVM}: a library for support vector machines.
\newblock {\em ACM Transactions on Intelligent Systems and Technology (TIST)},
  2(3):1--27, 2011.

\bibitem{ChoiPaigSaun11}
Sou-Cheng~T Choi, Christopher~C Paige, and Michael~A Saunders.
\newblock Minres-qlp: A krylov subspace method for indefinite or singular
  symmetric systems.
\newblock {\em SIAM Journal on Scientific Computing}, 33(4):1810--1836, 2011.

\bibitem{CottGuptPfei16}
Andrew Cotter, Maya Gupta, and Jan Pfeifer.
\newblock A light touch for heavily constrained sgd.
\newblock In {\em Conference on Learning Theory}, pages 729--771. PMLR, 2016.

\bibitem{CurtNoceWach10}
Frank~E Curtis, Jorge Nocedal, and Andreas W{\"a}chter.
\newblock A matrix-free algorithm for equality constrained optimization
  problems with rank-deficient jacobians.
\newblock {\em SIAM Journal on Optimization}, 20(3):1224--1249, 2010.

\bibitem{CurtOneiRobi21}
Frank~E Curtis, Michael~J O'Neill, and Daniel~P Robinson.
\newblock Worst-case complexity of an sqp method for nonlinear equality
  constrained stochastic optimization.
\newblock {\em arXiv preprint arXiv:2112.14799}, 2021.

\bibitem{CurtRobiZhou21}
Frank~E Curtis, Daniel~P Robinson, and Baoyu Zhou.
\newblock Inexact sequential quadratic optimization for minimizing a stochastic
  objective function subject to deterministic nonlinear equality constraints.
\newblock {\em arXiv preprint arXiv:2107.03512}, 2021.

\bibitem{CurtSche20}
Frank~E Curtis and Katya Scheinberg.
\newblock Adaptive stochastic optimization: A framework for analyzing
  stochastic optimization algorithms.
\newblock {\em IEEE Signal Processing Magazine}, 37(5):32--42, 2020.

\bibitem{FrieSchm12}
Michael~P Friedlander and Mark Schmidt.
\newblock Hybrid deterministic-stochastic methods for data fitting.
\newblock {\em SIAM Journal on Scientific Computing}, 34(3):A1380--A1405, 2012.

\bibitem{HashGhosPasu14}
Fatemeh~S Hashemi, Soumyadip Ghosh, and Raghu Pasupathy.
\newblock On adaptive sampling rules for stochastic recursions.
\newblock In {\em Proceedings of the Winter Simulation Conference 2014}, pages
  3959--3970. IEEE, 2014.

\bibitem{JinScheXie21}
Billy Jin, Katya Scheinberg, and Miaolan Xie.
\newblock High probability complexity bounds for line search based on
  stochastic oracles.
\newblock {\em arXiv preprint arXiv:2106.06454}, 2021.

\bibitem{MahdYangJinZhuYi12}
Mehrdad Mahdavi, Tianbao Yang, Rong Jin, Shenghuo Zhu, and Jinfeng Yi.
\newblock Stochastic gradient descent with only one projection.
\newblock {\em Advances in neural information processing systems}, 25:494--502,
  2012.

\bibitem{more2009benchmarking}
Jorge~J Mor{\'e} and Stefan~M Wild.
\newblock Benchmarking derivative-free optimization algorithms.
\newblock {\em SIAM Journal on Optimization}, 20(1):172--191, 2009.

\bibitem{NaAnitKola21}
Sen Na, Mihai Anitescu, and Mladen Kolar.
\newblock An adaptive stochastic sequential quadratic programming with
  differentiable exact augmented lagrangians.
\newblock {\em arXiv preprint arXiv:2102.05320}, 2021.

\bibitem{NandPathSingSing19}
Yatin Nandwani, Abhishek Pathak, Mausam Singla, and Parag Singla.
\newblock A primal dual formulation for deep learning with constraints.
\newblock In {\em Advances in Neural Information Processing Systems}, pages
  12157--12168, 2019.

\bibitem{PaigSaun75}
Christopher~C Paige and Michael~A Saunders.
\newblock Solution of sparse indefinite systems of linear equations.
\newblock {\em SIAM journal on numerical analysis}, 12(4):617--629, 1975.

\bibitem{PasuGlynGhosHash18}
Raghu Pasupathy, Peter Glynn, Soumyadip Ghosh, and Fatemeh~S Hashemi.
\newblock On sampling rates in simulation-based recursions.
\newblock {\em SIAM Journal on Optimization}, 28(1):45--73, 2018.

\bibitem{RaviDinhLokhSing19}
Sathya~N Ravi, Tuan Dinh, Vishnu~Suresh Lokhande, and Vikas Singh.
\newblock Explicitly imposing constraints in deep networks via conditional
  gradients gives improved generalization and faster convergence.
\newblock In {\em Proceedings of the AAAI Conference on Artificial
  Intelligence}, volume~33, pages 4772--4779, 2019.

\bibitem{ReesDollWath10}
Tyrone Rees, H~Sue Dollar, and Andrew~J Wathen.
\newblock Optimal solvers for pde-constrained optimization.
\newblock {\em SIAM Journal on Scientific Computing}, 32(1):271--298, 2010.

\bibitem{trefethen1997numerical}
Lloyd~N Trefethen and David Bau~III.
\newblock {\em Numerical linear algebra}, volume~50.
\newblock Siam, 1997.

\bibitem{XieBollByrdNoce20}
Yuchen Xie, Raghu Bollapragada, Richard Byrd, and Jorge Nocedal.
\newblock Constrained and composite optimization via adaptive sampling methods.
\newblock {\em arXiv preprint arXiv:2012.15411}, 2020.

\bibitem{zhu2019physics}
Yinhao Zhu, Nicholas Zabaras, Phaedon-Stelios Koutsourelakis, and Paris
  Perdikaris.
\newblock Physics-constrained deep learning for high-dimensional surrogate
  modeling and uncertainty quantification without labeled data.
\newblock {\em Journal of Computational Physics}, 394:56--81, 2019.

\end{thebibliography}
